\documentclass[opre,sglanonrev]{informs4}

\OneAndAHalfSpacedXI
\usepackage{booktabs}
\usepackage{tabularx}
\usepackage{array}

\usepackage{graphicx}   
\usepackage{wrapfig}    
\usepackage{subcaption} 
\usepackage{titlesec}  
\usepackage{float}     
\usepackage{framed}
\usepackage{lscape}
\usepackage{makecell}

\usepackage{threeparttable}

\usepackage{amsmath}    
\usepackage{amssymb}    
\usepackage{amsfonts}   
\usepackage{mathtools}  
\usepackage{bm}         
\usepackage{scalerel}   
\usepackage{dsfont}
\usepackage{mathrsfs}

\DeclareMathAlphabet{\msml}{OT1}{pzc}{m}n
\usepackage{bbm}
\usepackage{siunitx}

\usepackage{cite}
\usepackage{hyperref}
\usepackage{url} 
\usepackage{tikz}
\usepackage{pgfplots}
\pgfplotsset{width=10cm,compat=1.9}
\usepackage{makecell} 
\usepackage{threeparttable}

\usetikzlibrary{calc, positioning}
\usepackage{pdflscape}
\usepackage{bbm}
\usepackage{enumitem}
\usepackage{longtable}
\usepackage{adjustbox}
\usepackage{multirow}
\usepackage{multicol}
\usepackage{array}
\usepackage{booktabs}
\usepackage{wrapfig}

\usepackage{scalerel}

\DeclareMathOperator{\diag}{diag}

\allowdisplaybreaks[1]

\let\proof\relax

\let\theoremstyle\relax
\let\proof\relax

\usepackage{amsthm}

\usepackage{capt-of}

\newtheorem{theorem}{Theorem}
\newtheorem{proposition}{Proposition}
\newtheorem{corollary}{Corollary}
\newtheorem{lemma}{Lemma}

\newtheorem{definition}{Definition}

\newtheorem{assumption}{Assumption}

\newtheorem{remark}{Remark}

\usepackage{graphicx}

\usepackage{afterpage}
\graphicspath{{../../}}
\usepackage{algorithm}       
\usepackage{algpseudocode}   
\usepackage{hhline}

\DeclareCaptionSubType*{algorithm}

\makeatletter
\renewcommand\p@subalgorithm{\thealgorithm}
\makeatother

\DeclarePairedDelimiterX\Set[2]{\lbrace}{\rbrace}{ #1 \,\delimsize| \,\mathopen{} #2 }  

\newcommand{\mK}{\mathcal{K}}          
\newcommand{\mL}{\mathcal{L}}          
\newcommand{\mF}{\mathcal{F}}
\newcommand{\mD}{\mathcal{D}}
\newcommand{\mT}{\mathcal{T}}
\newcommand{\mZ}{\mathcal{Z}}
\newcommand{\mS}{\mathcal{S}}
\newcommand{\mQ}{\mathcal Q}

\newcommand{\mN}{\mathcal{N}}
\newcommand{\mG}{\mathcal{G}}
\newcommand{\mH}{\mathcal{H}}
\newcommand{\mM}{\mathcal{M}}
\newcommand{\mX}{\mathcal{X}}

\newcommand{\mW}{\mathcal{W}}
\newcommand{\mP}{\mathcal{P}}
\newcommand{\mB}{\mathcal{B}}

\newcommand{\mfD}{\mathfrak{D}}

\newcommand{\mfB}{\mathfrak{B}}

\newcommand{\mfZ}{\mathfrak{Z}}

\newcommand{\mbX}{\mathbb{X}}
\newcommand{\mbY}{\mathbb{Y}}
\newcommand{\mbR}{\mathbb R}

\newcommand{\mbS}{\mathbb S}
\newcommand{\mbE}{\mathbb{E}}
\newcommand{\mbP}{\mathbb{P}}

\newcommand{\mbN}{\mathbb{N}}

\newcommand{\mbGM}{\mathbb{GM}}
\newcommand{\ind}{\mathds{1}}

\newcommand{\hbbeta}{\hat{\bbeta}}

\newcommand{\hbs}{\hat{\bs}}
\newcommand{\zz}{\check{z}}

\newcommand{\bzz}{\bm{\zz}}
\newcommand{\bxi}{\bm{\xi}}

\newcommand{\bbeta}{\bm{\beta}}

\newcommand{\balpha}{\bm{\alpha}}

\newcommand{\bzero}{\bm{0}}

\newcommand{\bmean}{\bm{m}}

\newcommand{\bc}{\bm{c}}

\newcommand{\bq}{\bm{q}}

\newcommand{\bs}{\bm{s}}
\newcommand{\bt}{\bm{t}}
\newcommand{\bu}{\bm{u}}
\newcommand{\bv}{\bm{v}}
\newcommand{\bw}{\bm{w}}
\newcommand{\bx}{\bm{x}}
\newcommand{\by}{\bm{y}}
\newcommand{\bz}{\bm{z}}

\newcommand{\bA}{\bm{A}}
\newcommand{\bB}{\bm{B}}

\newcommand{\bH}{\bm{H}}
\newcommand{\bI}{\bm{I}}

\newcommand{\bQ}{\bm Q}
\newcommand{\bR}{\bm{R}}

\newcommand{\bU}{\bm{U}}

\newcommand{\bW}{\bm{W}}

\newcommand{\bZ}{\bm{Z}}

\newcommand{\bSigma}{\bm{\Sigma}}
\newcommand{\Tr}{\operatorname{tr}}
\newcommand{\Val}{\operatorname{Val}}

\usepackage{etoolbox}

\makeatletter

\g@addto@macro\normalsize{%
  \setlength{\abovedisplayskip}{2pt}%
  \setlength{\belowdisplayskip}{2pt}%
  \setlength{\abovedisplayshortskip}{2pt}%
  \setlength{\belowdisplayshortskip}{2pt}%
  \setlength{\jot}{2pt}%
}

\g@addto@macro\small{%
  \setlength{\abovedisplayskip}{2pt}%
  \setlength{\belowdisplayskip}{3pt}%
  \setlength{\abovedisplayshortskip}{2pt}%
  \setlength{\belowdisplayshortskip}{2pt}%
  \setlength{\jot}{2pt}%
}

\g@addto@macro\footnotesize{%
  \setlength{\abovedisplayskip}{2pt}%
  \setlength{\belowdisplayskip}{2pt}%
  \setlength{\abovedisplayshortskip}{2pt}%
  \setlength{\belowdisplayshortskip}{2pt}%
  \setlength{\jot}{2pt}%
}

\makeatother

\makeatletter
\patchcmd{\proof}{\trivlist}{%
  \setlength{\topsep}{0pt plus 0pt minus 0pt}%
  \setlength{\partopsep}{1pt}%
  \setlength{\parsep}{1pt}%
  \setlength{\itemsep}{1pt}%
  \trivlist
}{}{}
\makeatother

\AtBeginEnvironment{proof}{%
  \setlength{\abovedisplayskip}{0pt}%
  \setlength{\belowdisplayskip}{0pt}%
  \setlength{\abovedisplayshortskip}{0pt}%
  \setlength{\belowdisplayshortskip}{0pt}%
  \setlength{\jot}{0pt}%
}

\newcommand{\tightthmspacing}{%
  \setlength{\topsep}{1pt plus 1pt minus 1pt}%
  \setlength{\partopsep}{1pt}%
  \setlength{\parsep}{1pt}%
  \setlength{\itemsep}{0pt}%
  \setlength{\abovedisplayskip}{1pt}%
  \setlength{\belowdisplayskip}{1pt}%
  \setlength{\abovedisplayshortskip}{1pt}%
  \setlength{\belowdisplayshortskip}{1pt}%
  \setlength{\jot}{1pt}%
}

\AtBeginEnvironment{theorem}{\tightthmspacing}
\AtBeginEnvironment{proposition}{\tightthmspacing}
\AtBeginEnvironment{corollary}{\tightthmspacing}
\AtBeginEnvironment{lemma}{\tightthmspacing}
\AtBeginEnvironment{observation}{\tightthmspacing}
\AtBeginEnvironment{uproposition}{\tightthmspacing}

\AtBeginEnvironment{definition}{\tightthmspacing}
\AtBeginEnvironment{example}{\tightthmspacing}
\AtBeginEnvironment{instance}{\tightthmspacing}
\AtBeginEnvironment{assumption}{\tightthmspacing}
\AtBeginEnvironment{construction}{\tightthmspacing}
\AtBeginEnvironment{remark}{\tightthmspacing}
\AtBeginEnvironment{proof}{\tightthmspacing}

\BeforeBeginEnvironment{theorem}{\vspace{-1.0pt}}
\AfterEndEnvironment{theorem}{\vspace{-1.0pt}}

\BeforeBeginEnvironment{proposition}{\vspace{-1.5pt}}
\AfterEndEnvironment{proposition}{\vspace{-1.5pt}}

\BeforeBeginEnvironment{corollary}{\vspace{-1.5pt}}
\AfterEndEnvironment{corollary}{\vspace{-1.5pt}}

\BeforeBeginEnvironment{lemma}{\vspace{-1.5pt}}
\AfterEndEnvironment{lemma}{\vspace{-2.5pt}}

\BeforeBeginEnvironment{observation}{\vspace{-1.5pt}}
\AfterEndEnvironment{observation}{\vspace{-1.5pt}}

\BeforeBeginEnvironment{uproposition}{\vspace{-1.5pt}}
\AfterEndEnvironment{uproposition}{\vspace{-1.5pt}}

\BeforeBeginEnvironment{definition}{\vspace{-1.5pt}}
\AfterEndEnvironment{definition}{\vspace{-1.5pt}}

\BeforeBeginEnvironment{example}{\vspace{-1.5pt}}
\AfterEndEnvironment{example}{\vspace{-1.5pt}}

\BeforeBeginEnvironment{instance}{\vspace{-1.5pt}}
\AfterEndEnvironment{instance}{\vspace{-1.5pt}}

\BeforeBeginEnvironment{assumption}{\vspace{-1.5pt}}
\AfterEndEnvironment{assumption}{\vspace{-1.5pt}}

\BeforeBeginEnvironment{construction}{\vspace{-1.5pt}}
\AfterEndEnvironment{construction}{\vspace{-1.5pt}}

\BeforeBeginEnvironment{remark}{\vspace{-1.5pt}}
\AfterEndEnvironment{remark}{\vspace{-1.5pt}}

\BeforeBeginEnvironment{proof}{\vspace{-4pt}}
\AfterEndEnvironment{proof}{\vspace{-3pt}}


\titlespacing*{\section}{1pt}{1.0ex plus 0.6ex minus 0.5ex}{0.75ex}
\titlespacing*{\subsection}{1pt}{0.5ex plus 0.25ex minus 0.25ex}{0.25ex}
\titlespacing*{\subsubsection}{0.5pt}{.15ex plus 0.15ex minus 0.15ex}{0.15ex}

\tikzstyle{startstop} = [rectangle, rounded corners, minimum width=1cm, minimum height=1cm, text centered, draw=black]
\tikzstyle{process} = [rectangle, minimum width=15cm, minimum height=1cm, text centered, draw=black]
\tikzstyle{decision} = [diamond, minimum width=.5cm, minimum height=.5cm, text centered, draw=black]
\tikzstyle{arrow} = [thick,->,>=stealth]
\tikzstyle{rarrow} = [thick,<-,>=stealth]

\usepackage{bbm}      
\usepackage{fancyvrb} 
\usepackage{rotating} 
\newcommand{\ubar}[1]{\underline{#1}}
\allowdisplaybreaks[1]

\newcommand{\srevision}[1]{{\color{black}{#1}}}

\usepackage{rotating}
\usepackage{fancyvrb}

\EquationsNumberedThrough    
\ECRepeatTheorems  
\MANUSCRIPTNO{MOOR-0001-2024.00}

\begin{document}

\RUNAUTHOR{Dey and Mehrotra}

\RUNTITLE{Robust chance-Constrained optimization with Gaussian-mixture-based Wasserstein-2-type ambiguity set}

\TITLE{\centering Robust Chance-Constrained Optimization using a Continuous Parameter Space Wasserstein-2 Ambiguity Set of Gaussian Mixtures}

\ARTICLEAUTHORS{%
\AUTHOR{Shibshankar Dey}

\AFF{Department of Industrial Engineering and Management Sciences, Northwestern University \EMAIL{shibshankardey2025@u.northwestern.edu}}

\AUTHOR{Sanjay Mehrotra}

\AFF{Department of Industrial Engineering and Management Sciences, Northwestern University \EMAIL{mehrotra@northwestern.edu}}
}

\ABSTRACT{%

We study distributionally robust linear chance-constrained optimization models whose coefficients follow a Gaussian mixture model (GMM). We develop a novel computable Wasserstein-2-type metric for GMMs over continuous parameter sets, significantly generalizing a finite-set-based definition, and prove its equivalence to the standard data-space Wasserstein-2 metric between distributions. Using the resulting ambiguity set, we prove strong duality for the inner worst-case chance-constraint problem and derive its semi-infinite reformulation. We then develop an adaptive cutting-surface algorithm that endogenously determines mass-receiving mixture-component locations and the associated Gaussian means and covariances. We prove finite convergence to any prescribed optimality gap using finitely many linear segments to approximate the nonconvexity of Gaussian cumulative distribution functions. A block-alternating local-search procedure identifies new Gaussian distributions added to the current pool.

An electric vehicle charging station energy-allocation model serves as the case study for the proposed continuous-support robustification (CDR). For chance probability targets \(\theta\in\{0.95,0.97,0.99\}\) and radii \(\rho\in\{0.001,0.005,0.01\}\), CDR shows superior out-of-sample satisfaction (OSS) relative to no robustification and finite-support robustification (FDR) of the nominal GMM chance constraint.
While FDR fails to attain the prescribed target for every tested \((\theta, \rho) \) pair, CDR shows consistent \(\rho\)-based OSS improvement under \(\pm5\%\) or \(\pm10\%\) mean-support uncertainty. In particular, at \(\rho=0.001\), all CDR specifications satisfy \(\theta=0.95\), with OSS \(95.04\%\)--\(95.87\%\). At \(\rho=0.01\), CDR attains the \(97\%\) target in all continuous-support settings and reaches \(98.49\%\) and \(99.17\%\) OSS for \(\theta=0.99\) under \(\pm5\%\) and \(\pm10\%\), respectively. CDR model also changes the energy allocation solution structurally, whereas FDR remains close to that of the nominal model.
}
\FUNDING{This research was fully funded by the US National Science Foundation Grant DMS 2229410.}
\maketitle

\section{Introduction}
We study a distributionally robust linear chance-constrained program
with decision vector $\bx\in\mbR^n$:
\vspace{-0.5em}
\begin{align}\label{prob: DR-CCP}
\min_{\bx \in \mX}\;& \bc^\top \bx
\quad \mathrm{s.t.}\quad
\min_{\mbP\in\mfD}\ \mbP\!\left[\bxi^\top \bx \le b\right]\ \ge\ \theta,
\tag{DR-C}
\end{align}
where $\mbP[\cdot]$ denotes the probability of an event, $\theta\in(0,1)$ is the prescribed satisfaction probability of the uncertain requirement, and $\mX$ is assumed to be a compact feasible set specified by deterministic constraints. $\mfD$ denotes the ambiguity set of distributions representing the uncertainty of the random vector $\bxi$. In operational applications, the prescribed satisfaction probability $\theta$ is generally interpreted as a probability-based service-level target. This view is common in inventory and service operations, where service-level constraints are used to control stockout probabilities and fill rates, and in EV charging infrastructure planning, where service quality is measured using charging-request fulfillment or lost demand~\cite{chen2001inventory,rossi2013service,khaksari2021sizing}.

We introduce a novel construction of $\mfD$ by assuming that $\bxi$ follows a Gaussian mixture distribution. Our consideration of the Gaussian mixture-based uncertainty representation and the corresponding ambiguity set is motivated by the fact that Gaussian Mixture Models (GMMs) are dense in the space of probability distributions, meaning a mixture with enough components can approximate any probability distribution arbitrarily well \cite{lindsay1995mixture, li1999mixture}. 
Thus, a GMM acts as an approximator for modeling multimodal and non-Gaussian data. 
Our ambiguity set allows the prescribed service level to be hedged against statistical errors in Gaussian mixture parameter estimation, particularly when limited data are used to estimate the nominal mixture. 
For probability distributions $\mu_1$ and $\mu_2$, the Wasserstein-2 metric \cite{vaserstein1969markov} is defined as:
\begin{align}\label{def: wasser-dist}
W_2(\mu_1,\mu_2)
:=
\left(
\inf_{\substack{Y_1\sim\mu_1,\;Y_2\sim\mu_2}}
\mbE\!\left[\|Y_1-Y_2\|^2\right]
\right)^{1/2}
=
\left(
\inf_{\pi\in\Pi(\mu_1,\mu_2)}
\int_{\mbR^n\times\mbR^n}
\|\by_1-\by_2\|^2\,d\pi(\by_1,\by_2)
\right)^{1/2}.
\end{align}
Here, $\mu_1$ and $\mu_2$ are assumed to be multivariate Gaussian mixture distributions. 
The objective in \eqref{def: wasser-dist} minimizes the optimal transport cost obtained by minimizing the expected ``Euclidean'' distance of the transport cost over all joint probability distributions $\pi\in\Pi(\mu_1, \mu_2)$, called couplings, with marginals $Y_1\sim\mu_1$ and $Y_2\sim\mu_2$ \cite{monge1781memoire,kantorovich1942translocation,kantorovich_rubinstein1958,villani2008optimal}. 
We call the metric in \eqref{def: wasser-dist} the Wasserstein-2 metric in the data space, as it uses the space where samples (data) are observed in the metric definition. 
This metric can be used to describe a set $\{\mu_2 \;|\; W_2(\mu_1,\mu_2) \leq \rho\}$ of distributions called a Wasserstein-2 ball that are within $\rho$ Wasserstein-2 distance away from the nominal distribution $\mu_1$.  

\subsection{Gaussian Mixture Model for Multi-Modal Uncertainty Representation}
 \captionsetup{font=scriptsize}
\begin{wrapfigure}{r}{0.34\textwidth}
  \vspace{-5pt}
  \includegraphics[width=0.34\textwidth]{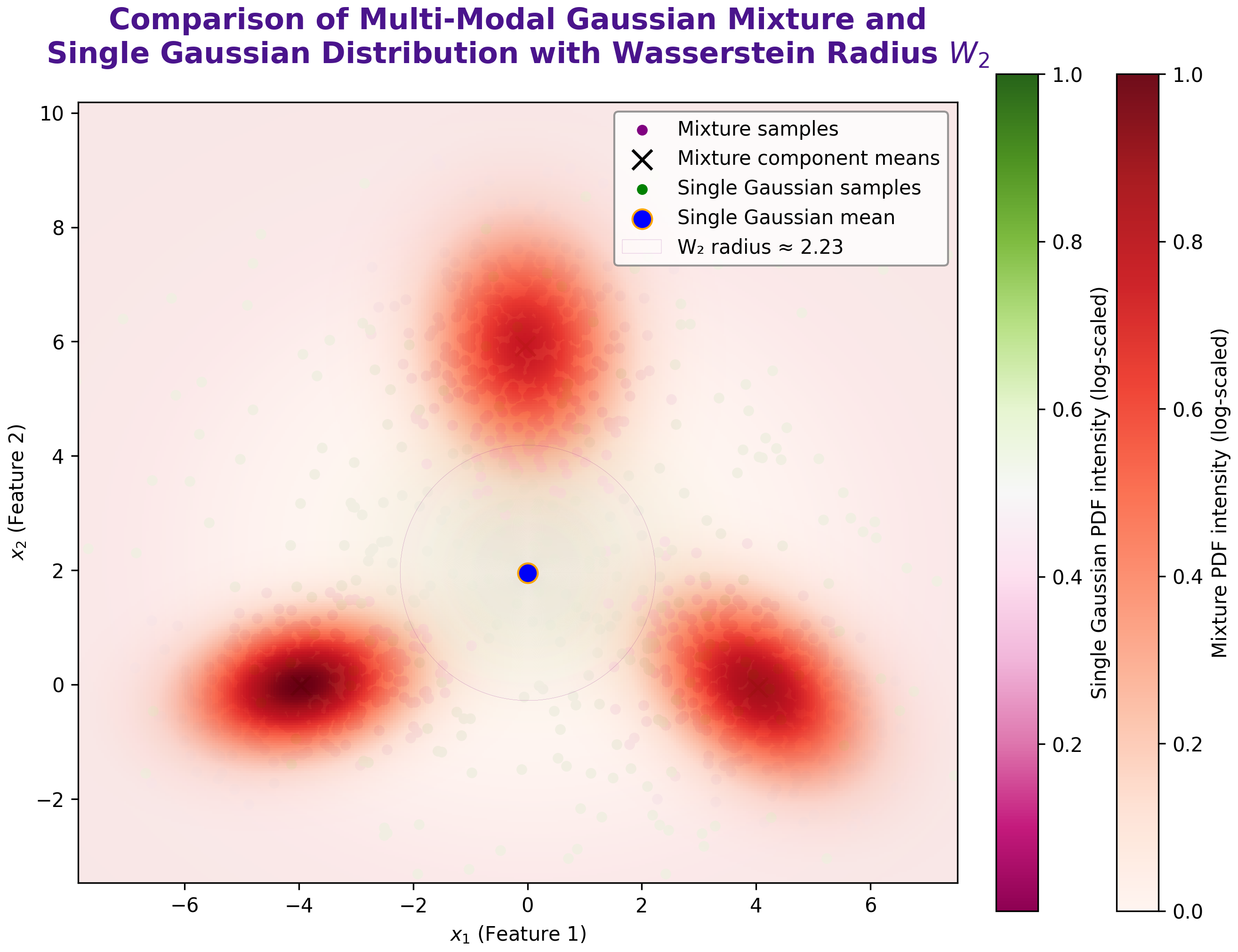}
  \caption{\centering Any uncertainty ball with $W_2 < 2.23$ around a single-modal distribution fitted using the data (see footnote) would not contain the three-modal distribution.}
  \label{fig:wasserstein-radius}
  \vspace{-5pt}
\end{wrapfigure}Multimodality and non-Gaussianity arise when the data reflect multiple clusters or regimes. Examples are
electric-vehicle charging demand~\cite{neuman2021workplace}, wind power generation density~\cite{fathabad2023asymptotically, yi2024discrete}, and short-horizon financial returns modeled using mixture or regime-switching distributions~\cite{kon1984models,perezquiros2000firmsize}. 
In such settings, population heterogeneity induces mixture distributions. 
Specifically, let $\bmean_k \in \mbR^n, \bQ_k \in \mbR^{n\times n}$ for $k=1,\ldots, K$ be the mean vectors and covariance matrices of component Gaussian distributions $\mN(\bmean_k, \bQ_k)$. 
Then the mixture distribution is given as $\sum_{k\in [K]} w_k \mN(\bmean_k, \bQ_k)$ where $\sum_{k\in [K]} w_k = 1$ and $w_k \geq 0$ for all $k \in [K]$.  
When the data-generating distribution is expected to have a mixture structure, a tighter ambiguity set can be obtained by modeling uncertainty over the mixing distribution on component parameters. For example, if a Gaussian distribution is used as a nominal distribution using the dataset\footnote{We generated a synthetic 2D dataset (sample size \(=1200\), seed \(=42\)) from an equally weighted 3-component Gaussian mixture with \(w_k=1/3\), \(\bmean_1=(-4,0)^\top\), \(\bmean_2=(4,0)^\top\), \(\bmean_3=(0,6)^\top\), and full covariance matrices \(\bQ_1=\bigl(\begin{smallmatrix}1.0&0.2\\0.2&0.5\end{smallmatrix}\bigr)\), \(\bQ_2=\bigl(\begin{smallmatrix}1.0&-0.3\\-0.3&0.7\end{smallmatrix}\bigr)\), and \(\bQ_3=\bigl(\begin{smallmatrix}0.8&0.0\\0.0&1.2\end{smallmatrix}\bigr)\). Using the \texttt{scikit-learn} \texttt{GaussianMixture} package in Python with full covariance, we fitted a 3-component GMM to represent the multimodal distribution and a 1-component Gaussian surrogate. The reported Wasserstein radius is the empirical \(W_2\) distance between these two fitted distributions, computed with the POT \texttt{ot.emd2} package under squared Euclidean transport cost from Monte Carlo samples.}, Fig.~\ref{fig:wasserstein-radius} shows that a Wasserstein ball of radius exceeding $2.23$ units is required for a unimodal Gaussian surrogate to cover the fitted multimodal Gaussian mixture. In this case, the ambiguity set radius can be much smaller when a Gaussian mixture is used as the nominal distribution, yielding a more localized ambiguity set around the nominal multimodal structure for service-level analysis.

\subsubsection{Wasserstein-2-type metric for finite Gaussian mixtures. }\, 
Let $\widehat\mbGM_n$ denote the set of Gaussian mixture distributions in $\mbR^n$, and $\widehat\mbGM_n(K)$ those with $K$ mixing terms. For $\mu_1 \in \widehat\mbGM_n(K)$ and $\mu_2 \in \widehat\mbGM_n(\hat{K})$, Delon and Desolneux \cite{delon2020wasserstein} proposed a Wasserstein-type metric by restricting the coupling distributions $\Pi(\mu_1,\mu_2)$ to $\widehat\mbGM_{2n}$: 
\begin{align}\label{def: mod-wasser-dist} 
{RW}^2_2 := \inf_{\pi \in \Pi(\mu_1, \mu_2) \cap \widehat\mbGM_{2n}} \int_{\mbR^n \times \mbR^n} \|y_1 - y_2\|^2 d\pi(\by_1, \by_2). 
\end{align}
In particular, if $\mu_1$, $\mu_2$ follow the Gaussian mixture distributions 
$\sum_{k\in[\hat{K}]}\hat{w}_k\,\hat{\mN}_k(\hat{\bmean}_k,\hat{\bQ}_k)$ and $\sum_{l\in[K]}w_l\,\mN_l(\bm m_l,\bQ_l)$ respectively, then \cite[Proposition~4]{delon2020wasserstein} showed that ${RW}^2_2=\inf_{q\in\Pi(\hat{w},w)}\sum_{k,l}q_{kl}\,d^2_{\mathrm{BW}}(\hat{\mN}_k,\mN_l),$ where $d_{\mathrm{BW}}$ is the Burer-Wasserstein metric between two Gaussian laws (see equation \eqref{def: distance-matric}). Here $\Pi(\bm{\hat{w}},\bw)$ is the standard coupling between two finitely supported probability distributions (represented by the transportation constraints). This result can be used to describe an ambiguity set around a nominal Gaussian-mixture distribution $\mu_1$ as follows:
\begin{equation}
\mathfrak{D}_d
:=\left\{(\bw, \bq) \middle| \,
\begin{array}{l}
\sum_{k\in \hat{K}} \sum_{l\in K}q_{kl}\,
d^2_{\mathrm{BW}}({\mN}_k, \mN_l)\le \rho \\
\sum_{l\in [K]} q_{kl}=\hat{w}_k,\ \forall k\in[\hat{K}] \\
\sum_{k\in [\hat{K}]} q_{kl}=w_l,\ \forall l\in[{K}] \\
\sum_{k\in [\hat{K}]} \hat{w}_k = 1, \,\, \hat{w}_k \geq 0 \,\, \forall k\in [\hat{K}]
\end{array}
\right\}.
\label{def: M-Wass-Ambiguity-Set}
\end{equation}
The solutions feasible to $\mathfrak{D}_d$ provide mixing weights $\bw$ for a Gaussian mixture distribution $\sum_{k\in[{K}]}{w}_k\,\mN_l(\bmean_k,\bQ_k)$ in the ambiguity set. The ambiguity set in \eqref{def: M-Wass-Ambiguity-Set} is referred to as the finite-parameter ambiguity set, and the corresponding distributionally robust model from~\eqref{prob: DR-CCP} is called FDR.

Note that $d^2_{\mathrm{BW}}({\mN}_k, \mN_l)$ is pre-computable, hence, the ambiguity set in \eqref{def: M-Wass-Ambiguity-Set} is interesting in its own right. However, the distributions in this set are restricted either to those means and covariances selected a priori or exactly to those of the nominal finite Gaussian mixture distribution. For example, the approach of selecting the unknown $K$ shown in~\cite[Proposition 3]{yoon2025data} restricts the components' means and covariances to those obtained from GMM fitting to available data. Fixing the mixture component distribution parameters to those of the nominal mixture and optimizing only over the mixing weights is similar to a scenario-reweighting approach (see, e.g., \cite{kammammettu2024distributionally}). Such support-restricted constructions may provide limited protection against realizations outside the observed support and underrepresent low-frequency (tail) regions unless the support is sufficiently rich \cite{ben-tal_den-hertog_dewaegenaere_2013,duchi_namkoong_2019,gao_kleywegt_2023,luo2019decomposition}. From the chance-constrained optimization perspective, the service-level assessment for any candidate decision $\bx$ is therefore confined to adversarial distributions supported on a pre-selected finite set of component mean--covariance parameters.
These limitations motivate us to study a generalization of \eqref{def: M-Wass-Ambiguity-Set} based on \eqref{def: wasser-dist} in which the component Gaussian distributions are not restricted to a fixed finite set of parameters, and the mixing law has a continuous support. 
The proposed construction allows the worst-case mixing law to place probability mass over the continuous parameter support, with component mean-covariances selected endogenously.
\subsection{Contributions}

\subsubsection{Contributions to optimal transport. }\, Since the identification of and computations with the Gaussian mixture distributions in the  Wasserstein-$2$ ball defined in the data space using \eqref{def: wasser-dist} are not practical, in this paper, we show that under suitable conditions, this data-space Wasserstein-$2$ ball can be equivalently described using a Wasserstein-like metric representation in the parameter space (see Theorems~\ref{thm: Wass2-Gaussian-EquivalenceTheorem} and~\ref{thm:Wass2-Gaussian-EquivalenceTheorem-restricted}). Although this paper focuses on the chance constraint optimization problem, from a distributional robust optimization and modeling perspective, this result applies to any optimization model that benefits from the ambiguity set representation based on a Gaussian mixture distribution.

\subsubsection{Algorithmic development. }\, For the ambiguity set developed in this paper, we establish a strong duality result for the robust chance constraint in \eqref{prob: DR-CCP} by showing that the sufficient conditions in \cite[Theorem 1(a)]{blanchet2019quantifying} are satisfied. This duality result allows us to reformulate \eqref{prob: DR-CCP} as a nonlinear semi-infinite program. The $\tau$-optimal convergence, after adding a finite number of cuts, of an adaptive cutting surface algorithm for this nonlinear semi-infinite program is proved. The solvability of the subproblem arising in the proposed cutting-surface algorithm is studied.

\subsubsection{Computational contributions. }\, We evaluate continuous- and finite-parameter distributionally robust formulations (CDR and FDR), along with their nominal counterparts, for a service-level-aware demand management problem in electric-vehicle charging station operations. To develop an implementation of the cutting-surface algorithm using the currently available commercial solvers, we develop a local-optimization heuristic to solve the cutting-surface identification problem (subproblem) efficiently. Across different experimental settings, the proposed adaptive cutting-surface algorithm's behavior indicates that the heuristic identifies the dominant violating mixture components in the first few iterations and then refines a reduced set. Beginning with five nominal cuts, the total cut count after four master iterations remains below 25.

\subsubsection{Experimental findings and managerial implications.  }\,  The computational study using a real dataset from repository~\cite{evcharging2025figshare} exhibits a quantitatively distinctive robustness frontier for the continuous-support ambiguity set. Table~\ref{tab:compact_target_summary} summarizes the key messages from experimental findings in terms of out-of-sample satisfaction (OSS) probability over target $\theta \in \{0.95, 0.97, 0.99\}$, incremental objective overhead, and runtime. While FDR approach protects only over finite empirical support points, CDR approach additionally allows the nominal GMM means and covariances to vary within compact support sets, controlled by the Wasserstein radius \(\rho \in \{0.001, 0.005, 0.01\}\), the mean-support expansion parameter \(\varsigma \in \{0.05,0.1\}\) that symmetrically expands the mixture component mean support around their respective nominal value, and the covariance-scaling bounds \((\underline{\lambda}_f,\overline{\lambda}_f) \in \{(1/2,2),\,(1/3,3)\}\). 

Table~\ref{tab:compact_target_summary} highlights the main gain of the proposed model with a clear demonstration of the added robustness by CDR. FDR does not achieve OSS under any settings. The same conclusion holds under the additional stand-alone benchmark runs at the larger radii $\rho \in \{5,10\}$: for each $\theta$, both radii yield identical objective premiums and identical OSS values of $0.9488$, $0.9530$, and $0.9604$ for $\theta=0.95$, $0.97$, and $0.99$, respectively, and all remain below their prescribed targets despite consuming the full 24-hour runtime budget. Whereas for the mean-support uncertainty parameter value $\varsigma = 0.1$, CDR achieves OSS for all values of $\theta$. The use of mean-covariance uncertainty beyond the Wasserstein radius in CDR expands the set of target-attaining solutions at a $1.5\%$-$5\%$ increase in the objective and a runtime increase from minutes to hours relative to the low-radius FDR instances. Enlarging the FDR radius to $\rho \in \{5,10\}$, however, also increases runtime and objective premium, while leaving OSS below the prescribed targets. Under the experimental settings considered for CDR, we further observe that the effect of changing the covariance-scaling multiplier is limited compared to that of the mean-uncertainty hyperparameter.

\begin{table}[!hbtp]
\centering
\caption{Comparative summary of prescribed target attainment, objective increase, and runtime under FDR and CDR. For CDR, the reported objective increase in \% is calculated by $100 \times |\operatorname{Obj}^{\mathrm{CDR}}-\operatorname{Obj}^{\mathrm{FDR}}|/\operatorname{Obj}^{\mathrm{FDR}}$. FDR performance is investigated under both the stand-alone (\(24\)-hour) solve and the initial iteration budget used to initialize CDR, so that ambiguity-set effects can be distinguished from solve-time effects. For FDR, the reported OSS and runtime ranges include both the cases with $\rho \in \{0.001,0.005,0.01\}$ and the additional stand-alone benchmark runs at $\rho \in \{5,10\}$. 
}
\label{tab:compact_target_summary}
\setlength{\tabcolsep}{4.5pt}
\renewcommand{\arraystretch}{1.05}
\scriptsize
\begin{tabular}{c l c c c c c}
\toprule
 \(\theta\) & \makecell{DR \\ Model} & \makecell{Mean-support \\ uncertainty \\ parameter \(\varsigma\)} & \makecell{OSS probability \\ range (\%)} & \makecell{Target attained at \\ ball radius \(\rho\)} & \makecell{Objective \\ increase (\%)} & Total time (hr) \\
\midrule
\multirow{3}{*}{0.95}
& FDR & -- & 92.76--94.88 & none & -- & 0.02--24.00 \\
& CDR & 0.05 & 95.04--96.41 & all & 1.49--2.31 & 4.64--12.01 \\
& CDR & 0.10 & 95.87--97.71 & all & 2.37--3.63 & 4.65--12.01 \\
\midrule
\multirow{3}{*}{0.97}
& FDR & -- & 94.09--95.30 & none & -- & 0.15--24.00 \\
& CDR & 0.05 & 96.21--97.36 & 0.01 & 1.64--2.48 & 15.70--24.00 \\
& CDR & 0.10 & 96.79--98.45 & 0.005, 0.01 & 2.53--3.99 & 16.08--20.00 \\
\midrule
\multirow{3}{*}{0.99}
& FDR & -- & 95.57--96.04 & none & -- & 24.00 \\
& CDR & 0.05 & 97.45--98.49 & none & 1.93--3.16 & 24.00 \\
& CDR & 0.10 & 98.10--99.17 & 0.01 & 2.82--5.00 & 24.00 \\
\bottomrule
\end{tabular}
\end{table}

The CDR model also provides decision-level guidance toward achieving higher service-level targets. The energy allocation diagnostics for the considered demand management case study further show that these gains are operationally substantive: FDR remains close to the nominal schedule as \(\rho\) varies, whereas CDR generates a distinct family of hourly energy allocations through targeted changes in a small set of recurring hours. Taken together, these findings indicate that allowing ambiguity over continuous mixture parameters, compared to only over empirical support points, provides complete control for service-level attainment, operating cost, and a meaningful structure of the charging schedule when the decision-maker can afford additional computational overhead.

\subsection{Organization}
Some basic notation used throughout the paper is given in Section~\ref{sec:NotationsAndBasicDefinitions}. Existing literature on the use of Gaussian mixture in data representation, chance-constraint optimization, distributional robustness, and optimal transport is reviewed in Section~\ref{sec: Literature}. Sections~\ref{sec: continuous-param-GMM}-\ref{sec: WassersteinTypeTransportEquivalence} present our proposed generalized representation of the Wasserstein-2 metric and ambiguity set~\eqref{def: mod-wasser-dist}-\eqref{def: M-Wass-Ambiguity-Set}. In these sections, we develop the optimal transport (OT) results for this set. Section~\ref{sec: formulation&Duality} provides strong duality results and reformulations resulting from its use for the distributionally robust GMM-based chance-constrained problem resulting from our proposed ambiguity set. Section~\ref{sec: Dual-formulation} provides the resulting semi-infinite program to ensure the convergence of a cutting-surface (Cut-S) algorithm, followed by a comprehensive computational study focusing on an electric vehicle charging capacity planning problem in Section~\ref{sec:numerical}.

\subsection{Notations and Basic Definitions}\label{sec:NotationsAndBasicDefinitions}
 For any function $h(\cdot \;;\;\cdot)$, arguments before and after $``;"$ respectively stand for variables and parameters of that function.  $\mP(\mbR^n)$ represents all sets of probability measures on $\mbR^n$, and $\mP_2(\mbR^n)$ represents those having finite second moment.   The notation $Y\sim\nu$ means that $Y$ is a random variable with probability distribution $\nu$. ``a.e." abbreviates almost everywhere. $\Tr(M)$ denotes the trace of a matrix~$M$. 
For any non-negative integer $n$, we use $[n]$ to denote the set $\{1, 2, \ldots, n\}$ and $[n]_0$ to denote $\{0, 1, 2, \ldots, n\}$. $\mF(\mathtt{\tau})$ represents a set defined by constraints, with the right-hand side of those constraints being $\tau$. 
$\|\cdot\|_F$ is used to denote the Frobenius norm of a matrix. $\ind$ denotes an indicator function. $C_b(\mbR^n)$ denotes the space of bounded continuous functional on $\mbR^n$. $\mbS^n_+$ represents the set of positive-definite matrices. 

$\phi(\cdot)$ and $\Phi(\cdot)$ represent the standard univariate Gaussian probability density and cumulative distribution function (CDF), respectively. The derivative of $\phi$ is denoted by $\phi'$. We let $\mathcal{S} := \mbR^n \times \mbS^n_+$ be the parameter space of $n$-variate Gaussian distribution  $\mathcal{N}(\bmean, \bQ)$ with $\bmean \in \mbR^n,\;\bQ \in \mbS^n_+$ be its mean vector and positive definite covariance matrix. $\mathcal{B}(\mathcal{S})$ is the Borel sigma algebra of $\mathcal{S}$. 
We let  $\bs := (\bmean, \bQ) \in \mathcal{S}$, and $\mathcal{N}(\bs) := \mathcal{N}(\bmean, \bQ)$.  
For 2n-variate Gaussian distributions, we use the notation  $\mathcal{Z}:=\mbR^n\times\mbR^n\times\mbS^{2n}_+$; $\bz := \bigl((\bmean_1,\bmean_2),\bQ\bigr)$ and $\mathcal{N}(\bz)=\mathcal{N}\bigl((\bmean_1,\bmean_2),\bQ\bigr)$.  
We use the symbol $\#$ to represent a pushforward operator. $\mathrm{pr}_i$ denotes the coordinate projection onto the $i$-th component of a product space; specific domains will be clear from the context.

\begin{definition}[Gaussian Mixture Measure]
    We call a measure $\mu \in \mP(\mbR^n)$ a Gaussian mixture measure if $\exists \, \lambda \in \mP(\mS)$ such that for all $C \in \mB(\mbR^n)$
    \[\mu(C) = \int_\mS \mN(\bs)(C) d\lambda(\bs),\]
\end{definition}
\noindent where $\mN(\bs)(C) = \int_{\mbR^n}\ind_C(\by) d\mN(\bs)$ for $\bs \in \mS$.

\begin{definition}[Mixing law via Borel sets]\label{def:mixing-law-set}
Let $\mu \in \mP(\mbR^n)$.  We say that a probability measure $\pi$ on $(\mathcal{S},\mathcal{B}(\mathcal{S}))$ is a mixing law  for a Gaussian mixture distribution $\mu$ if, for every Borel set
$C \in \mB(\mbR^n)$,
\begin{equation}\label{eq:mixing-law-pi1}
   \mu(C)
   =
   \int_{\mathcal{S}} \mN(\bs)(C)\,d\pi(\bs).
\end{equation}
\end{definition}
Note that the mixing law for a Gaussian mixture distribution may not be unique, and the second moment of a Gaussian mixture distribution may not be bounded.  We denote the set of all Gaussian mixture distributions with finite second moment as
\begin{equation}
\hspace{-1em}
\mbGM_n
  := \Bigl\{\nu \;\Big| \;\exists \, \lambda \in \mP(\mS) \,\,\, \mathrm{s.t.} \,\,\, \nu = \hspace{-0.3em}\int_{\bs \in \mS} \mN(\bs)\,d\lambda(\bs) \,\,\text{and}
       \int_{\bs \in \mS}\hspace{-0.5em}\left(\lVert \bmean(\bs)\rVert^2 + \Tr\big(\bQ(\bs)\big)\right)\,d\lambda(\bs) < \infty
       \Bigr\}.
  \label{def: GMMmix}
\end{equation}

\section{Literature Review}\label{sec: Literature}

This section summarizes three streams of relevant literature. The first stream provides the Wasserstein geometry, strong-duality reformulations, and finite-sample guarantees that justify data-driven radius calibration.    We then cover the distributionally robust chance-constrained optimization literature, emphasizing how the support structure of the random variables shapes the available deterministic reformulations and how the uncertainty law itself is represented, especially for multimodal data. The third identifies the related OT literature on Gaussian and mixture structure.  Reviewing these three streams provides a clear, distinctive indication between sample-supported mixture models and the continuous-parameter mixture class adopted in this work.

\subsection{Distributionally Robust Optimization with Wasserstein Set}

Foundational Wasserstein distributionally robust optimization (WDRO) results show that worst-case expectations over Wasserstein balls admit strong dual representations and, under standard growth and Lipschitz conditions, reduce to finite-dimensional programs \cite{gao_kleywegt_2023, esfahani_kuhn_2018, blanchet2019quantifying, luo2019decomposition, zhang2024shortDualityWassersteinDRO}. These results justify deterministic reformulations by viewing the Wasserstein-ball as a model-misspecification device and replacing the distributional supremum with a tractable dual program \cite{gao_kleywegt_2023, blanchet2019quantifying, luo2019decomposition, zhang2024shortDualityWassersteinDRO}. In parallel, data-driven WDRO centers the ball at an empirical law, yields convex reformulations in broad settings, and links radius calibration to finite-sample reliability \cite{esfahani_kuhn_2018}. Wasserstein neighborhoods also underpin robust learning objectives based on adversarial perturbations of the data law, with population-level robustness guarantees and statistical consistency analyses \cite{sinha2018certifyingRobustness, kwon2020wassersteinLocalPerturbations}. Algorithmically, the inner maximization has been handled in the literature via flow-based transport-map parameterizations \cite{xu2024flowdro}, entropic smoothing and Sinkhorn-type discrepancies \cite{wang2025sinkhorn}, and iterative minimax schemes over continuous densities with convergence guarantees \cite{zhu2024iterativeMinimaxWasserstein}.

\subsection{Distributionally Robust Chance-Constrained Optimization (DRCCO)}

A DRCCO problem replaces a nominal chance constraint under a fixed law by a worst-case chance constraint over an ambiguity set $\mP$ of probability measures \cite{shapiro_dentcheva_ruszczynski_2014}. When the ambiguity-set size is calibrated from data, one can obtain finite-sample guarantees that the unknown data-generating law lies in $\mP$ with prescribed confidence, which in turn supports out-of-sample reliability claims for decisions computed from the resulting DR model \cite{lam_2019}. Most continuous-support DRCCO literature falls into three families—Wasserstein-metric, $\varphi$-divergence, and moment-based sets—each with distinct reformulation mechanisms and complexity trade-offs.

 A theme of Wasserstein-metric ambiguity set is the deterministic reformulation of worst-case probability constraints through duality, leading to exact mixed-integer models in broad settings and MILP reductions under $\ell_1$ or $\ell_\infty$ transport costs \cite{xie_2019,chen_kuhn_wiesemann_2022}. Related developments treat individual and joint chance constraints and variants such as random right-hand sides, again relying on strong duality for the worst-case probability evaluation \cite{ji_lejeune_2020,ho_nguyen_kilic_2021}. Structural properties have also been identified under which Wasserstein robustification yields convex feasible regions \cite{shen_jiang_2021}.

$ \varphi$-divergence-based ambiguity sets (e.g., KL, $\chi^2$) have also been used to define confidence-region-based sets around a nominal law \cite{ben-tal_ghaoui_nemirovski_2009,duchi_namkoong_2019}. 
For chance constraints, robustification can reduce to a nominal chance constraint with a modified chance satisfaction level, which can be calculated explicitly for KL divergence. It also has a relationship with Bernstein-type convex approximations \cite{hu_hong_2013,nemirovski_shapiro_2006}. 
Divergence-based ambiguity set radii are calibrated for finite-sample coverage guarantees \cite{jiang_guan_2016,lam_2019,duchi_namkoong_2019}. Related robust-counterpart views for uncertain probability vectors and application-driven MILP reformulations (e.g., $\chi^2$-robust surgery scheduling) appear in \cite{ben-tal_den-hertog_dewaegenaere_2013,deng_shen_denton_2019}.
Moment-based ambiguity sets hedge over all laws consistent with the given moment (and possibly support) information. They typically yield conic reformulations when the first two moments are used in describing the ambiguity set formulation \cite{shapiro_dentcheva_ruszczynski_2014,delage_ye_2010}.  Such models also lead to safe approximations for (distributionally robust) chance constraints, with possible extensions to generalized dispersion and broader admissible moment/support sets \cite{hanasusanto_kuhn_wiesemann_zymler_2015,ding_etal_2020}. 

\subsubsection{Use of mixture models in uncertainty representation. }\, The distributionally robust optimization models developed in the literature have focused on representing distribution ambiguity by using a finite-sample support of the empirical distributions. Unless the sample size or the ambiguity set is very large, such finite-support representations may miss regions with nontrivial probability mass (e.g., those arising in rare event regimes), and thus provide only a coarse description of multivariate dependence \cite{shapiro_dentcheva_ruszczynski_2014,nemirovski_shapiro_2006,kleywegt_shapiro_homemdemello_2002,barrera_etal_2014_optimonline,blanchet_nam_zwart_2024}. A possible solution is to use mixture laws. Using finite support for the mean and variance parameters,  \cite{kammammettu2024distributionally} use a Gaussian mixture law to represent distribution ambiguity in the context of a chemical engineering application. In the chance-constrained setting, finite Gaussian-mixture models have been studied in \cite{fathabad2023asymptotically,hu2022chance,wei2024enhanced}, where convex relaxations are derived under strong assumptions. In a recent paper \cite{dey2025solving}, a finite Gaussian mixture was used for linear chance constraints, which established its solvability by using piecewise-linear inner and outer approximations of the nonlinearity arising from the Gaussian term in the mixture. The piecewise-linear approximation idea from \cite{dey2025solving} is further developed in the current paper in the analysis of the proposed cutting surface algorithm.  
\subsubsection{Service-level requirements through chance constraints.}\,
Beyond the illustrative case study of the electric vehicle charging center problem, the work in this paper has broad applicability. Chance constraint models arise in Operations Research from the service-level requirement perspective. In inventory and production systems, service-level constraints are often used to specify the probability of stockout, fill rate, or multi-period demand-fulfillment targets. Example settings are minimum-service inventory models \cite{chen2001inventory}, service-constrained stochastic lot-sizing formulations \cite{tarim_kingsman_2004}, and chance-constraint models that impose service targets jointly over the planning horizon \cite{rossi2008global}, cycle-service, fill rate, and $\alpha$-/$\beta$-service measures in stochastic inventory control \cite{rossi2013service,jiang2019service}. A more recent data-driven line of work emphasizes target attainment rather than only nominal service feasibility. In particular, empirical newsvendor decisions can undershoot prescribed service levels, motivating distributionally robust chance-constrained formulations to achieve on-target service levels~\cite {van_der_laan2022datadriven}. Probability-based quality-of-service (QoS) constraints also appear in call-center staffing~\cite{dam2022joint}, charging-station sizing with QoS requirements~\cite{khaksari2021sizing}, and peak charging-demand satisfaction constraints~\cite{casini2021chance}. Surgery-assignment and operating-room bin-packing models provide another service-level interpretation: elective cases with uncertain durations are packed into operating room blocks so that overtime or capacity violation remains below a prescribed probability~\cite{Shylo2012StochasticOR, Wang2021ChanceConstrainedBinPacking}; related surgery-planning models also impose distributionally robust chance constraints on waiting and overtime~\cite{Deng2019ChanceConstrainedSurgery}. Relative to these works, our focus is not to introduce a new service-level measure; rather, we study how a chance-constraint service target can be robustified when multimodal uncertainty is represented by a Gaussian mixture with continuous mean--covariance support.

\subsection{Optimal Transport (OT) with Gaussian and Gaussian-Mixture Models}
Wasserstein metric-based ambiguity sets are motivated from the theory of optimal transport. A key enabler for OT-based modeling with Gaussian uncertainty is that the $2$-Wasserstein distance between Gaussian laws admits a closed-form expression in terms of means and covariances, which yields a computable geometry on Gaussian families \cite{gelbrich1990formula,takatsu2011wasserstein}. 
It introduces a Riemannian viewpoint on the manifold of Gaussian densities that supports geodesics and gradient-based constructions directly in mean--covariance variables \cite{malago2018wasserstein}. There has been work on OT constructions with direct operations on Gaussian mixture models \cite{chen2018optimal} and structured mixture-restricted OT distances that preserve the mixture form while yielding finite-dimensional transport problems over component couplings \cite{delon2020wasserstein}. These results provide a principled finite-dimensional geometry for studying Gaussian-mixture uncertainty representations without reverting to empirical OT computations over samples in the data space.

 Because mixture-structured OT still requires solving an OT problem whose cost matrix involves pairwise Gaussian Wasserstein terms, recent work emphasizes scalable surrogates and approximations that retain mixture geometry. These include the minimized aggregated Wasserstein distance, developed for efficient mixture barycenter computation and later used to integrate heterogeneous datasets summarized by mixtures \cite{lin2023maw}, as well as sliced and multi-sliced constructions, which replace high-dimensional OT with collections of one-dimensional transport problems for mixture comparison and Wasserstein-based Gaussian mixture learning \cite{kolouri2018sliced,piening2025double}. 

These developments are related to robust modeling with heterogeneous data sources, where the empirical evidence indicates multiple regimes and a mixture representation is natural. In Wasserstein DRO with heterogeneous data sources, distributional ambiguity is modeled through OT neighborhoods around source-specific empirical laws and their composition, providing a principled mechanism to combine biased or non-identically distributed samples \cite{rychener2024heterogeneous}. Such multi-source OT developments also motivate tractable robust counterparts beyond finite support. Mean--covariance robust risk measurement uses Burer-Gelbrich balls in mean--covariance space to form conservative outer approximations of Wasserstein ambiguity sets that become exact under Gaussian structure and yield a reformulated model whose size does not scale with the nominal support \cite{nguyen2025meancov}. However, these advancements remain fragmented across distribution comparison, barycenter computation, and learning, and do not yet have a unified robust-optimization framework for multimodal uncertainty distributions having a continuous-support parameter set.

\section{Wasserstein-2 Metric for Gaussian Mixtures}\label{sec: continuous-param-GMM}
This section develops a theoretical basis for a metric between two Gaussian mixture distributions by allowing mixing weights and parameters of the Gaussians to vary over an admissible parameter set. The key mechanism is an optimal-transport coupling that transports mass between Gaussian components indexed by their parameters. This construction preserves analytic structure while remaining computationally amenable.

\subsection{Wasserstein-2 Metric between Two Gaussian Distributions in Data Space}
\label{sec:pre}

The following proposition highlights that the second moment of a Gaussian distribution is bounded. 

\begin{proposition}[Second moment of a Joint Gaussian in $\mbR^{2n}$]\label{prop: joint-gaussian-second-moment}
Let $Z \sim \mN\bigl((\bmean_1, \bmean_2),\,\bQ\bigr)$ in $\mbR^{2n}$ where $\bQ$ is a covariance matrix.
Then
\begin{equation}\label{eq: joint-gaussian-second-moment}
\mathbb E\bigl[\|Z\|^2\bigr]
\;=\;
\|(\bmean_1,\bmean_2)\|^2 \;+\; \Tr(\bQ)
\; \;=\; \|\bmean_1\|^2+\|\bmean_2\|^2 +\; \Tr(\bQ) <\;\infty.
\end{equation}
\end{proposition}
\begin{proof}
    See proof in Appendix~\ref{appndx:very-first-proof}.
\end{proof}

\begin{definition} Let $\mu_1, \mu_2$ be two n-variate probability distributions. Then, the Wasserstein-2 metric between $\mu_1$ and $\mu_2$ is defined as  
\begin{equation}
    W^2_2(\mu_1,\mu_2) :=
  \inf_{\nu \in \Pi(\mu_1, \mu_2)}
    \int_{\mbR^n \times \mbR^n}
      \|\by_1 - \by_2\|^2 \,d\nu(\by_1,\by_2), \label{eqn:Wass-2}
\end{equation}
where
\begin{align}\label{def: coupling-Euclidean}
  \Pi(\mu_1, \mu_2)
  =
  \bigl\{
    \nu \in \mP(\mbR^n \times \mbR^n)
    :
    \nu^{(1)} = \mu_1,\ \nu^{(2)} = \mu_2
  \bigr\},
\end{align}
\noindent and $\nu^{(1)}$ and $\nu^{(2)}$ are the marginal distributions of $\nu$ on $\mbR^n$. A 2n-variate distribution attaining the optimal value in \eqref{eqn:Wass-2} is called the Wasserstein-2 optimal coupling.
\end{definition}

\begin{theorem} \cite{GivensShortt1984, gelbrich1990formula}\label{thm: Gelbrich}
For any $\bs_1=(\bmean_1, \bQ_1) \in \mS$ and $\bs_2=(\bmean_2, \bQ_2) \in \mS$, the Wasserstein-2 optimal coupling between Gaussian distributions $\mN(\bs_1)$ and $\mN(\bs_2)$  is  given by the Gaussian distribution $\mN(\bs_1, \bs_2) \in \mP_2(\mbR^n \times \mbR^n)$, where
     $ \mN(\bs_1, \bs_2)= \mN\!\bigl((\bmean_1,\bmean_2),\bar{\bQ}\bigr) \in \mP_2(\mbR^n \times \mbR^n)$ with 
     \begin{equation}\label{def: gelbrich-off-diagonal}
      \bar{\bQ}=\begin{pmatrix}\bQ_{1}&\bSigma\\ \bSigma^\top &\bQ_{2}\end{pmatrix}\in\mbS^{2n}_+,    \quad \text{and} \quad  
\bSigma = {\bQ_1}^{\frac12}\,\bigl(\,{\bQ_1}^{\frac12}\,{\bQ_2}\,{\bQ_1}^{\frac12}\bigr)^{\frac12}\,{\bQ_1}^{-\frac12}.
     \end{equation}
The matrix $\bSigma\in\mbR^{n\times n}$ is the cross-covariance matrix of the optimal Gaussian coupling.
The Wasserstein-2 metric between $\mN(\bs_1)$ and $\mN(\bs_2)$ is given by:
\begin{equation}\label{def: distance-matric}
d^2_{\mathrm{BW}}\bigl(\mN(\bs_1),\mN(\bs_2)\bigr)
=
\|\bmean_1-\bmean_2\|^2
+\Tr \,\Bigl(
\bQ_1+\bQ_2
-2\bigl(\bQ_1^{\frac12}\,\bQ_2\,\bQ_1^{\frac12}\bigr)^{\frac12}
\Bigr),
\end{equation}
where the covariance term is the squared Bures--Wasserstein distance between
$\bQ_1$ and $\bQ_2$.
\end{theorem}

We denote the optimal coupling between $\mN(\bs_1)$ and $\mN(\bs_2)$ by $\mG(\bs_1,\bs_2)$ and refer to it as Gelbrich's coupling. 
Here and throughout, for any symmetric positive semidefinite matrix $\bQ\succeq 0$, $\bQ^{1/2}$ denotes the principal (symmetric PSD) square root, i.e., the unique matrix $\bQ^{1/2}\succeq 0$ such that $\bQ^{1/2}\bQ^{1/2}=\bQ$ \cite{bhatia2009positive}.
Moreover, whenever an inverse square root appears, we adopt the convention followed in Theorem~\ref{thm: Gelbrich}, i.e., $\bQ^{-1/2}\ \equiv\ \bQ^{\dagger/2}:=(\bQ^{1/2})^\dagger,$ where $(\cdot)^\dagger$ denotes the Moore--Penrose pseudoinverse with $\bQ^{\dagger/2}=\bQ^{-1/2}$ when $\bQ\succ 0$~\cite{Moore1920, Penrose1955}. In particular, $\bQ^\dagger$ is characterized as the unique matrix satisfying
$\bQ \bQ^\dagger \bQ = \bQ$, $\bQ^\dagger \bQ \bQ^\dagger = \bQ^\dagger$, $(\bQ \bQ^\dagger)^\top = \bQ \bQ^\dagger$, and $(\bQ^\dagger \bQ)^\top = \bQ^\dagger \bQ$.
Equivalently, if $\bQ=\bU\diag(\lambda_1,\ldots,\lambda_n)\bU^\top$ with $\lambda_i\ge 0$, then
\[
\bQ^{\dagger/2}=\bU\,\diag(\lambda_1^{-1/2}\mathbf{1}_{\{\lambda_1>0\}},\ldots,\lambda_n^{-1/2}\mathbf{1}_{\{\lambda_n>0\}})\,\bU^\top, \quad \text{and} \quad (\bQ^{\dagger/2})^2=\bQ^\dagger.
\]

\subsection{Construction of Transport Maps between Parameter Spaces}\label{sec: map-construct}
This section constructs two transport maps $\mT$ and $\mM$, respectively, for joint parameter construction and marginal parameter extraction.

\subsubsection{\texorpdfstring{Joint parameter construction map $\mT$. }{Joint parameter construction map}}\, 
For $\bs_i=(\bmean_i,\bQ_i)\in\mS$, define the induced parameter map
\begin{align}\label{def: OT-map-forward}
\mT:\mS\times\mS\to\mZ,\qquad
\mT(\bs_1,\bs_2) \coloneqq \bigl((\bmean_1,\bmean_2),\,\bar{\bQ}\bigr),
\end{align}
where $\bar{\bQ}$ is defined in~\eqref{def: gelbrich-off-diagonal} under the convention
$\bQ^{-1/2}:=(\bQ^{1/2})^\dagger$ for all $\bQ\succeq 0$. Consequently, $\mT$ is well-defined on $\mS\times\mS$.
Moreover, since $\bQ\mapsto \bQ^{1/2}$ and $\bQ\mapsto (\bQ^{1/2})^\dagger$ are Borel measurable on $\mbS_+^n$,
the map $\mT$ is Borel measurable.
Hence, for any $A\in\mB(\mZ)$, we define the preimage $\mT^{-1}(A)\in\mB(\mS\times\mS)$ as
\[
\mT^{-1}(A) \coloneqq \{(\bs_1,\bs_2)\in\mS\times\mS \mid\ \mT(\bs_1,\bs_2)\in A\}.
\]
Let $\pi\in\mP_2(\mS\times\mS)$ and $\mT_\#\pi\in\mP(\mZ)$ denote the probability measure such that
\begin{align*}
 (\mT_\#\pi)(A)=\pi(\mT^{-1}(A))
=\int_{\mS\times\mS}\ind_A\!\bigl(\mT(\bs_1,\bs_2)\bigr)\,d\pi(\bs_1,\bs_2)
\qquad \forall A\in\mB(\mZ).
\end{align*}
Then we call $\lambda := \mT_\# \pi \in \mP(\mZ)$ a push-forward measure. Consequently, for any measurable $\Psi$ that is either nonnegative or integrable (or, more strongly, bounded)~\cite{villani2008optimal},
\begin{align}\label{eq: push-integral}
\int_{\mS\times\mS}\Psi\!\bigl(\mT(\bs_1,\bs_2)\bigr)\,d\pi(\bs_1,\bs_2)
=\int_{\mZ}\Psi(\bz)\,d(\mT_\#\pi)(\bz)
=\int_{\mZ}\Psi(\bz)\,d\lambda(\bz).
\end{align}
The map $\mT$ preserves the second moment of the probability measure $\pi$, as the following proposition states.

\begin{proposition}[Moment preservation under $\mT$]\label{prop: finite-GMM-mmnt}
Let $\pi\in \mP(\mS\times \mS)$ and set $\lambda:=\mT_\#\pi\in \mP(\mZ)$. 
If
\begin{equation}\label{eq:assump-input}
\int_{\mS\times\mS}\Bigl(\|\bmean_1(\bs_1)\|^2+\|\bmean_2(\bs_2)\|^2+\|\bQ_1(\bs_1)\|_F^2+\|\bQ_2(\bs_2)\|_F^2\Bigr)\,
d\pi\bigl(\bs_1, \bs_2\bigr) < \infty,
\end{equation}
then $\lambda\in \mP_2(\mZ)$ with respect to the gauge
$\|((\bmean_1,\bmean_2),\bar{\bQ})\|_{\mZ}^2 := \|\bmean_1\|^2+\|\bmean_2\|^2+\|\bar{\bQ}\|_F^2$, i.e.,
\begin{equation}\label{eq:concl-output}
\int_{\mZ}\Bigl(\|\bmean_1(\bz)\|^2+\|\bmean_2(\bz)\|^2+\|\bar{\bQ}(\bz)\|_F^2\Bigr)\,
d\lambda\bigl(z\bigr) < \infty.
\end{equation}
\end{proposition}
\begin{proof}
    See proof in Appendix~\ref{appndx: second-moment-preservation}.
\end{proof}

\subsubsection{\texorpdfstring{Marginal parameter extraction map $\mM$. }{Marginal parameter extraction map}}\,
For $\bz=((\bmean_1,\bmean_2),\bQ)\in\mZ$ with
\begin{align}\label{def: Q-block}
\bQ=\begin{pmatrix}\bQ_{11}&\bQ_{12}\\ \bQ_{21}&\bQ_{22}\end{pmatrix}\in\mbS^{2n}_+,
\qquad \bQ_{11},\bQ_{22}\in\mbS^n_+, 
\end{align} define the marginal-parameter extraction map
\begin{align}\label{def: OT-map-backward}
\mM:\mZ\to \mS\times\mS,\qquad
\mM(\bz) \coloneqq \bigl((\bmean_1,\bQ_{11}),(\bmean_2,\bQ_{22})\bigr). 
\end{align}
The map $\mM$ is continuous (hence Borel measurable; see \cite[Theorem 1.6]{folland1999real}). For $B\in\mB(\mS\times\mS)$, the preimage of $\mM$ is
\vspace{-2.0em}
\[
\mM^{-1}(B)\coloneqq \{z\in\mZ \mid\ \mM(\bz)\in B\}\in\mB(\mZ).
\]
Given a probability measure $\lambda \in \mP_2(\mZ)$, define the pushforward measure
$\pi := \mM_\#\lambda \in \mP(\mS\times\mS)$ by
\begin{align*}
(\mM_\#\lambda)(B)
:=\lambda(\mM^{-1}(B))
=\int_{\mZ}\ind_{B}(\mM(\bz))\,d\lambda(\bz)
\qquad \forall B\in\mB(\mS\times\mS).
\end{align*}

Hence, for any measurable $\Psi$ that is either nonnegative or integrable (or, more strongly, bounded)~\cite{villani2008optimal},
\begin{align}\label{eq: push-integral-M}
\int_{\mZ}\Psi\!\bigl(\mM(\bz)\bigr)\,d\lambda(\bz)
=\int_{\mS \times \mS }\Psi(\bs_1, \bs_2)\,d(\mM_\#\lambda)(\bs_1, \bs_2)
=\int_{\mS\times \mS}\Psi(\bs_1, \bs_2)\,d\pi(\bs_1, \bs_2).
\end{align}
The transport map $\mM$ is also moment-preserving: if $\lambda$ has finite second moment, so does the pushforward measure $\pi$ as stated in the following proposition. 
\begin{proposition}[Moment preservation under $\mM$]\label{prop: finite-2nd-moment}
Let $\lambda\in \mP(\mZ)$ and set $\pi:=\mM_\#\lambda\in \mP(\mS\times \mS)$. If
\begin{equation}\label{eq:assump-Z-generalQ}
\int_{\mZ}\Bigl(\|\bmean_1(\bz)\|^2+\|\bmean_2(\bz)\|^2+\|\bQ(\bz)\|_F^2\Bigr)\,
d\lambda(\bz) < \infty,
\end{equation}
then $\pi\in \mP_2(\mS\times \mS)$ with respect to the product gauge $\|((\bmean_1,\bQ_{11}),(\bmean_2,\bQ_{22}))\|_{\mS\times \mS}^2
:=\|\bmean_1\|^2+\|\bmean_2\|^2+\|\bQ_{11}\|_F^2+\|\bQ_{22}\|_F^2$, i.e., \begin{equation}\label{eq:concl-SxS-generalQ}
\int_{\mS\times\mS}\Bigl(\|\bmean_1(\bs_1)\|^2+\|\bmean_2(\bs_2)\|^2+\|\bQ_{11}(\bs_1)\|_F^2+\|\bQ_{22}(\bs_2)\|_F^2\Bigr)\,
d\pi\bigl(\bs_1, \bs_2\bigr) < \infty.
\end{equation}
\end{proposition}
\begin{proof}
    See proof in Appendix~\ref{appndx: second-moment-preservation}.
\end{proof}

\subsubsection{Marginalization properties of Gaussian measures. }\, For any fixed Gaussian parameter, the corresponding Gaussian measure \emph{in the data space} satisfies the standard marginalization property. In particular,
for $(Y_1,Y_2)\sim \mN\!\bigl((\bmean_1,\bmean_2),\bQ\bigr) \in \mbR^n\times \mbR^n$, with block decomposition $\bQ$ in~\eqref{def: Q-block}, the marginals satisfy $Y_1\sim \mN(\bmean_1,\bQ_{11})$ and $Y_2\sim \mN(\bmean_2,\bQ_{22})$.
Define the continuous coordinate projections $\mathrm{pr}_i:\mbR^n\times\mbR^n\to \mbR^n$ by
$\mathrm{pr}_1(\by_1,\by_2)=y_1$ and $\mathrm{pr}_2(\by_1,\by_2)=y_2$. 
Their preimages for any $C\in\mB(\mbR^n)$ are
\begin{align}\label{eq: proj-preimage}
\mathrm{pr}^{-1}_1(C) \coloneqq \{(\by_1,\by_2)\in\mbR^n\times\mbR^n \mid\ \mathrm{pr}_1(\by_1,\by_2)\in C\}, \,\,
\mathrm{pr}^{-1}_2(C) \coloneqq \{(\by_1,\by_2)\in\mbR^n\times\mbR^n \mid\ \mathrm{pr}_2(\by_1,\by_2)\in C\}.   
\end{align}
Also, denote $\mM_1(\bz)\coloneqq \bs_1(\bz)=(\bmean_1,\bQ_{11}), \,\, 
\mM_2(\bz)\coloneqq \bs_2(\bz)=(\bmean_2,\bQ_{22})$. Then, by the definition of pushforward measure, for any $C\in\mB(\mbR^n)$,
\begin{equation}\label{def: proj-param}
\begin{aligned}
\mN(\bs_1(\bz))(C)
&=\mN(\mM_1(\bz))(C)
=\bigl((\mathrm{pr}_1)_{\#}\mN(\bz)\bigr)(C)
=\mN(\bz)\!\left(\mathrm{pr}^{-1}_1(C)\right),\\
\mN(\bs_2(\bz))(C)
&=\mN(\mM_2(\bz))(C)
=\bigl((\mathrm{pr}_2)_{\#}\mN(\bz)\bigr)(C)
=\mN(\bz)\!\left(\mathrm{pr}^{-1}_2(C)\right).
\end{aligned}
\end{equation}
Equivalently, $(\mathrm{pr}_1)_{\#}\mN(\bz)$ and $(\mathrm{pr}_2)_{\#}\mN(\bz)$ are the two marginals of $\mN(\bz)$ (in data space).

\subsection{Wasserstein-2-type Metric between Gaussian Mixtures}
We now define a Wasserstein-2-type metric between two Gaussian mixture distributions. This metric is related to the Wasserstein-2 metric. 

\begin{definition} Let $\mu_1, \mu_2 \in \mP(\mbR^n)$. Let $\Pi^\mathrm{mix}(\mu_1, \mu_2) = \{\Pi(\mu_1, \mu_2)\, \cap\  \mbGM_{2n}\}$. Consider the following optimization problem:
\begin{equation}
    W_2^\mathrm{mix}(\mu_1,\mu_2) := \inf_{\nu \in \Pi^\mathrm{mix}(\mu_1, \mu_2)}
    \int_{\mbR^n \times \mbR^n}
      \|\by_1 - \by_2\|^2 \,d\nu(\by_1,\by_2) \label{eqn: WassGausianMix}
\end{equation}
An optimal solution to~\eqref{eqn: WassGausianMix} exists (see Proposition~\ref{prop:intersection-weak-compact} in Appendix~\ref{appndx: ommited-proof}). We call $W_2^\mathrm{mix}(\mu_1,\mu_2)$ a Wasserstein-2-type metric between $\mu_1$ and $\mu_2$. The corresponding optimal solution is called an optimal Wasserstein-2-type coupling between $\mu_1$ and $\mu_2$. We will refer to $W_2^{\mathrm{mix}}(\mu_1,\mu_2)$  and the corresponding coupling as the Wasserstein-2 metric and coupling defined in the data space.
\end{definition}
\begin{remark}
Since $\{\Pi(\mu_1,\mu_2)\cap \mbGM_{2n} \} \subseteq \Pi(\mu_1,\mu_2)$,  $W_2(\mu_1,\mu_2) \leq W_2^{\mathrm{mix}}(\mu_1,\mu_2)$.  
\end{remark}
\section{Transport via Gaussian Mixture Distribution Parameters} 
\label{sec: WassersteinTypeTransportEquivalence}

We now define a metric that can be computed from an optimization model defined in the parameter space. For $\pi \in \mP(\mS \times \mS)$, let $\pi_1,\pi_2$ denote its first and second marginals:
$\pi_1(C) := \pi(C \times \mS)$, $\pi_2(C) := \pi(\mS \times C)$ for Borel $C \in \mB(\mS)$. Let
\begin{align}\label{def: coupling-param}
\hspace{-1.75em}\Gamma = 
  \Bigl\{
    \pi \in \mP_2(\mS \times \mS)
    \mid \hspace{-1pt}
   \int_{\mathcal{S}} \hspace{-1pt} \mN(\bs)(A)\,d\pi_1(\bs) = \mu_1(A),\hspace{-1pt}
   \int_{\mathcal{S}} \hspace{-1pt} \mN(\bs)(B)\,d\pi_2(\bs) = \mu_2(B), \forall A, B \in \mB{(\mbR^n)}
  \Bigr\}, 
\end{align}
\vspace{-1.0em}
\begin{equation}
\text{and} \qquad    W_2^\mathrm{prm} := \inf_{\pi\in\Gamma}
    \int_{\mS\times \mS} 
d^2_{\mathrm{BW}}\!\bigl(\mN(\bs_1),\,\mN(\bs_2)\bigr)\,d\pi(\bs_1, \bs_2)\srevision{.}\label{eqn:Wass2ParamSpace}
\end{equation}

\begin{theorem}\label{thm: Wass2-Gaussian-EquivalenceTheorem} Let $W_2^\mathrm{mix}(\mu_1,\mu_2)$ and $W_2^\mathrm{prm}$ be defined as in \eqref{eqn: WassGausianMix} and \eqref{eqn:Wass2ParamSpace}, respectively. Then
\begin{equation}
    W_2^\mathrm{mix}(\mu_1,\mu_2) = W_2^\mathrm{prm}. \label{eqn:Wass2-Dist-ParamSpace}
\end{equation}
\end{theorem}
The importance of Theorem~\ref{thm: Wass2-Gaussian-EquivalenceTheorem} is that it allows us to compute the Wasserstein-2-type metric defined in \eqref{eqn: WassGausianMix} by setting up an optimization model using the parameters of the Gaussians in the mixture. $W_2^\mathrm{prm}$ is more amenable to algorithmic development, as will be seen in subsequent sections. We further show in Theorem~\ref{thm:Wass2-Gaussian-EquivalenceTheorem-restricted} that the equivalence continues to hold when the mean--covariance parameter support is a compact subset of $\mbR^n\times\mbS^n_+$. We prove Theorem~\ref{thm: Wass2-Gaussian-EquivalenceTheorem} below; the restricted-support case is proved in Appendix~\ref{appndx: restricted}.

\subsection{Constructing Data-Space Coupling from Parameter Space Coupling}
The following result shows that given any parameter space coupling from $\Gamma$, we can construct a data space coupling in the form of a Gaussian mixture distribution $\nu_\pi$ with marginals $\mu_1, \mu_2$, i.e., $\nu_\pi \in \Pi^{\mathrm{mix}}(\mu_1, \mu_2)$. 

\begin{lemma}\label{lem: forward}
 Let $\bs_1, \bs_2 \in \mS$, \ $\mG(\bs_1, \bs_2)$ be the Gelbrich coupling, and $\,\pi \in \Gamma$. Then, for $\nu_{\pi}:=\int_{\mS\times \mS}\mG(\bs_1, \bs_2)\,d\pi(\bs_1,\bs_2)$
\begin{eqnarray}
\nu_{\pi}^{(1)}(A) 
&:=& \int_{\mbR^{2n}} \!\! \ind_{A}(\by_1)\, d\nu_{\pi}(\by_1,\by_2)
\;=\; \mu_1(A) \,\,\, \forall\, \text{Borel set } A\in \mB{(\mbR^n)},
\text{therefore}, \nu_{\pi}^{(1)} = \mu_1,
\label{eq:nu1-equality} \\[0.5em]
\nu_{\pi}^{(2)}(B) 
&:=& \int_{\mbR^{2n}} \!\! \ind_{B}(\by_2)\, d\nu_{\pi}(\by_1,\by_2)
\;=\; \mu_2(B) \,\,\,  \forall\, \text{Borel set } B\in \mB{(\mbR^n)},
\text{therefore}, 
\nu_{\pi}^{(2)} = \mu_2.
\label{eq:nu2-equality}
\end{eqnarray}
Moreover, $\nu_\pi \in \Pi^\mathrm{mix}(\mu_1, \mu_2)$.
\end{lemma}
\vspace{-2em}
\begin{proof}
First, for each $(\bs_1,\bs_2)\in \mS\times \mS$ and for Borel sets $A\in \mB(\mbR^n)$,
\begin{align}\label{eq: Gelbrich-Construct1}
\int_{\mbR^{2n}} \mathbf \ind_A(\by_1)\,d\mG(\bs_1, \bs_2)(\by_1,\by_2)
   =\mG(\bs_1, \bs_2)(A,\mbR^n) = \mN(\bs_1)(A),
\end{align}
\noindent where the first equality holds since $\ind_A(\by_1)$ does not depend on $y_2$, and the second equality holds since the marginal for the Gelbrich coupling is $\mN(\bs_1)$. Similarly,  $    \int_{\mbR^{2n}} \mathbf \ind_B(\by)\,d\mG(\bs_1, \bs_2)(\by_1,\by_2)
   =\mG(\bs_1, \bs_2)(\mbR^n,B) = \mN(\bs_2)(B)$
for any Borel set $B\in \mB(\mbR^n)$. Therefore, it is sufficient to prove \eqref{eq:nu1-equality}. Let $\Psi(\by_1,\by_2) :=\ind_A(\by_1)$. Clearly, $\Psi(\by_1,\by_2)$ is measurable and bounded since $0 \leq \Psi\leq 1$. Then, 
\begin{align*}
&\nu_{\pi}^{(1)}(A)
   = \int_{\mbR^{2n}}\! \ind_A(\by_1)\,d\nu_{\pi}(\by_1,\by_2) =  \int_{\mbR^{2n}}\!\Psi(\by_1,\by_2)\,d\nu_{\pi}(\by_1,\by_2)\\[2pt]
   &=  \int_{\mbR^{2n}}\!\Psi(\by_1,\by_2)\,d\big(\int_{\mS\times \mS}\mG(\bs_1, \bs_2)(\by_1,\by_2)\,d\pi(\bs_1,\bs_2)\big)  \text{\big(since $\,\nu_{\pi}(\by_1,\by_2)=\int_{\mS\times \mS}\mG(\bs_1, \bs_2)(\by_1,\by_2)\,d\pi(\bs_1,\bs_2)$\big)}\\[2pt]
    &= \int_{\mathcal{S}\times \mathcal{S}}
    \Bigl(\int_{\mbR^{2n}}\!\Psi(\by_1,\by_2)\,d\mG(\bs_1, \bs_2)(\by_1,\by_2)\Bigr)
              d\pi(\bs_1,\bs_2) \,\,\text{\big(using Kernel Fubini-Tonelli Theorem~\ref{thm: kernel-fubini-tonelli}\big)} \\[2pt]
    &= \int_{\mathcal{S}\times \mathcal{S}}
      \Bigl(\int_{\mbR^{2n}} \mathbf \ind_A(\by_1)\,d\mG(\bs_1, \bs_2)(\by_1,\by_2)\Bigr)
      d\pi(\bs_1,\bs_2) \quad \text{\big(using definition of $\Psi$\big)} \\[2pt]
   &= \int_{\mS\times \mS} \mN(\bs_1)(A)\,d\pi(\bs_1,\bs_2) \quad \text{\big(using \eqref{eq: Gelbrich-Construct1}\big)} \\
   &= \int_{\mS}   \mN(\bs_1)(A)\,d\pi_{1}(\bs_1)
        \quad \text{\big(integrate out the $\bs_2$-variable\big)} \\[2pt]
   &= \mu_1(A). \quad  \text{\big(since by definition of $\Gamma, \pi$ is a mixing law for $\mu_1$\big)} 
\end{align*}

Since this identity holds for every Borel set $A\in \mB(\mbR^n)$, the measures coincide, and $\nu_{\pi}^{(1)} = \mu_1$. Similarly, $\nu_{\pi}^{(2)}(B) = \mu_2(B)$. Hence, $\nu_\pi \in \Pi(\mu_1, \mu_2)$. 
We now prove that $\nu_{\pi} \in \mbGM_{2n}$. Since $\lambda := \mT_\#\pi \in \mP(\mZ)$, we obtain the following chain of equality:
\begin{align*}
    \nu_{\pi} = &\int_{\mS\times \mS}
        \mG(\bs_1, \bs_2)\,
        d\pi(\bs_1,\bs_2)
  \stackrel{\mathrm{Theorem}~\ref{thm: Gelbrich}}{=}\int_{\mS\times \mS}
\mN\!\bigl((\bmean_1(\bs_1),\bmean_2(\bs_2)),\bar{\bQ}(\bs_1, \bs_2)\bigr)\,
        d\pi(\bs_1,\bs_2) \\
        \stackrel{\eqref{def: OT-map-forward}}{=} &   \int_{\mS\times \mS}
        \mN\!\bigl(\mT(\bs_1, \bs_2)\bigr)\,
        d\pi(\bs_1,\bs_2) \stackrel{by~\eqref{eq: push-integral}}{=} \int_{\mZ} \mN(\bz)\,d(\mT_\# \pi)(\bz) \stackrel{by~\eqref{eq: push-integral}}{=}\displaystyle\int_{\mZ}\mN(\bz)\,d\lambda(\bz).
\end{align*}

We finally show finiteness of the second moment of $\nu_\pi$ leveraging Propositions~\ref{prop: joint-gaussian-second-moment} and~\ref{prop: finite-GMM-mmnt} from Section~\ref{sec: continuous-param-GMM}. 
First, since $\|\by\|^2\ge 0$, by the kernel Fubini-Tonelli Theorem~\ref{thm: kernel-fubini-tonelli} under the $\nu_\pi$ construction we have
\[
\int_{\mbR^{2n}} \|\by\|^2\,d\nu(\by)
=
\int_{\mZ}\left(\int_{\mbR^{2n}} \|\by\|^2\,
d\mN\!\bigl((\bmean_1,\bmean_2),\bar{\bQ}\bigr)(\by)\right)\,d\lambda(\bz).
\]
Additionally, by Proposition~\ref{prop: joint-gaussian-second-moment}, since
\begin{align*}
\int_{\mbR^{2n}} \|\by\|^2\,
d\mN\!\bigl((\bmean_1,\bmean_2),\bar{\bQ}\bigr)(\by)
=
\|(\bmean_1,\bmean_2)\|^2+\Tr(\bar{\bQ}),
\end{align*} it implies that 
\[
\int_{\mbR^{2n}}\|\by\|^2\,d\nu(\by)
=
\int_{\mZ}\Bigl(\|(\bmean_1,\bmean_2)\|^2+\Tr(\bar{\bQ})\Bigr)\,
d\lambda\bigl(z\bigr).
\]
Now note that for any $\bQ\succeq 0$, hence for $\bar{\bQ}\succeq 0$, $\Tr(\bar{\bQ}) \le \sqrt{2n}\,\|\bar{\bQ}\|_F$ and $\|\bar{\bQ}\|_F\le 1+\|\bar{\bQ}\|_F^2$~\cite{bhatia2009positive, horn2012matrix}. Hence, $\Tr(\bar{\bQ})
\le \sqrt{2n}\,\bigl(1+\|\bar{\bQ}\|_F^2\bigr).$
Therefore,
\begin{align*}
&     \|(\bmean_1,\bmean_2)\|^2+\Tr(\bar{\bQ})
\le
\|\bmean_1\|^2+\|\bmean_2\|^2+\sqrt{2n}\,\bigl(1+\|\bar{\bQ}\|_F^2\bigr) \le \sqrt{2n}\Bigl(\|\bmean_1\|^2+\|\bmean_2\|^2+ |\bar{\bQ}\|^2_F\Bigr) + \sqrt{2n} \implies \\
& \int_{\mZ} \hspace{-5pt}\Bigl(\|(\bmean_1,\bmean_2)(\bz)\|^2+\Tr(\bar{\bQ}(\bz))\Bigr)\,
d\lambda\bigl(z\bigr) \leq \sqrt{2n} \Bigl( \int_{\mZ} \hspace{-5pt}\bigl(\|\bmean_1(\bz)\|^2+\|\bmean_2(\bz)\|^2+\|\bar{\bQ}(\bz)\|_F^2\bigr)\,d\lambda(\bz) + \int_\mZ  \hspace{-5pt}\lambda(\bz) \Bigr),
\end{align*}
where the first integral on the right-hand side is finite by Proposition~\ref{prop: finite-GMM-mmnt} and the second integral is 1 since for any probability measure $\lambda \in \mP(\mZ), \int_\mZ d\lambda(\bz) = 1$. Thus, we obtain $\int_{\mbR^{2n}}\|\by\|^2\,d\nu(\by)<\infty$, i.e., $\nu\in \mP_2(\mbR^{2n})$.
Hence, $\nu_{\pi}\in \mbGM_{2n}$. Therefore, $\nu_\pi \in \Pi^{\mathrm{mix}}(\mu_1, \mu_2)$.
\end{proof}

\subsection{Construction of Parameter-space Coupling from Data Space Coupling}
The following lemma shows that a parameter-space coupling $\pi_\nu \in \Gamma$ can be constructed from a given data-space coupling $\nu \in \Pi^\mathrm{mix}(\mu_1, \mu_2)$.

\begin{lemma}\label{lem: Gamma-mixture-to-joint-gauss}
Let $\mM : \mZ \to \mS \times \mS$ be a Borel measurable map defined in~\eqref{def: OT-map-backward}. Assume $\nu \in \Pi^\mathrm{mix}(\mu_1, \mu_2)$ is given with corresponding mixing law $\lambda \in \mP(\mZ)$. Let $\pi_\nu = \mM_\# \lambda$ and $\pi_\nu^{(1)}, \, \pi_\nu^{(2)}$ be the marginals of $\pi_\nu$ on $\mP(\mS)$. Then \, (i) $\pi_\nu^{(1)}$ is a mixing law for $\mu_1$, \, (ii) $\pi_\nu^{(2)}$ is a mixing law for $\mu_2$.
Therefore, $\pi_\nu \in \Gamma$.
\end{lemma}

\begin{proof}
Since $\nu \in \Pi^\mathrm{mix}(\mu_1, \mu_2) = \{\Pi(\mu_1, \mu_2) \cap \mbGM_{2n}\}$, we have
\begin{align}\label{eq: given-identity}
 \nu^{(1)} = \mu_1, \quad \nu^{(2)} = \mu_2, \quad \nu = \int_\mZ \mN(\bz)\,d\lambda(\bz).     
\end{align}
Also, given a map $\mM$, $\pi_\nu(E) = \mM_{\#} \lambda(E)$ for all Borel set $E \in \mB(\mS \times \mS)$. It is sufficient to prove that the first marginal of $\pi_\nu$ is a mixing law for $\mu_1$.  In particular, we show that $   \mu_1(B)
= \int_{\mathcal{S}} \mN(\bs_1)(B)\,d\pi_\nu^{(1)}(\bs_1)$ for any Borel set $B\in \mB(\mbR^n)$. We proceed by first expressing $\mu_1(B)$ via $\nu$ and then via $\lambda$ through the following steps. Let $\Psi(\by_1,\by_2) :=\ind_A(\by_1)$. Clearly, $\Psi(\by_1,\by_2)$ is measurable and bounded since $0 \leq \Psi\leq 1$. Then, 
\begin{align}
       \mu_1(B)
   & =
   \int_{\mbR^n} \ind_B(\by_1)\,d\mu_1(\by_1) =
   \int_{\mbR^n} \ind_B(\by_1)\,d\nu^{(1)}(\by_1) \quad \text{\big(since $\nu^{(1)} = \mu_1$\big)} \nonumber\\
   & =
   \int_{\mbR^n\times\mbR^n} \Psi(\by_1,\by_2)\,d\nu(\by_1,\by_2) \quad \text{\big(By the definition of the first marginal $\nu^{(1)}$\big)} \nonumber\\
   & =\int_{\mbR^n\times\mbR^n} \Psi(\by_1,\by_2)\,d\biggl[ \int_{\mZ} \mN(\bz)\,d\lambda(\bz)\biggr](\by_1,\by_2) \, \, \text{\big(Using the mixture representation $\nu = \displaystyle\int_{\mZ} \mN(\bz)\,d\lambda(\bz)$\big)} \nonumber\\
&=  \int_{\mZ}
      \biggl[        \int_{\mbR^n\times\mbR^n} \Psi(\by_1,\by_2)\,d\mN(\bz)(\by_1,\by_2)
      \biggr]
   d\lambda(\bz) \quad \text{\big(By the kernel Fubini--Tonelli Theorem~\ref{thm: kernel-fubini-tonelli}\big)} \nonumber\\
   &=
   \int_{\mZ} \mN(\bs_1(\bz))(B)\,d\lambda(\bz), \label{eq:mu0-as-theta}
\end{align}
where we obtain the last equality as follows: since $(Y_1, Y_2) \sim \mN((\bmean_1, \bmean_2), \bQ)$, $ Y_1 \sim \mN(\bmean_1, \bQ_{11}) = \mN(\bs_1(\bz))$. Then, using the coordinate projection operation, for any Borel set $B\in\mB(\mbR^n)$, 
\begin{align*}
  \mN(\bs_1(\bz))(B)
\stackrel{by~\eqref{def: proj-param}}{=}\mN(\bz)\bigl(\mathrm{pr}_1^{-1}(B)\bigr) 
\stackrel{by~\eqref{eq: proj-preimage}}{=}\int_{\mbR^n\times\mbR^n}\ind_B(\by_1)\,d\mN(\bz)(\by_1,\by_2).  
\end{align*}

Now, our goal is to replace $\lambda$ by $\pi_\nu^{(1)}$ in~\eqref{eq:mu0-as-theta}. Since $\mathcal{M}_{\#}\lambda = \pi_\nu$ with $\mathcal{M}(\bz) = (\bs_1(\bz),\bs_2(\bz))$ and $\mathrm{pr}_1\bigl(\mM(\bz)\bigr) = \bs_1(\bz)$, by integration rule~\eqref{eq: push-integral-M} and finally by definition of $\pi_\nu^{(1)}$,
\begin{align*}
\int_{\mZ} \mN(\bs_1(\bz))(B)\,d\lambda(\bz)
& \stackrel{}{=} \int_{\mZ} \mN\!\Bigl(\mathrm{pr}_1\bigl(\mM(\bz)\bigr)\Bigr)(B)\,d\lambda(\bz) \stackrel{by~\eqref{eq: push-integral-M}}{=}  \int_{\mS\times\mS} \mN\!\bigl(\mathrm{pr}_1(\bs_1,\bs_2)\bigr)(B)\,d(\mM_\#\lambda)(\bs_1,\bs_2) \\
& \stackrel{by~\eqref{eq: push-integral-M}}{\;=\;}\int_{\mS\times\mS} \mN(\bs_1)(B)\,d\pi_\nu(\bs_1,\bs_2) = \int_{\mS} \mN(\bs_1)(B)\,d\pi_\nu^{(1)}(\bs_1).
\end{align*}

Combining this equality with \eqref{eq:mu0-as-theta}, we obtain
\[
   \mu_1(B)
   =
   \int_{\mathcal{S}} \mN(\bs_1)(B)\,d\pi_\nu^{(1)}(\bs_1),
   \qquad \forall\,B \in \mathcal{B}(\mbR^n),
\]
which proves that $\pi_\nu^{(1)}$ is a mixing law for $\mu_1$.

Additionally, by Proposition~\ref{prop: finite-2nd-moment}, for any given $\nu$, the second moment of $\pi_\nu$ is finite. Therefore, $\pi_\nu \in \Gamma$.
\end{proof}

\subsection{Parameter Space Wasserstein-2 Metric Equality}
We are now in a position to prove the main result. In the following, Proposition~\ref{thm: R<=L} shows that $W_2^\mathrm{mix} \leq W_2^\mathrm{prm}$, whereas Proposition~\ref{thm: R>=L} shows $W_2^\mathrm{mix} \geq W_2^\mathrm{prm}$ and thereby, proves Theorem~\ref{thm: Wass2-Gaussian-EquivalenceTheorem}.

\begin{proposition}\label{thm: R<=L}
 Let $\Pi(\mu_1, \mu_2), \mbGM_{2n}$, $\Gamma$ and $W_2^2$ be as defined  in~\eqref{def: coupling-Euclidean}, \eqref{def: GMMmix}, \eqref{def: coupling-param} and \eqref{def: distance-matric}, respectively. 
 Then 
 $$\,\, W_2^\mathrm{mix} = \inf_{\nu \in \Pi^\mathrm{mix}(\mu_1, \mu_2)} \int_{\mbR^n \times \mbR^n} \|\by_1 - \by_2\|^2 \,d\nu(\by_1,\by_2)  \leq \int_{\mS\times \mS} 
d^2_{\mathrm{BW}}\bigl(\mN(\bs_1),\mN(\bs_2)\bigr)\,d\pi(\bs_1,\bs_2)$$ for any $\pi \in \Gamma$. Hence, $  W_2^\mathrm{mix} \leq \inf_{\pi \in \Gamma} \int_{\mS\times \mS}
          d^2_{\mathrm{BW}}\bigl(\mN(\bs_1),\mN(\bs_2)\bigr)\,d\pi(\bs_1,\bs_2)$.
\end{proposition}

\begin{proof}
Since $\|\by_1 - \by_2\|^2 \ge 0$, and our constructed probability measure $\nu_{\pi}:=\int_{\mS\times \mS}\mG(\bs_1, \bs_2)\,d\pi(\bs_1,\bs_2)$ for a given $\pi \in \Gamma$, using Theorem~\ref{thm: kernel-fubini-tonelli} (Kernel Fubini-Tonelli), we have
\[
  \int_{\mbR^n\times\mbR^n} \|\by_1 - \by_2\|^2\,d\nu_{\pi}
    =
  \int_{\mS\times \mS}
        \Bigl(\int_{\mbR^n\times\mbR^n} \|\by_1 - \by_2\|^2\,
              d\mG(\bs_1, \bs_2)\Bigr)
        d\pi(\bs_1,\bs_2).
\]
Additionally, from Theorem~\ref{thm: Gelbrich}, for each fixed pair $(\bs_1,\bs_2)$,
\[
  \int_{\mbR^n\times\mbR^n} \|\by_1 - \by_2\|^2\,
        d\mG(\bs_1, \bs_2)
    \;=\;
  d^2_{\mathrm{BW}}\bigl(\mN(\bs_1),\mN(\bs_2)\bigr).
\]
\text{Hence,} 
\begin{align}\label{eq: integral_identity_dbw}
     \int_{\mbR^n\times\mbR^n} \|\by_1 - \by_2\|^2\,d\nu_{\pi}
     = \int_{\mS\times \mS}
          d^2_{\mathrm{BW}}\bigl(\mN(\bs_1),\mN(\bs_2)\bigr)\,d\pi(\bs_1,\bs_2).
\end{align}

Now, because $\nu_{\pi}\in \mbGM_{2n}$ by Lemma~\ref{lem: forward}, $\nu_{\pi}$ has finite second moment. Hence, the above integral is finite. Additionally,  considering all elements $\nu$ from $\Pi^\mathrm{mix}(\mu_1, \mu_2)$, one of which is $\nu_{\pi}$, we obtain $\inf_{\nu\in \Pi^\mathrm{mix}(\mu_1, \mu_2)}\int_{\mbR^{2n}}\|\by_1 - \by_2\|^2\,d\nu
  \;\le\;
  \int_{\mbR^n\times\mbR^n}
    \|\by_1 - \by_2\|^2\,d\nu_{\pi}$. Therefore,
\[
  W_2^\mathrm{mix}=\inf_{\nu\in \Pi^\mathrm{mix}(\mu_1, \mu_2)}\int_{\mbR^{2n}}\|\by_1 - \by_2\|^2\,d\nu
  \;\le\;
  \int_{\mbR^n\times\mbR^n}
    \|\by_1 - \by_2\|^2\,d\nu_{\pi}
  \stackrel{by~\eqref{eq: integral_identity_dbw}}{\;=\;}
  \int_{\mS\times \mS}d^2_{\mathrm{BW}}\bigl(\mN(\bs_1), \mN(\bs_2)\bigr)\,d\pi(\bs_1, \bs_2).
\]
Since this inequality holds for all $\pi \in \Gamma$,  $ W_2^\mathrm{mix}\,\, \leq \,\, W_2^\mathrm{prm}$.
\end{proof}

We now establish $ W_2^\mathrm{mix}\,\, \geq \,\, W_2^\mathrm{prm}$.


\begin{proposition}\label{thm: R>=L}
 Let $\Pi(\mu_1, \mu_2), \mbGM_{2n}$, $\Gamma$ and $W_2^2$ be as defined  in~\eqref{def: coupling-Euclidean}, \eqref{def: GMMmix}, \eqref{def: coupling-param} and \eqref{def: distance-matric}, respectively. Then 
\[\,\, \int_{\mbR^n \times \mbR^n}
      \|\by_1 - \by_2\|^2 \,d\nu(\by_1,\by_2)  \geq \inf_{\pi \in \Gamma}\int_{\mS\times \mS}           d^2_{\mathrm{BW}}\bigl(\mN(\bs_1),\mN(\bs_2)\bigr)\,d\pi(\bs_1, \bs_2) = W_2^\mathrm{prm}
      \]
for any $\nu \in \Pi^\mathrm{mix}(\mu_1, \mu_2)$. Hence, $W_2^\mathrm{mix} \geq W_2^\mathrm{prm}$.
\end{proposition}

\begin{proof}
Let $\nu_z := \mN(\bz)\in\mP(\mbR^n\times\mbR^n)$. We first show that the following inequality holds pointwise in $z\in \mZ$
\vspace{-1.5em}
\begin{align}\label{eq: pointwise-inq}
      \underbrace{\int_{\mbR^n\times\mbR^n}
                 \|\by_1 - \by_2\|^2\,
                 d\nu_z(\by_1, \by_2)}
              _{\text{cost of \emph{this} coupling}}
  \;\;\ge\;\;
  \underbrace{d^2_{\mathrm{BW}}\!\bigl(\mN(\bs_1(\bz)),\mN(\bs_2(\bz))\bigr)}
              _{\text{lowest possible cost}}.
\end{align}

By construction, the marginals of $\nu_z$ are $\nu^{(1)}_z=\mN(\bs_1(\bz)),\,\, \nu^{(2)}_z=\mN(\bs_2(\bz))$. Hence, $\nu_z\in\Pi\bigl(\mN(\bs_1(\bz)),\mN(\bs_2(\bz))\bigr)$. Then considering all the elements (coupling) of $\Pi\bigl(\mN(\bs_1(\bz)),\mN(\bs_2(\bz))\bigr)$,
\[
\int_{\mbR^n\times\mbR^n}
                 \|\by_1-\by_2\|^2\,
                 d\nu_z(\by_1, \by_2) 
\, \geq\,
\inf_{\nu\in\Pi(\mN(\bs_1(\bz)),\,\mN(\bs_2(\bz)))}
\int_{\mbR^n\times\mbR^n}\|\by_1-\by_2\|^{2}\,d\nu(\by_1,\by_2),
\]
where the right term is equal to $d^2_{\mathrm{BW}}\!\bigl(\mN(\bs_1(\bz)),\mN(\bs_2(\bz))\bigr)$ by definition. Hence, inequality \eqref{eq: pointwise-inq} follows. 

We next obtain
\(
  \displaystyle
  \int_{\mS\times \mS} d^2_{\mathrm{BW}}\bigl(\mN(\bs_1),\mN(\bs_2)\bigr)\,d\pi_\nu(\bs_1,\bs_2)
\)
as an equality to an integral of the right-hand side of~\eqref{eq: pointwise-inq} over $\lambda(\bz)$ using the push-forward technique.

Let $f(\bs_1(\bz),\bs_2(\bz)):=d^2_{\mathrm{BW}}\bigl(\mN(\bs_1(\bz)),\mN(\bs_2(\bz))\bigr)$. Then, by~\eqref{def: OT-map-backward} and integration identity~\eqref{eq: push-integral-M}
\begin{align}
   \hspace{-3em} & \int_{\mZ}      \hspace{-5pt}   d^2_{\mathrm{BW}}\bigl(\mN(\bs_1(\bz)),\mN(\bs_2(\bz))\bigr)\,d\lambda(\bz) = \int_{\mZ}  \hspace{-5pt} f\bigl(\bs_1(\bz), \bs_2(\bz)\bigr)\,d\lambda(\bz) 
    \stackrel{\eqref{def: OT-map-backward}}{=}  \int_{\mZ}  \hspace{-5pt} f \bigl(\mM(\bz)\bigr)\,d\lambda(\bz) \stackrel{by~\eqref{eq: push-integral-M}}{=} \label{eq: d_{BW}^2(s)-to-d_{BW}^2(s_2(z))}  \\
    & \int_{\mS\times \mS}  \hspace{-10pt} f(\bs_1, \bs_2)\,d(\mM_{\#}\lambda)(\bs_1, \bs_2) = 
    \int_{\mS\times \mS} \hspace{-15pt} f(\bs_1, \bs_2)\,d\pi_\nu(\bs_1, \bs_2) =\int_{\mS\times \mS}  \hspace{-15pt} d^2_{\mathrm{BW}}\bigl(\mN(\bs_1),\mN(\bs_2)\bigr)\,d\pi_\nu(\bs_1, \bs_2). \nonumber
\end{align}
From~\eqref{eq: pointwise-inq} we can write
      \[ \int_{\mZ}
           \biggl[
              \int_{\mbR^{2n}}
                   \|\by_1 - \by_2\|^2\,
                   d\nu_z(\by_1, \by_2)
           \biggr] d\lambda(\bz) \,\, \geq \,\, \int_{\mZ}
        d^2_{\mathrm{BW}}\bigl(\mN(\bs_1(\bz)),\mN(\bs_2(\bz))\bigr)\,d\lambda(\bz), \]
and using \eqref{eq: d_{BW}^2(s)-to-d_{BW}^2(s_2(z))} implies that
\[\int_{\mZ}
           \biggl[
              \int_{\mbR^{2n}}
                   \|\by_1 - \by_2\|^2\,
                   d\nu_z(\by_1, \by_2)
           \biggr] d\lambda(\bz) \,\, \geq \,\,   \int_{\mS\times \mS} d^2_{\mathrm{BW}}\bigl(\mN(\bs_1),\mN(\bs_2)\bigr)\,d\pi_\nu(\bs_1 \bs_2).
      \]
      
Additionally, since $\|y_1 - y_2\|^2 \geq 0$, $\nu_z := \mN(\bz)$ and  $\nu(\by_1, \by_2)=\int_{\mZ} \nu_z(\by_1, \by_2)\,d\lambda(\bz)=\int_{\mZ} \mN(\bz)(\by_1, \by_2)\,d\lambda(\bz)$,
\begin{align*}
 \int_{\mZ}
     &      \biggl[
              \int_{\mbR^{2n}}
                   \|\by_1 - \by_2\|^2\,
                   d\mN(\bz)(\by_1,\by_2)
           \biggr] d\lambda(\bz) =         
              \int_{\mbR^{2n}}
               \|\by_1 - \by_2\|^2\,d\,
               \biggl[\int_{\mZ}
                   \mN(\bz)(\by_1,\by_2)
           d\lambda(\bz)  \biggr] \\
   & =\int_{\mbR^{2n}}
           \|\by_1 - \by_2\|^2\,d\nu(\by_1, \by_2). \quad  \text{Thus}, \\
    \quad I(\nu) \,\, & \coloneqq \,\, \int_{\mbR^{2n}}
           \|\by_1 - \by_2\|^2\,d\nu(\by_1, \by_2) \,\, \geq \,\, \int_{\mS\times \mS} d^2_{\mathrm{BW}}\bigl(\mN(\bs_1),\mN(\bs_2)\bigr)\,d\pi_\nu(\bs_1, \bs_2) \,\, \forall \,  \nu\in\Pi^\mathrm{mix}(\mu_1, \mu_2).        
\end{align*}

Moreover, since $\pi_\nu$ is one admissible element of $\Gamma$ (see Lemma~\eqref{lem: Gamma-mixture-to-joint-gauss}), considering all elements $\pi \in \Gamma$ we have
\[
 I(\nu) \, \geq \, \int_{\mS\times \mS} d^2_{\mathrm{BW}}\bigl(\mN(\bs_1),\mN(\bs_2)\bigr)\,d\pi_\nu(\bs_1, \bs_2) \, \geq \,\, \inf_{\pi\in\Gamma} \int_{\mS\times \mS} d^2_{\mathrm{BW}}\bigl(\mN(\bs_1),\mN(\bs_2)\bigr)\,d\pi(\bs_1, \bs_2)  \,=\, W_2^\mathrm{prm}.  
\]
The above inequality holds for all $\nu\in\Pi^\mathrm{mix}(\mu_1, \mu_2)$. Therefore,
$ \inf_{\nu \in \Pi^\mathrm{mix}(\mu_1, \mu_2)} I(\nu) \, = \, W_2^\mathrm{mix} \geq W_2^\mathrm{prm}$. 
\end{proof}

In the following sections, we use the Gaussian-mixture ambiguity set induced by the Wasserstein-2-type metric introduced above to formulate the distributionally robust chance-constrained problem and to solve its strong-duality-based semi-infinite reformulation. In these developments, the chance-constraint expression is viewed as a function of both the decision vector $\bx$ and the Gaussian component parameters. 
To ensure that the resulting formulation is well-posed and for the algorithmic analysis, we require regularity conditions on the parameter space of the component distributions (see Assumption~\ref{assmp: compactness}). Theorem~\ref{thm:Wass2-Gaussian-EquivalenceTheorem-restricted} (provided in Appendix~\ref{appndx: restricted}) is a counterpart of Theorem~\ref{thm: Wass2-Gaussian-EquivalenceTheorem} under the assumption that the covariance matrices of the Gaussians are positive definite and the support for mean--covariance is compact.
\begin{theorem}[Wasserstein-type distance on compact parameter support]
\label{thm:Wass2-Gaussian-EquivalenceTheorem-restricted}
Let $\mS:=\Theta\times\Xi\subset\mbR^n_+\times\mbS^n_{++}$ be nonempty and compact. Define
\[
\Gamma_{\mS}
:=
\left\{
\pi\in\mP(\mS\times\mS):
\int_{\mS}\mN(\bs)(A)\,d\pi_1(\bs)=\mu_1(A),\ 
\int_{\mS}\mN(\bs)(B)\,d\pi_2(\bs)=\mu_2(B),\ 
\forall A,B\in\mB(\mbR^n)
\right\}.
\]
For this restricted support $\mS$, let
\[
W_{2,\mS}^{\mathrm{prm}}(\mu_1,\mu_2)
:=
\inf_{\pi\in\Gamma_{\mS}}
\int_{\mS\times\mS}
d^2_{\mathrm{BW}}\bigl(\mN(\bs_1),\mN(\bs_2)\bigr)\,d\pi(\bs_1,\bs_2),
\]
and
\begin{align}\label{def:W2-mix-restricted}
    W_{2,\mS}^{\mathrm{mix}}(\mu_1,\mu_2)
:=
\inf_{\nu\in\Pi(\mu_1, \mu_2)\cap\mbGM_{2n}^{\mS}}
\int_{\mbR^n\times\mbR^n}\|\by_1-\by_2\|_2^2\,d\nu(\by_1,\by_2),
\end{align}
where $\mbGM_{2n}^{\mS}$ is the class of Gaussian-mixture couplings whose component parameters are induced by elements of $\mS\times\mS$. Then $\Pi(\mu_1, \mu_2)\cap\mbGM_{2n}^{\mS}$ is non-empty and an optimal solution to~\eqref{def:W2-mix-restricted} exists. Additionally, every $\pi\in\Gamma_{\mS}$ has finite second moment, and
\begin{equation}
    W_{2,\mS}^{\mathrm{mix}}(\mu_1,\mu_2)
    =
    W_{2,\mS}^{\mathrm{prm}}(\mu_1,\mu_2).
    \label{eqn:Wass2-Dist-ParamSpace-restricted}
\end{equation} 
\end{theorem}

\section{Strong Duality of Robust CC and Semi-infinite Formulation}{\label{sec: formulation&Duality}

We now return to the distributionally robust chance-constrained problem~\eqref{prob: DR-CCP} where the ambiguity set is defined using the Wasserstein-type parameter space coupling developed in the previous section. From this section onward, we focus on the compact support set $\mS:=\Theta\times\Xi\subset\mbR^n_+\times\mbS^n_{++}$ as introduced above in Theorem~\ref{thm:Wass2-Gaussian-EquivalenceTheorem-restricted}. We first derive a strong dual form of the resultant robust chance constraint and thereby develop a semi-infinite program formulation to~\eqref{prob: DR-CCP}. We assume that the nominal distribution estimated from data is a mixture of finitely many Gaussians. In particular, the nominal distribution $\hat{P} \in 
\mbGM_n(\hat{\mS}):=\Bigl\{\sum_{k=1}^{\hat{K}} w^0_k\,\mN(\hat{\bs}_k) \mid
\hat{\bw}\in\hat{\mW}, \hat{\bs} \in \hat{\mS} \Bigr\}
$, where $\hat{\mW}:=\{\bw\in\mbR^{\hat{K}}_+ \mid \sum_{k=1}^{\hat{K}} w_k=1\}, \, \hat{\mS} := \bigl\{(\bmean, \bQ) \in \mS \mid |\hat{\mS} | = \hat{K} \bigr\}$ 
for some $\hat{K}\in\mathbb N$.  Note that the nominal distribution is continuous despite the mixing distribution being finite. 

We now define the ambiguity set of interest $\mfD$ using the coupling set $\Gamma$ and  metric $W^\mathrm{prm}_2$:
\begin{equation}\label{eq: W-type-ambiguity}
\hspace{-1.5em}\mathfrak{D}
:=\left\{\pi^{(2)} \in \mathcal P_2(\mS) \ \middle|\ 
\exists\,\pi \in \Gamma\subseteq \mP_2(\hat{\mS} \times \mS)\ \,\, \text{with}\quad
\begin{aligned}
& \sum_{k=1}^{\hat{K}} \int_{\mS} d^2_{\mathrm{BW}} \bigl(\mN(\hat{\bs}_k),  \mN(\bs) \bigr)\, d\pi(\hat{\bs}_k, \bs) \leq \rho, \\
& \pi^{(1)} := \int_{\mS} d\pi(\hat{\bs}, \bs) = \hat{\bm{w}}
\end{aligned}
\right\}, 
\end{equation}
where, each $\pi^{(2)}\in\mathcal P_2(\mS)$ induces a Gaussian mixture distribution on $\mbR^n$ via $\int \mN(\bmean,\bQ)\,d\pi^{(2)}(\bmean,\bQ)$. Observe that the mixing law $\pi^{(2)}\in\mathcal P_2(\mS)$ is allowed to freely vary over a continuum of measures on the parameter space $\mS$, unlike the restriction of marginal $\pi^{(1)}$ of $\pi$ to $\hat{\bw} \in \hat{\mW}$. Moreover, since \(\mathfrak D\) is the projection of the feasible transport plans onto the second marginal, and the function \(G(\bs;\bx)\) depends only on the second
coordinate \(\bs\), optimizing over \(\pi^{(2)}\in\mathfrak D\) is equivalent to
optimizing over the corresponding feasible plans \(\pi\in\mD\) with
\begin{align}
\hspace{-2em}    \mD:=\left\{\pi\in \Gamma \middle|  \int_{\mS} d\pi(\hat{\bs}_k, \bs) = \hat{w}_k \,\forall k \in [\hat{K}],
\int_{\hat{\mS}\times\mS}d^2_{\mathrm{BW}}\bigl(\mN(\hat{\bs}),\mN(\bs)\bigr)\,d\pi(\hat{\bs},\bs)\le \rho\right\}.\label{def: M-Wass-Ambiguity-Cont}
\end{align}
Hence, we equivalently represent distributionally robust chance-constrained program~\eqref{prob: DR-CCP} as~\eqref{prob: DR-CCP-GMM}: 
\begin{align}
 & \min_{\bx \in \mX} \;\; \bc^\top \bx \quad \mathrm{s.t.} \quad \sum_{k=1}^{\hat{K}} \int_{\bs \in \mS} G(\bs, \bx) \; d\pi \bigl(\hat{\bs}_k, \bs\bigr) \geq \theta \qquad \forall \pi \in \mD,  \label{prob: DR-CCP-GMM} \tag{CDR}  \\
&  \text{where}  \qquad  \; G(\bs, \bx) \coloneqq
    \begin{dcases*}\label{def: G}
        \mathbbm{1}_{\geq 0}(b), & if $\bx = \bzero$, \\
        \Phi \left( \frac{b - \bmean^\top \bx}{\sqrt{\bx^\top \bQ \bx}} \right) & otherwise.
    \end{dcases*}
\end{align}
An additional restriction in the definition of $G(\bs,\bx)$ is given in Assumption~\ref{assmp: compactness}. Also, the isolated instance at $\bx =\bzero$ with $\mathbbm{1}_{\geq 0}(b)$ is handled in Propositions~\ref{prop: infeasibility} and \ref{prop: infeasibility-1} below.
\begin{remark}
    We emphasize that the formulation~\eqref{prob: DR-CCP-GMM} does not take support in the data space $\mbR^n$ (also called observation or sample-value space) as in the classical Wasserstein definition. In the classical setting, the random element is $\bxi\in\mbR^n$. In contrast, the random element in~\eqref{prob: DR-CCP-GMM} is the Gaussian component parameter $\bs=(\bmean,\bQ)$, and therefore the relevant support set is $\mS\subseteq \mbR^n\times \mbS_+^n$ with $\bs\in\mS$. In the finite case, the classical model fixes a finite set of support points $\{\bxi_k\}_{k=1}^{\hat{K}}\subset\mbR^n$ and the distributional decision reduces to choosing weights on these points. Likewise, a finite counterpart of our model fixes a finite set of parameter support points $\{\bs_k\}_{k=1}^{\hat{K}}$  (typically the nominal component parameters $\{\hat{\bs}_k\}$), so the distributional decision also reduces to selecting weights on these known support points. However, for the continuous case as this study focuses on, the support points $\bs\in\mS$ are not fixed \emph{a priori}. Hence, the optimization needs to select: (i) where to place mass in $\mS$ (i.e., which Gaussian parameter values to use) and (ii) how much mass to assign.
\end{remark}
Theorem~\ref{thm: strong-duality} below provides an equivalent formulation of~\eqref{prob: DR-CCP-GMM} as a semi-infinite program  \eqref{prob: semif-infinite} by establishing strong duality for the robust chance constraint.  This result is developed under the assumption (Assumption~\ref{assmp: compactness}) that the Gaussian parameter set $\mS:=\Theta\times\Xi$ is compact and restricted to strictly positive definite covariance matrices with a lower bound on the smallest eigenvalue. Therefore, the Gaussian distributions on affine subspaces of $\mathbb{R}^n$ are excluded from the mixture components. This strict positive definiteness is used in proving the continuity result in Lemma~\ref{lem: G-continuity-mu-sigma}. Additionally, the compactness of the parameter set in Assumption~\ref{assmp: compactness} provides a condition needed for 
the (finite) convergence of the cutting surface algorithm we develop in Section~\ref{sec: Dual-formulation} to solve~\ref{prob: DR-CCP-GMM}. Problem \eqref{prob: DR-CCP-GMM} involves infinitely many decision variables, whereas its dual counterpart \eqref{prob: semif-infinite} has infinitely many constraints. 

\begin{assumption}\label{assmp: compactness}
Covariance matrices $\bQ$ appearing in the Gaussian mixture distribution belong to $\mbS^n_{\hat{\delta}} := \{\bQ \in \mbS^n_{++}: \bQ \succeq \hat{\delta} \bm{I}\}$ for some $\hat{\delta} >0$. Additionally, the Gaussian parameters $(\bmean,\bQ)$ belong to a compact set $\mS := (\Theta\times\Xi) \subseteq \mbR^n\times\mbS^n_{\hat\delta}$.
\end{assumption}


\begin{theorem}[Semi-infinite program formulation]
\label{thm: strong-duality}
Let Assumption~\ref{assmp: compactness} hold, and suppose Wasserstein radius \(\rho>0\). Then,
for any \(\bx\in\mX\), \(\bx\) satisfies the robust constraint in
\eqref{prob: DR-CCP-GMM} if and only if there exists
\(\bbeta\in\mathbb R^{\hat K}\times\mathbb R_+\) such that $\sum_{k=1}^{\hat K}\hat w_k\beta_k-\rho\beta_{\hat K+1}\ge\theta$ and $\beta_k-\beta_{\hat K+1}
d_{\mathrm{BW}}^2
\left(
\mN(\hat{\bs}_k),\mN(\bs)
\right)
\le
G(\bs,\bx),
\,\,\,
\forall \bs\in\mS,\ \forall k\in[\hat K]$. Consequently, solving~\eqref{prob: DR-CCP-GMM} is equivalent to solving the following semi-infinite program
\begin{equation}\label{prob: semif-infinite}
\renewcommand{\arraystretch}{1.15}
\begin{array}{@{}l l@{}}
\multirow{2}{*}{$\begin{aligned}
\hspace{-0.8em}\min_{\bx\in\mX,\;\bbeta\in\mathbb R^{\hat K}\times\mathbb R_+}\quad
& \bc^\top\bx \quad \mathrm{s.t.}\quad
\\
\end{aligned}$}
&
\displaystyle
\sum_{k=1}^{\hat K}\hat w_k\beta_k-\rho\beta_{\hat K+1}\ge\theta,
\\[4pt]
&
\displaystyle
\beta_k-\beta_{\hat K+1}
d_{\mathrm{BW}}^2
\left(
\mN(\hat{\bs}_k),\mN(\bs)
\right)
\le
G(\bs,\bx),
\qquad
\forall \bs\in\mS,\ \forall k\in[\hat K].
\end{array}
\tag{S-Inf}
\end{equation}
In particular, the formulations~\eqref{prob: DR-CCP-GMM} and~\eqref{prob: semif-infinite} have the same optimal objective value, and their
optimal \(\bx\)-solutions coincide whenever an optimal solution exists.
\end{theorem}

The ~\eqref{prob: DR-CCP-GMM} and~\eqref{prob: semif-infinite}  formulations are infeasible at $\bx=\bzero$ whenever $b<0$, as shown in Propositions~\ref{prop: infeasibility} and~\ref{prop: infeasibility-1} respectively below. 

\begin{proposition}\label{prop: infeasibility}
If $b<0$, then $\bx=\bzero$ is infeasible for \eqref{prob: DR-CCP-GMM}. 
Additionally, if $b \geq 0$, then for  $\bx=\bzero$, \eqref{prob: robust-CC} admits an optimal solution $\check{\pi}$ with $\bw=\hat{\bw}$, i.e., $\check{\pi}$ is a self-coupling of $\hat{\bw}$.
\end{proposition}
\begin{proof}
    See proof in Appendix~\ref{appndx:cut-alg-proof}.
\end{proof}

\begin{proposition}\label{prop: infeasibility-1}
    Let $\bbeta \in \mbR^{\hat{K}}$. Then the following holds:
    \begin{enumerate}[leftmargin=1.5em]
        \item[i)] if $\|\bbeta\| = 0$, then ~\eqref{prob: semif-infinite} is infeasible.
        \item[ii)] If $b <0$, then $\bx = 0$ cannot be a feasible solution to~\eqref{prob: semif-infinite}.
        \item[iii)] if $ b \geq 0$, for $\bx = \bzero$ and $k \in [\hat{K}]$, $\hat{\bs}_k$  is an optimal solution to the following problem: 
        $$\max_{\bs \in (\Theta \times \Xi)} \beta_k - \beta_{\hat{K}+1} d^2_{\mathrm{BW}}\bigl(\mN(\hat{\bs}_k), \mN(\bs) \bigr) - G(\bs; \bzero).$$
        \vspace{-2em}
    \end{enumerate}
\end{proposition}
\begin{proof}
    See proof in Appendix~\ref{appndx:cut-alg-proof}.
\end{proof}

Theorem~\ref{thm: strong-duality} follows from considering a dual of: 
\begin{align}\label{prob: robust-CC}
     \min_{\pi \in \mD} \;\sum_{k=1}^{\hat{K}} \int_{\bs \in \Theta \times \Xi} G(\bs, \bx) \; d\pi \bigl(\hat{\bs}_k, \bs\bigr).
\end{align}
Lemma~\ref{lem: strong-duality} establishes this intermediate result.
\begin{lemma}[Strong Duality]\label{lem: strong-duality}
    Let us consider the following problem:
\begin{align}\label{prob: robust-CC-dual}
\max_{\bbeta \in \mbR^{\hat{K}+1}} &\quad  \sum_{k=1}^{\hat{K}} \hat{w}_k\beta_k - \rho \beta_{\hat{K}+1}, \;\; \beta_{\hat{K}+1} \geq 0, \\ 
\mathrm{s.t.} &  \quad \beta_k - \beta_{\hat{K}+1} d^2_{\mathrm{BW}}\bigl(\mN(\hat{\bs}_k), \mN(\bs) \bigr) \leq G(\bs; \bx), \qquad \forall \bs\in \mS, \;\; \forall k \in [\hat{K}] \nonumber
\end{align}
where $G(\bs; \bx)$ is as defined in~\eqref{def: G}. Then, under Assumption~\ref{assmp: compactness}, $\Val\eqref{prob: robust-CC} = \Val\eqref{prob: robust-CC-dual}$.
\end{lemma}

We prove the strong-duality claim in Lemma~\ref{lem: strong-duality} using
\cite[Theorem~1(a)]{blanchet2019quantifying}. To apply this
result, we first show the continuity of both
\(d^2_{\mathrm{BW}}\bigl(\mN(\hat{\bs}),\mN(\bs)\bigr)\) and
\(G(\bs; \bx)\) in \(\bs\), respectively, in
Lemma~\ref{lem: Delta-Continuity} and
Lemma~\ref{lem: G-continuity-mu-sigma}. We then complete the proof of
Lemma~\ref{lem: strong-duality} by using the boundedness
\(0\le G(\bs;\bx)\le 1,
\, \forall \bs\in\mS,\ \forall \bx\in\mX\), which verifies the integrability
condition in \cite[Theorem~1(a)]{blanchet2019quantifying} (see (A3) in
 Appendix~A(vi)).


\begin{lemma}\label{lem: Delta-Continuity}
Let $(\Theta \times \Xi)$ be as defined in Assumption~\ref{assmp: compactness}. Fix $(\hat{\bmean}, \hat{\bQ}) \in (\Theta \times \Xi)$ and define $d^2_{\mathrm{BW}}\bigl(\bmean, \bQ; (\hat{\bmean}, \hat{\bQ})\bigr) \coloneqq \|\hat{\bmean} - \bmean\|^2 + \Tr\Bigl(\hat{\bQ}+\bQ -2\,(\bQ^{1/2}\hat{\bQ}\bQ^{1/2})^{1/2}\Bigr)$.
Then the map $(\bmean, \bQ) \mapsto d^2_{\mathrm{BW}}\big(\bmean,\bQ;(\hat{\bmean},\hat{\bQ})\big)$ is convex on $\mbR^n\times \mbS_+^n$. Moreover, $d^2_{\mathrm{BW}}\big(\bmean,\bQ;(\hat{\bmean},\hat{\bQ})\big)$ is continuous on $(\Theta \times \Xi)$.
\end{lemma}
\begin{proof}
    See proof in Appendix~\ref{appndx:proof-of-lem: Delta-Continuity}.
\end{proof}

Note that the continuity statement in the following lemma is with respect to $\bs$, treating $\bx$ as a parameter.
\begin{lemma}\label{lem: G-continuity-mu-sigma}
Let $\bs\;\mapsto\;G(\bs; \bx)$
be a mapping where $G(\bs; \bx)$ is defined in~\eqref{def: G}. Then under Assumptions~\ref{assmp: compactness}, $G(\bs; \bx)$ is continuous on $(\Theta \times \Xi)$.
\end{lemma}

\begin{proof}
 Let us consider (case-1): $\bx \neq 0$ and (case-2): $\bx = 0$ with (case-2a): $b \neq 0$ and (case-2b): $b=0$.

To show the claim in (case-1), let $g(\bx) = \frac{b-\bmean^\top \bx}{\sqrt{\bx^\top \bQ \bx}}$ and $\{(\bmean_t, \bQ_t)\}^\infty_{t=1}$ be any sequence in $(\Theta \times \Xi)$ converging to $(\check{\bmean}, \check{\bQ})$. Since $\bQ, \check{\bQ} \in \mbS^n_{\hat{\delta}} \subset \mbS^n_{++}$ and hence $\bx^\top \bQ \bx \geq \hat{\delta} \|\bx\|^2 > 0$ for any $\bx \neq \bzero$,
\begin{align*}
    \lim_{t \to \infty} g(\bmean_t, \bQ_t) = \lim_{t \to \infty} \frac{b-\bmean_t^\top \bx}{\sqrt{\bx^\top \bQ_t \bx}} = \frac{\lim_{t \to \infty} (b-\bmean^\top_t \bx)}{\lim_{t \to \infty} \sqrt{\bx^\top \bQ_t \bx}} = \frac{b-\check{\bmean}^\top \bx}{\sqrt{\bx^\top \check{\bQ} \bx}} = g(\check{\bmean}, \check{\bQ}).
\end{align*}

\noindent The second equality holds by \cite[Theorem~4.4]{rudin2021principles} and the second-to-last equality follows due to the continuity of the linear $b-\bmean^\top \bx$ and quadratic functions $\bx^\top \bQ \bx$ in $(\bmean, \bQ)$, and since $\bx^\top \bQ \bx > 0$ for any $\bx \in \mbR^n \backslash \{\bzero\}$ by Assumption~\ref{assmp: compactness}. Thus $g(\bmean, \bQ)$ is continuous on $\Theta \times \Xi$. Additionally, since $\Phi(\cdot)$ is continuous on $\mbR$, by \cite[Theorem~4.7]{rudin2021principles}, composition $\Phi\big(g(\bmean, \bQ)\big)$ is also continuous at all $(\bmean, \bQ) \in (\Theta \times \Xi)$.

To prove in (case-2a),  we use $\epsilon-\delta$ definition of continuity\footnote{A function is continuous at any arbitrary point $(\bmean_{0},\bQ_0)$, if for any given $\epsilon > 0$, $\exists$ some $\delta(\epsilon) > 0$ such that for any $(\bmean, \bQ) \in (\Theta \times \Xi)$ with $\bigl\|(\bmean,\bQ)-(\bmean_{0},\bQ_0)\bigr\|<\delta(\epsilon)$ implies $|\Phi(\frac{b-\bmean^\top \bx}{\sqrt{\bx^\top \bQ \bx}}) - \Phi(\frac{b-\bmean_0^\top \bx}{\sqrt{\bx^\top \bQ_0 \bx}})| < \epsilon$.}. Observe that, since $\bx = 0$ with $b\neq 0$, $\Phi(\frac{b-\bmean^\top \bx}{\sqrt{\bx^\top \bQ \bx}})$ is defined and 
\[
|\Phi(\frac{b-\bmean^\top \bx}{\sqrt{\bx^\top \bQ \bx}})-\Phi(\frac{b-\bmean_0^\top \bx}{\sqrt{\bx^\top \bQ_0 \bx}})|=0
\quad\text{for \emph{every}}\quad (\bmean,\bQ).
\]
Thus the $\epsilon-\delta$ condition
\[
\exists\,\delta>0\;:\;
\bigl\|(\bmean,\bQ)-(\bmean_{0},\bQ_0)\bigr\|<\delta
\;\Longrightarrow\;
|\Phi(\frac{b-\bmean^\top \bx}{\sqrt{\bx^\top \bQ \bx}})-\Phi(\frac{b-\bmean_0^\top \bx}{\sqrt{\bx^\top \bQ_0 \bx}})|
   <\epsilon
\]
is satisfied by \emph{any} \(\delta > 0\). 

Finally in (case-2b), $G(\bs; \bzero)$ is a constant function in $(\bmean, \bQ)$ with value $1$ by definition~\eqref{def: G}. Hence, it is continuous. Therefore, \(G(\bmean, \bQ; \bx)\) is continuous at every point in
$(\Theta \times \Xi)$.
\end{proof}
\normalsize
\begin{remark}
We note that when $\bx \neq \bzero$, if one allows either $(\bmean, \bQ) \to (\bzero, \bzero)$ with $b = 0$ or simultaneously $\bQ\to \bzero$ and $\bmean\to\bmean_0 \neq \bzero$ so that $b - \bmean_0^\top \bx=0$, then $\lim_{(\bmean,\bQ)\to(\mu_0,\bzero)}g(\bmean,\bQ)$
is path–dependent and hence $\Phi(g(\bmean, \bQ))$ may fail to be continuous at $(\mu_0, \bzero)$ (see \cite[Chapter 9]{rudin2021principles} for counterexamples of continuity). 
We require regularity condition $\bx^\top \bQ \bx \geq \|\delta\|^2 > 0$ for some $\delta > 0$ to overcome such an issue since $\bQ \succeq \bzero, \bQ \neq \bzero$ can still lead to $\bx^\top \bQ \bx = 0$ if $\bx$ lies entirely in the null-space of $\bQ$. 
\end{remark}

Based on the continuity results and the boundedness of \(G\), the proof of Lemma~\ref{lem: strong-duality} is immediate:

\begin{proof}
 \textbf{(of Lemma~\ref{lem: strong-duality})}
 Proofs of Lemmas~\ref{lem: Delta-Continuity} and
 \ref{lem: G-continuity-mu-sigma} respectively establish the continuity of
 \(d^2_{\mathrm{BW}}\) and \(G(\bs;\bx)\) in \(\bs\). Hence, for each nominal
 support point \(\hat{\bs}_k\), the transport-cost term $d^2_{\mathrm{BW}}\bigl(\mN(\hat{\bs}_k),\mN(\bs)\bigr)$
 is lower semicontinuous on $\mS = \Theta \times \Xi$, and \(d^2_{\mathrm{BW}}\bigl(\mN(\bs),\mN(\bs)\bigr)=0\) for every
\(\bs\in\mS\).

 Since \eqref{prob: robust-CC} is a minimization problem, we apply
 \cite[Theorem~1(a)]{blanchet2019quantifying} to \(-G(\cdot;\bx)\). By
 Lemma~\ref{lem: G-continuity-mu-sigma}, \(G(\cdot;\bx)\) is continuous, and
 therefore \(-G(\cdot;\bx)\) is upper semicontinuous.

 Moreover, for any fixed \(\bx\), by definition, $ 0\le G(\bs;\bx)\le 1,
 \, \forall \bs\in\Theta\times\Xi.$
 Thus, the integrability condition of
 \cite[Theorem~1(a)]{blanchet2019quantifying} (also see condition (A3) in
 Appendix~A(vi)) is satisfied by choosing the scalar transport-cost multiplier to be zero, because \(-G(\bs;\bx)\in[-1,0]\) and hence $ -1
 \le
 \sup_{\bs\in\mS}\{-G(\bs;\bx)\}
 \le
 0.$  Therefore, all conditions of \cite[Theorem~1(a)]{blanchet2019quantifying} are satisfied, and strong duality follows.
\end{proof}


\begin{proof} \textbf{(of Theorem~\ref{thm: strong-duality})}
    See proof in Appendix~\ref{appndx:proof-of-lem: Delta-Continuity}.
\end{proof}

\noindent As a consequence of Theorem~\ref{thm: strong-duality}, we also obtain a dual formulation of a finitely supported distributionally robust chance constraint model:

\begin{corollary}
\label{cor: FDR-dual}
Let \(\mD_d
:=
\left\{
\boldsymbol{\pi}\in\mathbb R_+^{\hat K\times K}:
\sum_{k=1}^{\hat K}\sum_{l=1}^K
d^2_{\mathrm{BW}}
\left(
\mN(\hat{\bs}_k),\mN(\hat{\bs}_l)
\right)\pi_{kl}\le\rho
\right\}, \,\, \sum_{l=1}^K\pi_{kl}=\hat w_k \,\,\, \forall k\in[\hat K]\).
Consider the optimization problem
\begin{align}
 & \min_{\bx \in \mX} \;\; \bc^\top \bx \quad \mathrm{s.t.} \quad \sum_{k=1}^{\hat{K}} \sum_{l=1}^K G(\bs_l, \bx) \; \pi_{k l} \geq \theta \qquad \forall \pi \in \mD_d,  \label{prob: FDR-CCP-GMM} \tag{FDR}  
\end{align}
Then, for any \(\bx\in\mX\) satisfying the robust constraint in
\eqref{prob: FDR-CCP-GMM} if and only if there exists
\(\bbeta\in\mathbb R^{\hat K}\times\mathbb R_+\) such that $\sum_{k=1}^{\hat{K}}\hat w_k\beta_k
-
\rho\beta_{\hat K+1}
\ge \theta$
and $\beta_k
-
\beta_{\hat K+1}
d^2_{\mathrm{BW}}
\left(
\mN(\hat{\bs}_k),\mN(\hat{\bs}_l)
\right)
\le
G(\hat{\bs}_l;\bx),
\,\,
\forall k\in[\hat K], \forall l\in[K].$
Consequently, solving~\eqref{prob: FDR-CCP-GMM} is equivalent to
solving~\eqref{prob: semif-infinite-FDR}:
\begin{equation}
\label{prob: semif-infinite-FDR}
\renewcommand{\arraystretch}{1.15}
\begin{array}{@{}l l@{}}
\multirow{2}{*}{$\begin{aligned}
\min_{\bx\in\mX,\;\bbeta\in\mathbb R^{\hat K}\times\mathbb R_+} \quad
& \bc^\top\bx
\quad  \mathrm{s.t.}\quad\\
\end{aligned}$}
&
\displaystyle
\sum_{k=1}^{\hat K}\hat w_k\beta_k
-
\rho\beta_{\hat K+1}
\ge \theta,
\\[6pt]
&
\displaystyle
\beta_k
-
\beta_{\hat K+1}
d^2_{\mathrm{BW}}
\left(
\mN(\hat{\bs}_k),\mN(\bs_l)
\right)
\le
G(\bs_l;\bx),
\quad
\forall k\in[\hat K], \forall l\in[K].
\end{array}
\end{equation}
In particular, \eqref{prob: FDR-CCP-GMM} and
\eqref{prob: semif-infinite-FDR} have the same optimal value, and their optimal \(\bx\)-solutions coincide whenever an optimal solution exists.
\end{corollary}

Note that, unlike Theorem~\ref{thm: strong-duality},
Corollary~\ref{cor: FDR-dual} does not require \(\rho>0\), because its inner
worst-case problem, when feasible, is a finite-dimensional linear program with
finite optimal value, and dual attainment follows from LP duality. We also note that a special case of \eqref{prob: FDR-CCP-GMM} is when the Wasserstein ball is supported on the same finite set of parameter values $\bs_k$ that are used in defining the nominal distribution, with $K=\hat{K}$.

\section{A Cutting Surface Algorithm and its Convergence}\label{sec: Dual-formulation}

The cutting surface algorithm (CSA) framework is a general-purpose approach for solving semi-infinite programs~\cite{hettich1993semi, reemtsen1998sip, goberna2001semi, mehrotra2014cutting, luo2019decomposition}. This section focuses on developing CSA for the semi-infinite formulation~\eqref{prob: semif-infinite}. Algorithm~\ref{alg: cutting-surface} gives pseudo-code based on the main steps of the CSA, and its use to obtain a solution of the user-desired tolerance $\hat{\tau}$ for \eqref{prob: semif-infinite}. The user-desired tolerance $\hat{\tau}$ gives a parameter that is iteratively reduced to set the accuracy level $\epsilon_j$ when developing outer and inner approximations of the master and subproblems. The convergence of the detailed algorithm depends on the solvability of the master problem and the subproblems to this accuracy, which are studied in Sections~\ref{sec: Cut-S-algorithm} and~\ref{sec: subproblem}, respectively. The respective problems are solved exactly after obtaining the outer and inner approximations. The accuracy $\epsilon_j$ is iteratively set because achieving a user-defined tolerance $\hat{\tau}$ depends on an unknown constant, which is known to exist, but computing its value is not practical. 

\subsection{Cutting-Surface Algorithm}\label{sec: Cut-S-algorithm}
For a given accuracy parameter $\epsilon_j$, the CSA Algorithm \ref{alg: cutting-surface} generates solutions to the increasingly tighter relaxations of \eqref{prob: semif-infinite}. The master problem~\eqref{prob: master-DRCC-cont-dual} is an outer approximation of \eqref{prob: semif-infinite} using finitely many constraints
generated from a finite subset of elements in the set $(\Theta \times \Xi)$.
Suppose, at a master iteration $\ell$, set $\mS_\ell$ contains $l$   elements of $(\Theta \times \Xi)$ generated from previous $\ell$-iterations. At an initialization step, these elements can be the parameters of the Gaussians in the nominal mixture distribution. Let a solution of the corresponding master problem: 
\begin{equation}\label{prob: master-DRCC-cont-dual}
\renewcommand{\arraystretch}{1.0}
\begin{array}{@{}l l@{}}
\multirow{2}{*}{$\begin{aligned}
\hspace{-1em} \min_{\bx, \bbeta} \;\, &  \bc^\top \bx \quad
\mathrm{s.t.}  &
\end{aligned}$}
& \bx \in \mX, \;\;\; \bbeta \in \mfB := \Bigl\{ \bbeta \in \mbR^{\hat{K}} \times \mbR_+ \; | \;\sum_{k=1}^{\hat{K}} \hat{w}_k\beta_k - \rho \beta_{\hat{K}+1} \geq \theta \; \Bigr\},
\\[3pt]
&  J_k(\bs_l) := \beta_k - \beta_{\hat{K}+1} d^2_{\mathrm{BW}}\bigl(\mN(\hat{\bs}_k), \mN(\bs_l) \bigr) \leq G(\bs_l; \bx), \; \forall \; \bs_l \in \mS_\ell, \;  k \in [\hat{K}] \;
\end{array}
\end{equation}
be $(\bx^{(\ell)}, \bbeta^{(\ell)})$. Given this solution, CSA solves $[\hat{K}]$ subproblems~\eqref{prob: subproblem}: 
\begin{align}\label{prob: subproblem}
  \hspace{-1.25em} \max_{\bs \in (\Theta \times \Xi)} J_k\big(\bs; (\bx^{(\ell)}, \bbeta^{(\ell)}) \big) = \beta^{(\ell)}_k - \beta^{(\ell)}_{\hat{K}+1} d^2_{\mathrm{BW}}\bigl(\mN(\hat{\bs}_k), \mN(\bs) \bigr) - G(\bs; \bx^{(\ell)}) \,\,\text{for each} \;\; k \in [\hat{K}] 
\end{align}
for all $k \in [\hat{K}]$. Let the solutions to these subproblems be $\bs^{(\ell+1)}_k$. We check the termination criteria based on sufficient constraint violation (i.e., $J_k\big(\bs_k^{(\ell+1)}; (\bx^{(\ell)}, \bbeta^{(\ell)}) > \epsilon_j$). If violated constraints are identified, Step 2 in Algorithm~\ref{alg: cutting-surface} adds a set of new points $\{\bs^{(\ell+1)}_k\}$ to the existing set $\mS_\ell$, i.e., $\mS_{\ell+1} \gets \mS_\ell \cup \{\bs^{(\ell+1)}_k\}$ that provide these violated constraints. Note that prior to entering into Step 2, Algorithm~\ref{alg: cutting-surface} checks if the current iterate is small to handle the singularity at $\bx = 0$. If one or more new violated constraints are found, then the algorithm adds these constraints to the master problem and proceeds to the next master iteration.

\begin{algorithm}[t]
\scriptsize
  \caption{A cutting‐surface algorithm (CSA) to solve~\eqref{prob: DR-CCP-GMM}.}
  \label{alg: cutting-surface}
  \begin{algorithmic}[1]
    \Require Optimality Tolerance $\hat{\tau}$
    \Ensure A $\hat{\tau}$-optimal solution or infeasibility
    \State Set $\quad \mS_0 \gets \{\hat{\bs}_1,\ldots,\hat{\bs}_{\hat K}\},\quad \ell \gets 0, \quad j \gets 1, \quad \delta_\circ \gets 0 \quad lb \gets -\infty.$
    \While{\textbf{true}}
      \State Set $\epsilon_j \gets \hat{\tau}/2^j$
      \State Set $z_{\epsilon_j}
\gets\Phi^{-1}(1-\min\left\{\epsilon_j,\frac12\right\})$
      \State \textbf{Step 1. Solve Master Problem:} Solve~\eqref{prob: master-CS-algrthm} to obtain an optimal solution to
      \begin{align}\label{prob: master-CS-algrthm}
        \min\Bigl\{\bc^\top \bx \ \colon\ \bx \in \mX,\ \bbeta \in \mfB,\ 
        J_k(\bs_l)\le \epsilon_j,\ \bs_l \in \mS_\ell,\ k\in[\hat{K}]\Bigr\}.
      \end{align}
      \If{\eqref{prob: master-CS-algrthm} is infeasible}
        \State \Return Infeasibility Status
      \EndIf
      \State Denote $(\bx^{(\ell)},\bbeta^{(\ell)})$ as the $\epsilon_j$-feasible solution of~\eqref{prob: master-CS-algrthm}
      \State Set $
\delta_{\epsilon_j}
\gets
\frac{b}{\overline{m}_\Theta+\sqrt{\overline{\lambda}_\Xi}\,z_{\epsilon_j}}$ for 
\(\quad\overline{m}_\Theta:=\max_{\bmean\in\Theta}\|\bmean\|,
\quad \overline{\lambda}_\Xi
:=
\max_{\bQ\in\Xi}\lambda_{\max}(\bQ)\).
    \If{$\|\bx^{(\ell)}\|\le \delta_{\epsilon_j}$}
        \State Set $\bs^{(\ell+1)}\gets \bs^{(\ell)}$ and go to Step 3.
    \EndIf
      \State \textbf{Step 2. Solve Subproblems:}  Solve all $\hat{K}$ sub-problems~\eqref{prob: subproblem} to $\epsilon_j$-accuracy and using their solution define $\mathrm{S}_{\ell+1}$ and $\bs^{(\ell+1)}$ as follows
\begin{align*}
&   \mathrm{S}_{\ell+1}
:= \bigcup_{k\in[\hat K]}
\left\{\bs^{(\ell+1)}_k \mid \bs^{(\ell+1)}_k \, \text{is an}\, \epsilon_j\text{-accurate solution to} \,  \max_{\bs\in\mS} J_k\!\left(\bs;(\bx^{(\ell)},\bbeta^{(\ell)})\right), \,\, 
J_k\!\left(\bs^{(\ell+1)}_k;(\bx^{(\ell)},\bbeta^{(\ell)})\right)>\epsilon_j
\right\} \\
& 
\bs^{(\ell+1)}\in \arg\max_{\bs\in \mathrm{S}_{\ell+1}} 
J_k\!\left(\bs;(\bx^{(\ell)},\bbeta^{(\ell)})\right)
\quad \text{(break tie arbitrarily)}.
\end{align*}
\vspace{-0.5em}
     \State \textbf{Step 3. Check the Termination Criteria:} 
      \If{$J\bigl(\bs^{(\ell+1)};(\bx^{(\ell)},\bbeta^{(\ell)})\bigr)\le \epsilon_j$}
      \State Solve~\eqref{prob: master-DRCC-cont-dual} to $\epsilon_j$-optimality and obtain a feasible solution $(\tilde\bx, \tilde\bbeta)$ to~\eqref{prob: semif-infinite}
      \State $lb \gets \bc^\top \tilde\bx$
        \If{$\bc^\top (\tilde\bx - \bx^{(\ell)}) \leq \hat{\tau}$}
        \State \Return $\bx^{(\ell)}$
        \Else 
        \State Set $j \gets j + 1$
        \EndIf
      \EndIf
    \State $\mS_{\ell+1}\gets \mS_\ell\cup\mathrm{S}_{\ell+1}$,\quad $\ell\gets \ell+1$.
    \EndWhile
  \end{algorithmic}
\end{algorithm}

The following definitions are used in establishing the finite convergence of Algorithm~\ref{alg: cutting-surface} as stated in Theorem~\ref{thm: Cut-S-finite-convergence}.

\begin{definition}[$\hat{\tau}$-accurate and $\hat{\tau}$-optimal solutions]
Assume that an optimization problem $\min_{\bx \in \mathcal{X}} f(\bx)$ with a nonempty feasible region $\mathcal{X}$ has a finite optimal objective value $f(\bx^*)$ attained at a solution $\bx^*$.
A solution $\bx^*_{\hat{\tau}} \in\mathcal{X}$ is called \emph{$\hat{\tau}$-accurate} for some $\hat{\tau} > 0$ if it is $\hat{\tau}$-feasible and its objective discrepancy from the true optimum satisfies \(\big| f(\bx^*_{\hat{\tau}}) - f(\bx^*) \big| \leq \hat{\tau}.\) Moreover,  $\bx^*_{\hat{\tau}}$ is called \emph{$\hat{\tau}$-optimal} if it is $\hat{\tau}$-feasible and \(f(\bx^*_{\hat{\tau}}) - f(\bx^*) \leq \hat{\tau}.\)
\end{definition}

\begin{definition}\label{def: accurate-sol}
    Let $\mfB := \Bigl\{ \bbeta \in \mbR^{\hat{K}} \times \mbR_+ \; | \;\sum_{k=1}^{\hat{K}} \hat{w}_k\beta_k - \rho \beta_{\hat{K}+1} \geq \theta \; \Bigr\}$ and, for some given $\hat{\bx}$, $g(\bs; \hat{\bx}) := (b-\bmean^\top \hat{\bx})/\sqrt{{\hat{\bx}}^\top \bQ \hat{\bx}}$. Further, let us denote the constraint set of the master problem~\eqref{prob: master-DRCC-cont-dual} and the semi-infinite problem~\eqref{prob: semif-infinite} respectively as\vspace{-0.25em} 
    \[
    \mF_\ell(\mathtt{0}) := \{ (\bx, \bbeta) \; \, | \; \; \bx \in \mX, \;\; \bbeta \in \mfB, \;\; J_k(\bs_l)\le 0,\;\bs_l \in \mS_\ell, \,\, k \in [\hat{K}]\}\bigr\}
    \]
    and\vspace{-1.8em}   
    \[
    \mF(\mathtt{0}) := \{ (\bx, \bbeta) \; \, | \; \; \bx \in \mX, \;\; \bbeta \in \mfB, \;\; J_k(\bs)\le 0,\;\bs \in \mS, \,\, k \in [\hat{K}]\}\bigr\}.
    \] 
    Then we define the optimal objective value as a function of some right-hand side parameter $\tau \in \mbR$ \, for the master problem~\eqref{prob: master-DRCC-cont-dual} and subproblem~\eqref{prob: subproblem}, respectively, as:
    \begin{align*}
    & V^\star_\ell(\tau) \coloneqq \min \left\{\; \bc^\top \bx \;\; |\;\; (\bx, \bbeta) \in  \mF(\tau) := \bigl\{ (\bx, \bbeta) \; \, | \; \; \bx \in \mX, \;\; \bbeta \in \mfB, \;\; J_k(\bs_l)\le \tau,\;\bs_l \in \mS_\ell, \,\, k \in [\hat{K}] \bigr\}\right\}  \\ 
    & V^\star_{sub}(\tau) \coloneqq \max \left\{ \; J(\bs, \zeta'; \, \hat{\bx}, \hat{\bbeta}) \; | \; (\bs, \zeta') \in  \mQ(\tau) := \bigl\{(\bs, \zeta') \; | \; \bs \in \mS, \;\zeta' \in [-1, 0], \; \tau \geq \Phi(g(\bs; \hat{\bx})) + \zeta' \bigr\}\right\}, 
    \end{align*}
    with $(\bx^*(\mathtt{\tau}), \bbeta^*(\mathtt{\tau}))$ and $(\bs^*(\mathtt{\tau}), \zeta^*(\tau))$ being the solution corresponding to $V^\star_\ell(\tau)$ and $V^\star_{sub}(\tau)$.
\end{definition}

\begin{theorem}\label{thm: Cut-S-finite-convergence}
Consider the adaptive cutting-surface Algorithm~\ref{alg: cutting-surface}.
Suppose that the master problem and separation subproblems in every iteration are solved, respectively, to optimality and $\epsilon$-accuracy for some $\epsilon > 0$ as specified in Step 2 (line 14). Further assume that, for every
\(\epsilon>0\), there exists an oracle \(\mathcal O\) returning $(\widetilde{\bx},\widetilde{\bbeta}) \in \argmin\{\bc^\top\bx:(\bx,\bbeta)\in\mF(-\epsilon)\}$.  Then given a user-specified tolerance $\hat{\tau} > 0$,

\begin{enumerate}[leftmargin=0.975em]
    \item[i)] Algorithm~\ref{alg: cutting-surface} attains a \( \hat{\tau}/2 \)-feasible solution \( (\hat{\bx}, \hbbeta) \) of \eqref{prob: semif-infinite} with \( \bc^\top \hat{\bx} \leq \text{val\eqref{prob: semif-infinite}} \) after adding finitely many points from $\mS$ to $\mS_\ell$.

\item[ii)] Algorithm~\ref{alg: cutting-surface} terminates in finite number of iterations \(\ell^\star\), returns a
\(\hat{\tau}\)-accurate certificate \(\bx^{(\ell^\star)}\), and yields a
\(\hat{\tau}\)-optimal feasible solution \(\widetilde{\bx}\) to
\eqref{prob: semif-infinite}.
\end{enumerate} 
\end{theorem}

Theorem~\ref{thm: Cut-S-finite-convergence} requires solving both the master problem~\eqref{prob: master-DRCC-cont-dual} and the subproblems~\eqref{prob: subproblem} at each master iteration. The function $G(\bx, \bs)$ defined in~\eqref{def: G} appears in both the master problem and subproblems, and it involves the function $\Phi$ for which a closed analytical expression is not known. Hence, we first develop approximate formulations of these problems together with an explicit mechanism to control the approximation accuracy of the resulting feasible set. To ensure that all constraints in~\eqref{prob: semif-infinite} are satisfied with $\hat{\tau}/2$ feasibility, we study the master and subproblem optimal-value functions by analyzing $G(\bx, \bs)$  as a function of both $\bx$ and $\bs$, and through properties of $\Phi(z)$ for $ z \in \mbR$.


The univariate function $\Phi(z)$ is convex in $z$ when $z \leq 0$ and concave when $z \geq 0$. Using convexity, concavity, and curvature properties of $\Phi(z)$, \cite{dey2025solving} recently developed outer and inner approximations of a finite Gaussian mixture-based chance-constrained optimization problem. We restate results from \cite{dey2025solving} here for completeness. We consider a piecewise-linear outer (PWL-O) approximation-based formulation for~\eqref{prob: master-DRCC-cont-dual}, which itself is a relaxation of~\eqref{prob: semif-infinite}, while the piecewise-linear inner approximation (PWL-I) is used in~\eqref{prob: subproblem}. Since~\eqref{prob: subproblem} is a maximization problem, using its inner approximation provides an overall outer approximation to the feasible set of \eqref{prob: semif-infinite}.  In the PWL-O approximation of  $\Phi(z)$, we approximate the convex part of $\Phi(z)$ with secant lines and its concave part with tangent lines. For the PWL-I approximation, the convex part is approximated by tangent lines, while the concave part is approximated by secant lines. 
 
 Let $\bzz = (\zz_{-L}, \zz_{-L+1}, \ldots, \zz_{-1},  \zz_{0}, \zz_{1}, \ldots, \zz_{R-1}, \zz_{R}) \in \mbR^{L+R+1}$ be some ordered array of breakpoints on $z-$axis, parameterized by integers $L, R >0 $. Let the set of tangent and secant lines for PWL-O be as defined below so that the curve $\bar{\Phi}(\cdot; \bzz)$ always lies above that of $\Phi$, i.e., $\bar{\Phi}(\cdot; \bzz) \geq \Phi(z), \,\, \forall z \in \mbR$ (see Figure~\ref{fig:pwl_phi_outer}):
\begin{equation}\label{def: phi_pwl_outer}
\bar{\Phi}(z; \bzz) \coloneqq
\begin{dcases*}
    \min\left\{1, \min_{i \in [R]_0} \left\{ \mathtt{g}_i z + \mathtt{g}_i^0 \right\} \right\},
    & if $z \geq 0$, $\mathtt{g}_i \coloneqq \phi(\zz_i), \;\mathtt{g}_i^0 = \Phi(\zz_i) - \phi(\zz_i) \zz_i, \;$  \text{(Tangent)}\\
    \max\left\{\Phi(\zz_{-L}), \max_{i \in [L]} \left\{ h_i z + h_i^0 \right\} \right\},
    & otherwise \; where, $ h_i \coloneqq \frac{\Phi\left(\zz_{-i+1}\right) - \Phi\left(\zz_{-i}\right)}{\zz_{-i+1} - \zz_{-i}}, \;$ \\
    & $h_i^0 = \Phi(\zz_{-i}) - \frac{\Phi\left(\zz_{-i+1}\right) - \Phi\left(\zz_{-i}\right)}{\zz_{-i+1} - \zz_{-i}} \zz_{-i}, \;$ \quad \text{(Secant).}
\end{dcases*}
\end{equation}

We now provide a relaxation of the master problem~\eqref{prob: master-DRCC-cont-dual} by replacing $\Phi$ with $\bar{\Phi}$, followed by its equivalent mixed-integer reformulation using~\eqref{eq: reform-outer-master}. Subsequent result (Theorem~\ref{thm: complexity-breakpoint}) claims that finding an optimal solution of such a relaxation is equivalent to obtaining an optimal solution to~\eqref{prob: master-CS-algrthm}. Section~\ref{sec: subproblem} then provides an approximate formulation of subproblem~\eqref{prob: subproblem} by using inner approximation $\underline{\Phi}(\cdot, \bzz)$ of $\Phi(\cdot)$ so that the desired accurate solution in Step 2 of Algorithm~\ref{alg: cutting-surface} is achieved.

\begin{proposition}\label{prop: master-prob-reform}
Let $\bzz$ be an ordered array of breakpoints of $z-$axis and $\Phi(\cdot)$ be replaced with its PWL outer approximation $\bar{\Phi}(\cdot; \bzz)$ defined in~\eqref{def: phi_pwl_outer}. Then, the resulting formulation~\eqref{prob: master-DRCC-cont-dual-O} is a outer approximation to master problem~\eqref{prob: master-DRCC-cont-dual}:\vspace{-1em}
\begin{subequations}
\begin{equation}\label{prob: master-DRCC-cont-dual-O}
\renewcommand{\arraystretch}{1.15}
\begin{array}{@{}l l@{}}
\multirow{4}{*}{$\begin{aligned}
\min \quad &  \bc^\top \bx \quad
\mathrm{s.t.} \quad &
\end{aligned}$}
&  \bx \in \mX, \qquad \sum_{k=1}^{\hat{K}} \hat{w}_k\beta_k - \beta_{\hat{K}+1} \; \rho \geq \theta, \qquad \bbeta \in \mbR^{\hat{K}} \times \mbR_+ 
\\[2pt]
& \beta_k - \beta_{\hat{K}+1} \,\, d^2_{\mathrm{BW}}\bigl(\mN(\hat{\bs}_k), \mN(\bs_l) \bigr) -\zeta_l \leq 0 \qquad  \forall  k \in [\hat{K}] \\[2pt]
& \bar{\Phi}\big(z_l; \bzz\big) \geq \zeta_l , \quad    b - \bmean_l^\top \bx \geq z_l \;\sigma_l, \quad \bx^\top \bQ_l \bx = \sigma^2_l, \\[2pt]
& \zeta_l \in [0, 1], \qquad z_l \in \mbR, \qquad \sigma_l \in \mbR_{+} \qquad \forall \; \bs_l \in \mS_\ell. \;
\end{array}
\end{equation}

Additionally (ignoring subscript $l$), $\exists (z, \zeta) \in \mbR \times \mbR$ satisfying $\zeta \leq \bar{\Phi} \left(z ; \bzz\right)$ if and only if $\mH^O(\bzz) \neq \emptyset$, where\vspace{-1em}
	\begin{equation}\label{eq: reform-outer-master}
		\mH^O(\bzz) \coloneqq 
		\Set*{
			\begin{aligned}
				&\balpha \in \mbR_{+}^{L}, \\
				& \bt \in \{0, 1\}^3, \\
				& \by \in \mbR^3, \\
				& z, \zeta \in \mbR
			\end{aligned}
		}
		{
			\begin{aligned}
				& \Phi(\zz_{-L})t_1 + t_2 + \mathtt{g}_i y_3 + \mathtt{g}_i^0 t_3  \geq \zeta \;\; i \in [R]_0, \quad \Phi(\zz_{-L})t_1 + \sum_{i=0}^{L} \alpha_{i} \Phi(\zz_{-i})  + t_3 \geq \zeta, \\
				& t_1 + t_2 + t_3 = 1, \;\;\; z = y_1 + y_2 + y_3, \;\;\;  t_1 = 0 \implies y_1 = 0, \;\;\; t_3 = 0 \implies y_3 = 0,  \\
				& y_1 \leq t_1 \zz_{-L}, \qquad y_2 = \sum_{i=0}^{L} \alpha_{i} \zz_{-i}, \quad t_2 = \sum_{i=0}^{L} \alpha_{i}, \quad \bm{\alpha} \in \mathop{\mathrm{SOS2}}
			\end{aligned}
		}.
	\end{equation}
\end{subequations}
\normalsize
\end{proposition}
\vspace{-1.0em}
\begin{proof}
    See proof in Appendix~\ref{appndx:cut-alg-proof}.
\end{proof}

\begin{theorem}(~\cite[Lemma~1, Theorem~1]{dey2025solving})\label{thm: complexity-breakpoint}
 Let Algorithms~\ref{alg:breakpoint_find} and ~\ref{alg:breakpoint_find_negative} be used to generate breakpoint array $\zz$ for PWL-O approximation of $\Phi(\cdot)$. Then $O(\frac{1} {\sqrt{\tau}} \sqrt{\log(\frac{1}{\tau})})$ number of breakpoints are sufficient to approximate $\Phi(z)$   so that $0 \leq \bar{\Phi}(z; \bzz) - \Phi(z) \leq \tau$ for all $z \in \mbR$. 
\end{theorem}

\subsection{Solving subproblem to a desired accuracy using approximation}\label{sec: subproblem}
This section shows that, given a tolerance $\tau > 0$, solving a $\tau$-approximate formulation of subproblem~\eqref{prob: subproblem} yields a solution whose objective value is within $\tau$ of the exact model~\eqref{prob: subproblem}. We pursue an inner approximation for~\eqref{prob: subproblem} to obtain the desired $\tau$-approximate formulation. Together with the PWL-O approximation used to solve the master problem~\eqref{prob: master-DRCC-cont-dual}, such an inner approximation ensures that the master model remains a relaxation of~\eqref{prob: semif-infinite}. In contrast to PWL-O, for PWL-I, the curve $\underline{\Phi}(\cdot, \bzz)$ lies below $\Phi$ and is defined as:
\begin{equation}\label{def: phi_pwl_inner}
	\underline{\Phi}(z; \bzz)
	=
	\begin{dcases*}
		\min\left\{\Phi(\zz_{R}), \min_{i \in [R]} \left\{ \mathtt{g}_i z + \mathtt{g}_i^0 \right\} \right\},
		& if $z \geq 0$, \quad \; \text{(Secant)}\\
		\max\left\{ 0, \max_{i \in [L]_0}\left\{ h_i z + h_i^0 \right\} \right\},
		& otherwise, \; \text{(Tangent)}
	\end{dcases*}
\end{equation}
where
\vspace{-0.5em}
\begin{equation*}
\begin{alignedat}{3}
    & \mathtt{g}_i := \frac{\Phi(\zz_i) - \Phi(\zz_{i-1})}{\zz_i - \zz_{i-1}}, \; &&  \mathtt{g}^0_i := \Phi(\zz_{i-1}) - \frac{\Phi(\zz_i) - \Phi(\zz_{i-1})}{\zz_i - \zz_{i-1}} \zz_{i-1}, \; && i \in [R], \\ 
\text{and} \quad	& h_i \coloneqq \phi(\zz_i), \; && h_i^0 = \Phi(\zz_i) - \phi(\zz_i) \zz_i, \; && i \in [L]_0.
\end{alignedat}
\end{equation*}
A counterpart of Theorem~\ref{thm: complexity-breakpoint} for PWL-I can then be stated as follows.
\begin{theorem}(~\cite[Lemma~5, Theorem~3]{dey2025solving})\label{thm: complexity-breakpoint-inner}
   Consider counterparts of Algorithms~\ref{alg:breakpoint_find} and ~\ref{alg:breakpoint_find_negative} to generate a breakpoint array $\zz$ for PWL-I approximation of $\Phi(\cdot)$. Then $O(\frac{1} {\sqrt{\tau}} \sqrt{\log(\frac{1}{\tau})})$ breakpoints are sufficient when $\Phi(z)$ is approximated as $\underline{\Phi}(z; \zz)$ so that $0 \leq \Phi(z) -\underline{\Phi}(z; \bzz)  \leq \tau$ for all $z \in \mbR$.  
\end{theorem}

Since subproblem~\eqref{prob: subproblem} is a maximization problem, to follow the same convention for inequality direction between $\bar{\Phi}\big(z_l; \bzz\big)$ and $\zeta$ as used in~\eqref{prob: master-DRCC-cont-dual-O} and~\eqref{eq: reform-outer-master}, we use $\zeta':= -\zeta$ here. This leads to $-\underline{\Phi}(z; \zz) \geq \zeta'$ in~\eqref{prob: subproblem-1.b} since $\underline{\Phi}(z; \zz) \leq \zeta \Longleftrightarrow -\underline{\Phi}(z; \zz) \geq -\zeta = \zeta'$. Hence, the next proposition provides a reformulation of the relaxation of the subproblem~\eqref{prob: subproblem} obtained by taking a piecewise linear approximation of $\Phi$.
\begin{proposition}\label{prop: sub-prob-reform}
Let $\bzz$ be an ordered array of breakpoints of $z-$axis and $\Phi(\cdot)$ be replaced with its PWL inner approximation $\underline{\Phi}(\cdot; \bzz)$ defined in~\eqref{def: phi_pwl_inner}. Then, for any given $\bx \neq \bzero$ from $\mX$, the resulting formulation~\eqref{prob: subproblem-1.a}-\eqref{prob: subproblem-1.c} is a relaxation to subproblem~\eqref{prob: subproblem}. 
\begin{subequations}
\begin{align}
     & \max \quad J\bigl(\bs,\zeta'; (\bx, \bbeta)\bigr) \coloneqq \beta_k - \beta_{\hat{K}+1} \;d^2_{\mathrm{BW}}\bigl(\mN(\hat{\bs}_k), \mN(\bmean, \bQ) \bigr) + \zeta', \; \label{prob: subproblem-1.a} \\
    & \;\;\; \mathrm{s.t.}  \quad - \underline{\Phi}\big(z; \bzz\big) \geq \zeta',  \;\; b - \bmean^\top \bx \leq z \; \sigma, \;\; \bx^\top \bQ \bx = \sigma^2, \label{prob: subproblem-1.b}\\
    & \qquad \qquad \bmean \in \Theta, \;\; \bQ \in \Xi, \quad \zeta' \in [-1, 0], \quad z \in \mbR, \quad \sigma \geq \sqrt{\hat{\delta}} \; \|\bx\|. \label{prob: subproblem-1.c}
\end{align}
Additionally, $\exists \, (z, \zeta) \in \mbR \times \mbR$ satisfying $-\underline{\Phi} \left(z ; \bzz\right) \geq \zeta'$ if and only if $\mH^I(\bzz) \neq \emptyset$, where
	\begin{equation}\label{prob: subproblem-reform-outer}
		\mH^I(\bzz) \coloneqq 
		\Set*{
			\begin{aligned}
				&\balpha \in \mbR_{+}^{L}, \\
				& \bt \in \{0, 1\}^3, \\
				& \by \in \mbR^3, \\
				& z, \zeta' \in \mbR
			\end{aligned}
		}
		{
			\begin{aligned}
				&  - \mathtt{h}_i y_1 - \mathtt{h}_i^0 t_1 -t_2 - \Phi(\zz_{R})t_3 \geq \zeta' \;\; i \in [L_0], \;\; - t_1 -\sum_{i=0}^{R} \alpha_{i} \Phi(\zz_{i})  - \Phi(\zz_{R})t_3\geq \zeta', \\
				& t_1 + t_2 + t_3 = 1, \;\;\; z = y_1 + y_2 + y_3, \;\;\;  t_1 = 0 \implies y_1 = 0, \;\;\; t_3 = 0 \implies y_3 = 0,  \\
				& y_2 = \sum_{i=0}^{R} \alpha_{i} \zz_{i}, \quad t_2 = \sum_{i=0}^{R} \alpha_{i}, \quad \bm{\alpha} \in \mathop{\mathrm{SOS2}}, \qquad y_3 \geq t_3 \zz_{R}
			\end{aligned}
		}.
	\end{equation}
\end{subequations}
\end{proposition}
\vspace{-1.25em}
\begin{proof}
    See proof in Appendix~\ref{appndx:cut-alg-proof}.
\end{proof}
We are now in a position to formally state and establish the claim this section starts with.
\begin{proposition}\label{prop: feasibility-tau}
For a master iterate $\bx^{(\ell)}$ at iteration $\ell$, let $g(\bs;\bx^{(\ell)}) \coloneqq \frac{b-\bmean^\top \bx^{(\ell)}}{\sqrt{{\bx^{(\ell)}}^\top \bQ \bx^{(\ell)}}}$. Define
\[
\mQ \coloneqq \{(\bs,\zeta') \,|\, \bs\in\mS, \zeta'\in[-1,0], -\Phi\!(g(\bs;\bx^{(\ell)}))\ge \zeta' \}, \, \tilde{\mQ} \coloneqq \{(\bs,\zeta') \,|\, \bs\in\mS, \zeta'\in[-1,0],-\underline{\Phi}\!(g(\bs;\bx^{(\ell)}))\ge \zeta' \},
\]
and  $v^\star \coloneqq\max_{(\bs,\zeta')\in\mQ} J(\bs,\zeta')$, $\tilde v^\star \coloneqq \max_{(\bs,\zeta')\in\tilde{\mQ}} J(\bs,\zeta')$.  Then, if $ \,\, 0 \le \Phi(z)-\underline{\Phi}(z) \le \hat{\tau}/2 \,\,$ for all $z\in\mbR$,   (i) $\mQ \subseteq \Tilde{\mQ}$ and (ii) $0 \le \tilde v^\star - v^\star \le \hat\tau/2.$
\end{proposition}

\begin{proof}
   We begin by noting that $g(\bs; \bx^{(\ell)})$ is well defined since an iterate $\bx^{(\ell)}$ is always considered for the subproblem after filtering it through the near-zero case through Lines~11-13 of Algorithm~\ref{alg: cutting-surface} and hence satisfies ${\bx^{(\ell)}}^\top \bQ \bx^{(\ell)} > 0$.
   
   First part of the claim is immediate since for any $\tau > 0$ and $z \in \mbR$, we know $0 \leq -\underline{\Phi}(z) + \Phi(z) \leq \tau$.  It implies
    $- \Phi(z) \leq -\underline{\Phi}(z) \leq \tau - \Phi(z)$. Thus, $\mQ \subseteq \Tilde{\mQ}$. Now for \emph{part-(ii)}, let $(\tilde{\bs},\tilde{\zeta}')$ be optimal for the approximate problem, and define
\[
\Delta \coloneqq \Phi\!\big(g(\tilde{\bs};\bx^{(\ell)})\big)-\underline{\Phi}\!\big(g(\tilde{\bs};\bx^{(\ell)})\big)\in[0,\hat\tau/2],
\qquad
\zeta^\circ \coloneqq \max\{-1,\ \tilde{\zeta}'-\Delta\}.
\]
Since $-\underline{\Phi}(g(\tilde{\bs};\bx^{(\ell)}))\ge \tilde{\zeta}'$ and $-\Phi = -\underline{\Phi}-\Delta$, we have
\[
-\Phi\!\big(g(\tilde{\bs};\bx^{(\ell)})\big)
= -\underline{\Phi}\!\big(g(\tilde{\bs};\bx^{(\ell)})\big) - \Delta
\ge \tilde{\zeta}'-\Delta
\ge \zeta^\circ,
\]
and clearly $\zeta^\circ\in[-1,0]$. Hence, $(\tilde{\bs},\zeta^\circ)\in\mQ$. Moreover, because only $\zeta'$ changes in $J(\bs,\zeta')$, we obtain
\[
J(\tilde{\bs},\zeta^\circ) \ge J(\tilde{\bs},\tilde{\zeta}') - \Delta \ge \tilde v^\star - \hat\tau/2.
\]
Since $\mQ\subseteq \tilde{\mQ}, v^\star = \max_{(\bs,\zeta')\in\mQ} J(\bs,\zeta') \geq J(\tilde{\bs},\zeta^\circ)$. Combining the two inequalities yields $0 \leq \tilde v^\star - v^\star \le \hat\tau/2$.
\end{proof}

\subsection{Finite convergence of the Cutting-Surface Algorithm}
We make some additional assumptions (Assumptions~\ref{assmp: strict-strict-interior}-\ref{assmp: LICQ}) required in our convergence analysis. The results for the master problem, along with those in Section~\ref{sec: subproblem} for subproblems, provide the tools required to prove Theorem~\ref{thm: Cut-S-finite-convergence}. Assumption~\ref{assmp: strict-strict-interior} excludes the optimal solution from being in a neighborhood of $\bzero$. Assumption~\ref{assmp: non-empty-interior} is a regularity condition requiring that our distributionally robust chance constraint program is strictly feasible.

\begin{assumption}[Exclusion of a zero-neighborhood]\label{assmp: strict-strict-interior} Let $\mF(\mathtt{0})$ be as in Definition~\ref{def: accurate-sol}. For some $\hat{\delta}$, define $\mF_{\hat{\delta}}(\mathtt{0})
:=
\mF(\mathtt{0})\cap
\{(\bx,\bbeta):\|\bx\|\ge\hat{\delta}\}$. Then for $b=0$ and $(\bzero, \bbeta) \in \mF(\mathtt{0})$,  there exists
$\hat{\delta}>0$ such that $\min_{(\bx,\bbeta)\in\mF_{\hat{\delta}}(\mathtt{0})}\bc^\top\bx
=
\min_{(\bx,\bbeta)\in\mF(\mathtt{0})}\bc^\top\bx.$
\end{assumption}


\begin{assumption}[Local nonemptiness]
\label{assmp: non-empty-interior}
There exists \(\varepsilon>0\) such that $\mF(\tau)\neq\emptyset,
\qquad
\forall \tau\in[-\varepsilon,\varepsilon].$
\end{assumption}

\begin{assumption}[LICQ]\label{assmp: LICQ}
 Let $(\bx^*(\mathtt{0}), \bbeta^*(\mathtt{0}))$ be an optimal solution to $\min \{\;\bc^\top \bx \, |\, (\bx, \bbeta) \in \mF(\mathtt{0})\}$. Linear independence constraint qualification (LICQ) holds at $(\bx^*(\mathtt{0}), \bbeta^*(\mathtt{0}))$. 
\end{assumption}


\begin{remark} Assumption~\ref{assmp: strict-strict-interior} is needed because of the continuity and differentiability requirements in our analysis (see Lemma~\ref{lem: G-diff-continuity-x}). Together, Assumptions~\ref{assmp: non-empty-interior} and~\ref{assmp: LICQ} provide the regularity conditions commonly assumed in nonlinear optimization. Under Assumption~\ref{assmp: non-empty-interior}, one can ensure the existence of the oracle $\mathcal{O}$ assumed in part (ii) of Theorem~\ref{thm: Cut-S-finite-convergence} by taking a PWL-inner approximation for the master problem and a PWL-outer approximation for the subproblem. 
\end{remark}

Recall that $\mX$ is assumed to be compact. 
Lemma~\ref{lem:beta-compactification} ensures that when considering the master problem \eqref{prob: master-DRCC-cont-dual}, we are dealing with compact feasible sets.
\begin{lemma}
\label{lem:beta-compactification}
Suppose that in the nominal Gaussian mixture distribution, $\hat w_k>0$ for all $k\in[\hat K]$,
$\sum_{k=1}^{\hat K}\hat w_k=1$, Wasserstein radius $\rho>0$, and
$\theta\in(0,1]$. Suppose also that
$\hat{\bs}_k\in\mS$ for all $k\in[\hat K]$ and that, at master iteration
$\ell$, $\{\hat{\bs}_1,\ldots,\hat{\bs}_{\hat K}\}\subseteq \mS_\ell.$ Then every feasible solution $(\bx,\bbeta)$ of the master problem \eqref{prob: master-DRCC-cont-dual} satisfies $\bbeta\in\overline{\mathfrak B},$
where $\overline{\mathfrak B}
:=
\prod_{k=1}^{\hat K}
\left[
1-\frac{1-\theta}{\hat w_k},
1
\right]
\times
\left[
0,\frac{1-\theta}{\rho}
\right]$ 
is compact.
\end{lemma}

\begin{lemma}(\cite[Lemma~2-3]{dey2025solving})\label{lem: G-diff-continuity-x}
Suppose Assumption~\ref{assmp: strict-strict-interior} holds. Then given any $\bs \in \mS$, $G(\bx, \bbeta; \bs)$ is continuously differentiable on $\mX \times \overline{\mfB}$.
\end{lemma}

\begin{lemma}\label{lem: optimality-master-problem}
Let Assumptions~\ref{assmp: compactness}-\ref{assmp: LICQ} hold, and at some master iteration $\ell$, $(\bx^{(\ell)}_*, \bbeta^{(\ell)}_*)$ and $(\bx^{(\ell)}, \bbeta^{(\ell)})$ respectively be the exact optimal solutions to~\eqref{prob: master-DRCC-cont-dual} and~\eqref{prob: master-DRCC-cont-dual-O}. Then there exists a finite PWL-O approximation of~\eqref{prob: master-DRCC-cont-dual} using  $O(\frac{1} {\sqrt{\epsilon}} \sqrt{\log(\frac{1}{\epsilon})})$ breakpoints so that  solution $(\bx^{(\ell)}, \bbeta^{(\ell)})$ obtained from~\eqref{prob: master-DRCC-cont-dual-O} is $\epsilon$-feasible to~\eqref{prob: master-DRCC-cont-dual} for some tolerance $\epsilon > 0$. The corresponding optimal value function $V^\star_\ell(\epsilon)$ is continuous at $\mathtt{0}$ and \[\bc^\top \bx^{(\ell)}_* - \bc^\top \bx^{(\ell)} \leq \epsilon.\] 
\end{lemma}

\begin{proof}
  See proof in Appendix~\ref{appndx:cut-alg-proof}. 
\end{proof}

When $b>0$, a master iterate close to $\bzero$ can make the normal-CDF term approach one. In numerical computation, this may appear as $\Phi\!\left(
\frac{b-\bmean^\top \bx^{(\ell)}}
{\sqrt{{\bx^{(\ell)}}^\top\bQ\bx^{(\ell)}}}
\right) \approx 1$ for all points in $\mS_\ell$. This numerical situation, however, is not used as a standalone termination certificate. Instead, we
treat it as a \emph{near-zero case} prior to Step 2 in Algorithm~\ref{alg: cutting-surface} and use a deterministic norm-based certificate.
Lemma~\ref{lem:near-zero-certificate} shows that for a feasibility
tolerance $\epsilon,$ if the iterate $\bx^{(\ell)}$ is in the $\epsilon$-neighborhood
of $\bzero$, then it is already $\epsilon$-feasible for all the constraints of the semi-infinite constraint set in~\eqref{prob: master-DRCC-cont-dual}. Hence, the separation subproblem step (Step 2 in Algorithm~\ref{alg: cutting-surface}) is
unnecessary, and the algorithm proceeds directly to the inner-approximation-based termination criteria checking step (Step 3).

\begin{lemma}[Near-zero feasibility certificate]
\label{lem:near-zero-certificate}
Suppose $b>0$, $\mS=\Theta\times\Xi$ is compact, and define
\[
\overline{m}_\Theta:=\max_{\bmean\in\Theta}\|\bmean\|,
\qquad
\overline{\lambda}_\Xi
:=
\max_{\bQ\in\Xi}\lambda_{\max}(\bQ).
\]
Let $\epsilon \in  (0, \min\left\{\hat{\tau}/2,\frac12\right\})$, and define $z_\epsilon
:=\Phi^{-1}(1-\epsilon),
\,\,\,
\delta_{\epsilon}
:=
\frac{b}{\overline{m}_\Theta+\sqrt{\overline{\lambda}_\Xi}\,z_\epsilon}$.
If $(\bx^{(\ell)},\bbeta^{(\ell)})$ is feasible for the finite master problem
at iteration $\ell$ and $\|\bx^{(\ell)}\|\le\delta_{\epsilon},$
then $J_k(\bs;(\bx^{(\ell)},\bbeta^{(\ell)}))\le\epsilon,
\,\,\,
\forall \bs\in\mS,\; k\in[\hat K]$. 
Equivalently, $(\bx^{(\ell)},\bbeta^{(\ell)})\in\mF(\epsilon).$
\end{lemma}
\begin{proof}
    
Fix any support point $\bs=(\bmean,\bQ)\in\mS$. Since
$\mS=\Theta\times\Xi$ is compact, the quantities
\[
\overline{m}_\Theta:=\max_{\bmean\in\Theta}\|\bmean\|,
\qquad
\overline{\lambda}_\Xi
:=
\max_{\bQ\in\Xi}\lambda_{\max}(\bQ)
\]
are finite. For $\bx^{(\ell)}\neq\bzero$, the  positive definiteness Assumption
on $\bQ\in\Xi$ gives ${\bx^{(\ell)}}^\top\bQ\bx^{(\ell)}>0.$ Moreover, $b-\bmean^\top\bx^{(\ell)}
\ge
b-\overline{m}_\Theta\|\bx^{(\ell)}\|,$
and $\sqrt{{\bx^{(\ell)}}^\top\bQ\bx^{(\ell)}}
\le
\sqrt{\overline{\lambda}_\Xi}\,\|\bx^{(\ell)}\|.$ 
Since
\[
\|\bx^{(\ell)}\|\le\delta_{\epsilon}
:=
\frac{b}{\overline{m}_\Theta+\sqrt{\overline{\lambda}_\Xi}\,z_\epsilon},
\qquad
z_\epsilon:=\Phi^{-1}(1-\epsilon),
\]
\[
b-\overline{m}_\Theta\|\bx^{(\ell)}\|
\ge
z_\epsilon\sqrt{\overline{\lambda}_\Xi}\,\|\bx^{(\ell)}\|.
\]
Therefore, $\frac{b-\bmean^\top\bx^{(\ell)}}
{\sqrt{{\bx^{(\ell)}}^\top\bQ\bx^{(\ell)}}}
\ge
z_\epsilon.$ Hence, $G(\bs;\bx^{(\ell)})
=
\Phi\!\left(
\frac{b-\bmean^\top\bx^{(\ell)}}
{\sqrt{{\bx^{(\ell)}}^\top\bQ\bx^{(\ell)}}}
\right)
\ge
\Phi(z_\epsilon)
=
1-\epsilon.$
The same conclusion holds for $\bx^{(\ell)}=\bzero$ since $G(\bs;\bzero)=1$ for $b>0$ by definition~\eqref{def: G}.

Next, since the algorithm is initialized with $\mS_0=\widehat{\mS}
:=
\{\hat{\bs}_1,\ldots,\hat{\bs}_{\hat K}\}$, we have $\widehat{\mS}\subseteq\mS_\ell
\,\,
\text{for every iteration }\ell.$
Therefore, for each $k\in[\hat K]$, the finite master problem contains the cut
generated by the nominal support point $\hat{\bs}_k$: $\beta_k^{(\ell)}
-
\beta_{\hat K+1}^{(\ell)}
d^2_{\mathrm{BW}}
\left(
\mN(\hat{\bs}_k),\mN(\hat{\bs}_k)
\right)
\le
G(\hat{\bs}_k;\bx^{(\ell)})$.
Since $d^2_{\mathrm{BW}}
\left(
\mN(\hat{\bs}_k),\mN(\hat{\bs}_k)
\right)=0$
and \(G(\hat{\bs}_k;\bx^{(\ell)})\le 1\), it follows that $\beta_k^{(\ell)}\le 1,
\,\,\,
\forall k\in[\hat K]$.

Now fix arbitrary $k\in[\hat K]$ and $\bs\in\mS$. Then
\[
\begin{aligned}
J_k(\bs;(\bx^{(\ell)},\bbeta^{(\ell)}))
&=
\beta_k^{(\ell)}
-
\beta_{\hat K+1}^{(\ell)}
d^2_{\mathrm{BW}}
\left(
\mN(\hat{\bs}_k),\mN(\bs)
\right)
-
G(\bs;\bx^{(\ell)})\le
\beta_k^{(\ell)}
-
G(\bs;\bx^{(\ell)}) \le 
1-(1-\epsilon) = \epsilon,
\end{aligned}
\]
where we used $\beta_{\hat K+1}^{(\ell)}\ge0$ and $d^2_{\mathrm{BW}}
\left(
\mN(\hat{\bs}_k),\mN(\bs)
\right)\ge0$. Since $k$ and $\bs$ were arbitrary, $J_k(\bs;(\bx^{(\ell)},\bbeta^{(\ell)}))\le\epsilon,
\,\,
\forall \bs\in\mS,\; k\in[\hat K].$
Thus, $(\bx^{(\ell)},\bbeta^{(\ell)})\in\mF(\epsilon)$.
\end{proof}

\subsubsection{\texorpdfstring{Proof of the optimality and finite termination claim of Theorem~\ref{thm: Cut-S-finite-convergence}.}{Proof of the optimality and finite termination claim}}
\label{sec: finite-convergence-proof}

\begin{proof}
\emph{\underline{Part-(i).}} Proof arguments for this part are established in the following order by setting $j=1$ and hence feasibility tolerance $\hat{\tau}/2$:
\begin{itemize}[leftmargin=1.0em]
    \item[--] claim that satisfying the Step-3 condition $J\bigl(\bs^{(\ell+1)};(\bx^{(\ell)},\bbeta^{(\ell)})\bigr)\le \hat{\tau}/2$ in Line~16 of Algorithm~\ref{alg: cutting-surface} is equivalent to obtaining $\hat{\tau}/2$-feasible solution to the semi-infinite program~\eqref{prob: semif-infinite}. This is proved by separately considering the regular case and the near-zero case in Line~11 of Algorithm~\ref{alg: cutting-surface}
\item[--]  show the objective value obtained from the current iterate $\bx^{(\ell)}$ is bounded above by \(\Val\eqref{prob: semif-infinite}\) based on its relaxation relation with master model~\eqref{prob: master-DRCC-cont-dual}
\item[--]  use compactness and continuity property to conclude that only finitely many points from \(\mS\) are needed
\end{itemize} 

First, suppose that, at iteration
\(\ell\), Step~2 of Algorithm~\ref{alg: cutting-surface} is executed. For
fixed \((\bx^{(\ell)},\bbeta^{(\ell)})\), define
\[
\check J_k(\bs;(\bx^{(\ell)},\bbeta^{(\ell)}))
:=
\beta_k^{(\ell)}
-
\beta_{\hat K+1}^{(\ell)}
d_{\mathrm{BW}}^2(\mN(\hat{\bs}_k),\mN(\bs))
-
\underline{\Phi}(g(\bs;\bx^{(\ell)})),
\]
and\vspace{-2em}
\[
J_k(\bs;(\bx^{(\ell)},\bbeta^{(\ell)}))
:=
\beta_k^{(\ell)}
-
\beta_{\hat K+1}^{(\ell)}
d_{\mathrm{BW}}^2(\mN(\hat{\bs}_k),\mN(\bs))
-
\Phi(g(\bs;\bx^{(\ell)})).
\]
Because \(\underline{\Phi}\le \Phi\), we have
\(J_k(\bs;(\bx^{(\ell)},\bbeta^{(\ell)}))
\le
\check J_k(\bs;(\bx^{(\ell)},\bbeta^{(\ell)}))\) for every
\(\bs\in\mS\) and \(k\in[\hat K]\). Let
\((\bar k,\bs^{(\ell+1)})\) solve $\max_{k\in[\hat K]}\max_{\bs\in\mS}
\check J_k(\bs;(\bx^{(\ell)},\bbeta^{(\ell)})).$
If the condition in Line~16 is satisfied, then
\(\check J_{\bar k}(\bs^{(\ell+1)};(\bx^{(\ell)},\bbeta^{(\ell)}))
\le \hat{\tau}/2\). Since \((\bar k,\bs^{(\ell+1)})\) maximizes
\(\check J_k\) over \(k\in[\hat K]\) and \(\bs\in\mS\), it follows that
\(\check J_k(\bs;(\bx^{(\ell)},\bbeta^{(\ell)}))\le\hat{\tau}/2\) for all
\(\bs\in\mS\) and \(k\in[\hat K]\). Combining this with \(J_k\le\check J_k\)
gives \(J_k(\bs;(\bx^{(\ell)},\bbeta^{(\ell)}))\le\hat{\tau}/2\) for all
\(\bs\in\mS\) and \(k\in[\hat K]\). Hence
\((\bx^{(\ell)},\bbeta^{(\ell)})\in\mF(\hat{\tau}/2)\).

Now suppose that Line~11 is invoked and the subproblem step is skipped. Let
\(\delta_{\hat{\tau}/2}\) be the radius in
Lemma~\ref{lem:near-zero-certificate} with \(\epsilon=\hat{\tau}/2\). If
\(b>0\) and \(\|\bx^{(\ell)}\|\le\delta_{\hat{\tau}/2}\), then
Lemma~\ref{lem:near-zero-certificate} gives
\(J_k(\bs;(\bx^{(\ell)},\bbeta^{(\ell)}))\le\hat{\tau}/2\) for all
\(\bs\in\mS\) and \(k\in[\hat K]\). Thus
\((\bx^{(\ell)},\bbeta^{(\ell)})\in\mF(\hat{\tau}/2)\) also in this case.
This proves that the condition used to enter Step 3 is a valid feasibility certificate for the
exact semi-infinite constrained problem~\eqref{prob: semif-infinite}.

We next prove the objective bound. Let \(\mF_\ell(\mathtt{0})\) and
\(\mF_\ell(\hat{\tau}/2)\) be the finite-cut feasible sets in
Definition~\ref{def: accurate-sol}, and let
\(\widetilde{\mF}_\ell(\mathtt{0})\) be the master feasible set defined by the constraints in~\eqref{prob: master-DRCC-cont-dual-O} using \(\bar\Phi\) with an approximation error tolerance $\hat{\tau}/2$. By construction and Theorem~\ref{thm: complexity-breakpoint},
\(\mF_\ell(\mathtt{0})\subseteq
\widetilde{\mF}_\ell(\mathtt{0})\subseteq\mF_\ell(\hat{\tau}/2)\). At iteration
\(\ell\), the master problem solves
\(\min\{\bc^\top\bx:(\bx,\bbeta)\in\widetilde{\mF}_\ell(\mathtt{0})\}\).
Therefore, for an optimal solution
\((\bx^{(\ell)},\bbeta^{(\ell)})\) of the master problem~\eqref{prob: master-DRCC-cont-dual-O}, $\min\{\bc^\top\bx:(\bx,\bbeta)\in\mF_\ell(\hat{\tau}/2)\} \le \bc^\top\bx^{(\ell)}
\le
\min\{\bc^\top\bx:(\bx,\bbeta)\in\mF_\ell(\mathtt{0})\}$. 
Since \(\mS_\ell\subseteq\mS\), set
\(\mF_\ell(\mathtt{0})\) is a relaxation of the exact feasible set
\(\mF(\mathtt{0})\). Hence
\(\min\{\bc^\top\bx:(\bx,\bbeta)\in\mF_\ell(\mathtt{0})\}
\le
\Val\eqref{prob: semif-infinite}\), and consequently
\(\bc^\top\bx^{(\ell)}\le\Val\eqref{prob: semif-infinite}\).

It remains to show that only finitely many cuts are needed to reach
\(\mF(\hat{\tau}/2)\). By Lemma~\ref{lem:beta-compactification}, because all
nominal cuts are included in the initial set \(\mS_0\), every feasible master
solution satisfies \(\bbeta\in\overline{\mathfrak B}\), where
\(\overline{\mathfrak B}\) is compact. Thus, the master variables lie in the compact set \(\mathcal Y=
\begin{cases}
\mX\times\overline{\mathfrak B}, & b>0,\\[2pt]
\mX_{\hat\delta}\times\overline{\mathfrak B}, & b=0,
\end{cases}\)
where the second case uses the zero-excluded domain in
Assumption~\ref{assmp: strict-strict-interior}. By
Lemma~\ref{lem: Delta-Continuity}, Lemma~\ref{lem: G-continuity-mu-sigma}, and
Lemma~\ref{lem: G-diff-continuity-x}, \(J_k(\bs,\bx,\bbeta)\) is continuous on
\(\mS\times\mathcal Y\) for every \(k\in[\hat K]\). Since \(\mS\) is compact
and \([\hat K]\) is finite, the semi-infinite system is a finite collection of
compactly indexed continuous constraints. Therefore, \(\mF(\mathtt{0})\) is a
closed subset of the compact set \(\mathcal Y\), and each finite-cut feasible
set \(\mF_\ell(\mathtt{0})\) is compact. The hypotheses of
\cite[Theorem~7.2]{hettich1993semi} hold. Applying that result with tolerance
\(\hat{\tau}/2\), Algorithm~\ref{alg: cutting-surface} reaches
\(\mF(\hat{\tau}/2)\) after adding finitely many points from \(\mS\). This proves part-(i).

\emph{\underline{Proof of Part-(ii).}} 
Let \(\epsilon_j:=\hat{\tau}/2^j, \, j\geq 1\). Note that, up to this point, we have worked with an outer approximation of the master model toward obtaining an \(\epsilon_j\)-feasible solution to \eqref{prob: semif-infinite}; in this part, we additionally consider an inner approximation of the master model to have an \(\epsilon_j\)-optimal solution to~\eqref{prob: semif-infinite}.  Also note that the proof of part-(i) does not rely on any condition on \(\hat{\tau}/2\) except that \(\hat{\tau}/2 > 0\), so the same argument there applies
with \(\epsilon_j\) in place of \(\hat{\tau}/2\). Thus, for each fixed \(j\),
only finitely many cuts are added before the algorithm obtains an
\(\epsilon_j\)-feasible point for the exact semi-infinite constrained problem~\eqref{prob: semif-infinite}.

Now, let \(\bx^*\) be an optimal solution of~\eqref{prob: semif-infinite}, and
\(\bx_*^{(\ell)}\) be an exact optimal solution of the finite master problem~\eqref{prob: master-CS-algrthm} at
iteration \(\ell\). Since the finite master problem contains only a subset of
the constraints of~\eqref{prob: semif-infinite},
\(\bc^\top\bx_*^{(\ell)}\le \bc^\top\bx^*=\Val\eqref{prob: semif-infinite}\).
Let \(V_\ell^\star(\epsilon_j)\) denote the optimal value over
\(\mF_\ell(\epsilon_j)\), and let \(V^\star(\mathtt{0})\) be the optimal value
of~\eqref{prob: semif-infinite} over $\mF(\mathtt{0})$. The master step gives
\(V_\ell^\star(\epsilon_j)\le
\bc^\top\bx^{(\ell)}\le
\bc^\top\bx_*^{(\ell)}
\le
\bc^\top\bx^*=V^\star(\mathtt{0})\).

In step 3 of Algorithm~\ref{alg: cutting-surface}, Line~17 uses the oracle $\mathcal{O}$
to return a feasible solution \((\tilde{\bx},\tilde{\bbeta})\) of
\eqref{prob: semif-infinite} by solving~\eqref{prob: master-DRCC-cont-dual} at $\epsilon_j$ optimality. Specifically, for any prescribed \(\tau>0\), we solve~\eqref{prob: master-DRCC-cont-dual}
using a piecewise-linear inner approximation of \(\Phi\), yielding the restricted
feasible set \(\check{\mF}_{\ell}(\mathtt{0})\). Assumption~\ref{assmp: non-empty-interior} ensures that \(\min\{\bc^\top\bx:(\bx,\bbeta)\in\widetilde{\mF}_\ell(\mathtt{0})\}\) has a solution.  Moreover, similar to the piecewise-linear outer
approximation, inner approximation construction and
Theorem~\ref{thm: complexity-breakpoint-inner} imply \(\mF_\ell(-\tau) \subseteq \check{\mF}_\ell(0) \subseteq\mF_\ell(\mathtt{0})\).

Now, Lemma~\ref{lem: optimality-master-problem} guarantees that iterative refinement of such piecewise linear approximation accuracy $\tau$ achieves any desired $\epsilon_j$-optimal solution to~\eqref{prob: master-DRCC-cont-dual} under finite approximation accuracy through establishing continuity of
the optimal value function \(V^\star(\cdot)\) at the origin.
Hence, for every \(\epsilon_j>0\), there exists \(\tau^+(\epsilon_j)>0\) such
that \(0<\tau\le\tau^+(\epsilon_j)\) implies
\(|V^\star(\mathtt{0})-V^\star(\tau)|\le\epsilon_j\) and
\(|V^\star(\mathtt{0})-V^\star(-\tau)|\le\epsilon_j\). Therefore,
\[
|V^\star(-\tau)-V^\star(\tau)|
\le
|V^\star(-\tau)-V^\star(\mathtt{0})|
+
|V^\star(\mathtt{0})-V^\star(\tau)|
\le
2\epsilon_j
=
2\hat{\tau}/2^j \leq \hat{\tau} .
\]
Because \(\epsilon_j\) is halved whenever \(j\gets j+1\), a finite \(j\) is
reached for which the corresponding tolerance lies inside the required
continuity neighborhood and \(2\epsilon_j\le\hat{\tau}\). For this \(j\), the
feasible solution \(\tilde{\bx}\) and the certificate iterate
\(\bx^{(\ell^\star)}\) satisfy
\(|\bc^\top\tilde{\bx}-\bc^\top\bx^{(\ell^\star)}|\le\hat{\tau}\) at some terminating iteration $\ell^\star$.

For each fixed \(j\), the preceding paragraph and \emph{part (i)} imply finite
termination of the cut-generation step. Since only finitely many updates of
\(j\) are required, the total number of master iterations is finite. Hence
\(j<\infty\), \(\ell^\star<\infty\), and
Algorithm~\ref{alg: cutting-surface} returns a \(\hat{\tau}\)-accurate solution in finitely many iterations.
\end{proof}
\section{Case Study}
\label{sec:numerical}
We now discuss our computational experience with continuous support-based robustification of the chance constraint, specified using the Gaussian mixture distribution, as studied in the earlier sections. We discuss this in the context of a service-level planning case study for EV charging stations, developed using a public dataset from~\cite{evcharging2025figshare}.  Our case study is motivated by the broader EV charging literature on the problem of aggregate charging-profile control and planning over time, as well as service-level-based charging infrastructure planning; see, for example,~\cite{liu2022data, altaha2022aggregate}. We introduce the model in Section~\ref{sec: EV-model} and describe our computational settings in Section~\ref{sec: compt-setup}. This section also specifies Wasserstein-2 ambiguity set radii and the mean--covariance support sets used in our experiments. Section~\ref{sec: implementation} presents the implementation details for both the finite-support DR (FDR) and continuous-support DR (CDR) formulations of the EV model. Section~\ref{sec:numerical-results} compares properties of the solutions generated by the two DR approaches along with the nominal distribution in terms of objective value, out-of-sample performance, and computation time. Section~\ref{sec:compare-solution} compares the energy-allocation profile in the nominal, FDR, and CDR solutions using different metrics and visual tools to provide insights into how robustness reshapes the energy profile in these models.

\subsection{An Aggregate EV Charging Service-Level Planning Model}\label{sec: EV-model}

We consider a least-cost EV charging-station energy allocation model subject to a minimum daily demand-fulfillment requirement as follows:


\vspace{-1em}
\begin{equation}
\begin{aligned}
& \min_{\bx\in [0, 1]^T}\quad \sum_{t=1}^T\Bigl(c^f_t\,x_t + C^o_t(x_t)\Bigr)
\quad \mathrm{s.t.}\quad
\mbP\!\left(\sum_{t=1}^T \xi_t\,x_t\ \ge\ E_{\min}\right)\ \ge\ \theta.
\end{aligned}
\label{eq:ev_model_generic}
\end{equation}
Here \(\xi_t\) denotes the random charging demand in hour \(t\), and \(x_t\in[0,1]\) denotes the fraction of that demand served in that hour. \(\xi_t x_t\) is the served demand in hour \(t\), and \(\sum_{t=1}^T \xi_t x_t\) is the total demand served over the day. The parameter \(E_{\min}\) specifies a minimum daily served-energy requirement, and the chance constraint requires this requirement to be met with probability at least \(\theta\). The objective consists of a baseline linear cost \(c_t^f x_t\), where \(c_t^f\) captures the hour-\(t\) time-of-use cost, and an additional operational cost term \(C_t^o(x_t)\), which allows the cost of serving demand to vary nonlinearly with the demand fulfillment fraction. The convex piecewise-linear form provides a regime-dependent penalty across service levels.

\subsubsection{Instance Generation}. 
To instantiate the EV demand-management model~\eqref{eq:ev_model_generic} from the data in \cite{evcharging2025figshare}, each charging session was partitioned across clock hours in proportion to the time spent in each hour and then aggregated into a daily 24-dimensional demand vector.  The entire dataset was divided for training and out-of-sample testing using a 60/40 split. The setting of the cost parameters and the procedure for estimating necessary parameters for the Gaussian distribution model~\eqref{eq:ev_model_generic} from the source data and its training/testing split are elaborated in Appendix~\ref{appndx: data-preparation}.

 The nominal Gaussian mixture model fitted to the training data had $\hat{K} = 5$ components. It was determined as follows. Gaussian mixtures with \(k=1,\ldots,20\) components were fitted using the \texttt{GaussianMixture} class in the \texttt{scikit-learn} package. In the package-specific configuration, the covariance type was set to \texttt{full} and diagonal regularization was set to \(10^{-6}\). For each fixed order \(k\), the native expectation-maximization (EM) algorithm was run from ten different initializations, and the fit with the largest log-likelihood value was retained. This entire order-selection procedure was then repeated over ten independent random seeds. In each replicate, the selected order was the value of \(k\) minimizing the Bayesian information criterion (BIC). Sorting these ten replicate-specific selections, we took the lower of the two middle values as the final choice, yielding \(\hat K=5\) for the fitted demand data. The resulting nominal distribution had mixture weights \(\hat{\bw}=(0.167,0.167,0.204,0.241,0.222)\).

 \subsubsection{Data preparation for out-of-sample performance evaluation.}\label{appndx: OSS-data}\,  For the out-of-sample performance analysis, energy allocation solutions obtained from the FDR and CDR models were evaluated against demand samples generated from a Gaussian mixture model fitted on the holdout data. Gaussian mixture models with orders \(k=1,\ldots,20\) were fitted using the EM algorithm, with 15 initializations per order and up to 100 EM iterations per fit; for each \(k\), the best fit among the 15 initializations was retained.  Let \(K^\circ\) correspond to the one with the smallest BIC over \(k=1,\ldots,20\). out-of-sample satisfaction (OSS) probability was then computed from samples drawn for \(K^\circ, K^\circ-1\) and \(K^\circ+1\), using Monte Carlo sample sizes of 500 for each, and reporting the proportion of sampled demand vectors \(\bxi\) for which \(\bxi^\top \bx \ge E_{\min}\) with candidate allocation $\bx$. Such an approach is statistically motivated: finite-mixture likelihoods are non-convex, EM estimates may depend on initialization, and adjacent mixture orders often have nearly indistinguishable information-criterion values on moderate sample sizes. Evaluating OSS at \(K^\circ-1\), \(K^\circ\), and \(K^\circ+1\) therefore serves as a robustness check against model-order and local-optimum sensitivity rather than relying on a single fitted mixture.

\subsection{Experimental Settings}\label{sec: compt-setup}
All computations were performed on a server with Intel Xeon 2.80 GHz CPUs using Python 3.11. Gurobi 11.0.1 was used to solve master and subproblems that do not involve positive semi-definite (PSD) constraints; Mosek 11.1.6 was used otherwise. A thread count of 2 was set for the master problem and each subproblem. For both the finite- and continuous-support formulations, a violation tolerance of $10^{-4}$
 was used as the stopping threshold for cut generation. The target satisfaction probability \(\theta\), representing the desired service level, was set to take values in \(\{0.95,0.97,0.99\}\). 
 The cutting-surface algorithm used to solve CDR was run for four master iterations. For each of the five master problems, Gurobi’s \texttt{MIPGap} parameter was set to 0.1\%. Since solving nonlinear mixed-integer programs generated in the cutting surface algorithm to optimality is time-consuming (and not required), an iteration-dependent time limit, as specified in implementation detail Section~\ref{sec:approaches_cont}, was followed to keep the total computation time $\leq 24$ hours for each instance.  All other Gurobi and Mosek settings were kept at their default values.  

\subsubsection{Wasserstein-2 set. }\label{sec: mean-cov-spec}\, 
We experimented with Wasserstein-radius parameters \(\rho \in \{0.001,0.005,0.01\}\). For the CDR case, the compact support $\mS:=\Theta\times\Xi$ was generated as follows. The mean support set \(\Theta\) was specified by componentwise bounds \(\underline{\bmean} \le \bmean_k \le \bar{\bmean}\) for all \(k \in [K]\) which along each coordinate \(j \in [n]\) (for the case study $j$ and $[n]$ respectively correspond to $t$ and $[T]$) are obtained from the nominal component means as:\vspace{-1.5em}
\[
\underline{m}_j
=
\min_{k \in \hat K} m_{jk}
-
\varsigma\left|\min_{k \in \hat K} m_{jk}\right|,
\qquad
\bar{m}_j
=
\max_{k \in \hat K} m_{jk}
+
\varsigma\left|\max_{k \in \hat K} m_{jk}\right|.
\]
Similarly, the covariance support set \(\Xi\) was specified by imposing multiplicative semidefinite bounds around the nominal covariance matrices \(\hat{\bQ}_k\), namely, $\underline{\lambda}_f \hat{\bQ}_k \preceq \bQ_k \preceq \overline{\lambda}_f \hat{\bQ}_k,
\,\, \forall k\in[K]$. We considered two settings for the mean and covariance bounds: \(\varsigma \in \{0.05,\,0.1\}\) and \((\underline{\lambda}_f,\overline{\lambda}_f)\in \{(1/2,2),\,(1/3,3)\}\).
 
\subsection{Implementation Details of Nominal Model and DR Variants}
\label{sec: implementation}

\subsubsection{Nominal model. }\,
Given the GMM fitted to the EV-station demand data, the nominal counterpart of model~\eqref{eq:ev_model_generic} is $\min_{\bx\in[0,1]^T}\sum_{t=1}^T\bigl(c_t^f x_t + C_t^o(x_t)\bigr)
\,\,\, \mathrm{s.t.} \,\,\,
\sum_{k=1}^{\hat K} \hat w_k \,
\Phi\!\left(\frac{E_{\min}-\bmean_k^\top \bx}{\sqrt{\bx^\top \bQ_k \bx}}\right)\ge \theta,$ which we solved via a piecewise-linear approximation of \(\Phi(\cdot)\) and the corresponding chance-constraint reformulation in~\eqref{eq: reform-outer-master}. Note that obtaining the array of breakpoints for the piecewise-linear approximation of $\Phi(\cdot)$ from Algorithm~\ref{alg:breakpoint_find}-\ref{alg:breakpoint_find_negative} requires the approximation accuracy tolerance $\tau$; the same approximation tolerance $\hat{\tau}:=\tau = 10^{-4}$ was used to decide the number of points in the piecewise linear approximations of the univariate Gaussian for the nominal model and the DR instances. The outermost loop in Algorithm~\ref{alg: cutting-surface} that iteratively reduces the approximation accuracy and tolerances was not performed.

\subsubsection{Distributionally robust models. }\,
\label{sec:approaches_cont}

We solve~\eqref{prob: semif-infinite-FDR} as the FDR benchmark with the unknown support set restricted to the nominal GMM components. Since formulation~\eqref{prob: semif-infinite-FDR} contains only finitely many constraints, it can be handled directly by an off-the-shelf solver such as Gurobi and does not require iterative subproblem separation. 
By Corollary~\ref{cor: FDR-dual} in Section~\ref{sec: Dual-formulation} with $K = \hat{K}$, the iteration-0 master problem in the cutting-surface algorithm corresponds to the FDR model~\eqref{prob: semif-infinite-FDR}. An iteration-specific computational time budget was used in the cutting surface algorithm implementation.

\subparagraph{Iteration-dependent master iteration time limits.}

For the CDR model, one separation problem is solved for each \(k\in[\hat K]\). These \(\hat K\) subproblem solutions may generate multiple violated cuts, and each admitted cut introduces a nonconvex constraint together with a new block of binary variables. As a result, the master problem becomes increasingly computationally demanding across iterations. To manage this computational burden, rather than solving each master problem to a prescribed optimality gap, we impose iteration-dependent time limits of 2, 2, 4, 6, and 8 hours for iterations 0 through 4, respectively. This practical adaptation does not compromise the CSA algorithmic framework, as an optimal solution of the master problem in this framework is required when proving optimality. Sub-optimal solutions to the master problem are acceptable for generating the cutting surfaces. Since we have FDR solutions available from two separate contexts: under iteration-0 of the cutting surface algorithm as well as from its stand-alone run of 24 hours, we report both for performance comparisons.


Since the iteration-adaptive schedule allocates only up to 2 hours to iteration 0, which may be insufficient to obtain an FDR solution with  \texttt{MIPGap}=0.1\%, we solved~\eqref{prob: semif-infinite-FDR} as a standalone benchmark run with a time limit set to 24 hours, which equals the time allocated to the first four CDR iterations.

\subparagraph{Separation problem.}
Let $\bx^{(\ell)}\neq\bzero$ and $\bbeta^{(\ell)}$ be the master solution at some iteration $\ell$. For each nominal component \(k\in[\hat K]\), the separation problem searches over a support point \(\bs=(\bmean,\bQ)\in\mS\) to identify a violated semi-infinite constraint.  We suppress the subscript $k$ in the following description. Thus,  \(\hat{\bmean}\), \(\hat{\bQ}\), and \(\beta^{(\ell)}\) denote \(\hat{\bmean}_k\), \(\hat{\bQ}_k\), and \(\beta_k^{(\ell)}\), respectively. Note that \(\beta_{K+1}^{(\ell)}\) does not depend on $k \in [K]$. Additionally, after replacing \(\Phi(\cdot)\) by its PWL inner approximation and using \(\zeta:=-\zeta'\), and the mean-covaraince support set $\Theta \times \Xi$ specified in Section~\ref{sec: mean-cov-spec}, the separation problem~\eqref{prob: subproblem-1.a}-\eqref{prob: subproblem-1.c} can be written in the minimization form as\vspace{-1.0em}
\begin{equation}
\begin{aligned}
\hspace{-1em}\min_{\bmean,\bQ,\sigma,z,\zeta,\balpha,\bt,\by}
&
\hspace{-0.5em}-\beta^{(\ell)}
+
\beta_{\hat K+1}^{(\ell)}
\Big[
\|\bmean-\hat{\bmean}\|_2^2
+
\Tr(\bQ)+\Tr(\hat{\bQ})
-
2\Tr\!\left(
\big(\bQ^{1/2}\hat{\bQ}\bQ^{1/2}\big)^{1/2}
\right)
\Big]
+
\zeta
\\
\mathrm{s.t.}\quad
&
 \underline{\bmean} \, \le \bmean\le \bar{\bmean}, \qquad \underline{\lambda_f}\hat{\bQ}\preceq \bQ\preceq \overline{\lambda}_f\hat{\bQ},\\
 & {\bx^{(\ell)}}^\top\bQ\bx^{(\ell)}=\sigma^2,
\,\,\, 
b-\bmean^\top\bx^{(\ell)}\le z\,\sigma,
\,\,\,
(z,\zeta,\balpha,\bt,\by)\in\mH^I(\bzz)_+ ,
\end{aligned}
\label{prob: sep-original-form}
\end{equation}
where, \(\mH^I(\bzz)_+\) denotes the set $\mH^I(\bzz)$ except using $\zeta := -\zeta'$ for $\zeta \in [0, 1]$.

The above model is not directly representable using the algebraic modeling language of the off-the-shelf solver Mosek. Although the mean-distance term and the trace terms \(\Tr(\bQ)+\Tr(\hat{\bQ})\) are directly representable, the difficulty arises from the matrix square-root term $\Tr\!\big(
(\bQ^{1/2}\hat{\bQ}\bQ^{1/2})^{1/2}
\big),$ which couples the variable covariance matrix \(\bQ\) with the nominal covariance matrix \(\hat{\bQ}\). To express this term through semidefinite constraints, we use the identity
\[
\Tr\!\left(
\left(\bQ^{1/2}\hat{\bQ}\bQ^{1/2}\right)^{1/2}
\right)
=
\max_{\bW}
\left\{
\Tr(\bW):
\begin{bmatrix}
\bQ & \bW\\
\bW^\top & \hat{\bQ}
\end{bmatrix}\succeq0
\right\}.
\]
Substituting this representation into~\eqref{prob: sep-original-form} gives the following \(\bW\)-based form:
\begin{equation}
\begin{aligned}
\min_{\bmean,\bQ,\bW,\sigma,z,\zeta,\balpha,\bt,\by}
\quad
&
-\beta^{(\ell)}
+
\beta_{\hat K+1}^{(\ell)}
\Bigl[
d_{m,k}(\bmean)+d_{Q,k}(\bQ,\bW)
\Bigr]
+
\zeta
\\
\mathrm{s.t.}\quad
&
\underline{\bmean} \, \le \bmean\le \bar{\bmean}, \qquad \underline{\lambda_f}\hat{\bQ}\preceq \bQ\preceq \overline{\lambda}_f\hat{\bQ},\\
&
\begin{bmatrix}
\bQ & \bW\\
\bW^\top & \hat{\bQ}
\end{bmatrix}\succeq0,
\\
&
{\bx^{(\ell)}}^\top\bQ\bx^{(\ell)}=\sigma^2, \quad
b-\bmean^\top\bx^{(\ell)}\le z\,\sigma, \quad
(z,\zeta,\balpha,\bt,\by)\in\mH^I(\bzz)_+ ,
\end{aligned}
\label{prob: sep-W-form}
\end{equation}
where\vspace{-1.5em}
\[
d_{m}(\bmean):=\|\bmean-\hat{\bmean}\|_2^2,
\qquad
d_{Q}(\bQ,\bW)
:=
\Tr(\bQ)+\Tr(\hat{\bQ})-2\Tr(\bW).
\]
The resultant model~\eqref{prob: sep-W-form} is a reformulation to~\eqref{prob: sep-original-form}. Indeed, for a fixed \(\bQ\), minimizing \(d_{Q}(\bQ,\bW)\) is equivalent to maximizing \(\Tr(\bW)\) over the block PSD constraint. Since \(\beta_{\hat K+1}^{(\ell)}\ge0\), the objective selects such a maximizer whenever the covariance-distance term is active. If \(\beta_{\hat K+1}^{(\ell)}=0\), the Bures--Wasserstein distance term is absent from the objective, and the choice of \(\bW\) does not affect the separation objective.

The separation problem reformulation~\eqref{prob: sep-W-form} is a mixed-binary problem with PSD constraints, quadratic equalities ($\sigma^2=\bx^\top \bQ \bx$), and bilinear terms. Despite the simplification through reformulation~\eqref{prob: sep-W-form}, Gurobi and Mosek (and other commercial mixed-integer solvers, to our knowledge) are not advanced enough to accept this formulation in its native form. Therefore, we developed an alternating block minimization heuristic to obtain solutions to ~\eqref{prob: sep-W-form}. It is acceptable (possibly desirable) to use a heuristic in the cut generation subproblems at an earlier stage of CSA since its optimal solution is required only in proving optimality. 

\subparagraph{An alternating block minimization heuristic.} We propose an alternating block minimization heuristic by leveraging the fact that the cross-block coupling arising in the subproblem~\eqref{prob: sep-W-form} only couples through the link constraint $b(\bmean;\bx^{(\ell)})  \coloneq b - \bmean^\top \bx^{(\ell)} \le z\,\sigma,$ which decomposes in two blocks of variables: $\bmean$-block variables $(\bmean, z, \zeta)$ and $\bQ$-block variables $(\bQ, \bW, \sigma)$. All other constraints and objectives are
block-separable, and the interaction between the two blocks is confined to the
product of the scalar variables \(z\) and \(\sigma\). At iteration $r$ of this alternating block minimization heuristic, we present the formulations below for both blocks:
\begin{equation*}
\begin{minipage}[t]{0.5\linewidth}
\textbf{\(\bmean\)-block at sub-iteration \ \(r\):}\\[-5pt]
\[
\hspace{-4em}
\begin{aligned}
\min_{\bmean,z,\zeta,\,\balpha,\bt,\by}
&-\beta^{(\ell)}+\beta_{K+1}^{(\ell)}\,d_m(\bmean)
+\beta_{K+1}^{(\ell)}\,\bar c_Q^{(r)}+\zeta\\
\mathrm{s.t.}\quad
& \underline{\bmean} \, \le \bmean\le \bar{\bmean},\\
&b(\bmean;\bx^{(\ell)})\le z\,\sigma^{(r)},\\
&(z,\zeta,\balpha,\bt,\by)\in \mH^I(\bzz)_+,\\
\text{where,} \quad &\bar c_Q^{(r)}:=d_Q\!\big(\bQ^{(r)},\bW^{(r)}\big),\\
&\sigma^{(r)}:=\sqrt{{\bx^{(\ell)}}^\top \bQ^{(r)}\bx^{(\ell)}}.
\end{aligned}
\]
\end{minipage}\hfill
\hspace{-2em}
\begin{minipage}[t]{0.5\linewidth}
\textbf{\(\bQ\)-block at sub-iteration\ \(r+1\):}\\[-5pt]
\[
\begin{aligned}
\min_{\bQ,\bW,\sigma}
&-\beta^{(\ell)}+\beta^{(\ell)}_{K+1}\,\bar c_m^{(r+1)}
+\beta^{(\ell)}_{K+1}\,d_Q(\bQ,\bW)+\zeta^{(r+1)}\\
\mathrm{s.t.}\quad
&\underline{\lambda_f}\hat{\bQ}\preceq \bQ\preceq \overline{\lambda}_f\hat{\bQ},\\
&\begin{bmatrix}\bQ&\bW\\ \bW^\top&\hat{\bQ}\end{bmatrix}\succeq 0,\\
&\sigma^2\le {\bx^{(\ell)}}^\top \bQ \bx^{(\ell)},\\
& b(\bmean^{(r+1)};\bx^{(\ell)})\le z^{(r+1)}\sigma, \\
\text{where,} \,\, 
&\bar c_m^{(r+1)}:=d_m\!\big(\bmean^{(r+1)}\big).
\end{aligned}
\]
\end{minipage}
\label{eq:abm-two-blocks}
\end{equation*}

In our implementation, the heuristic was initialized on the covariance side by setting
\(\bQ^{(0)}=\hat{\bQ}\) and \(\bW^{(0)}=\hat{\bQ}\), that is, at the nominal
covariance pair of the Gaussian component under consideration. With
\((\bQ^{(0)},\bW^{(0)})\) fixed, the algorithm started by solving the
\(\bmean\)-block to obtain \((\bmean^{(1)}, z^{(1)}, \zeta^{(1)})\). Fixing these coefficients it then
solves the \(\bQ\)-block problem to obtain \((\bQ^{(1)},\bW^{(1)})\), and iterates from thereon.

Mosek was used to solve the $\bQ$-block of the subproblem containing PSD constraints with a tolerance of $10^{-6}$.  The alternating block minimization heuristic was run for up to $r=15$ iterations. A resulting candidate solution was added only if its violation exceeded the threshold $10^{-4}$.

\subsection{Numerical Findings}
\label{sec:numerical-results}

This section reports computational insights from the CDR and FDR experiments.
Across all parameter settings of the CDR model, cumulative solution times and
cut counts over the master iterations of the cutting-surface approach indicate
that the computational burden is driven primarily by repeated master-problem reoptimization rather than by separation. For both mean-support expansion values
\(\varsigma=0.05\) and \(\varsigma=0.1\), the separation heuristic accounts for
less than \(1\%\) of total computation time. By contrast, the master problems
become increasingly difficult as the algorithm proceeds, because each admitted
cut adds a nonconvex constraint together with an additional block of binary
variables. This effect becomes more pronounced at larger target satisfaction
levels \(\theta\).  The number of admitted cuts remains relatively stable across all the tested settings. After four CDR iterations, the total number of cuts is observed to be between \(21\) and \(23\). Early CDR
iterations are solved within minutes, whereas later iterations may require hours
and, in the most difficult cases, the cumulative runtime reaches the \(24\)-hour
budget. Detailed iteration-wise breakdown of solution times and cut counts is reported in Tables~\ref{tab:cumulative-time-cut-0.1}
and~\ref{tab:cumulative-time-cut-0.05} (see Appendix~\ref{appndx: visuals}) separately for
\(\varsigma=0.1\) and \(\varsigma=0.05\). Setting \(\varsigma=0.1\) generally requires more computation time than \(\varsigma=0.05\), which is noticeable for \(\theta\geq0.97\). This additional expense is mainly attributable to the larger cut space, which makes identifying highly violated cuts more costly. Runtime differences, however, are
less distinguishable for different covariance-scaling parameter \(\underline{\lambda}_f\) values across all these test settings.

The FDR model is evaluated and compared under two computational contexts. The first is a
benchmark setting that solves the stand-alone FDR formulation to the desired
optimality gap of \(10^{-4}\) within the same \(24\)-hour time budget. The
second uses the CDR iteration-0 master problem as an FDR reference, where runtime
is capped below \(2\) hours by the iteration-adaptive schedule. This distinction
is important because the stand-alone FDR solve may still consume the full
\(24\)-hour budget toward the desired optimality gap, whereas the iteration-0
CDR master is deliberately restricted to a shorter runtime. For FDR, the effect
of a larger ambiguity radius \(\rho\) on computation time is seen for
higher satisfaction targets $\theta \geq 0.97$. For CDR, this effect is further amplified by larger
mean-support uncertainty. These two factors together make master problem reoptimization more computationally demanding.

\subsubsection{Out-of-Sample Performance and the Cost-Reliability Tradeoff}. 
Table~\ref{tab:ev-objdiff-oss-theta-rho-iter7-pmax560-vsdisc} quantifies the associated cost--reliability tradeoff. The nominal objective is constant across \(\rho\) for each fixed \(\theta\) because \(\rho\) enters only through the ambiguity set. Its out-of-sample satisfaction (OSS) probability in the hold-out data is \(92.52\%\), \(94.02\%\), and \(95.20\%\), which is significantly below the desired probabilities at \(\theta=0.95\), \(0.97\), and \(0.99\), respectively. The FDR model changes this baseline only modestly.  Across all nine \((\theta,\rho)\) combinations, the FDR objective increase remains \(0.06\%\) to \(0.77\%\), and the corresponding OSS probability gain over the nominal model is at most \(1.39\%\). In none of the reported cases does FDR increase the empirical satisfaction probability to the target level irrespective of runtime settings. Also note that increasing the FDR runtime beyond the iteration-0 budget minimally improves OSS probability: OSS values under both runtime settings are identical in five of the nine \((\theta,\rho)\), with small differences in the case \(\theta=0.99\). Particularly at \(\theta=0.99\), although the additional computatonal time was about \(20\) hours, the OSS improvement was only \(0.28\%\) to \(0.33\%\) points. Hence, increasing the provable accuracy of a solution to the FDR model has some, but incremental value.


\begin{table}[!hbtp]
\caption{Nominal objective values and out-of-sample satisfaction probabilities (OSS, \%) are reported as baselines for all $\rho\in\{0.001,0.005,0.01\}$ and $\theta\in\{0.95,0.97,0.99\}$. For FDR, OSS values are reported both from a dedicated benchmark FDR run with \texttt{MIPGap}$=0.1\%$ and the same 24-hour runtime budget as CDR, and from iteration 0 of the cutting surface approach. Relative objective difference with respect to nominal are calculated as $100\frac{|\operatorname{Obj}^{\mathrm{nominal}}-\operatorname{Obj}^{\mathrm{DR}}|}{\operatorname{Obj}^{\mathrm{nominal}}}$ and represented under Obj Diff (\%) column. For the CDR approach, OSS probability and relative objective differences are reported for four different hyperparameter $\text{(mean perturbation, covariance scaling factor)} \, = \,(\varsigma,\underline{\lambda}_f)$ settings.
}
\label{tab:ev-objdiff-oss-theta-rho-iter7-pmax560-vsdisc}
\centering
\begin{threeparttable}
\resizebox{\textwidth}{!}{%
\begin{tabular}{cc|cc|ccc|cc|cc|cc|cc}
\toprule
&  & \multicolumn{2}{c|}{Nominal} & \multicolumn{3}{c|}{FDR} & \multicolumn{8}{c}{CDR} \\
\cmidrule(lr){5-7}\cmidrule(lr){8-15}
& & & & & \multicolumn{2}{c|}{OSS (\%)} 
& \multicolumn{2}{c|}{$\varsigma=0.05,\ \underline{\lambda}_f=0.333$}
& \multicolumn{2}{c|}{$\varsigma=0.05,\ \underline{\lambda}_f=0.5$}
& \multicolumn{2}{c|}{$\varsigma=0.10,\ \underline{\lambda}_f=0.333$}
& \multicolumn{2}{c}{$\varsigma=0.10,\ \underline{\lambda}_f=0.5$} \\
\cmidrule(lr){6-7}\cmidrule(lr){8-9}\cmidrule(lr){10-11}\cmidrule(lr){12-13}\cmidrule(lr){14-15}
$\theta$ & $\rho$ & Obj & OSS (\%) & \makecell{Obj\tnote{a} \\ Diff (\%)} 
& Benchmark & \makecell{From \\CDR\\Iter 0}
& \makecell{Obj \\Diff (\%)\\ (FDR)\tnote{b}} & OSS (\%)
& \makecell{Obj \\Diff (\%)\\ (FDR)\tnote{b}} & OSS (\%)
& \makecell{Obj \\Diff (\%)\\ (FDR)\tnote{b}} & OSS (\%)
& \makecell{Obj \\Diff (\%)\\ (FDR)\tnote{b}} & OSS (\%) \\
\midrule
& 0.001 & 7011.376 & 92.52 & 0.10 & 92.76 & 92.76 & 1.59 (1.49) & 95.040 & 1.59 (1.49) & 95.060 & 2.47 (2.37) & 95.867 & 2.47 (2.37) & 95.867 \\
0.95 & 0.005 & 7011.376 & 92.52 & 0.43 & 93.45 & 93.45 & 2.42 (1.98) & 95.940 & 2.42 (1.98) & 95.940 & 3.72 (3.28) & 97.187 & 3.74 (3.29) & 97.173 \\
 & 0.010 & 7011.376 & 92.52 & 0.77 & 93.91 & 93.91 & 3.10 (2.31) & 96.413 & 3.10 (2.31) & 96.413 & 4.43 (3.63) & 97.673 & 4.43 (3.63) & 97.713 \\
\midrule
 & 0.001 & 7073.915 & 94.02 & 0.08 & 94.09 & 94.09 & 1.73 (1.65) & 96.207 & 1.72 (1.64) & 96.207 & 2.62 (2.53) & 96.787 & 2.62 (2.53) & 96.787 \\
0.97 & 0.005 & 7073.915 & 94.02 & 0.37 & 94.53 & 94.53 & 2.49 (2.12) & 96.793 & 2.49 (2.12) & 96.793 & 3.98 (3.60) & 98.107 & 3.98 (3.60) & 98.107 \\
 & 0.010 & 7073.915 & 94.02 & 0.65 & 95.19 & 94.69 & 3.14 (2.48) & 97.360 & 3.14 (2.48) & 97.353 & 4.66 (3.99) & 98.453 & 4.66 (3.99) & 98.440 \\
\midrule
& 0.001 & 7176.816 & 95.20 & 0.06 & 95.57 & 95.27 & 1.99 (1.93) & 97.447 & 1.99 (1.93) & 97.447 & 2.89 (2.82) & 98.100 & 2.89 (2.82) & 98.100 \\
0.99 & 0.005 & 7176.816 & 95.20 & 0.25 & 95.60 & 95.32 & 2.86 (2.60) & 98.160 & 2.86 (2.60) & 98.173 & 4.51 (4.25) & 98.833 & 4.51 (4.25) & 98.833 \\
 & 0.010 & 7176.816 & 95.20 & 0.41 & 95.59 & 95.26 & 3.58 (3.16) & 98.487 & 3.59 (3.16) & 98.487 & 5.44 (5.00) & 99.173 & 5.44 (5.00) & 99.173 \\
\bottomrule
\end{tabular}%
}
\begin{tablenotes}
\scriptsize
\item[a] Identical across the two FDR settings up to two decimals and hence reported once.
\item[b] \parbox[t]{0.95\textwidth}{%
The non-parenthesized value is relative to the nominal solution, while the parenthesized value is relative to FDR.}
\end{tablenotes}
\end{threeparttable}
\end{table}

\begin{table}[!hbtp]
\centering
\caption{FDR results for the larger ambiguity radii $\rho\in\{5,10\}$ under the same OSS evaluation setting used in Table~\ref{tab:ev-objdiff-oss-theta-rho-iter7-pmax560-vsdisc}. For each entry, the non-parenthesized objective difference is computed relative to the nominal objective at the same $\theta$, while the parenthesized value is computed relative to the FDR baseline at $\rho=0.001$ for the same $\theta$.}
\label{tab:fdr-large-rho-supp-compact}
\setlength{\tabcolsep}{5pt}
\renewcommand{\arraystretch}{1}
\scriptsize
\begin{tabular}{l|cc|cc|cc}
\toprule
& \multicolumn{2}{c|}{$\theta=0.95$}
& \multicolumn{2}{c|}{$\theta=0.97$}
& \multicolumn{2}{c}{$\theta=0.99$} \\
\cmidrule(lr){2-3}\cmidrule(lr){4-5}\cmidrule(lr){6-7}
& $\rho=5$ & $\rho=10$ & $\rho=5$ & $\rho=10$ & $\rho=5$ & $\rho=10$ \\
\midrule
Obj Diff (\%) & 2.47 (2.37) & 2.47 (2.37) & 2.03 (1.95) & 2.03 (1.95) & 1.42 (1.36) & 1.42 (1.36) \\
OSS (\%)      & 94.88       & 94.88       & 95.30       & 95.30       & 96.04       & 96.04       \\
\bottomrule
\end{tabular}
\end{table}

By contrast, the CDR results exhibit a systematic progression toward the prescribed service-level targets as \(\rho\) increases within each fixed continuous-support setting \((\varsigma,\underline{\lambda}_f)\). For \(\theta=0.95\), all CDR settings satisfy the target already at \(\rho=0.001\), with OSS values ranging from \(0.9504\) to \(0.9587\). For \(\theta=0.97\), the \(\varsigma=0.05\) settings improve OSS from \(0.9621\) to \(0.9735\)--\(0.9736\), crossing the target at \(\rho=0.01\), while the \(\varsigma=0.10\) settings always exceed the target from \(\rho=0.005\) onward. For \(\theta=0.99\), with increasing $\rho$ the \(\varsigma=0.05\) settings increases OSS from \(0.9745\) to \(0.9849\), and the \(\varsigma=0.10\) setting reaches \(0.9917\) at \(\rho=0.01\). These reliability gains are obtained with CDR solution objective value increases of \(1.59\%\)--\(5.44\%\) relative to nominal and \(1.49\%\)--\(5.00\%\) relative to the FDR model. This increase in the objective values under the CDR model is accompanied by substantial OSS gains:
CDR exceeds the stand-alone FDR OSS by \(1.88\)--\(3.80\) percentage points
overall, and by \(2.50\)--\(3.80\), \(2.16\)--\(3.26\), and \(2.90\)--\(3.58\)
percentage points at \(\rho=0.01\) for \(\theta=0.95,0.97,\) and \(0.99\),
respectively. Therefore, CDR not only produces higher OSS values than FDR, but
also exhibits a consistent response to the ambiguity set radii \(\rho\): under the same
\((\varsigma,\underline{\lambda}_f)\) setting, OSS moves closer to the prescribed service level and often exceeds it, whereas FDR remains below the targets throughout the reported settings. The effect is particularly evident for the larger mean-support with parameter \(\varsigma=0.10\), under which CDR either attains the target or remains closest to it across all three service levels.


Within the continuous-support family of solutions, two features stand out. First, increasing \(\rho\) raises both cost and OSS monotonically. Second, \(\varsigma\) is the economically meaningful hyperparameter, whereas \(\underline{\lambda}_f\) is secondary across our settings. For example, at \((\theta,\rho)=(0.99,0.01)\), moving from \(\varsigma=0.05\) to \(\varsigma=0.10\) increases the objective difference from about \(3.58\%\) to \(5.44\%\) and OSS from \(0.9849\) to \(0.9917\). By contrast, changing \(\underline{\lambda}_f\) from \(0.333\) to \(0.5\) leaves the reported objective and OSS values almost unchanged. In these EV instances, the extent of uncertainty in the mean matters much more than the uncertainty in the covariance. 

Additional FDR runs with $\rho=5$ and $\rho=10$ are reported in Table~\ref{tab:fdr-large-rho-supp-compact} to assess whether, at $\theta=0.95$ and $\theta=0.97$, enlarging the FDR ambiguity radius improves OSS at a lower computational burden than CDR. However, at both radii, the FDR runs exhaust the full $24.00$-hour runtime budget while delivering only marginal OSS improvement. At $\theta=0.95$, FDR OSS probability increases from $0.9391$ at $\rho=0.01$ to $0.9488$ at both $\rho=5$ and $\rho=10$, but could not reach the target $\theta=0.95$. At the same time, the FDR objective premium increases from $0.77\%$ to $2.47\%$. That premium is larger than the smallest target-attaining CDR premium in Table~\ref{tab:ev-objdiff-oss-theta-rho-iter7-pmax560-vsdisc}, which is $1.59\%$, while that CDR solution already exceeds the target with $0.9504$ and $0.9506$ OSS values. The same pattern persists at the higher service levels. At $\theta=0.97$ and $\theta=0.99$, increasing the FDR radius from $\rho=0.01$ to $\rho \in \{5,10\}$ raises benchmark OSS only marginally, from $0.9519$ to $0.9530$ and from $0.9560$ to $0.9604$, respectively, while always consuming $24.00$ hours. The corresponding objective premiums are $2.03\%$ and $1.42\%$, yet the prescribed targets are still missed by $1.70$ and $2.96$ percentage points. Finally, Table~\ref{tab:fdr-large-rho-supp-compact} reports identical FDR objective premiums and identical OSS values at $\rho=5$ and $\rho=10$ for each $\theta$. Under the imposed $24.00$-hour budget, increasing the radius from $5$ to $10$ therefore does not produce a stronger reported FDR outcome.

\subsection{Solution Attributes and Managerial Insights}
\label{sec:compare-solution}

The out-of-sample performance and cost--reliability trade-off of the CDR solutions provide aggregate evidence of improved robustness, but they do not reveal how robustness is implemented in the energy allocation profile. This distinction is important in the EV demand planning problem because two solutions can have similar objective values or OSS probabilities while reallocating energy across different hours. In order to understand the source of improvement exhibited in the CDR solutions, we compared the nominal, FDR, and CDR allocation vectors directly. Let \(\bx^{\mathrm N}\) denote the nominal allocation and let \(\bx^{\mathrm{DR}}\) denote the allocation returned by a DR formulation. For each hour \(t\), define the signed allocation difference $\Delta_t:=x_t^{\mathrm{DR}}-x_t^{\mathrm N}.$

We use three diagnostic plots, each designed to capture and explain \(\boldsymbol \Delta\) from different aspects. The first diagnostic is a violin plot of the absolute deviations \(\{|\Delta_t|\}_t\), which displays the empirical distribution of hourly change magnitudes; it is used to distinguish whether robustness is achieved through broad small-scale adjustments or through a few large reallocations. The second diagnostic is a heatmap of the signed deviations \(\Delta_t\), which preserves the hour index and the direction of change; it is used to identify when the DR solution increases or decreases allocation relative to the nominal schedule. The third diagnostic is a JS--cosine scatter plot, which compares each DR allocation profile with the nominal profile using two full-vector similarity measures: cosine similarity captures directional alignment between \(\bx^{\mathrm N}\) and \(\bx^{\mathrm{DR}}\), while Jensen--Shannon divergence, computed after normalizing the nonnegative allocation profiles, captures distributional separation in the allocation pattern. Together, the three diagnostics connect aggregate robustness outcomes to operational changes in the charging schedule: the violin plot measures how large the hourly changes are, the heatmap identifies where and in which direction they occur through hour-by-hour signed differences \(\Delta_t\), and the JS--cosine plot summarizes the overall distance of the DR profile from the nominal profile. Figure~\ref{fig:combined_three_rows-0.01-0.005} reports these diagnostics for \(\rho=0.01\), while Figure~\ref{fig:combined_three_rows-0.001} in Appendix~\ref{appndx: visuals} reports the corresponding results for \(\rho=0.005\) and \(\rho=0.001\) in top-down order.


\subsubsection{Discussion on violin plot, thermal map, and scatter plot}. 
Across all nine \((\theta,\rho)\) settings, the FDR solution remains much more similar to the nominal benchmark than any CDR solution. Its violin plots have the tightest spread, its heatmap column stays close to zero in most hours, and its JS--cosine point lies nearest to the upper-left corner of the scatter plots. Quantitatively, the FDR solution does not move beyond about \(0.0045\) in Jensen- Shannon divergence and remains between about \(0.994\) and \(1.000\) in cosine similarity. The CDR solutions, on the other hand, are uniformly farther away, with JS divergences roughly staying between \(0.003\) and \(0.018\) and cosine similarities between about \(0.973\) and \(0.998\). The cost and reliability differences in Table~\ref{tab:ev-objdiff-oss-theta-rho-iter7-pmax560-vsdisc} therefore correspond to genuinely different allocation policies rather than to small perturbations of the nominal schedule.

The heatmaps show that these allocation-profile differences are highly structured. The robust reallocations are concentrated in a small number of recurring time blocks. Positive shifts appear repeatedly around hours \(11\)–\(12\), \(17\), and \(19\)–\(21\), often extending to hour \(22\); in the more stringent \(\theta=0.99\) cases, these increases are frequently offset by reductions at around hour \(10\). The CDR model thus improves reliability primarily through temporal redistribution. It shifts energy allocation away from hours whose nominal protection is less contributive and toward hours that repeatedly emerge as critical under distributional perturbation. This is also why the violin plots broaden mainly through their upper tails. Thus, CDR policy need not be uniformly different from the nominal schedule, but selectively different where protection matters most.

The similarity plots further indicate that a larger ambiguity radius does not necessarily correspond to the greatest global dissimilarity from the nominal allocation. The major difference in the allocation profile appears between the nominal and the CDR solutions. Within the CDR family, however, the ordering in JS-cosine space is not uniform. In particular, at \(\theta=0.99\) and \(\rho=0.005\), the \(\varsigma=0.10\) solutions are closer to the nominal profile than the \(\varsigma=0.05\) solutions under both similarity measures, even though both of them are more conservative in objective value and attain higher OSS. The heatmaps help explain this distinction: a more conservative model may obtain its gain through concentrated changes in a small set of influential hours, whereas a less conservative model may spread smaller adjustments over a broader portion of the day. Allocation profile similarity and robustness cost, therefore, measure related but distinct aspects of the solution.

\begin{figure}[!hbtp]
\centering
\begin{subfigure}[t]{0.30\textwidth}
    \centering
    \includegraphics[width=\linewidth,height=8em]{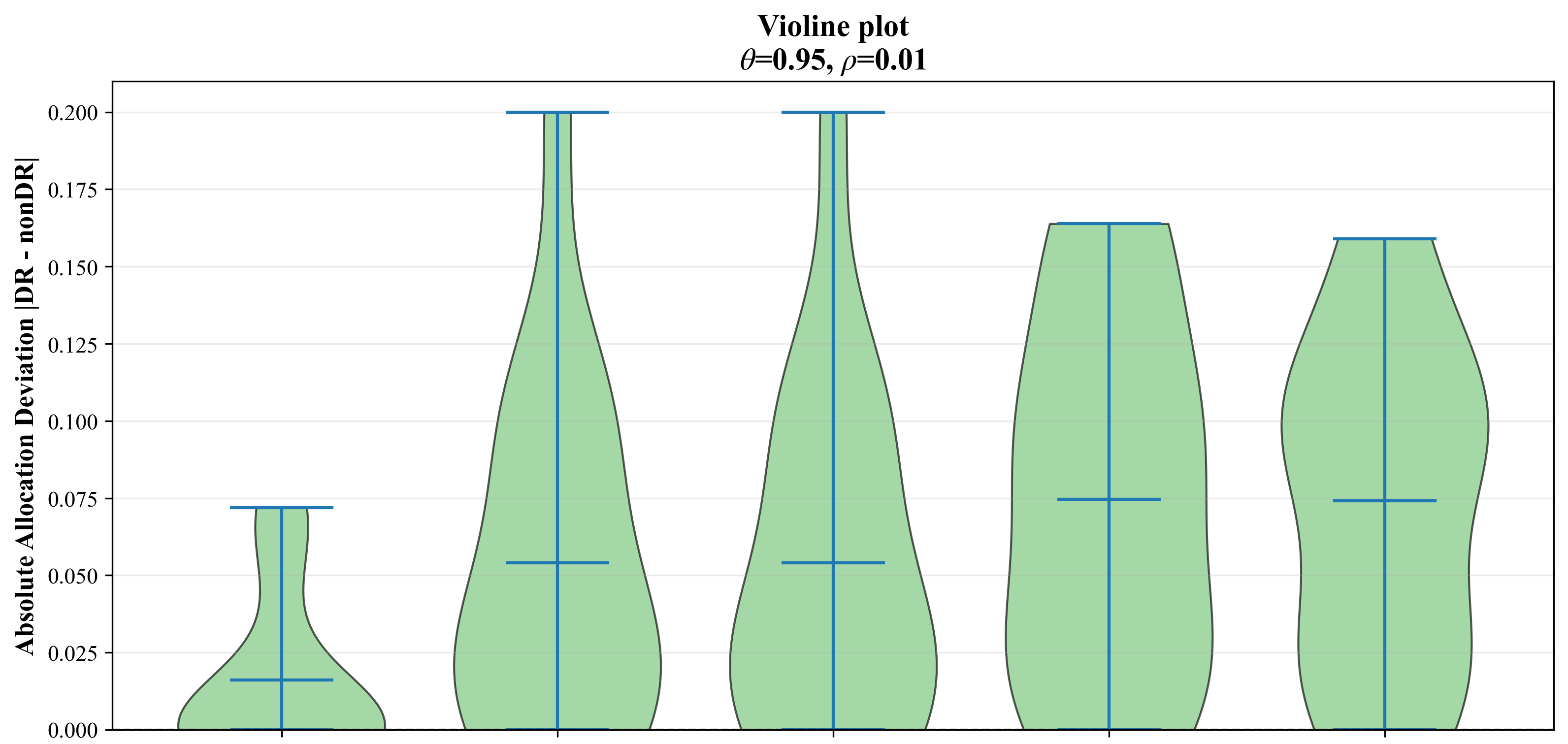}
    \label{fig:violin-0.95-0.01}
\end{subfigure}\hfill
\begin{subfigure}[t]{0.345\textwidth}
    \centering
    \includegraphics[width=\linewidth,height=8em]{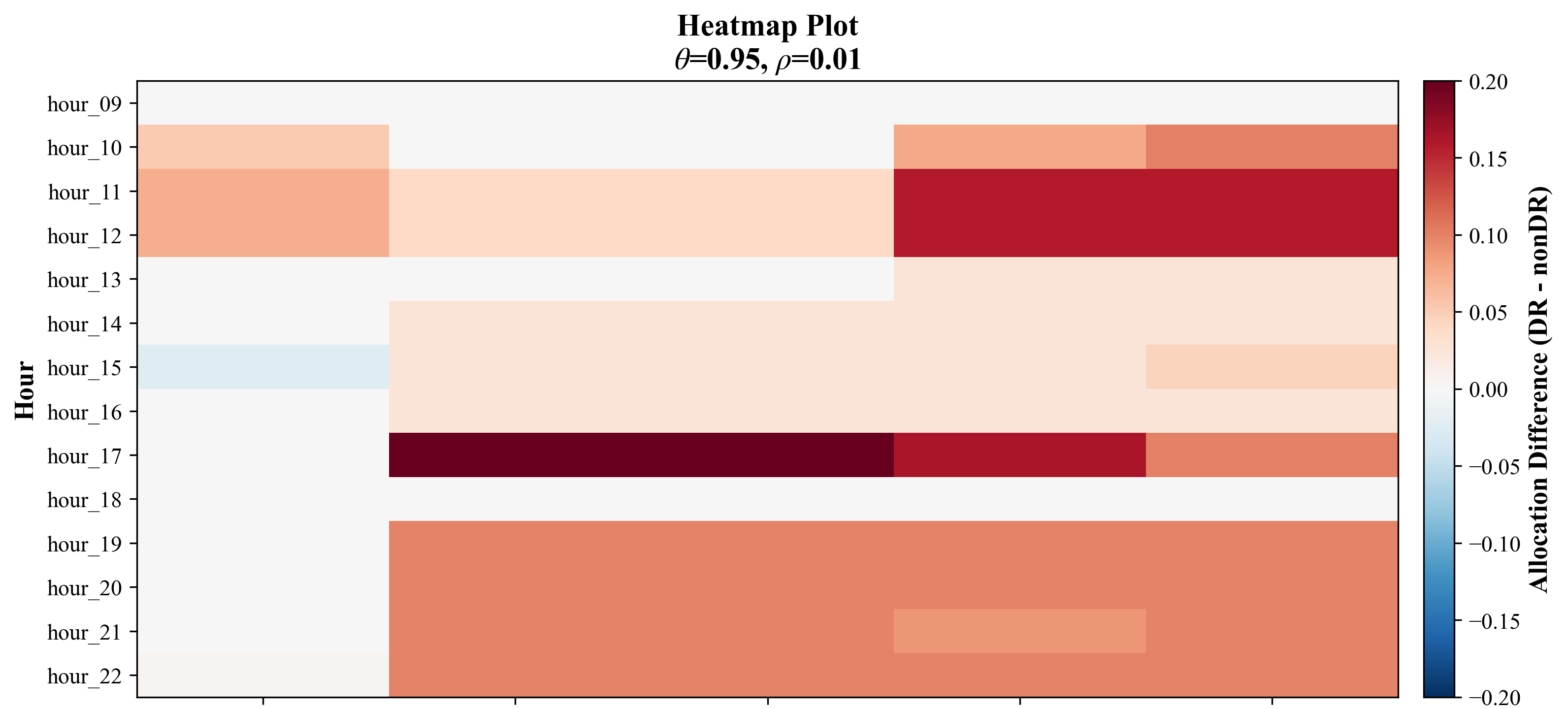}
    \label{fig:heatmap-0.95-0.01}
\end{subfigure}\hfill
\begin{subfigure}[t]{0.345\textwidth}
    \centering
    \includegraphics[width=0.95\linewidth,height=8em]{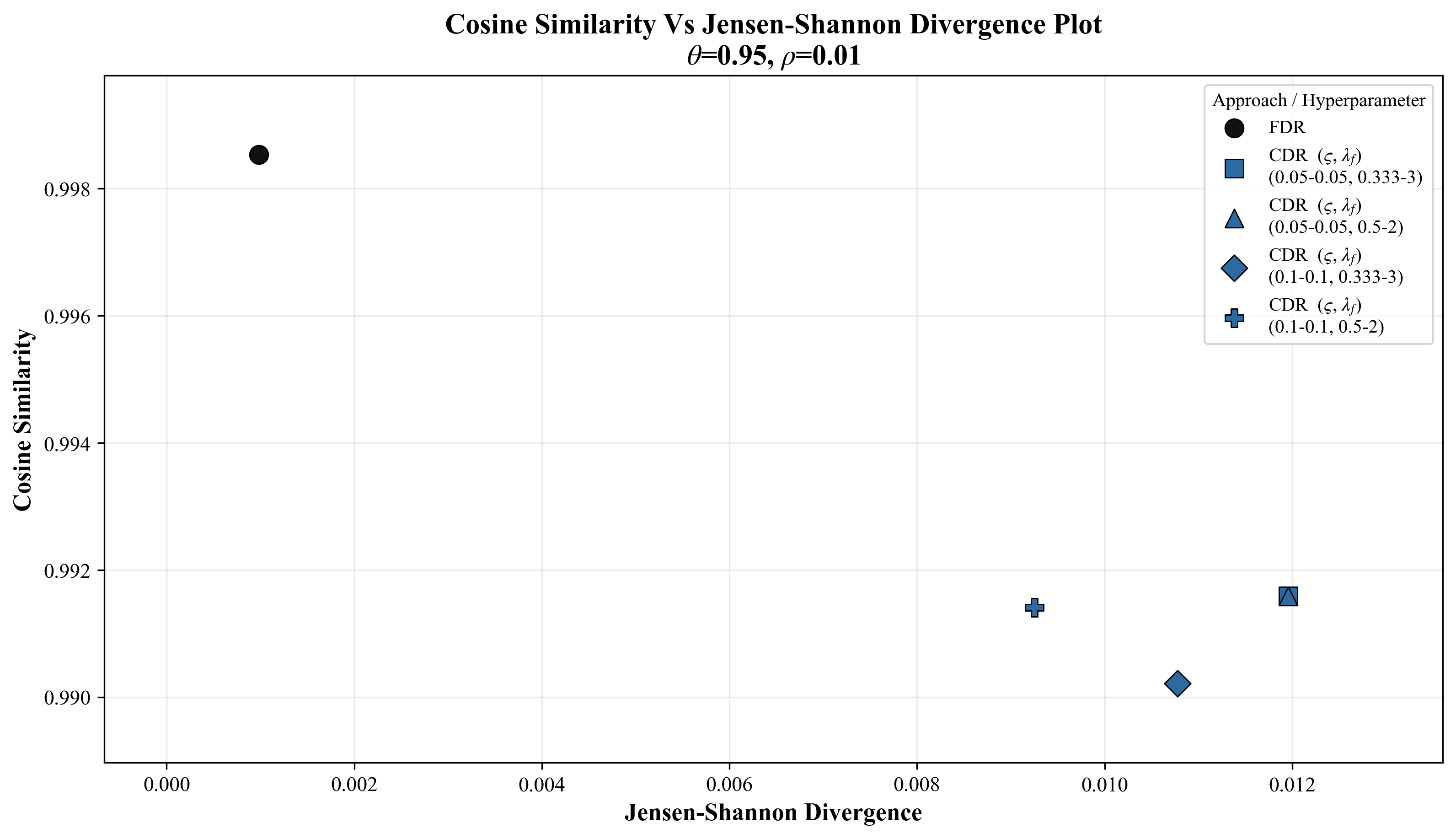}
    \label{fig:scatter-0.95-0.01}
\end{subfigure}

\vspace{-10pt}
\begin{subfigure}[t]{0.30\textwidth}
    \centering
    \includegraphics[width=0.9\linewidth,height=8em]{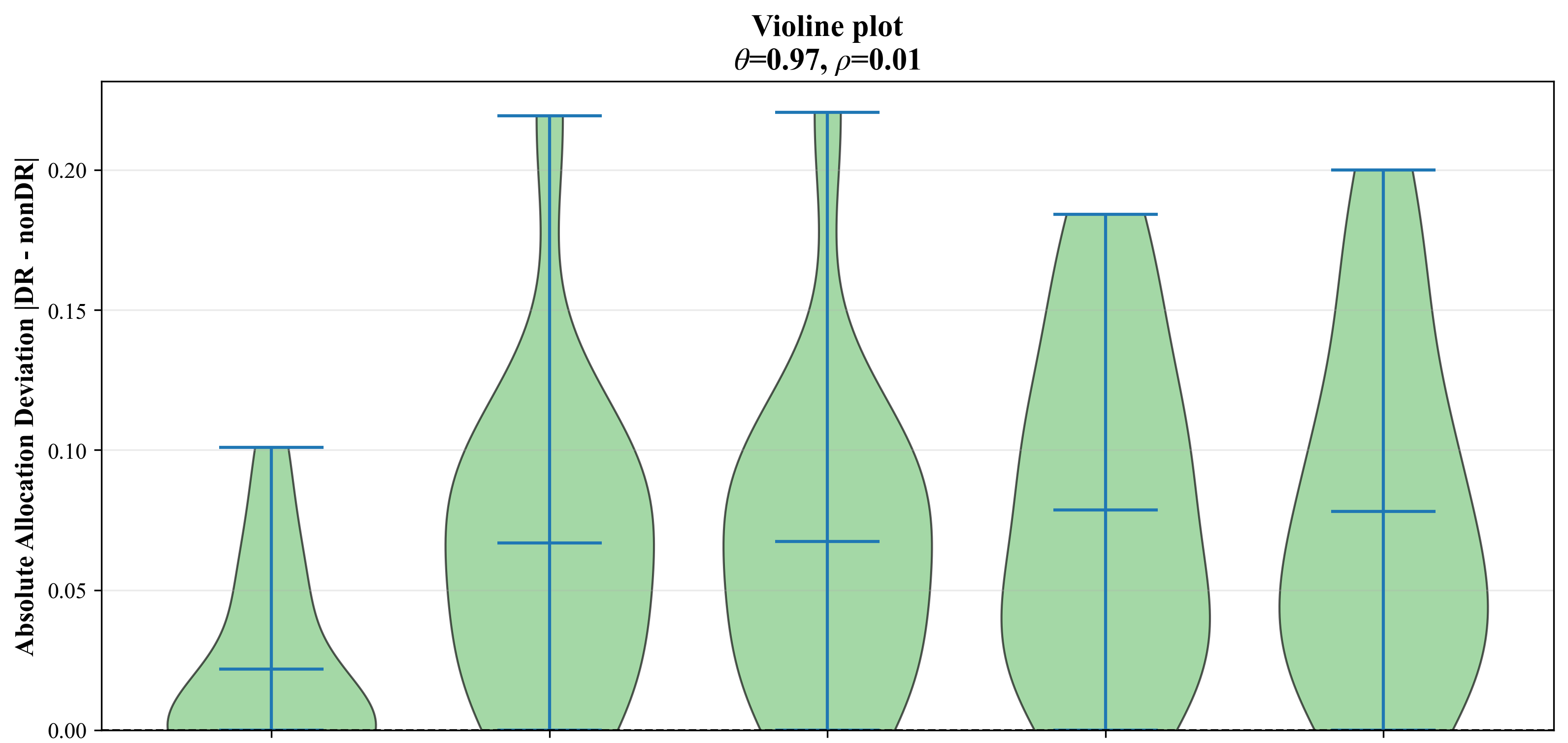}
     \label{fig:violin-0.97-0.01}
\end{subfigure}\hfill
\begin{subfigure}[t]{0.345\textwidth}
    \centering
    \includegraphics[width=\linewidth,height=8em]{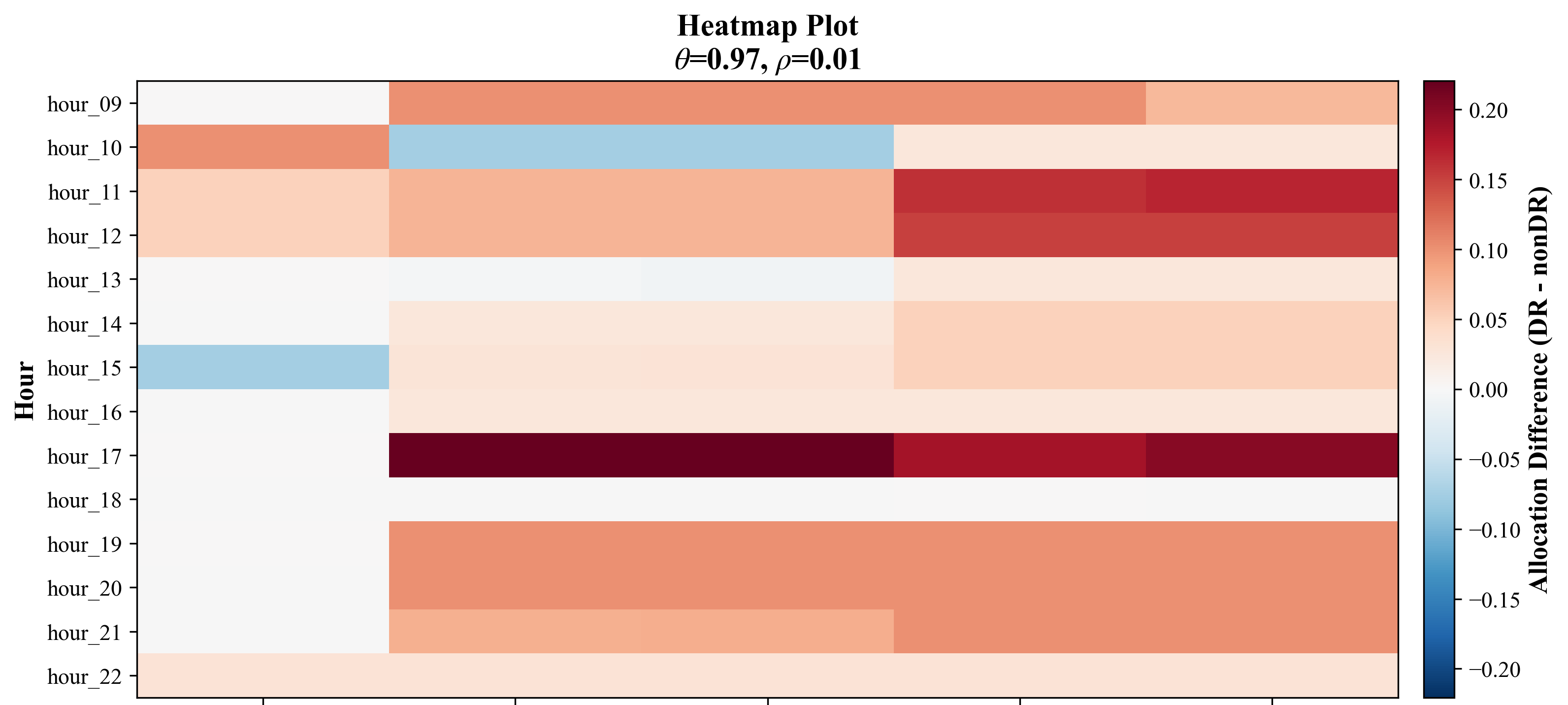}
    \label{fig:heatmap-0.97-0.01}
\end{subfigure}\hfill
\begin{subfigure}[t]{0.345\textwidth}
    \centering
    \includegraphics[width=0.95\linewidth,height=8em]{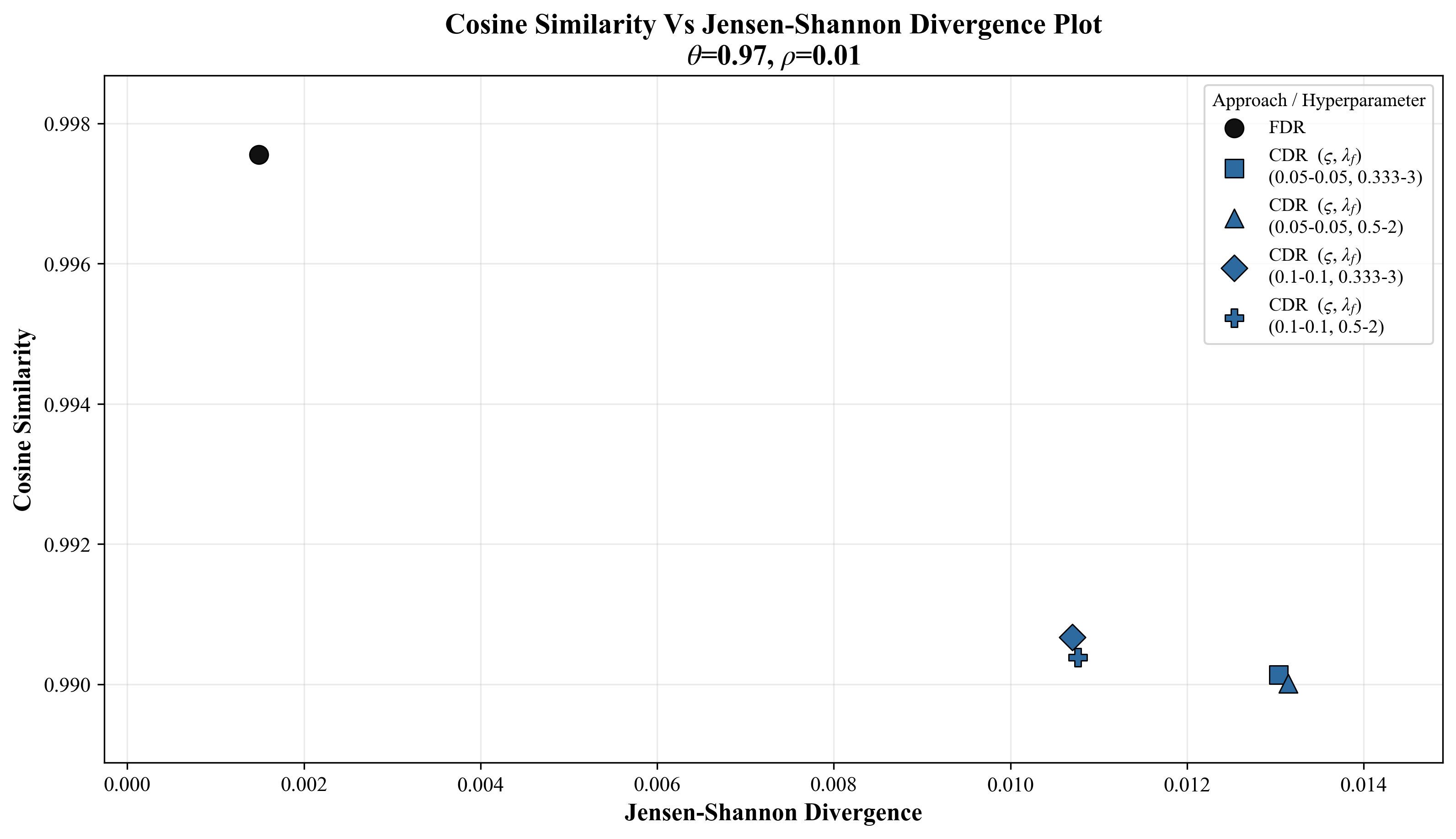}
    \label{fig:scatter-0.97-0.01}
\end{subfigure}

\vspace{-10pt}
\begin{subfigure}[t]{0.30\textwidth}
    \centering
    \includegraphics[width=0.9\linewidth,height=9.5em]{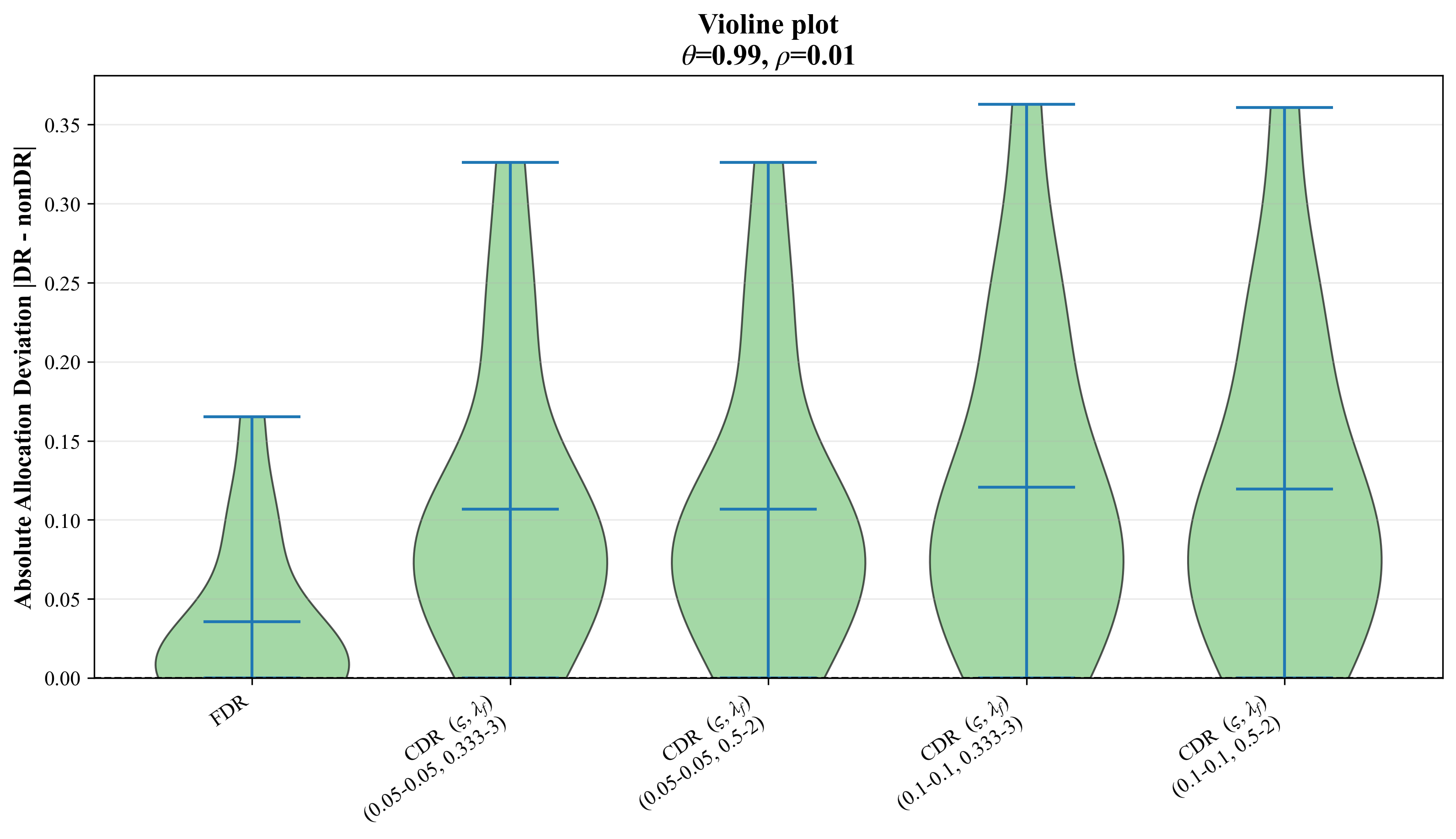}
     \label{fig:violin-0.99-0.01}
\end{subfigure}\hfill
\begin{subfigure}[t]{0.345\textwidth}
    \centering
    \includegraphics[width=\linewidth,height=9.5em]{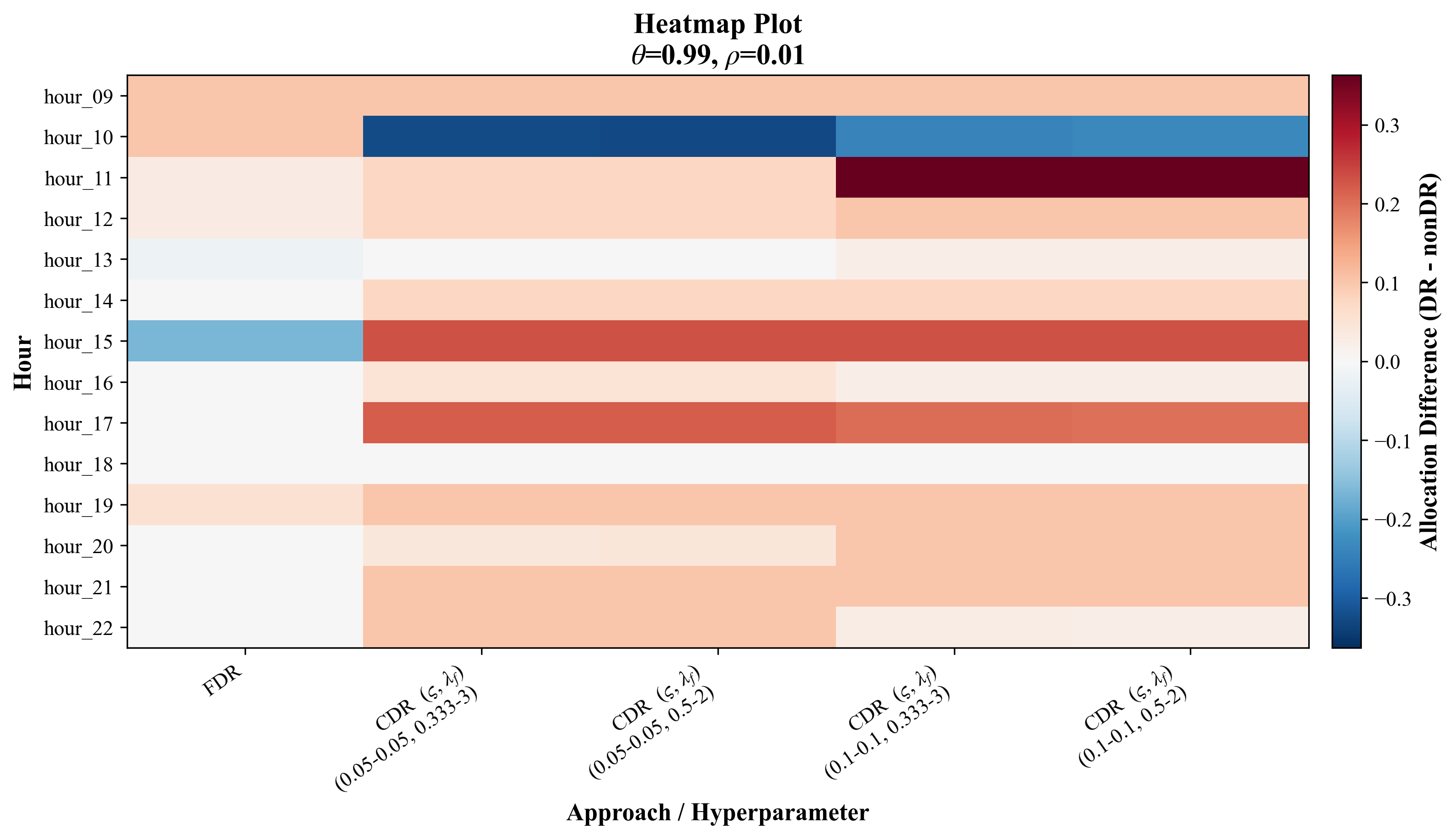}
    \label{fig:heatmap-0.99-0.01}
\end{subfigure}\hfill
\begin{subfigure}[t]{0.345\textwidth}
    \centering
    \includegraphics[width=0.95\linewidth,height=9.5em]{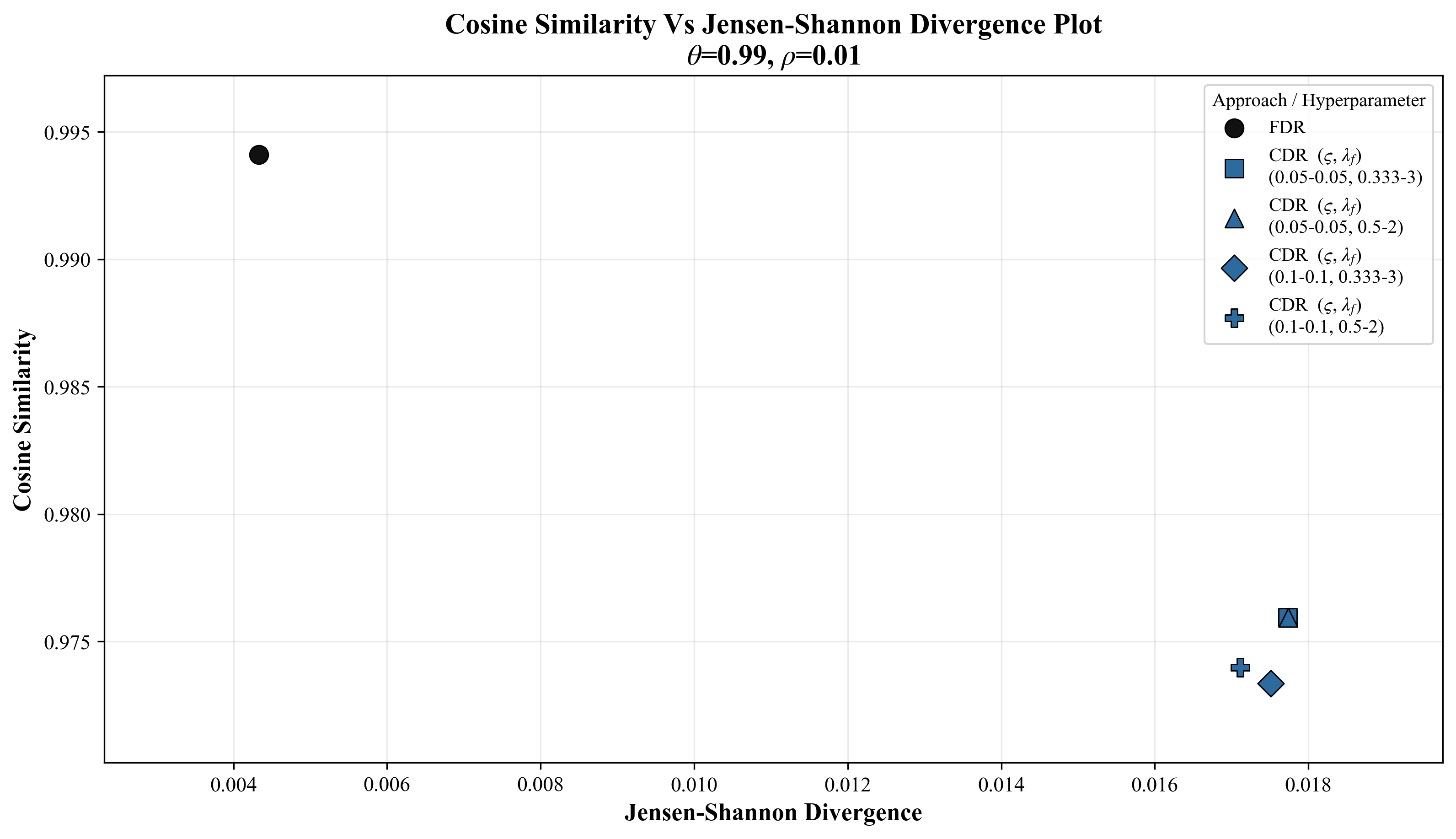}
    \label{fig:scatter-0.99-0.01}
\end{subfigure}
\vspace{-1em}
\caption{\scriptsize 
Comparison of finite- and continuous-support solutions with $(\varsigma,\lambda_{\mathrm{min}}) \in \{(0.05, 0.333), (0.05, 0.5), (0.1, 0.333), (0.1, 0.5)\}$ for $\rho = 0.01$, with rows ordered from top to bottom by $\theta=0.95,0.97,0.99$. 
Within each row, the left panel shows the violin plot, i.e., the distribution of absolute deviations from the nominal solution; the middle panel shows the heatmap of allocation differences relative to the nominal solution, and the right panel shows cosine similarity versus Jensen-Shannon divergence.}
\label{fig:combined_three_rows-0.01-0.005}
\end{figure}

\vspace{-1.25em}
\subsubsection{Managerial insights.}\, For the EV demand planning problem, a clear managerial message is that robustness is fundamentally a timing decision. The CDR model does not only recommend ``more'' protection; it identifies when protection is most valuable. Since the main changes are concentrated in a limited subset of hours, higher service reliability need not be pursued through a uniform expansion of hourly allocations. It can instead be achieved by redesigning the charging schedule toward the hours that matter most for maintaining feasibility under distributional misspecification.

The results also clarify the distinct roles of the two DR formulations. FDR is best viewed as a low-cost hedge around the empirical support of the fitted model. It is computationally attractive and remains close to the nominal solution, but its gains in empirical reliability are limited. CDR is the more relevant model when the planner is also concerned about mis-specification in the fitted mixture parameters themselves. In that case, robustness protects not only against variation in realized demand but also against error in the estimated distributional representation used to construct the plan.

Within the continuous-support model, the mean perturbation parameter \(\varsigma\) acts as a practical robustness control. In the examined instances, \(\varsigma=0.05\) produces a moderate-robustness regime: it increases OSS probability materially above FDR while keeping the additional objective premium over FDR between about \(1.5\%\) and \(3.2\%\). The setting \(\varsigma=0.10\) defines a high-protection regime: it yields the strongest empirical service levels, including the only cases that exceed \(99\%\) OSS when \(\theta=0.99\), but it can require up to about a \(5\%\) cost premium over FDR. The computational results thus suggest a corresponding workflow: FDR can serve as a fast screening model, while CDR is most valuable when the benefit of higher reliability justifies the additional solve time and cost premium.

Taken together, the numerical evidence shows that CDR adds a distinct robustness layer beyond FDR. It improves empirical satisfaction markedly; it does so through interpretable changes in the timing of the charging plan, and it provides a transparent cost-reliability tradeoff through the design of the parameter-support set.

\section{Conclusion}
We studied a distributionally robust chance-constrained framework with continuous support in Gaussian-mixture parameters and developed a cutting-surface-based solution approach for its implementation. Relative to the nominal and finite-support DR models, the continuous-support formulation delivers stronger out-of-sample reliability and can induce a meaningfully different allocation pattern, showing that protection against broader mean--covariance misspecification may require more than a mild perturbation of the nominal plan. The computational results (not shown here) also indicate that the implemented approach is practically meaningful, with early iterations identifying the dominant violating distributions and later iterations performing refinement.

An important direction for future work is computational acceleration. In the current implementation, the master problem relies on off-the-shelf solvers, with limited control over the piecewise-linear approximation and the associated branching behavior on binary variables, and it is re-solved from scratch after each newly added cut. These features create a significant runtime burden. Developing warm-started master reoptimization, cut bundling, adaptive management of the PWL approximation, and more tailored branch-and-cut or decomposition strategies would substantially improve scalability and strengthen the practical applicability of the proposed framework. Overall, the study shows that continuous-support robustification is both practically relevant and computationally approachable, while also highlighting the need for more specialized algorithms to fully unlock its benefits.

\bibliography{validated_references}
\bibliographystyle{plain}

\begin{APPENDICES}

\section{Known Results and Tools}\label{appndx: known-results}

\begin{itemize}[leftmargin=1em]
\item[\textbf{A.1}] \textbf{Integration interchangeability.}
\begin{theorem}[Tonelli-Fubini Theorem for product measure\cite{folland1999real}]
Let $(\mbX,\mB(\mbX), \mu)$ and $(\mbY,\mB(\mbY), \nu)$ be probability measure spaces
with probability measures $\mu$ and $\nu$ respectively. If
$f:\mbX\times \mbY\to[0,\infty]$ be a nonnegative measurable function (Tonelli Condition) or 
$f:\mbX\times\mbY\to\mbR$ be integrable, i.e.
$\displaystyle \int_{\mbX\times \mbY} |f|\,d(\mu\times\nu)<\infty$ (Fubini Condition),
Then:

\begin{enumerate}
[label=\textit{\roman*)},leftmargin=1em]
\item In the Tonelli case \(f\ge 0\), the maps
      \[
        \msml{x} \mapsto \int_\mbY f(\msml{x},\msml{y})\,d\nu(\msml{y}), \qquad
        \msml{y} \mapsto \int_\mbX f(\msml{x}, \msml{y})\,d\mu(\msml{x})
      \]
      are well-defined as \([0,\infty]\)-valued measurable functions. In the
      Fubini case \(f\in L^1(\mu\otimes\nu)\), these maps are finite almost
      everywhere and integrable with respect to \(\mu\) and \(\nu\), respectively.
  \item the iterated integrals and the product-space integral satisfy
  \[
     \int_{\mbX\times \mbY} f(\msml{x},\msml{y})\,(\mu\otimes\nu)(d\msml{x},d\msml{y})
     \;=\;
     \int_\mbX \Bigl( \int_\mbY f(\msml{x},\msml{y})\, d\nu(\msml{y}) \Bigr) d\mu(\msml{x})
     \;=\;
     \int_\mbY \Bigl( \int_\mbX f(\msml{x},\msml{y})\, d\mu(\msml{x}) \Bigr) d\nu(\msml{y}).
  \]
\end{enumerate}
\end{theorem}
\begin{theorem}[Tonelli--Fubini Theorem for Markov Kernels]\label{thm: kernel-fubini-tonelli} \srevision{(\cite[Section~10.7]{bogachev_measure_2007};\cite[Section 1.6]{ccinlar2011probability};\cite[Section~8.3]{klenke2008probability})}
Let \((\mbX,\mB(\mbX))\) and \((\mbY,\mB(\mbY))\) be measurable spaces, and let
\(\nu_\msml{x}\) be a probability measure on \((\mbX,\mB(\mbX))\). Let
\(\mK:\mbX\times\mB(\mbY)\to[0,1]\) be a transition kernel, that is:
\begin{enumerate}[label=\textit{\roman*)},leftmargin=1em]
\item[a)] For every \(\msml{x}\in \mbX\), the map \(C\mapsto \mK(\msml{x},C)\) is a probability measure on \((\mbY,\mB(\mbY))\), and
\item[b)] For every \(C\in\mB(\mbY)\), the map
\(\msml{x}\mapsto \mK(\msml{x},C)\) is \(\mB(\mbX)\)-measurable, hence
\(\nu_\msml{x}\)-measurable for any probability measure \(\nu_\msml{x}\) on
\((\mbX,\mB(\mbX))\).
\end{enumerate}
Define the induced probability measure \(\nu_\msml{y}\) on
\((\mbY,\mB(\mbY))\) by
\[
\nu_\msml{y}(C):=\int_{\mbX}\mK(\msml{x},C)\,d\nu_\msml{x}(\msml{x}),
\qquad C\in\mB(\mbY).
\]
Then, if \(\Psi:\mbY\to\mbR\) is measurable, the following two integral rules hold.

\medskip

\noindent\textbf{(Tonelli case).}
If \(\Psi\ge 0\), then
\begin{equation}\label{eq:kernel-tonelli}
\int_{\mbY} \Psi(\msml{y})\,d\nu_\msml{y}(\msml{y})
=
\int_{\mbX}
     \left(\int_{\mbY} \Psi(\msml{y})\,\mK(\msml{x},d\msml{y})\right)
     d\nu_\msml{x}(\msml{x}).
\end{equation}
Both sides are equal in \([0,\infty]\).

\medskip

\noindent\textbf{(Fubini case).}
If \(\Psi\) is integrable with respect to \(\nu_\msml{y}\), i.e.,
\(\int_\mbY |\Psi(\msml{y})|\,d\nu_\msml{y}(\msml{y}) < \infty\), then:
\begin{enumerate}[label=\textit{\roman*)},leftmargin=1em]
\item for \(\nu_\msml{x}\)-almost every \(\msml{x}\in \mbX\), the map
      \(\msml{y}\mapsto \Psi(\msml{y})\) is
      \(\mK(\msml{x},\cdot)\)-integrable, i.e.,
      \(\int_\mbY |\Psi(\msml{y})| \, \mK(\msml{x}, d\msml{y}) < \infty\),
\item the map \(\msml{x} \longmapsto \int_\mbY \Psi(\msml{y})\,\mK(\msml{x},d\msml{y})\)
is integrable with respect to \(\nu_\msml{x}\), (since
\[
\int_\mbX
\left|
\int_\mbY \Psi(\msml{y})\,\mK(\msml{x},d\msml{y})
\right|
d\nu_\msml{x}(\msml{x})
\le
\int_\mbX\int_\mbY
|\Psi(\msml{y})|\,\mK(\msml{x},d\msml{y})\,d\nu_\msml{x}(\msml{x})
=
\int_\mbY|\Psi(\msml{y})|\,d\nu_\msml{y}(\msml{y})
<\infty).
\]
\item and identity~\eqref{eq:kernel-tonelli} holds with finite values.
\end{enumerate}
\end{theorem}

\begin{remark}[Reading~\eqref{eq:kernel-tonelli}]
For \(\Psi=\ind_C\), the definition of \(\nu_\msml{y}\) gives
\[
\int_{\mbY}\ind_C(\msml{y})\,d\nu_\msml{y}(\msml{y})
=
\int_{\mbX}
\left(\int_{\mbY}\ind_C(\msml{y})\,\mK(\msml{x},d\msml{y})\right)
d\nu_\msml{x}(\msml{x}),
\]
i.e., identity~\eqref{eq:kernel-tonelli} reduces exactly to the
defining relation
\[
\nu_\msml{y}(C)
=
\int_{\mbX}\mK(\msml{x},C)\,d\nu_\msml{x}(\msml{x}).
\]
Thus, \eqref{eq:kernel-tonelli} is the function-level extension of the setwise
definition of \(\nu_\msml{y}\); the extension from indicators to nonnegative
measurable functions follow from the standard simple-function/monotone
convergence argument, and the integrable case follows from the decomposition
\(\Psi=\Psi^+-\Psi^-\)~\cite{folland1999real}.
\end{remark}

\begin{remark}[Interchanging the kernel integration order]
The kernel \(\mK(\msml{x},d\msml{y})\) defines the joint measure
\[
\eta(d\msml{x},d\msml{y})
=
\nu_\msml{x}(d\msml{x})\,\mK(\msml{x},d\msml{y}).
\]
If a reverse conditional kernel \(\mL(\msml{y},d\msml{x})\) exists, i.e.,
\(\eta(d\msml{x},d\msml{y})
=
\nu_\msml{y}(d\msml{y})\,\mL(\msml{y},d\msml{x})\), then for every nonnegative
measurable \(F:\mbX\times\mbY\to[0,\infty]\),
\[
\int_{\mbX\times\mbY}F\,d\eta
=
\int_{\mbX}\!\int_{\mbY}F(\msml{x},\msml{y})\,\mK(\msml{x},d\msml{y})\,d\nu_\msml{x}(\msml{x})
=
\int_{\mbY}\!\int_{\mbX}F(\msml{x},\msml{y})\,\mL(\msml{y},d\msml{x})\,d\nu_\msml{y}(\msml{y}).
\]
Thus, the reverse-order expression requires the reverse kernel \(\mL\), not
\(\mK(d\msml{x},\msml{y})\).
\end{remark}

\item[\textbf{A.2}]\textbf{Gaussian instantiation of transition kernel.}\,  
We start with a standard result on preservation of measurability under integration, which will be invoked throughout the sequel.
\begin{proposition}\label{prop: measurability}
Let \((\mbX,\mB(\mbX))\) and \((\mfZ,\mB(\mfZ))\) be measurable spaces, let \(\gamma\) be a measure on \((\mfZ,\mB(\mfZ))\), and let $g:\mbX\times \mfZ\to [0,\infty]$ be \(\mB(\mbX)\otimes \mB(\mfZ)\)-measurable. If
\[
\mathscr{G}(\msml{x}):=\int_{\mfZ} g(\msml{x},\bz)\,d\gamma(\bz),
\]
then \(\mathscr{G}:\mbX\to[0,\infty]\) is \(\mB(\mbX)\)-measurable.
\end{proposition}
Now, leveraging the above Proposition~\ref{prop: measurability}, we show that Gaussian is a special instantiation of the Markov kernel.
Let
\[
\mbX := \mbR^n \times \mbS_+^n, 
\qquad
\mfZ := \mbR^n,
\qquad
\mbY := \mbR^n,
\]
equipped with their Borel \(\sigma\)-algebras. Let $\bA(\bQ):=\bQ^{1/2},$
where \(\bQ^{1/2}\) is the principal symmetric positive semidefinite square root
of \(\bQ\). The map \(\bQ\mapsto \bQ^{1/2}\) is Borel measurable on
\(\mbS_+^n\), and
\[
\bA(\bQ)\bA(\bQ)^\top=\bQ\qquad\text{for all }\bQ\in\mbS_+^n.
\]
Fix \(\gamma := \mN(\bzero,\bI)\) with identity matrix $\bI$ on \((\mfZ,\mB(\mfZ))\). For \(\msml{x}=(\bmean,\bQ)\in\mbX\) and \(C\in\mB(\mbY)\), define the map
\(\mK:\mbX\times\mB(\mbY)\to[0,1]\) by
\[
\mK(\msml{x},C)\;:=\;\int_{\mfZ} g_C(\msml{x},\bz)\,d\gamma(\bz),
\]
where the integrand \(g_C\) is $g_C(\msml{x},\bz)\;:=\;\ind_C\!\big(\bmean + \bA(\bQ)z\big).$

We now verify the two Markov Kernel conditions $(a)$-$(b)$ stated in Theorem~\ref{thm: kernel-fubini-tonelli} for the Gaussian instantiation:

\begin{enumerate}[leftmargin=0.5em]
\item \textbf{Condition (a) of Theorem~\ref{thm: kernel-fubini-tonelli}: For each fixed \(\msml{x}\), \(C\mapsto \mK(\msml{x},C)\) is a probability measure on \((\mbY,\mB(\mbY))\).} 

Define \(\mN(\bmean,\bQ)\) as a pushforward of \(\gamma=\mN(0,\bI)\) as follows
\[
\mN(\bmean,\bQ):=(T_\msml{x})_\#\gamma,
\qquad T_\msml{x}(\bz):=\bmean + \bA(\bQ)\bz.
\]
Then for every \(C\in\mB(\mbY)\),
\[
\mN(\bmean,\bQ)(C)=(T_\msml{x})_\#\gamma(C)
=\gamma\big(T_\msml{x}^{-1}(C)\big)
=\int_{\mfZ}\ind_C\big(T_\msml{x}(\bz)\big)\,d\gamma(\bz)
=\int_{\mfZ}\ind_C\big(\bmean + \bA(\bQ)z\big)\,d\gamma(\bz) = \mK(\msml{x}, C).
\]

Clearly, $\mN(\bmean, \bQ)$ is a probability measure and $\mN(\bmean, \bQ)(C)$ is the evaluation of that measure on the set $C$. So \(C\mapsto \mK(\msml{x},C) \bigl(= \mN(\bmean, \bQ)(C)\bigr)\) is the pushforward measure \((T_\msml{x})_\#\gamma\). Pushforwards of probability measures are probability measures. Concretely:
\begin{itemize}[leftmargin=1.3em]
\item[--] \(\mK(\msml{x},\emptyset)=0\), \(\mK(\msml{x},\mbY)=1\),
\item[--] countable additivity holds because \(C\mapsto \gamma(T_\msml{x}^{-1}(C))\) is a measure.
\end{itemize}
So condition (a) in Theorem~\ref{thm: kernel-fubini-tonelli} is satisfied.

\begin{remark}
    We emphasize that no Lebesgue density was used, so a singular \(\bQ\) is allowed. Also, \(\mN(\bmean,\bQ)\) is a probability measure on \((\mbY,\mB(\mbY))\) while \(\mK(\msml{x},C)\) is the evaluation of that measure on the set \(C\). Hence, $\mK(\msml{x},C)=\mN(\bmean,\bQ)(C),
\qquad \msml{x}=(\bmean,\bQ)\in\mbX,\; C\in\mB(\mbY).$
\end{remark}

\item \noindent\textbf{Condition (b) of Theorem~\ref{thm: kernel-fubini-tonelli}: For each fixed \(C\), \(\msml{x}\mapsto \mK(\msml{x},C)\) is \(\nu_\msml{x}\)-measurable.}

For each fixed
\(C\in\mB(\mbY)\), we want to prove $\nu_\msml{x}$ measurability of $\msml{x}\longmapsto \mK(\msml{x},C)=\int_{\mfZ} g_C(\msml{x},\bz)\,d\gamma(\bz).$
By Proposition~\ref{prop: measurability}, it is enough to show that $g_C:\mbX\times\mfZ\to[0,1],\quad g_C(\msml{x},\bz)=\ind_C(\bmean + \bA(\bQ)\bz)$
is \(\mB(\mbX)\otimes\mB(\mfZ)\)-measurable.

But the map $h:\mbX\times\mfZ\to \mbY,\,\,\, h((\bmean,\bQ),\bz)=\bmean + \bA(\bQ)\bz$
is measurable because it is built from measurable operations (addition, matrix–vector multiplication) and \(\bA(\cdot)\) is measurable by construction. Then $g_C = \ind_C\circ h$
is measurable since \(\ind_C\) is Borel measurable on \(\mbY\). Therefore, by the lemma, \(\msml{x}\mapsto\int g_C(\msml{x},\bz)\,d\gamma(\bz)\) is measurable. This proves condition (b).
\end{enumerate}
\item[\textbf{A.3}]\textbf{Notation correspondence of Markov transition kernel for the proof $W_2^\mathrm{prm} = W_2^\mathrm{mix}$.}\,
Let \(\Psi\) be a measurable map. Then the following allows us to apply
the Markov-kernel Tonelli--Fubini theorem, Theorem~\ref{thm: kernel-fubini-tonelli}:
\[
\begin{array}{c|c}
\hline
\text{Kernel-based Theorem~\ref{thm: kernel-fubini-tonelli} Notation} & \text{Mapped Notation given $\pi \in \Gamma$} \\[0.3em]
\hline
\mbX & \mS \times \mS \\[0.3em]
\mbY & \mbR^{2n}\\[0.3em]
\mK(\msml{x},\cdot) & \mG(\bs_1, \bs_2)(\cdot) \\[0.3em]
\nu_\msml{x} & \pi \in \mP(\mS \times \mS) \\[0.3em]
\nu_\msml{y}(C) & \displaystyle \nu_{\pi}(C)
  = \int_{\mS\times\mS}
        \mG(\bs_1, \bs_2)(C)\,d\pi(\bs_1,\bs_2) \\[0.3em]
\Psi(\msml{y}) & \Psi(\by_1,\by_2) =  \ind_A(\by_1) \\
\hline
\end{array}
\]

Similarly, for the other directions:
\[
\begin{array}{c|c}
\hline
\text{Kernel-based Theorem~\ref{thm: kernel-fubini-tonelli} Notation} & \text{Mapped Notation given $\nu \in \Pi^\mathrm{mix}(\mu_1, \mu_2)$} \\ \hline
\mbX & \mZ = \mbR^n \times \mbR^n \times \mbS^{2n}_+ \\[0.3em]
\mbY & \mbR^{2n}\\[0.3em]
\mK(\msml{x},\cdot) & \mN(\bz)(\cdot) \\[0.3em]
\nu_\msml{x} & \lambda \in \mP(\mZ) \\[0.3em]
\nu_\msml{y}(C) & \displaystyle \nu(C) = \int_{\mZ} \mN(\bz)(C)\,d\lambda(\bz) \\[1em]
\Psi(\msml{y}) & \Psi(\by_1, \by_2) = \ind_B(\by_2) \\
\hline
\end{array}
\]
For objective identities, the same mapping is used with
\(\Psi(\by_1,\by_2)=\|\by_1-\by_2\|^2\), which is nonnegative, so the Tonelli case applies. For the first set of mappings, intuition to define $\nu_\pi$ comes from the Gelbrich OT map Theorem~\ref{thm: Gelbrich}, while for the second set of maps, it is a Gaussian mixture measure defined in~\eqref{def: GMMmix}.

\item[\textbf{A.4}] \textbf{Optimal-transport discrepancy and ambiguity set\srevision{~\cite[Theorem 1]{blanchet2019quantifying}.}}
Let $(\mS,\mathcal{B}(\mS))$ be a Polish space with its Borel $\sigma$-algebra,
and let $\mP(\mS)$ be the set of Borel probability measures on $\mS$. Assume $c: \mS \times \mS \to [0,\infty]$ is a measurable transportation cost function. Define
the optimal transport cost between $\mu_1,\mu_2\in \mP(\mS)$ by
\[
W_c(\mu_1,\mu_2)
:=
\inf_{\pi\in \Pi(\mu_1,\mu_2)}
\int_{\mS \times \mS} c(\bs_1,\bs_2)\,\pi(d\bs_1,d\bs_2),
\]
where $\Pi(\mu_1,\mu_2)$ is the set of all couplings of $\mu_1$ and $\mu_2$.

Fix a nominal distribution $\hat{\mu}\in \mP(\mS)$ and a radius $\rho \ge 0$. The corresponding OT ambiguity set is
\[
\mathfrak{D}(\hat{\mu},\rho)
:=
\left\{\mu\in \mP(\mS): W_c(\hat{\mu},\mu)\le \rho \right\}.
\]
Let $G: \mS\to \overline{\mbR}:=\mbR\cup\{+\infty\}$ be a measurable payoff (loss) function. Consider the worst-case expectation
\[
\sup_{\mu \in\mathfrak{D}(\hat{\mu},\rho)} \int_{\bs \in \mS} G(\bs)\, \mu(d\bs).
\]

\noindent\textbf{(A1) Regularity of transport cost.} $c$ is lower semicontinuous on
$\mS\times \mS$ and satisfies $c(\bs,\bs)=0$ for every $\bs\in \mS$.

\noindent\textbf{(A2) Upper semicontinuity of the payoff.} $G$ is upper semicontinuous on $\mS$.

\noindent\textbf{(A3) Integrability condition.} There exists some scalar $\lambda\ge 0$ such that the following is finite:
\[
\int_\mS \sup_{\bs\in \mS}\left\{G(\bs)-\lambda\,c(\hbs,\bs)\right\}\ \hat{\mu}(d\hbs)\ <\ \infty.
\]
Then strong duality holds, i.e.,
\[
\sup_{\mu\in\mathfrak{D}(\hat{\mu},\rho)} \int_\mS G(\bs)\, \mu(d\bs)
\ =\ 
\inf_{\lambda\ge 0}\left\{\lambda\,\rho\ +\ \int_\mS \sup_{\bs\in \mS}\left(G(\bs)-\lambda\,c(\hbs,\bs)\right)\ \hat{\mu}(d\hbs)\right\}.
\]
Moreover, whenever the integral is well-defined, the right-hand side is a
convex optimization problem in the scalar dual variable $\lambda\ge 0$,
and the inner term
\[
\sup_{\bs\in \mS}\left(G(\bs)-\lambda\,c(\hbs,\bs)\right)
\]
is the $c$-transform of $G$ (up to sign conventions).

\noindent\textbf{Empirical nominal distribution $\widehat P$ with point-mass support}

If $\hat{\mu}=\sum_{k=1}^{\hat{K}}\hat{w}_k\ \delta_{\hbs_k}$, where $\delta_{\hbs_k}$ denotes the Dirac measure, i.e., a point mass at support point $\hbs_k$, then the dual becomes
\[
\sup_{\mu\in\mathfrak{D}(\hat{\mu},\rho)} \int_{\mS}G(\bs)\,\mu(d\bs)
=
\inf_{\lambda\ge 0}
\left\{
\lambda\rho
+
\sum_{k=1}^{\hat{K}}\hat{w}_k
\sup_{\bs\in \mS}
\left(G(\bs)-\lambda\,c(\hbs_k,\bs)\right)
\right\}.
\]
\textbf{Special case:} If $c(\hbs,\bs)=d(\hbs,\bs)^p$ for a metric $d$ on
\(\mS\), then \(W_c(\hat{\mu},\mu)=W_p(\hat{\mu},\mu)^p\). The dual above remains the same, given that the radius $\rho$ remains consistent with the convention adopted.

\end{itemize}

\section{Omitted Proofs}\label{appndx: ommited-proof}

\subsection{Proof of Proposition~\ref{lem: Gamma-mixture-to-joint-gauss}}\label{appndx:very-first-proof}

Let $\bmean:=(\bmean_1,\bmean_2)\in\mbR^{2n}$ and define $X:=Z-\bmean$. Then $X$ is an $\mbR^{2n}$-valued random vector with
\[
\mbE[X]=0
\qquad\text{and}\qquad
\mbE[XX^\top]=\bQ.
\]
Using $\|\bu\|^2=\bu^\top \bu$ for $\bu\in\mbR^{2n}$, we expand
\begin{align}
\mbE\bigl[\|\bZ\|^2\bigr]
&=\mbE\bigl[\|\bmean+X\|^2\bigr] 
=\mbE\bigl[(\bmean+X)^\top(\bmean+X)\bigr] \notag\\
&=\mbE\bigl[\bmean^\top\bmean + 2\bmean^\top X + X^\top X\bigr] \notag\\
&= \bmean^\top\bmean + 2\bmean^\top\mbE[X] + \mbE[X^\top X] \notag\\
&= \|\bmean\|^2 + \mbE[X^\top X]. \label{eq:Gaussian-moment-proof}
\end{align}
Next, note that $X^\top X=\Tr(X^\top X)=\Tr(X X^\top)$, hence by linearity of
$\Tr(\cdot)$ and $\mbE[\cdot]$,
\[
\mbE[X^\top X]
=\mbE\bigl[\Tr(X X^\top)\bigr]
=\Tr\bigl(\mbE[X X^\top]\bigr)
=\Tr(\bQ).
\]
Substituting this into~\eqref{eq:Gaussian-moment-proof} yields \eqref{eq: joint-gaussian-second-moment}.
\hfil \qed

\subsection{Attainability of the Optimal Solution}
We use $\mX:=\mbR^n\times \mbR^n$ with $\mB(\mX)$ its Borel $\sigma$-algebra throughout this section toward establishing the attainability claim in Theorem~\ref{thm: Wass2-Gaussian-EquivalenceTheorem}. On $\mX$, the Gaussian-mixture class we define is:
\begin{equation}\label{eq:def-GMMmix-2n}
\mbGM_{2n}
:=
\Bigl\{
\nu\ \Big|\ 
\exists\,\lambda\in\mP(\mZ)\ \mathrm{s.t.}\ 
\nu(A)= \hspace{-0.35em}\int_{\mZ}\mN(\bz)(A)\,\lambda(d\bz)\,\, \forall A\in\mB(\mX),
\ \ \text{and} \hspace{-0.25em}
\int_{\mZ}\!\bigl(\|\bmean(\bz)\|^2+\Tr(\bQ(\bz))\bigr)\,\lambda(d\bz)<\infty
\Bigr\}.
\end{equation}

\begin{lemma}[Moment identity for Gaussian mixtures]\label{lem:mixture-moment-identity}
Let $\lambda\in\mP(\mZ)$ and define $\nu\in\mP(\mX)$ by
\[
\nu(A)=\int_{\mZ}\mN(\bz)(A)\,\lambda(d\bz)\qquad \forall A\in\mB(\mX).
\]
If $\int_{\mZ}(\|\bmean(\bz)\|^2+\Tr(\bQ(\bz)))\,\lambda(d\bz)<\infty$, then
\begin{equation}\label{eq:mixture-moment-identity}
\int_{\mX}\|\bu\|^2\,\nu(d\bu)
=
\int_{\mZ}\bigl(\|\bmean(\bz)\|^2+\Tr(\bQ(\bz))\bigr)\,\lambda(d\bz)
<\infty.
\end{equation}
\end{lemma}

\begin{proof}
Let $f_M(\bu):=\min\{\|\bu\|^2,M\}$ for $M\in\mbN$. Then $0\le f_M\uparrow \|\bu\|^2$ pointwise.
Using the probability kernel Tonelli's theorem (nonnegative integrand),
\[
\int_{\mX} f_M(\bu)\,\nu(d\bu)
=
\int_{\mX} f_M(\bu)\left(\int_{\mZ}\mN(\bz)(d\bu)\,\lambda(d\bz)\right)
=
\int_{\mZ}\left(\int_{\mX} f_M(\bu)\,\mN(\bz)(d\bu)\right)\lambda(d\bz).
\]
Letting $M\to\infty$ and applying monotone convergence theorem\srevision{~\cite[Theorem 2.14]{folland1999real}} to the inner integral (in $u$) and then the outer
integral (in $z$) gives
\[
\int_{\mX}\|\bu\|^2\,\nu(d\bu)
=
\int_{\mZ}\left(\int_{\mX}\|\bu\|^2\,\mN(\bz)(d\bu)\right)\lambda(d\bz).
\]
By Proposition~\ref{prop: joint-gaussian-second-moment}, $\int_{\mX}\|\bu\|^2\,\mN(\bz)(d\bu)=\|\bmean(\bz)\|^2+\Tr(\bQ(\bz))$,
yielding \eqref{eq:mixture-moment-identity}.
\end{proof}

\begin{lemma}[Weak compactness of $\Pi(\mu_1,\mu_2)$]\label{lem:Pi-weak-compact}
Let $\mu_1,\mu_2\in\mP(\mbR^n)$. Then $\Pi(\mu_1,\mu_2)$ is weakly compact in $\mP(\mX)$.
\end{lemma}

\begin{proof}
\textit{\underline{(i) Tightness.}}
Fix $\varepsilon>0$. Choose compact sets $K_1,K_2\subset \mbR^n$ such that
$\mu_1(K_1)\ge 1-\varepsilon/2$ and $\mu_2(K_2)\ge 1-\varepsilon/2$ (see\srevision{~\cite[Vol. II, Theorem 7.4.3]{bogachev_measure_2007}} for the existence of such compact high-probability sets). Let $K^c_1$ and $K_2^c$ respectively be the complement set of $K_1$ and $K_2$ with respect to $\mbR^n$.
For any $\nu\in\Pi(\mu_1,\mu_2)$,
\[
\nu\bigl((K_1\times K_2)^c\bigr)
\le \nu(K_1^c\times \mbR^n)+\nu(\mbR^n\times K_2^c)
= \mu_1(K_1^c)+\mu_2(K_2^c)\le \varepsilon,
\]
so $\Pi(\mu_1,\mu_2)$ is tight.

\textit{\underline{(ii) Closedness.}}
Let $\nu_k\xrightarrow{\mathcal D}\nu$ weakly in $\mP(\mX)$ with $\nu_k\in\Pi(\mu_1,\mu_2)$.
For any bounded continuous $\varphi:\mbR^n\to\mbR$, the map $(\bu_1, \bu_2)\mapsto\varphi(\bu_1)$ is bounded continuous on $\mX$,
hence
\[
\int_{\mX}\varphi(\bu_1)\,\nu(d\bu_1,d\bu_2)
=\lim_{k\to\infty}\int_{\mX}\varphi(\bu_1)\,\nu_k(d\bu_1,d\bu_2)
=\int_{\mbR^n}\varphi(\bu_1)\,\mu_1(d\bu_1),
\]
so $\nu^{(1)}=\mu_1$. Similarly $\nu^{(2)}=\mu_2$. Thus $\Pi(\mu_1,\mu_2)$ is weakly closed.

Since $\mX$ is Polish, tightness and weak closedness together establish weak compactness of $\Pi(\mu_1,\mu_2)$.
\end{proof}

\begin{lemma}[Weak continuity of Gaussian laws in parameters]\label{lem:gaussian-weak-continuity}
Let $d\in\mathbb N$. If $\bmean_k\to \bmean$ in $\mbR^d$ and $\bQ_k\to \bQ$ in $\mbS^{d}_+$ (entry-wise, equivalently in any matrix norm),
then
\[
\mN(\bmean_k,\bQ_k)\ \xrightarrow{\mathcal D} \, \mN(\bmean, \bQ)
\quad\text{in }\, \mP(\mbR^d).
\]
Consequently, for every bounded continuous $f\in C_b(\mbR^d)$, the map
\[
(\bmean, \bQ)\longmapsto \int_{\mbR^d} f(\bu)\,\mN(\bmean, \bQ)(d\bu)
\]
is continuous on $\mbR^d\times \mbS^{d}_+$.
\end{lemma}

\begin{proof}
Fix $t\in\mbR^d$. The characteristic function of $\mN(\bmean_k,\bQ_k)$ equals
\[
\varphi_k(t)
=\exp\!\Bigl(i\,t^\top \bmean_k-\tfrac12\, t^\top \bQ_k t\Bigr).
\]
and $\varphi(t):=\exp\!\bigl(i\,t^\top \bmean-\tfrac12\, t^\top \bQ t\bigr)$ is the characteristic function of $\mN(\bmean, \bQ)$, possibly degenerate Gaussian law since \(\bQ\succeq0\). 
Since $\bmean_k\to \bmean$ and $\bQ_k\to Q$, we have $\varphi_k(t)\to \varphi(t)$ for every $t$.
By \srevision{Lévy's continuity theorem~\cite[Theorem 15.23]{klenke2008probability}}, $\mN(\bmean_k,\bQ_k)\xrightarrow{\mathcal D}\mN(\bmean, \bQ)$.
Therefore, the claimed continuity of $(\bmean, \bQ)\mapsto \int f\,d\mN(\bmean, \bQ)$ follows because weak convergence implies
convergence of integrals against all $f\in C_b$.
\end{proof}

\begin{proposition}[Weak compactness of $\Pi(\mu_1,\mu_2)\cap \mbGM_{2n}$]\label{prop:intersection-weak-compact}
The set $\{\Pi(\mu_1,\mu_2)\cap \mbGM_{2n}\}$ is weakly compact in $\mP(\mX)$.
\end{proposition}

\begin{proof}
Let \(\mathcal F:=\Pi(\mu_1,\mu_2)\cap\mbGM_{2n}\). Note that $\mF \neq \emptyset$ since the product coupling \(\mu_1\otimes\mu_2\in\Pi(\mu_1,\mu_2)\) belongs to \(\mbGM_{2n}\) through the point-mass Gaussian representation allowed by \(\mbS^{2n}_+\).   
Hence, throughout the proof we consider \(\mathcal F\neq\varnothing\). 
By Lemma~\ref{lem:Pi-weak-compact}, $\Pi(\mu_1,\mu_2)$ is tight. Hence, $\mathcal F$  is also tight. Therefore, it suffices to show that $\mathcal F$ is weakly closed. Now, since $\mathcal F\neq\varnothing$, choose $\bar\nu\in\mathcal F$. Since
$\bar\nu\in\mbGM_{2n}$, Lemma~\ref{lem:mixture-moment-identity} implies
\[
\int_{\mX}\|\bu\|^2\,\bar\nu(d\bu)<\infty.
\]
Because $\bar\nu\in\Pi(\mu_1,\mu_2)$, we have
\[
\int_{\mX}\|\bu\|^2\,\bar\nu(d\bu)
=
\int_{\mbR^n}\|\bu_1\|^2\,\mu_1(d\bu_1)
+
\int_{\mbR^n}\|\bu_2\|^2\,\mu_2(d\bu_2)<\infty.
\]
Therefore, $\mu_1,\mu_2\in\mathcal P_2(\mbR^n)$. Consequently, every
$\nu\in\Pi(\mu_1,\mu_2)$ satisfies
\[
\int_{\mX}\|\bu\|^2\,\nu(d\bu)
=
\int_{\mbR^n}\|\bu_1\|^2\,\mu_1(d\bu_1)
+
\int_{\mbR^n}\|\bu_2\|^2\,\mu_2(d\bu_2)
=:M<\infty.
\]

Since $\mX$ is Polish, the weak topology on $\mathcal P(\mX)$ is metrizable. Therefore, to prove weak closedness of $\mathcal F$, it suffices to prove sequential weak closedness \srevision{by~\cite[Vol II, Theorem 8.6.2]{bogachev_measure_2007}.} Let $\nu_k\in\mathcal F$ and suppose
$\nu_k\xrightarrow{\mathcal D}\nu$ weakly in $\mathcal P(\mX)$. Because $\nu_k\in\mathcal F\subseteq\Pi(\mu_1,\mu_2)$ for every $k$, and because
$\Pi(\mu_1,\mu_2)$ is weakly closed by Lemma~\ref{lem:Pi-weak-compact}, the weak
limit satisfies $\nu\in\Pi(\mu_1,\mu_2)$. Therefore, it remains to prove that
$\nu\in\mbGM_{2n}$.

For each $k$, pick $\lambda_k\in\mP(\mZ)$ such that $\nu_k\in\mbGM_{2n}$, i.e.,
\begin{equation}\label{eq:nu-k-mixture}
\nu_k(A)=\int_{\mZ}\mN(\bz)(A)\,\lambda_k(d\bz)\qquad \forall A\in\mB(\mX),
\end{equation}
and
\begin{equation}\label{eq:lambda-k-moment}
\int_{\mZ}\bigl(\|\bmean(\bz)\|^2+\Tr(\bQ(\bz))\bigr)\,\lambda_k(d\bz)<\infty.
\end{equation}
Moreover, since $\nu_k\in\Pi(\mu_1,\mu_2)$ and $\|(\bu_1, \bu_2)\|^2=\|\bu_1\|^2+\|\bu_2\|^2$, by the preceding derivation of \(M<\infty\),
\begin{equation}\label{eq:uniform-second-moment-nu-k}
\int_{\mX}\|\bu\|^2\,\nu_k(d\bu)
=
\int_{\mbR^n}\|\bu_1\|^2\,\mu_1(d\bu_1)+\int_{\mbR^n}\|\bu_2\|^2\,\mu_2(d\bu_2)
=:M<\infty.
\end{equation}
By Lemma~\ref{lem:mixture-moment-identity} applied to \eqref{eq:lambda-k-moment}, we obtain
\begin{equation}\label{eq:uniform-moment-lambda-k}
\int_{\mZ}\bigl(\|\bmean(\bz)\|^2+\Tr(\bQ(\bz))\bigr)\,\lambda_k(d\bz)
=
\int_{\mX}\|\bu\|^2\,\nu_k(d\bu)
\le M
\qquad \forall k.
\end{equation}

\textit{\underline{Step 1: Tightness of $\{\lambda_k\}$.}}
Let $g(\bz):=\|\bmean(\bz)\|^2+\Tr(\bQ(\bz))$. For $R>0$, the sublevel set $K_R:=\{z\in\mZ:\ g(\bz)\le R\}$ is compact in $\mZ$. \(K_R\) is closed because \(\mZ=\mbR^{2n}\times\mbS^{2n}_+\) is closed in the finite-dimensional Euclidean space \(\mbR^{2n}\times\mbS^{2n}\), and \(g\) is continuous on \(\mZ\). Moreover, if \(\bz=(\bmean,\bQ)\in K_R\), then \(|\bmean|^2\le R\) and \(\operatorname{tr}(\bQ)\le R\). Since \(\bQ\succeq0\), its eigenvalues are nonnegative, and hence \[ |\bQ|_F=\left(\sum_i\lambda_i(\bQ)^2\right)^{1/2} \le \sum_i\lambda_i(\bQ)=\operatorname{tr}(\bQ)\le R. \] Thus \(K_R\) is bounded. Hence, \(K_R\) is closed and bounded in a finite-dimensional Euclidean space, and therefore compact by the \srevision{Heine--Borel theorem~\cite{rudin2021principles}.}
By Markov's inequality and \eqref{eq:uniform-moment-lambda-k},
\[
\lambda_k(K_R^c)=\lambda_k(\{g>R\}) \le \frac{1}{R}\int_{\mZ} g(\bz)\,\lambda_k(d\bz)\le \frac{M}{R},
\]
uniformly in $k$. Given \(\varepsilon>0\), choosing \(R>M/\varepsilon\) yields \(\lambda_k(K_R)\ge 1-\varepsilon\) for all \(k\). Since \(K_R\) is compact, \(\{\lambda_k\}\) is tight on $\mZ$.
By \srevision{Prokhorov's theorem~\cite[Theorem 13.29]{klenke2008probability},} there exists a subsequence $\lambda_k$ and $\lambda\in\mP(\mZ)$ such that
\[
\lambda_k\xrightarrow{\mathcal D}\lambda \quad\text{weakly in }\mP(\mZ).
\]


\textit{\underline{Step 2: Identify the weak limit $\nu$ as a Gaussian mixture.}}
Fix any bounded continuous \(f\in C_b(\mX)\). Let
\[
\Psi_f:\mZ\to\mbR,\qquad
\Psi_f(\bz):=\int_{\mX} f(\bu)\,\mN(\bz)(d\bu).
\]
By Lemma~\ref{lem:gaussian-weak-continuity} applied with \(d=2n\), if \(\bz_k\to \bz\) in \(\mZ\), then
\[
\mN(\bz_k)\xrightarrow{\mathcal D}\mN(\bz).
\]
Therefore,
\[
\Psi_f(\bz_k)
=
\int_{\mX} f(\bu)\,\mN(\bz_k)(d\bu)
\longrightarrow
\int_{\mX} f(\bu)\,\mN(\bz)(d\bu)
=
\Psi_f(\bz).
\]
Hence, \(\Psi_f\) is continuous on \(\mZ\). Moreover, since
\[
|\Psi_f(\bz)|\le \|f\|_\infty
\qquad \forall z\in\mZ,
\]
we have \(\Psi_f\in C_b(\mZ)\).

Hence, using \eqref{eq:nu-k-mixture} and the kernel Fubini theorem for bounded
measurable functions,
\[
\int_{\mX} f(\bu)\,\nu_k(d\bu)
=
\int_{\mZ}\left(\int_{\mX} f(\bu)\,\mN(\bz)(d\bu)\right)\lambda_k(d\bz)
=
\int_{\mZ}\Psi_f(\bz)\,\lambda_k(d\bz).
\]
Since \(\lambda_k\xrightarrow{\mathcal D}\lambda\) and
\(\Psi_f\in C_b(\mZ)\), we have
\[
\int_{\mZ}\Psi_f(\bz)\,\lambda_k(d\bz)
\longrightarrow
\int_{\mZ}\Psi_f(\bz)\,\lambda(d\bz).
\]
On the other hand, since \(\nu_k\xrightarrow{\mathcal D}\nu\) and
\(f\in C_b(\mX)\), \srevision{by~\cite[Vol II, Corollary~8.2.10]{bogachev_measure_2007}}
\[
\int_{\mX}f(\bu)\,\nu_k(d\bu)
\longrightarrow
\int_{\mX}f(\bu)\,\nu(d\bu).
\]
Therefore,
\begin{equation}\label{eq:limit-identity-f}
\int_{\mX} f(\bu)\,\nu(d\bu)
=
\int_{\mZ}\Psi_f(\bz)\,\lambda(d\bz)
\qquad \forall f\in C_b(\mX).
\end{equation}

Now define
\[
\tilde\nu(A)
:=
\int_{\mZ}\mN(\bz)(A)\,\lambda(d\bz),
\qquad A\in\mB(\mX).
\]
Because \(\bz\mapsto\mN(\bz)(\cdot)\) is a probability kernel and
\(\lambda\in\mathcal P(\mZ)\), the set function \(\tilde\nu\) is a probability
measure on \((\mX,\mB(\mX))\), i.e., \(\tilde\nu\in\mathcal P(\mX)\).
Moreover, by the kernel Fubini theorem again, for every \(f\in C_b(\mX)\),
\[
\int_{\mX} f(\bu)\,\tilde\nu(d\bu)
=
\int_{\mZ}
\left(\int_{\mX} f(\bu)\,\mN(\bz)(d\bu)\right)\lambda(d\bz)
=
\int_{\mZ}\Psi_f(\bz)\,\lambda(d\bz).
\]
Combining this identity with \eqref{eq:limit-identity-f} yields
\begin{align}\label{eq:nu-mixture-intermediate}
    \int_{\mX}f(\bu)\,\nu(d\bu)
=
\int_{\mX}f(\bu)\,\tilde\nu(d\bu)
\qquad \forall f\in C_b(\mX).
\end{align}
Let
\[
\bar{C}_b(\mX)
:=
\{f:\mX\to\mathbb R:\ f \text{ is bounded and uniformly continuous}\}.
\]
Since $\bar{C}_b(\mX)\subseteq C_b(\mX)$,~\eqref{eq:nu-mixture-intermediate} implies
\[
\int_{\mX}f(\bu)\,\nu(d\bu)
=
\int_{\mX}f(\bu)\,\tilde\nu(d\bu)
\qquad \forall f\in \bar{C}_b(\mX).
\]
Since \(\mX=\mathbb R^n\times\mathbb R^n\) is a metric space, the class
\(\bar{C}_b(\mX)\) of bounded uniformly continuous test functions is
measure-determining for Borel probability measures on \(\mX\)~\cite[Vol.~II, Remark~8.3.1; see also Theorem~8.2.3 and Corollary~8.2.4]{bogachev_measure_2007}. Therefore, equality of integrals against all functions in
\(\bar{C}_b(\mX)\) implies equality of the underlying probability measures.
Therefore, \(\nu=\tilde\nu\). Equivalently,
\begin{equation}\label{eq:nu-mixture}
\nu(A)
=
\int_{\mZ}\mN(\bz)(A)\,\lambda(d\bz)
\qquad \forall A\in\mB(\mX).
\end{equation}

\textit{\underline{Step 3: verify the moment condition for $\lambda$.}}
Let $g(\bz):=\|\bmean(\bz)\|^2+\Tr(\bQ(\bz))$. Since $g$ is lower semicontinuous on $\mZ$ and
$\lambda_k\xrightarrow{\mathcal D}\lambda$, \srevision{Portmanteau theorem~\cite[Theorem~13.16]{klenke2008probability}} yields
\[
\int_{\mZ} g(\bz)\,\lambda(d\bz)\le \liminf_{k\to\infty}\int_{\mZ} g(\bz)\,\lambda_k(d\bz)\le M<\infty.
\]
Together with \eqref{eq:nu-mixture}, this shows $\nu\in\mbGM_{2n}$.
Hence, $\nu\in\mathcal{F}$, so $\mathcal{F}$ is weakly closed.

\noindent Therefore, $\mathcal{F}$ is a weakly closed subset of the weakly compact set $\Pi(\mu_1,\mu_2)$,
and is thus weakly compact.
\end{proof}

\begin{corollary}[Attainment of solution in Theorem~\ref{thm: Wass2-Gaussian-EquivalenceTheorem}]\label{cor:attainment-intersection}
Let $c:\mX\to(-\infty,\infty]$ be lower semicontinuous and bounded below. Then the optimization problem
\[
\inf_{\nu\in \Pi(\mu_1,\mu_2)\cap \mbGM_{2n}} \int_{\mX} c(\bu)\,\nu(d\bu)
\]
admits a minimizer in $\Pi(\mu_1,\mu_2)\cap \mbGM_{2n}$.
In particular, this holds for $c(\bu_1,\bu_2)=\|\bu_1-\bu_2\|^2$.
\end{corollary}

\begin{proof}
By Proposition~\ref{prop:intersection-weak-compact}, the feasible set $\Pi(\mu_1,\mu_2)\cap\mbGM_{2n}$ is weakly compact. The feasible set is nonempty because the product coupling \(\nu:=\mu_1\otimes\mu_2\in\Pi(\mu_1,\mu_2)\) belongs to \(\mbGM_{2n}\) through the point-mass Gaussian representation allowed by \(\mbS^{2n}_+\). Additionally, since
\(c\) is lower semicontinuous and bounded below, the map \(\nu\mapsto\int_{\mX}c\,d\nu\) is weakly lower semicontinuous (cf. the direct-method argument in\srevision{~\cite[Theorem~4.1]{villani2008optimal}).} Hence, the infimum is attained over the nonempty weakly compact feasible set.
\end{proof}

\subsection{Second Moment Finiteness Preservation}\label{appndx: second-moment-preservation}

\begin{lemma}[Cross-covariance Frobenius bounds]
\label{lem:Sigma-bound}
Let $\bQ_1,\bQ_2\in \mbS^n_+$ and let $\bSigma$ be defined by~\eqref{def: gelbrich-off-diagonal}. Then
\begin{equation}
\label{eq:Sigma-Fbound}
\|\bSigma\|_F
\le
\sqrt{\|\bQ_1\|_F\,\|\bQ_2\|_F}
\le
\frac{\|\bQ_1\|_F+\|\bQ_2\|_F}{2} \,\,\, \text{and} \,\,\, \|\bar{\bQ}\|_F^2
\le
2\bigl(
\|\bQ_1\|_F^2+\|\bQ_2\|_F^2
\bigr).
\end{equation}
where $\bar{\bQ}
=
\begin{psmallmatrix}
\bQ_1 & \bSigma\\
\bSigma^\top & \bQ_2
\end{psmallmatrix}
\in \mbS^{2n}_+.$
\end{lemma}

\begin{proof}

A standard characterization of positive semidefinite block matrices states that $\begin{pmatrix}
A & X\\
X^\top & B
\end{pmatrix}
\succeq 0$ if and only if there exists a contraction $K$ such that $X=A^{1/2}KB^{1/2},
\,\,
\|K\|_{\mathrm{op}}\le 1,$
where $\|K\|_{\mathrm{op}}
:=
\sup_{\|x\|_2=1}\|Kx\|_2
=
\sigma_{\max}(K)$
denotes the operator norm induced by the Euclidean norm; see, e.g.,
\cite[Proposition~1.3.2]{bhatia2009positive}. 
Applying this result to $\bar{\bQ}\succeq 0$, there exists a matrix $K$
such that 
\[
\bSigma
=
\bQ_1^{1/2}K\bQ_2^{1/2},
\qquad
\|K\|_{\mathrm{op}}\le 1.
\]
Hence, we obtain
\begin{align*}
\|\bSigma\|_F^2
=
\Tr(\bSigma\bSigma^\top)=
\Tr\!\left(
\bQ_1^{1/2}
K\bQ_2K^\top
\bQ_1^{1/2}
\right)=
\Tr\!\left(
\bQ_1K\bQ_2K^\top
\right),
\end{align*}
where the last equality follows from the cyclicity of the trace.

Using the Frobenius inner product $\langle A,B\rangle_F
:=
\Tr(A^\top B),$ we write $\|\bSigma\|_F^2
=
\left\langle
\bQ_1,
K\bQ_2K^\top
\right\rangle_F.$
By the Cauchy--Schwarz inequality for the Frobenius inner product, $\|\bSigma\|_F^2
\le
\|\bQ_1\|_F
\,
\|K\bQ_2K^\top\|_F.$ 

We now use the standard mixed-norm inequality
\[
\|AB\|_F
\le
\|A\|_{\mathrm{op}}\|B\|_F.
\]
Applying it twice yields
\begin{align*}
\|K\bQ_2K^\top\|_F
\le
\|K\|_{\mathrm{op}}
\|\bQ_2K^\top\|_F\le
\|K\|_{\mathrm{op}}
\|\bQ_2\|_F
\|K^\top\|_{\mathrm{op}}=
\|K\|_{\mathrm{op}}^2
\|\bQ_2\|_F\le
\|\bQ_2\|_F.
\end{align*}
Consequently, $\|\bSigma\|_F^2
\le
\|\bQ_1\|_F
\|\bQ_2\|_F,$ and therefore $\|\bSigma\|_F
\le
\sqrt{
\|\bQ_1\|_F
\|\bQ_2\|_F
}.$ 
The scalar arithmetic-geometric mean inequality gives
\[
\sqrt{
\|\bQ_1\|_F
\|\bQ_2\|_F
}
\le
\frac{
\|\bQ_1\|_F+\|\bQ_2\|_F
}{2},
\]
which proves~\eqref{eq:Sigma-Fbound}.

Finally, since $\bar{\bQ}
=
\begin{pmatrix}
\bQ_1 & \bSigma\\
\bSigma^\top & \bQ_2
\end{pmatrix},$ the Frobenius norm of a block matrix gives
\begin{align*}
    \|\bar{\bQ}\|_F^2 &=
\|\bQ_1\|_F^2
+
\|\bQ_2\|_F^2
+
\|\bSigma\|_F^2
+
\|\bSigma^\top\|_F^2 =
\|\bQ_1\|_F^2
+
\|\bQ_2\|_F^2
+
2\|\bSigma\|_F^2 \\
&\le
\|\bQ_1\|_F^2
+
\|\bQ_2\|_F^2
+
2\|\bQ_1\|_F\|\bQ_2\|_F=
\bigl(
\|\bQ_1\|_F+\|\bQ_2\|_F
\bigr)^2 \leq 2\bigl(
\|\bQ_1\|_F^2+\|\bQ_2\|_F^2
\bigr).
\end{align*}
\end{proof}

\subsubsection{Proof of Proposition~\ref{prop: finite-GMM-mmnt}.}
Let $f\bigl(\bz\bigr)
:=\|\bmean_1(\bz)\|^2+\|\bmean_2(\bz)\|^2+\|\bar{\bQ}(\bz)\|_F^2.$
Similar to~\eqref{eq: push-integral}, by the pushforward (change-of-variables) identity for nonnegative measurable map $f$,
\begin{equation}\label{eq:pushforward-id}
\int_{\mZ} f(\bz)\,d \lambda(\bz) = \int_{\mZ} f(\bz)\,d(\mT_\#\pi)(\bz)
=\int_{\mS\times \mS} f(\mT(\bs_1,\bs_2))\,d\pi(\bs_1,\bs_2),
\end{equation}
where the first term is the one we have in the claim~\eqref{eq:concl-output}.

Now, for $(\bs_1,\bs_2)=((\bmean_1,\bQ_1),(\bmean_2,\bQ_2))$ we have
$f(\mT(\bs_1,\bs_2))=\|\bmean_1(\bs_1)\|^2+\|\bmean_2(\bs_2)\|^2+\|\bar{\bQ}(\bs_1, \bs_2)\|_F^2$.
By Lemma~\ref{lem:Sigma-bound}, $\|\bar{\bQ}(\bs_1, \bs_2)\|_F^2 \le 2\bigl(\|\bQ_1(\bs_1)\|_F^2+\|\bQ_2(\bs_2)\|_F^2\bigr).$
Therefore,
\[
f(\mT(\bs_1,\bs_2))
\le
\|\bmean_1(\bs_1)\|^2+\|\bmean_2(\bs_2)\|^2
+2\|\bQ_1(\bs_1)\|_F^2+2\|\bQ_2(\bs_2)\|_F^2.
\]
Integrating both sides with respect to $\pi$ yields
\[
\int_{\mS\times\mS} f\bigl(\mT(\bs_1, \bs_2)\bigr) \, d\pi(\bs_1, \bs_2)
\le
\int_{\mS\times\mS}
\Bigl(\|\bmean_1(\bs_1)\|^2+\|\bmean_2(\bs_2)\|^2+2\|\bQ_1(\bs_1)\|_F^2+2\|\bQ_2(\bs_2)\|_F^2\Bigr)\,d\pi(\bs_1, \bs_2).
\]
Combining this with~\eqref{eq:pushforward-id} gives us:
\[
\int_{\mZ} f(\bz)\,d\lambda(\bz)
\le
\int_{\mS\times\mS}
\Bigl(\|\bmean_1(\bs_1)\|^2+\|\bmean_2(\bs_2)\|^2+2\|\bQ_1(\bs_1)\|_F^2+2\|\bQ_2(\bs_2)\|_F^2\Bigr)\,d\pi(\bs_1, \bs_2).
\]
The right-hand side is finite by assumption~\eqref{eq:assump-input}. Hence~\eqref{eq:concl-output} holds. \hfill \qed

\subsubsection{Proof of Proposition~\ref{prop: finite-2nd-moment}.}
Let $\,\, f\bigl(\bz):=\|\bmean_1(\bz)\|^2+\|\bmean_2(\bz)\|^2+\|\bQ(\bz)\|_F^2,
\quad
g\bigl(\bs_1, \bs_2\bigr):=\|\bmean_1(\bs_1)\|^2+\|\bmean_2(\bs_2)\|^2+\|\bQ_{11}(\bs_1)\|_F^2+\|\bQ_{22}(\bs_2)\|_F^2.$
 Since the measurable map $g$ is nonnegative, similar to the pushforward identity~\eqref{eq: push-integral-M},
\begin{equation}\label{eq:pushforward-M-generalQ}
\int_{\mS\times \mS} g(\bs_1, \bs_2) d \pi(\bs_1, \bs_2) = \int_{\mS\times \mS} g(\bs_1, \bs_2)\,d(\mM_\#\lambda)(\bs_1, \bs_2)
=
\int_{\mZ} g(\mM(\bz))\,d\lambda(\bz),
\end{equation}
where the left-most term is exactly the one desired in our claim~\eqref{eq:concl-SxS-generalQ}. Now fix $\bz=((\bmean_1,\bmean_2),\bQ)\in\mZ$ and write $\bQ$ in blocks as defined in~\eqref{def: Q-block}. Then,
\[
g(\mM(\bz))=\|\bmean_1(\bz)\|^2+\|\bmean_2(\bz)\|^2+\|\bQ_{11}(\bz)\|_F^2+\|\bQ_{22}(\bz)\|_F^2.
\]

Additionally, the following inequality holds pointwise since by Frobenius norm definition $\|\bQ\|_F^2=\|\bQ_{11}\|_F^2+\|\bQ_{22}\|_F^2+\|\bQ_{12}\|_F^2+\|\bQ_{21}\|_F^2$\srevision{,}
\begin{align*}
& \|\bQ_{11}(\bz)\|_F^2+\|\bQ_{22}(\bz)\|_F^2 \le \|\bQ(\bz)\|_F^2 \\ 
\implies  \|\bmean_1(\bz)\|^2+\|\bmean_2(\bz)\|^2 + &  \|\bQ_{11}(\bz)\|_F^2+\|\bQ_{22}(\bz)\|_F^2 \le  \|\bQ(\bz)\|_F^2 +\|\bmean_1(\bz)\|^2+\|\bmean_2(\bz)\|^2\srevision{.}
\end{align*}
Hence, for all $z \in \mZ$,  $g(\mM(\bz))\le f(\bz)$. Integrating both sides with respect to $\lambda$,
\[ \quad \int_\mZ g(\mM(\bz)) d \lambda(\bz) \leq \int_\mZ f(\bz) d\lambda(\bz).\]
Then, combining the above inequality with~\eqref{eq:pushforward-M-generalQ} we obtain
\[
\int_{\mS\times \mS} g(\bs_1, \bs_2)\,d\pi(\bs_1, \bs_2)
\le
\int_{\mZ} f(\bz)\,d\lambda(\bz).
\]
The right-hand side is finite by assumption~\eqref{eq:assump-Z-generalQ}, thus, \eqref{eq:concl-SxS-generalQ} holds. \hfill \qed

\subsection{Proof of Wasserstein-Metric Equivalence with Restricted Parameter Support Set}\label{appndx: restricted}

Throughout this subsection, we assume \(\mX:=\mbR^n\times\mbR^n\) and a compact support set $\mathcal S:=\Theta\times\Xi\subset \mbR^n\times\mbS^n_{++}$. For a joint covariance matrix\vspace{-1.9em}
\begin{align}\label{def:Q-def-appndx}
    \bQ=
\begin{pmatrix}
\bQ_{11} & \bQ_{12}\\
\bQ_{21} & \bQ_{22}
\end{pmatrix}
\in\mbS^{2n}_+,
\end{align} define the restricted joint-parameter set
\[
\mathcal Z_{\mathcal S}
:=
\left\{
\bz=((\bmean_1,\bmean_2),\bQ)\in
\mbR^n\times\mbR^n\times\mbS^{2n}_+
:
(\bmean_1,\bQ_{11})\in\mathcal S,\ 
(\bmean_2,\bQ_{22})\in\mathcal S
\right\}.
\]
We define the restricted \(2n\)-variate Gaussian-mixture class by
\[
\mbGM_{2n}^{\mathcal S}
:=
\left\{
\nu\in\mathcal P(\mbR^{2n})
:
\nu(A)=\int_{\mathcal Z_{\mathcal S}}\mN(\bz)(A)\,\lambda(d\bz)
\ \forall A\in\mathcal B(\mbR^{2n})
\text{ for some }\lambda\in\mathcal P(\mathcal Z_{\mathcal S})
\right\}.
\]

\begin{remark}
    We note that \(\mathcal Z_{\mathcal S}\) is not the full set of all joint PSD
covariances; it is the set of valid joint Gaussian parameters whose two marginal
Gaussian parameters lie in \(\mathcal S\). Thus, for
\(\bz=((\bmean_1,\bmean_2),\bQ)\in\mathcal Z_{\mathcal S}\), the full joint
covariance \(\bQ\) is required only to satisfy \(\bQ\succeq0\), while its diagonal blocks satisfy
\[
\bQ_{11},\bQ_{22}\in\Xi\subseteq\mathbb S^n_{\hat\delta}.
\]
The joint covariance may therefore be singular even though the marginal
covariances are uniformly positive definite. For example, if
\((\bmean,\bQ_0)\in\mathcal S\), then
\[
\begin{pmatrix}
\bQ_0 & \bQ_0\\
\bQ_0 & \bQ_0
\end{pmatrix}
\succeq0
\]
is singular, while both diagonal blocks equal \(\bQ_0\in\Xi\). Hence
\[
((\bmean,\bmean),
\begin{pmatrix}
\bQ_0 & \bQ_0\\
\bQ_0 & \bQ_0
\end{pmatrix})
\in\mathcal Z_{\mathcal S}.
\]

On the other hand, \(\mathcal Z_{\mathcal S}\) also contains the independent
Gaussian couplings. Indeed, if \(s_i=(\bmean_i,\bQ_i)\in\mathcal S\), \(i=1,2\),
then
\[
\bQ^{\mathrm{ind}}
:=
\begin{pmatrix}
\bQ_1 & 0\\
0 & \bQ_2
\end{pmatrix}
\]
satisfies \(\bQ^{\mathrm{ind}}\succeq0\), and
\[
((\bmean_1,\bmean_2),\bQ^{\mathrm{ind}})
\in\mathcal Z_{\mathcal S}.
\]
Thus \(\mathcal Z_{\mathcal S}\) contains at least the independent Gaussian
couplings of pairs of Gaussian components in \(\mathcal S\).
\end{remark}

\begin{lemma}[Compactness of the restricted joint-parameter set]
\label{lem:ZS-compact}
Under Assumption~\ref{assmp: compactness}, the set \(\mathcal Z_{\mathcal S}\)
is compact.
\end{lemma}

\begin{proof}
Since \(\mathcal S\) is compact, the sets of admissible marginal means and
diagonal covariance blocks are bounded. Thus \(\bmean_1,\bmean_2\),
\(\bQ_{11}\), and \(\bQ_{22}\) are bounded over
\(\mathcal Z_{\mathcal S}\). Moreover, if $\bQ\succeq0$,
then every entry of the off-diagonal block satisfies
\[
|(\bQ_{12})_{ij}|^2
\le
(\bQ_{11})_{ii}(\bQ_{22})_{jj}.
\]
Hence, \(\bQ_{12}\) and \(\bQ_{21}\) are also bounded. Therefore,
\(\mathcal Z_{\mathcal S}\) is bounded.

It remains to show that \(\mathcal Z_{\mathcal S}\) is closed. Indeed,
\(\mbS_+^{2n}\) is closed, and the conditions
\((\bmean_1,\bQ_{11})\in\mathcal S\) and
\((\bmean_2,\bQ_{22})\in\mathcal S\) are closed because
\(\mathcal S\) is compact, hence closed. Thus, \(\mathcal Z_{\mathcal S}\) is closed and bounded in a finite-dimensional Euclidean space. By \srevision{Heine--Borel~\cite{rudin2021principles},}
\(\mathcal Z_{\mathcal S}\) is compact.
\end{proof}

\begin{lemma}[Nonemptiness under restricted Gaussian-mixture support]
\label{lem:nonempty-restricted-GM-coupling}
Suppose that \(\mu_1,\mu_2\in\mathcal P(\mbR^n)\) admit mixing laws over
\(\mathcal S\); that is, there exist \(\eta_1,\eta_2\in\mathcal P(\mathcal S)\)
such that, for every \(C\in\mathcal B(\mbR^n)\),
\[
\mu_1(C)=\int_{\mathcal S}\mN(\bs)(C)\,\eta_1(d\bs),
\qquad
\mu_2(C)=\int_{\mathcal S}\mN(\bs)(C)\,\eta_2(d\bs).
\]
Then $\quad \Pi(\mu_1,\mu_2)\cap \mbGM_{2n}^{\mathcal S}\neq\varnothing.$
\end{lemma}

\begin{proof}
Let\vspace{-2em}
\[
\pi:=\eta_1\otimes\eta_2\in\mathcal P(\mathcal S\times\mathcal S).
\]
For \(\bs_i=(\bmean_i,\bQ_i)\in\mathcal S\), \(i=1,2\), let
\(\mathcal G(\bs_1,\bs_2)\) denote the Gelbrich coupling between
\(\mN(\bs_1)\) and \(\mN(\bs_2)\). Thus
\[
\mathcal G(\bs_1,\bs_2)
=
\mN\!\left(
(\bmean_1,\bmean_2),
\overline{\bQ}(\bs_1,\bs_2)
\right),
\]
where\vspace{-1.75em}
\[
\overline{\bQ}(\bs_1,\bs_2)
=
\begin{pmatrix}
\bQ_1 & \bSigma(\bQ_1,\bQ_2)\\
\boldsymbol\Sigma(\bQ_1,\bQ_2)^\top & \bQ_2
\end{pmatrix}
\in\mbS^{2n}_+.
\]
Since the diagonal covariance blocks of \(\overline{\bQ}(\bs_1,\bs_2)\) are
\(\bQ_1\) and \(\bQ_2\), the map $\mathcal T(\bs_1,\bs_2)
:=
\left(
(\bmean_1,\bmean_2),
\overline{\bQ}(\bs_1,\bs_2)
\right)$ satisfies $\mathcal T(\bs_1,\bs_2)\in\mathcal Z_{\mathcal S}
\quad
\forall (\bs_1,\bs_2)\in\mathcal S\times\mathcal S$. Hence, $\lambda:=\mathcal T_{\#}\pi\in\mathcal P(\mathcal Z_{\mathcal S})$. 

Define \(\nu\in\mathcal P(\mbR^{2n})\) by
\[
\nu(A):=\int_{\mathcal Z_{\mathcal S}}\mN(\bz)(A)\,\lambda(d\bz),
\qquad
A\in\mathcal B(\mbR^{2n}).
\]
Equivalently,
\[
\nu(A)
=
\int_{\mathcal S\times\mathcal S}
\mathcal G(\bs_1,\bs_2)(A)\,\pi(ds_1,ds_2).
\]
Therefore, \(\nu\in\mbGM_{2n}^{\mathcal S}\).

It remains to show that \(\nu\in\Pi(\mu_1,\mu_2)\). Let
\(B\in\mathcal B(\mbR^n)\). Since \(\mathcal G(\bs_1,\bs_2)\) is a coupling of
\(\mN(\bs_1)\) and \(\mN(\bs_2)\),
\[
\mathcal G(\bs_1,\bs_2)(B\times\mbR^n)
=
\mN(\bs_1)(B).
\]
Thus\vspace{-1.85em}
\[
\nu(B\times\mbR^n)
=
\int_{\mathcal S\times\mathcal S}
\mN(\bs_1)(B)\,(\eta_1\otimes\eta_2)(ds_1,ds_2)
=
\int_{\mathcal S}\mN(\bs_1)(B)\,\eta_1(ds_1)
=
\mu_1(B).
\]
Hence, \(\nu^{(1)}=\mu_1\). The same argument gives
\[
\nu(\mbR^n\times B)
=
\int_{\mathcal S}\mN(\bs_2)(B)\,\eta_2(ds_2)
=
\mu_2(B),
\]
so \(\nu^{(2)}=\mu_2\). Therefore, \(\nu\in\Pi(\mu_1,\mu_2)\).

Combining \(\nu\in\Pi(\mu_1,\mu_2)\) and
\(\nu\in\mbGM_{2n}^{\mathcal S}\), we obtain
\[
\nu\in\Pi(\mu_1,\mu_2)\cap\mbGM_{2n}^{\mathcal S}.
\]
Hence, the intersection is nonempty.
\end{proof}

\begin{corollary}[Weak compactness under restricted support]
\label{cor:restricted-intersection-weak-compact}
Under Assumption~\ref{assmp: compactness}, $\Pi(\mu_1,\mu_2)\cap\mbGM_{2n}^{\mathcal S}$ is weakly compact.
\end{corollary}

\begin{proof}
The set \(\mathcal Z_{\mathcal S}\) is compact. Indeed, \(\mathcal S\) is compact,
so the diagonal blocks \(\bQ_{11}\) and \(\bQ_{22}\), as well as the
means \(\bmean_1,\bmean_2\), are bounded. Moreover, for any
\(\bQ\succeq0\),
\[
|(\bQ_{12})_{ij}|^2
\le
(\bQ_{11})_{ii}(\bQ_{22})_{jj},
\]
so the off-diagonal block \(\bQ_{12}\) is also bounded. The defining
conditions of \(\mathcal Z_{\mathcal S}\) are closed, because
\(\mathcal S\) and \(\mbS^{2n}_+\) are closed. Hence,
\(\mathcal Z_{\mathcal S}\) is closed and bounded in a finite-dimensional Euclidean space and therefore compact.

Now the proof of Proposition~\ref{prop:intersection-weak-compact} applies with
\(\mathcal Z\) replaced by \(\mathcal Z_{\mathcal S}\). In fact, the tightness
of the mixing laws is immediate because every \(\lambda_k\) is supported on the
compact set \(\mathcal Z_{\mathcal S}\). The weak limit of a subsequence of
\(\lambda_k\)'s remains in \(\mathcal P(\mathcal Z_{\mathcal S})\), and the
weak continuity of Gaussian laws again identifies the weak limit as a Gaussian
mixture supported on \(\mathcal Z_{\mathcal S}\). Hence,
\(\Pi(\mu_1,\mu_2)\cap\mbGM_{2n}^{\mathcal S}\) is weakly closed inside the
weakly compact set \(\Pi(\mu_1,\mu_2)\), and is therefore weakly compact.
\end{proof}

\begin{corollary}[Attainment of solution in Theorem~\ref{thm:Wass2-Gaussian-EquivalenceTheorem-restricted}]
\label{cor:attainment-restricted-intersection}
Assume the setting of Lemma~\ref{lem:nonempty-restricted-GM-coupling}. Let
\(c:\mX\to(-\infty,\infty]\) be lower semicontinuous and bounded below. Then
\[
\inf_{\nu\in \Pi(\mu_1,\mu_2)\cap \mbGM_{2n}^{\mathcal S}}
\int_{\mX} c(\bu)\,\nu(d\bu)
\]
admits a minimizer in
\(\Pi(\mu_1,\mu_2)\cap \mbGM_{2n}^{\mathcal S}\). In particular, this holds for
\(c(\bu_1,\bu_2)=\|\bu_1-\bu_2\|^2\).
\end{corollary}

\begin{proof}
By Corollary~\ref{cor:restricted-intersection-weak-compact}, the feasible set $\Pi(\mu_1,\mu_2)\cap \mbGM_{2n}^{\mathcal S}$ is weakly compact. By Lemma~\ref{lem:nonempty-restricted-GM-coupling}, it is nonempty. Since \(c\) is lower semicontinuous and bounded below, the map $\nu\mapsto \int_{\mX}c(\bu)\,\nu(d\bu)$ is weakly lower semicontinuous. Therefore, the infimum is attained over the nonempty weakly compact feasible set.
\end{proof}

\subsubsection{Second-moment finiteness preservation under restricted parameter support. }\, 
In the restricted-support setting, the two parameter maps $\mT$ and $\mM$ from the full-support case will be used with restricted domains. First, for
\(\bs_i=(\bmean_i,\bQ_i)\in\mathcal S\), the Gelbrich joint covariance
has diagonal blocks \(\bQ_1\) and \(\bQ_2\). Therefore,  $\mathcal T(\mathcal S\times\mathcal S)\subseteq\mathcal Z_{\mathcal S},$
and the forward map is $\mathcal T:\mathcal S\times\mathcal S\to\mathcal Z_{\mathcal S}$. Second, for
\(\bz=((\bmean_1,\bmean_2),\bQ)\in\mathcal Z_{\mathcal S}\)
with $\bQ$ as denoted in~\eqref{def:Q-def-appndx}, 
define $\mathcal M(\bz)
:=
\bigl((\bmean_1,\bQ_{11}),
      (\bmean_2,\bQ_{22})\bigr)$. By the definition of \(\mathcal Z_{\mathcal S}\), this is well-defined as $\mathcal M:\mathcal Z_{\mathcal S}\to\mathcal S\times\mathcal S$ and  $\mathcal M(\mathcal Z_{\mathcal S})\subseteq\mathcal S\times\mathcal S.$

Because \(\mathcal S\) is compact in the finite-dimensional space
\(\mathbb R^n\times\mathbb S^n\), the product set
\(\mathcal S\times\mathcal S\) is compact and hence bounded. Therefore, under the product gauge $\|((\bmean_1,\bQ_{11}),(\bmean_2,\bQ_{22}))\|_{\mS\times \mS}^2$, 
every
\(\pi\in\mathcal P(\mathcal S\times\mathcal S)\) satisfies
\[
\int_{\mathcal S\times\mathcal S}
\left(
\|\bmean_1\|^2+\|\bQ_1\|_F^2
+\|\bmean_2\|^2+\|\bQ_2\|_F^2
\right)
\,d\pi(s_1,s_2)<\infty.
\]
Similarly, by Lemma~\ref{lem:ZS-compact}, \(\mathcal Z_{\mathcal S}\) is
compact. Hence, every \(\lambda\in\mathcal P(\mathcal Z_{\mathcal S})\) has a finite second moment with respect to a gauge on
\(\mathcal Z_{\mathcal S}\). We formally state the claim below.

\begin{corollary}[Restricted moment preservation]
\label{cor:restricted-moment-preservation}
Under Assumption~\ref{assmp: compactness}, the moment-preservation statements
of Propositions~\ref{prop: finite-GMM-mmnt} and
\ref{prop: finite-2nd-moment} remain valid after replacing
\(\mathcal Z\) by \(\mathcal Z_{\mathcal S}\). In particular,
\[
\pi\in\mathcal P(\mathcal S\times\mathcal S)
\quad\Longrightarrow\quad
\mathcal T_{\#}\pi\in\mathcal P(\mathcal Z_{\mathcal S}),
\]
and
\[
\lambda\in\mathcal P(\mathcal Z_{\mathcal S})
\quad\Longrightarrow\quad
\mathcal M_{\#}\lambda\in\mathcal P(\mathcal S\times\mathcal S).
\]
Moreover, all parameter second moments appearing in
Propositions~\ref{prop: finite-GMM-mmnt} and
\ref{prop: finite-2nd-moment} are finite.
\end{corollary}

\begin{proof}
The two push-forward inclusions follow from the restricted domains of
\(\mathcal T\) and \(\mathcal M\):
\[
\mathcal T(\mathcal S\times\mathcal S)\subseteq\mathcal Z_{\mathcal S},
\qquad
\mathcal M(\mathcal Z_{\mathcal S})\subseteq\mathcal S\times\mathcal S.
\]
By Assumption~\ref{assmp: compactness}, \(\mathcal S\) is compact. By
Lemma~\ref{lem:ZS-compact}, \(\mathcal Z_{\mathcal S}\) is compact. Therefore, the gauges used in Propositions~\ref{prop: finite-GMM-mmnt} and
\ref{prop: finite-2nd-moment} are bounded on the relevant supports. Hence, the
corresponding moment integrals are finite for every probability measure
supported on these compact sets.
\end{proof}

\subsubsection{Containment under restricted support}

For clarity, let
\[
\mathcal S_{\rm full}:=\mathbb R^n\times\mathbb S^n_+,
\qquad
\mathcal Z_{\rm full}:=\mathbb R^n\times\mathbb R^n\times\mathbb S^{2n}_+.
\]
Then \(\mathcal S\subseteq\mathcal S_{\rm full}\) and
\(\mathcal Z_{\mathcal S}\subseteq\mathcal Z_{\rm full}\).

Define
\[
\Gamma_{\mathcal S}
:=
  \Bigl\{
    \pi \in \mP_2(\mathcal S \times \mathcal S)
    \ \Big|\ 
   \int_{\mathcal S} \mN(\bs)(A)\,d\pi_1(\bs) = \mu_1(A),\quad
   \int_{\mathcal S} \mN(\bs)(B)\,d\pi_2(\bs) = \mu_2(B),
   \ \forall A,B\in\mB(\mbR^n)
  \Bigr\}.
\]
We keep the notation \(\pi\in\mathcal P_2(\mathcal S\times\mathcal S)\) only to remain consistent with \eqref{def: coupling-param}.

\begin{corollary}
\label{cor:restricted-forward-containment}
For every \(\pi\in\Gamma_{\mathcal S}\), define
\[
\nu_\pi(A)
:=
\int_{\mathcal S\times\mathcal S}
\mathcal G(\bs_1,\bs_2)(A)\,d\pi(\bs_1,\bs_2),
\qquad A\in\mathcal B(\mbR^{2n}).
\]
Then $\nu_\pi\in\Pi(\mu_1,\mu_2)\cap\mbGM_{2n}^{\mathcal S}.$
\end{corollary}

\begin{proof}
Fix \(\pi\in\Gamma_{\mathcal S}\).  Since \(\pi\in\Gamma_{\mathcal S}\), its
marginals satisfy
\[
\int_{\mathcal S}\mN(\bs)(A)\,d\pi_1(\bs)=\mu_1(A),
\qquad
\int_{\mathcal S}\mN(\bs)(B)\,d\pi_2(\bs)=\mu_2(B), \quad \forall \, A,B\in\mathcal B(\mbR^n)
\]
Since \(\pi_1\) and \(\pi_2\) are
supported on \(\mathcal S\), these are the same marginal mixing-law identities
used in Lemma~\ref{lem: forward} viewing \(\pi\) as a probability measure on
\(\mathcal S_{\rm full}\times\mathcal S_{\rm full}\) supported on \(\mathcal S\times\mathcal S\). Moreover, since \(\mathcal S\) is compact, the second moment of \(\pi\) is finite. Therefore Lemma~\ref{lem: forward} applies and
gives $\nu_\pi\in\Pi(\mu_1,\mu_2)\cap\mbGM_{2n}.$

It remains to check that the Gaussian-mixture representation of \(\nu_\pi\) is
supported on \(\mathcal Z_{\mathcal S}\). Let $\lambda:=\mathcal T_{\#}\pi.$ Since $\mathcal T(\mathcal S\times\mathcal S)\subseteq\mathcal Z_{\mathcal S}$, $\lambda\in\mathcal P(\mathcal Z_{\mathcal S})$. Now fix any \(A\in\mathcal B(\mbR^{2n})\), and define 
\[\Phi_A(\bz):=\mathcal N(\bz)(A).
\]
The map \(\Phi_A\) is measurable because
\(\bz\mapsto\mathcal N(\bz)(\cdot)\) is a Gaussian transition kernel. Additionally, since
\(\mathcal T(\bs_1,\bs_2)\) is the joint Gaussian parameter corresponding to the Gelbrich coupling \(\mathcal G(\bs_1,\bs_2)\), we have
\[
\Phi_A(\mathcal T(\bs_1,\bs_2))
=
\mathcal N(\mathcal T(\bs_1,\bs_2))(A)
=
\mathcal G(\bs_1,\bs_2)(A).
\]
Since \(\lambda=\mathcal T_{\#}\pi\), the pushforward identity gives
\[
\int_{\mathcal S\times\mathcal S}
\Phi_A(\mathcal T(\bs_1,\bs_2))\,d\pi(\bs_1,\bs_2)
=
\int_{\mathcal Z_{\mathcal S}}
\Phi_A(\bz)\,d\lambda(\bz).
\]
Then for $A\in\mathcal B(\mbR^{2n})$, we obtain
\[
\begin{aligned}
\nu_\pi(A)
&=
\int_{\mathcal S\times\mathcal S}
\mathcal G(\bs_1,\bs_2)(A)\,d\pi(\bs_1,\bs_2)\\
&=
\int_{\mathcal S\times\mathcal S}
\Phi_A(\mathcal T(\bs_1,\bs_2))\,d\pi(\bs_1,\bs_2)\\
&=
\int_{\mathcal Z_{\mathcal S}}
\Phi_A(\bz)\,d\lambda(\bz)\\
&=
\int_{\mathcal Z_{\mathcal S}}
\mathcal N(\bz)(A)\,d\lambda(\bz).
\end{aligned}
\]
Since this identity holds for every \(A\in\mathcal B(\mbR^{2n})\), we have $\nu_\pi\in\mbGM_{2n}^{\mathcal S}$. Together with the conclusion from Lemma~\ref{lem: forward} that
\(\nu_\pi\in\Pi(\mu_1,\mu_2)\), this proves $\nu_\pi\in\Pi(\mu_1,\mu_2)\cap\mbGM_{2n}^{\mathcal S}.$
\end{proof}

\begin{corollary}
\label{cor:restricted-backward-containment}
Let $\nu\in\Pi(\mu_1,\mu_2)\cap\mbGM_{2n}^{\mathcal S},$
and let \(\lambda\in\mathcal P(\mathcal Z_{\mathcal S})\) be a restricted
mixing law for \(\nu\), i.e.,
\[
\nu(A)
=
\int_{\mathcal Z_{\mathcal S}}\mN(\bz)(A)\,\lambda(d\bz),
\qquad A\in\mathcal B(\mbR^{2n}).
\]
If $\pi_\nu:=\mathcal M_{\#}\lambda$, then $\pi_\nu\in\Gamma_{\mathcal S}.$
\end{corollary}

\begin{proof}
Since \(\mathcal Z_{\mathcal S}\subseteq\mathcal Z\), we can regard
\(\lambda\in\mathcal P(\mathcal Z_{\mathcal S})\) as a probability measure on
the joint-parameter space \(\mathcal Z\), supported on
\(\mathcal Z_{\mathcal S}\). By Corollary~\ref{cor:restricted-moment-preservation}, the parameter second-moment requirements stated in Proposition~\ref{prop: finite-2nd-moment} are finite under the restricted support. Hence, Lemma~\ref{lem: Gamma-mixture-to-joint-gauss} applies and shows
that the first marginal of \(\pi_\nu:=\mathcal M_{\#}\lambda\) is a mixing law
for \(\mu_1\), and the second marginal is a mixing law for \(\mu_2\).
Moreover, Corollary~\ref{cor:restricted-moment-preservation} gives
\(\pi_\nu\in\mathcal P(\mathcal S\times\mathcal S)\), and compactness of
\(\mathcal S\) gives \(\pi_\nu\in\mathcal P_2(\mathcal S\times\mathcal S)\).
Therefore, by definition, \(\pi_\nu\in\Gamma_{\mathcal S}\).
\end{proof}

\subsubsection{Objective identity under restricted parameter support. } \,
We now show the objective identity~\eqref{eqn:Wass2-Dist-ParamSpace-restricted} of Theorem~\ref{thm:Wass2-Gaussian-EquivalenceTheorem-restricted}.

\noindent\textit{\textbf{Proof of Theorem~\ref{thm:Wass2-Gaussian-EquivalenceTheorem-restricted}.}}
We prove the two inequalities.

First, let \(\pi\in\Gamma_{\mathcal S}\). By
Corollary~\ref{cor:restricted-forward-containment}, the measure
\[
\nu_\pi(A)
:=
\int_{\mathcal S\times\mathcal S}
\mathcal G(\bs_1,\bs_2)(A)\,d\pi(\bs_1,\bs_2),
\qquad A\in\mathcal B(\mathbb R^{2n}),
\]
belongs to $\Pi(\mu_1,\mu_2)\cap\mbGM_{2n}^{\mathcal S}.$
Therefore, by the definition of \(W_{2,\mathcal S}^{\mathrm{mix}}\),
\[
W_{2,\mathcal S}^{\mathrm{mix}}
\le
\int_{\mathbb R^n\times\mathbb R^n}
\|\by_1-\by_2\|^2\,d\nu_\pi(\by_1,\by_2).
\]
Since \(\|\by_1-\by_2\|^2\ge0\), the kernel Tonelli theorem gives
\[
\int_{\mathbb R^n\times\mathbb R^n}
\|\by_1-\by_2\|^2\,d\nu_\pi(\by_1,\by_2)
=
\int_{\mathcal S\times\mathcal S}
\left[
\int_{\mathbb R^n\times\mathbb R^n}
\|\by_1-\by_2\|^2\,d\mathcal G(\bs_1,\bs_2)(\by_1,\by_2)
\right]
d\pi(\bs_1,\bs_2).
\]
For each fixed \((\bs_1,\bs_2)\), \(\mathcal G(\bs_1,\bs_2)\) is the
Wasserstein-2 optimal Gaussian coupling between \(\mathcal N(\bs_1)\) and
\(\mathcal N(\bs_2)\). Hence,
\[
\int_{\mathbb R^n\times\mathbb R^n}
\|\by_1-\by_2\|^2\,d\mathcal G(\bs_1,\bs_2)(\by_1,\by_2)
=
d^2_{\mathrm{BW}}\bigl(\mathcal N(\bs_1),\mathcal N(\bs_2)\bigr).
\]
Consequently,
\[
W_{2,\mathcal S}^{\mathrm{mix}}
\le
\int_{\mathcal S\times\mathcal S}
d^2_{\mathrm{BW}}\bigl(\mathcal N(\bs_1),\mathcal N(\bs_2)\bigr)\,d\pi(\bs_1,\bs_2).
\]
Taking the infimum over \(\pi\in\Gamma_{\mathcal S}\) yields $W_{2,\mathcal S}^{\mathrm{mix}}
\le
W_{2,\mathcal S}^{\mathrm{prm}}.$

Conversely, let $\nu\in\Pi(\mu_1,\mu_2)\cap\mbGM_{2n}^{\mathcal S},$
and \(\lambda\in\mathcal P(\mathcal Z_{\mathcal S})\) be a restricted
mixing law for \(\nu\), so that 
\[\nu(A)
=
\int_{\mathcal Z_{\mathcal S}}\mathcal N(\bz)(A)\,\lambda(d\bz),
\,\,
A\in\mathcal B(\mathbb R^{2n}).
\]
By Corollary~\ref{cor:restricted-backward-containment}, $\pi_\nu:=\mathcal M_{\#}\lambda\in\Gamma_{\mathcal S}.$

For each \(\bz\in\mathcal Z_{\mathcal S}\), the Gaussian law
\(\mathcal N(\bz)\) has marginals $\mathcal N(\mathcal M_1(\bz))
\,\,\text{and}\,\,
\mathcal N(\mathcal M_2(\bz)).$
Therefore, \(\mathcal N(\bz)\) is a feasible coupling between these two
Gaussian marginals. By the optimality of the Gelbrich coupling,
\[
\int_{\mathbb R^n\times\mathbb R^n}
\|\by_1-\by_2\|^2\,d\mathcal N(\bz)(\by_1,\by_2)
\ge
d^2_{\mathrm{BW}}\bigl(
\mathcal N(\mathcal M_1(\bz)),
\mathcal N(\mathcal M_2(\bz))
\bigr).
\]
Since \(\|\by_1-\by_2\|^2\ge0\), the kernel Tonelli theorem and the restricted
mixture representation of \(\nu\) give
\[
\begin{aligned}
\int_{\mathbb R^n\times\mathbb R^n}
\|\by_1-\by_2\|^2\,d\nu(\by_1,\by_2)
&=
\int_{\mathcal Z_{\mathcal S}}
\left[
\int_{\mathbb R^n\times\mathbb R^n}
\|\by_1-\by_2\|^2\,d\mathcal N(\bz)(\by_1,\by_2)
\right]
d\lambda(\bz)\\
&\ge
\int_{\mathcal Z_{\mathcal S}}
d^2_{\mathrm{BW}}\bigl(
\mathcal N(\mathcal M_1(\bz)),
\mathcal N(\mathcal M_2(\bz))
\bigr)
\,d\lambda(\bz).
\end{aligned}
\]
Now use the pushforward identity for $\pi_\nu=\mathcal M_{\#}\lambda.$
With $\Phi(\bs_1,\bs_2)
:=
d^2_{\mathrm{BW}}\bigl(\mathcal N(\bs_1),\mathcal N(\bs_2)\bigr),$
we have
\[
\int_{\mathcal Z_{\mathcal S}}
\Phi(\mathcal M(\bz))\,d\lambda(\bz)
=
\int_{\mathcal S\times\mathcal S}
\Phi(\bs_1,\bs_2)\,d\pi_\nu(\bs_1,\bs_2).
\]
Equivalently,
\[
\int_{\mathcal Z_{\mathcal S}}
d^2_{\mathrm{BW}}\bigl(
\mathcal N(\mathcal M_1(\bz)),
\mathcal N(\mathcal M_2(\bz))
\bigr)
\,d\lambda(\bz)
=
\int_{\mathcal S\times\mathcal S}
d^2_{\mathrm{BW}}\bigl(\mathcal N(\bs_1),\mathcal N(\bs_2)\bigr)
\,d\pi_\nu(\bs_1,\bs_2).
\]
Therefore,
\[
\int_{\mathbb R^n\times\mathbb R^n}
\|\by_1-\by_2\|^2\,d\nu(\by_1,\by_2)
\ge
\int_{\mathcal S\times\mathcal S}
d^2_{\mathrm{BW}}\bigl(\mathcal N(\bs_1),\mathcal N(\bs_2)\bigr)
\,d\pi_\nu(\bs_1,\bs_2).
\]
Since \(\pi_\nu\in\Gamma_{\mathcal S}\), the right-hand side is bounded below by
\(W_{2,\mathcal S}^{\mathrm{prm}}\). Hence
\[
\int_{\mathbb R^n\times\mathbb R^n}
\|\by_1-\by_2\|^2\,d\nu(\by_1,\by_2)
\ge
W_{2,\mathcal S}^{\mathrm{prm}}.
\]
Taking the infimum over
\(\nu\in\Pi(\mu_1,\mu_2)\cap\mbGM_{2n}^{\mathcal S}\) yields $W_{2,\mathcal S}^{\mathrm{mix}}
\ge
W_{2,\mathcal S}^{\mathrm{prm}}.$
Combining the two inequalities proves identity~\eqref{eqn:Wass2-Dist-ParamSpace-restricted} \hfill\qed

\subsection{\texorpdfstring{Properties of Burer-Wasserstein Distance $d^2_{\mathrm{BW}}$ and Duality}{Properties of Burer-Wasserstein Distance and Duality}}\label{appndx:proof-of-lem: Delta-Continuity}

\subsubsection{Proof of Lemma~\ref{lem: Delta-Continuity}. } 
Fix \((\hat{\bmean},\hat{\bQ})\in\Theta\times\Xi\). For
\(\bQ\in\mbS^n_+\), define
\begin{align}\label{eq:def-F-trace}
    F(\bQ)
:=
\Tr(\hat{\bQ})+\Tr(\bQ)
-
2\Tr\!\left((\bQ^{1/2}\hat{\bQ}\bQ^{1/2})^{1/2}\right).
\end{align}
Hence,$ d^2_{\mathrm{BW}}\bigl(\bmean,\bQ;(\hat{\bmean},\hat{\bQ})\bigr)
=
\|\bmean-\hat{\bmean}\|_2^2+F(\bQ)$. We prove the two claims in the order stated: first, convexity on
\(\mbR^n\times\mbS^n_+\), and then continuity on
\(\Theta\times\Xi\).

\medskip
\noindent\textbf{Convexity proof.} Since the map \(\bmean\mapsto \|\bmean-\hat{\bmean}\|_2^2\) is convex on
\(\mbR^n\), it remains to prove that \(F\) is convex on
\(\mbS^n_+\). We first justify the following trace identity
\begin{equation}\label{eq:BW-trace-switch}
\Tr\!\left((\bQ^{1/2}\hat{\bQ}\bQ^{1/2})^{1/2}\right)
=
\Tr\!\left((\hat{\bQ}^{1/2}\bQ\hat{\bQ}^{1/2})^{1/2}\right).
\end{equation}
Let $\bR:=\bQ^{1/2}\hat{\bQ}^{1/2}$. Then,\vspace{-2em}
\begin{align}\label{eq:trace-identity}
\quad \bR\bR^\top=\bQ^{1/2}\hat{\bQ}\bQ^{1/2},
\qquad
\bR^\top \bR=\hat{\bQ}^{1/2}\bQ\hat{\bQ}^{1/2}
\end{align}
Both \(\bR\bR^\top\) and \(\bR^\top \bR\) are symmetric positive semidefinite. We show that they have the same characteristic polynomial. For every \(t\neq0\),
Sylvester's determinant identity~\cite{horn2012matrix} gives
\[
\begin{aligned}
\det(t \bI-\bR\bR^\top) \,\,=\,\,t^n\det\!\left(\bI-\frac{1}{t}\bR\bR^\top\right) \,\,=\,\, t^n\det\!\left(\bI-\frac{1}{t}\bR^\top \bR\right) \,\,=\,\,\det(t \bI-\bR^\top \bR).
\end{aligned}
\]
The two sides are polynomials in \(t\). Since they agree for every \(t\neq0\),
they agree for every \(t\in\mbR\). Hence,
\[
\det(t \bI-\bR\bR^\top)=\det(t \bI-\bR^\top \bR)
\qquad \forall t\in\mbR.
\]
Therefore, \(\bR\bR^\top\) and \(\bR^\top \bR\) have the same characteristic polynomial.

Since \(\bR\bR^\top\) and \(\bR^\top \bR\) are symmetric positive semidefinite, the spectral theorem~\cite{horn2012matrix} gives orthogonal matrices \(U,V\) and scalars
\(\lambda_1,\ldots,\lambda_n\ge0\) such that
\[
\bR\bR^\top
=
U\operatorname{diag}(\lambda_1,\ldots,\lambda_n)U^\top,
\qquad
\bR^\top \bR
=
V\operatorname{diag}(\lambda_1,\ldots,\lambda_n)V^\top.
\]
By the definition of the principal symmetric positive semidefinite square root,
\[
(\bR\bR^\top)^{1/2}
=
U\operatorname{diag}(\sqrt{\lambda_1},\ldots,\sqrt{\lambda_n})U^\top, \quad \text{and} \quad (\bR^\top \bR)^{1/2}
=
V\operatorname{diag}(\sqrt{\lambda_1},\ldots,\sqrt{\lambda_n})V^\top.
\]
Using the invariance of trace under orthogonal similarity, $\Tr\!\left((\bR\bR^\top)^{1/2}\right)
=
\sum_{i=1}^n\sqrt{\lambda_i}
=
\Tr\!\left((\bR^\top \bR)^{1/2}\right).$
Substitution with~\eqref{eq:trace-identity} gives \eqref{eq:BW-trace-switch}.

We next prove that  $H(\bQ)
:=
\Tr\!\left((\hat{\bQ}^{1/2}\bQ\hat{\bQ}^{1/2})^{1/2}\right)$ is concave on \(\mbS^n_+\).  
Let \(\bQ^{(1)},\bQ^{(2)}\in\mbS^n_+\), \(\alpha\in[0,1]\) and 
\begin{align} \label{eq:linear-comb}
   \bQ^{(\alpha)}
:=
(1-\alpha)\bQ^{(1)}+\alpha\bQ^{(2)}.
\end{align}
Since the map $\mathcal L:\mbS^n\to\mbS^n,
\,\,
\mathcal L(\bQ):=\hat{\bQ}^{1/2}\bQ\hat{\bQ}^{1/2},$ is linear, for any \(\bQ^{(1)},\bQ^{(2)}\in\mbS^n\) and any
\(a,b\in\mbR\),
\[
\begin{aligned}
\mathcal L(a\bQ^{(1)}+b\bQ^{(2)})
&=
\hat{\bQ}^{1/2}(a\bQ^{(1)}+b\bQ^{(2)})\hat{\bQ}^{1/2}=
a\,\hat{\bQ}^{1/2}\bQ^{(1)}\hat{\bQ}^{1/2}
+
b\,\hat{\bQ}^{1/2}\bQ^{(2)}\hat{\bQ}^{1/2}=
a\,\mathcal L(\bQ^{(1)})+b\,\mathcal L(\bQ^{(2)}).
\end{aligned}
\]
Therefore, following \eqref{eq:linear-comb} we obtain $\hat{\bQ}^{1/2}\bQ^{(\alpha)}\hat{\bQ}^{1/2}
=
(1-\alpha)\hat{\bQ}^{1/2}\bQ^{(1)}\hat{\bQ}^{1/2}
+
\alpha\hat{\bQ}^{1/2}\bQ^{(2)}\hat{\bQ}^{1/2}.$ Additionally, since the principal square-root map is operator concave on \(\mbS^n_+\)~\cite[Theorem~4.2.3]{bhatia2009positive},
\[
\begin{aligned}
&\left(\hat{\bQ}^{1/2}\bQ^{(\alpha)}\hat{\bQ}^{1/2}\right)^{1/2} \,\, \succeq
(1-\alpha)
\left(\hat{\bQ}^{1/2}\bQ^{(1)}\hat{\bQ}^{1/2}\right)^{1/2}
+
\alpha
\left(\hat{\bQ}^{1/2}\bQ^{(2)}\hat{\bQ}^{1/2}\right)^{1/2}.
\end{aligned}
\]
Taking traces preserves the Löwner order~\cite{bhatia2009positive}: if \(\bA\succeq\bB\), then
\(\Tr(\bA)\ge\Tr(\bB)\). Hence, $H(\bQ^{(\alpha)})
\ge
(1-\alpha)H(\bQ^{(1)})+\alpha H(\bQ^{(2)}).$ Thus \(H\) is concave on \(\mbS^n_+\). 

Finally, we show the convexity of \(F\) on \(\mbS^n_+\). Using \eqref{eq:BW-trace-switch}, \eqref{eq:def-F-trace} gives us $F(\bQ)=\Tr(\hat{\bQ})+\Tr(\bQ)-2H(\bQ),$
where \(\Tr(\hat{\bQ})\) is constant in \(\bQ\) and \(\Tr(\bQ)\) is linear in
\(\bQ\) since $\Tr(\bQ^{(\alpha)})
=
(1-\alpha)\Tr(\bQ^{(1)})+\alpha\Tr(\bQ^{(2)})$.
Also, concavity of \(H\) implies \(-2H(\bQ)\) is convex. Hence, $F(\bQ^{(\alpha)})
\le
(1-\alpha)F(\bQ^{(1)})+\alpha F(\bQ^{(2)}),$
i.e., \(F\) is convex on \(\mbS^n_+\). Combining this with the convexity of
\(\bmean\mapsto\|\bmean-\hat{\bmean}\|_2^2\), the map
\begin{align}\label{eq:d-BW-map}
    (\bmean,\bQ)\mapsto
d^2_{\mathrm{BW}}\bigl(\bmean,\bQ;(\hat{\bmean},\hat{\bQ})\bigr)
\end{align}
is convex on \(\mbR^n\times\mbS^n_+\).

\noindent\textbf{Continuity proof.}\, Let $\mathcal U:=\mbR^n\times\mbS^n_{++}$. The set \(\mathcal U\) is open and convex in the finite-dimensional vector space
\(\mbR^n\times\mbS^n\). Since $\mbS^n_{\hat\delta}
=
\{\bQ\in\mbS^n_{++}:\bQ\succeq \hat\delta I\}$, we have $\Theta\times\Xi\subseteq \mathcal U$. From the convexity result above, the map~\eqref{eq:d-BW-map} is convex on \(\mbR^n\times\mbS^n_+\). Hence, its restriction to
\(\mathcal U\) is a finite convex function on the open convex set \(\mathcal U\).
By the standard continuity theorem for finite convex functions on open convex
sets \cite[Corollary~2.36]{rockafellar2009variational}, this restriction is
continuous on \(\mathcal U\). Since \(\Theta\times\Xi\subseteq\mathcal U\), the further restriction to
\(\Theta\times\Xi\) is continuous. Hence, $d^2_{\mathrm{BW}}\bigl(\bmean,\bQ;(\hat{\bmean},\hat{\bQ})\bigr)$ is continuous on \(\Theta\times\Xi\).

\subsubsection{Convexity on restricted domains. }
\begin{corollary}[Convexity on restricted domains]
\label{cor:Delta-convex-restricted-domains}
Fix \((\hat{\bmean},\hat{\bQ})\in\Theta\times\Xi\), and define
\[
d^2_{\mathrm{BW}}\bigl(\bmean,\bQ;(\hat{\bmean},\hat{\bQ})\bigr)
=
\|\hat{\bmean}-\bmean\|_2^2
+
\Tr\!\left(
\hat{\bQ}+\bQ
-2(\bQ^{1/2}\hat{\bQ}\bQ^{1/2})^{1/2}
\right).
\]
Then the following holds.
\begin{enumerate}[label=\textit{\roman*)},leftmargin=1.5em]
\item The map $(\bmean,\bQ)\mapsto
d^2_{\mathrm{BW}}\bigl(\bmean,\bQ;(\hat{\bmean},\hat{\bQ})\bigr)$ is convex on \(\mbR^n_+\times\mbS^n_{++}\).

\item If \(\Theta\times\Xi\) is convex, then the same map is convex on
\(\Theta\times\Xi\).
\end{enumerate}
\end{corollary}

\begin{proof}
By Lemma~\ref{lem: Delta-Continuity}, the map is convex on
\(\mbR^n\times\mbS^n_+\).

We first show that \(\mbR^n_+\times\mbS^n_{++}\) is convex. Let $(\bmean^{(1)},\bQ^{(1)}),(\bmean^{(2)},\bQ^{(2)})
\in
\mbR^n_+\times\mbS^n_{++},
\,\,
\alpha\in[0,1].$
Since \(\mbR^n_+\) is convex, $(1-\alpha)\bmean^{(1)}+\alpha\bmean^{(2)}\in\mbR^n_+.$
Moreover, for every \(\bx\neq0\),
\[
x^\top\bigl((1-\alpha)\bQ^{(1)}+\alpha\bQ^{(2)}\bigr)x
=
(1-\alpha)x^\top\bQ^{(1)}x+\alpha x^\top\bQ^{(2)}x>0.
\]
Thus, $(1-\alpha)\bQ^{(1)}+\alpha\bQ^{(2)}
\in
\mbS^n_{++}.$ Therefore, \(\mbR^n_+\times\mbS^n_{++}\) is a convex subset of
\(\mbR^n\times\mbS^n_+\). The restriction of a convex function to a
convex subset of its domain is convex, so the first claim follows.

For the second claim, suppose \(\Theta\times\Xi\) is convex. Then it is a convex
subset of \(\mbR^n\times\mbS^n_+\). Therefore, the restriction of the
convex function from Lemma~\ref{lem: Delta-Continuity} to \(\Theta\times\Xi\) is
convex. Finally, if \(\Theta\) and \(\Xi\) are convex, then
\(\Theta\times\Xi\) is convex because convex combinations are taken componentwise.
\end{proof}

\begin{remark}[A convex compact specification of \(\Theta\times\Xi\)]
\label{rem:convex-compact-support-specification}
A way to specify the compact mean--covariance support is to take $\Theta
:=
\left\{
\bmean\in\mbR^n:
|\bmean_i-\hat{\bmean}_i|
\le
\varsigma |\hat{\bmean}_i|,
\quad i=1,\ldots,n
\right\},$
for some \(\varsigma\ge0\), and $\Xi
:=
\left\{
\bQ\in\mbS^n:
\ubar \lambda_f\,\hat{\bQ}
\preceq
\bQ
\preceq
\overline \lambda_f\,\hat{\bQ}
\right\},$
where \(\hat{\bQ}\succ0\) and
\(0<\ubar \lambda_f\le\overline \lambda_f\). Under this specification,
\(\Theta\) is a Cartesian product of closed intervals and is therefore compact
and convex. The set \(\Xi\) is also compact and convex. Indeed, if
\(\bQ^{(1)},\bQ^{(2)}\in\Xi\) and \(\alpha\in[0,1]\), then $\ubar \lambda_f\,\hat{\bQ}
\preceq
(1-\alpha)\bQ^{(1)}+\alpha\bQ^{(2)}
\preceq
\overline \lambda_f\,\hat{\bQ},$
so \((1-\alpha)\bQ^{(1)}+\alpha\bQ^{(2)}\in\Xi\). Therefore,
\(\Theta\times\Xi\) is compact and convex. Moreover, with $\hat\delta:=\ubar \lambda_f\,\lambda_{\min}(\hat{\bQ})>0$, $\bQ\succeq \ubar \lambda_f\,\hat{\bQ}
\succeq
\ubar \lambda_f\,\lambda_{\min}(\hat{\bQ})I.$ Hence, \(\Xi\subseteq\mbS^n_{\hat\delta}\).
\end{remark}

\subsubsection{Proof of Theorem~\ref{thm: strong-duality}.}\, 
Fix any \(\bx\in\mX\).  Since \(\mD\neq\emptyset\), the robust constraint in~\eqref{prob: DR-CCP-GMM}, $\sum_{k=1}^{\hat K}
\int_{\mS}G(\bs,\bx)\,d\pi(\hat{\bs}_k,\bs)
\ge\theta,\,
\forall \pi\in\mD.$ is equivalent to $\inf_{\pi\in\mD}
\sum_{k=1}^{\hat K}
\int_{\mS}G(\bs,\bx)\,d\pi(\hat{\bs}_k,\bs)
\ge\theta$.
By  Lemma~\ref{lem: Delta-Continuity}, for any fixed outer decision \(\bx\), the strong duality result for this inner problem  provides the following objective-value identity
\begin{align}\label{eq: strong-dual-eq}
&
\inf_{\pi\in\mD}
\sum_{k=1}^{\hat K}
\int_{\mS}G(\bs,\bx)\,d\pi(\hat{\bs}_k,\bs)
\\
&=\hspace{-1em}
\sup_{\bbeta\in\mathbb R^{\hat K},\;\beta_{\hat{K}+1}\ge0}
\Big\{
\sum_{k=1}^{\hat K}\hat w_k\beta_k-\rho\beta_{\hat{K}+1}:
\beta_k-\beta_{\hat{K}+1}
d_{\mathrm{BW}}^2
\left(
\mN(\hat{\bs}_k),\mN(\bs)
\right)
\le
G(\bs,\bx),\
\forall \bs\in\mS,\ \forall k\in[\hat K]
\Big\}. \nonumber
\end{align}


We next show that the supremum on the right-hand side of~\eqref{eq: strong-dual-eq} is attained. This
attainment is needed because the semi-infinite formulation uses an actual
decision vector \(\bbeta\), not a supremum operator. The point
\((\bbeta,\beta_{\hat{K}+1})=(\bzero,0)\) is feasible because \(G(\bs,\bx)\ge0\) on
\(\mS\), so the supremum is at least zero. No component \(\beta_k\) needs to be
negative: if \(\beta_k<0\), replacing it by \(0\) preserves feasibility, since
\(0-\beta_{\hat{K}+1} d_{\mathrm{BW}}^2(\mN(\hat{\bs}_k),\mN(\bs))\le0\le G(\bs,\bx)\),
and weakly improves the objective. Hence a maximizing sequence may be chosen
with \(\beta_k\ge0\) for every \(k\). Evaluating the dual constraint at
\(\bs=\hat{\bs}_k\) gives \(\beta_k\le G(\hat{\bs}_k,\bx)\le1\), since
\(d_{\mathrm{BW}}^2(\mN(\hat{\bs}_k),\mN(\hat{\bs}_k))=0\). Thus
\(0\le\beta_k\le1\) for all \(k\). Along any maximizing sequence with
nonnegative objective value, $0\le
\sum_{k=1}^{\hat K}\hat w_k\beta_k-\rho\beta_{\hat{K}+1}
\le
1-\rho\beta_{\hat{K}+1},$
where \(\hat{\bw}\in\Delta_{\hat K}
:=
\left\{
\hat{\bw}\in\mathbb R_+^{\hat K}:
\sum_{k=1}^{\hat K}\hat w_k=1
\right\}\)
and \(\beta_k\le1\) were used. Since
\(\rho>0\), we obtain \(0\le\beta_{\hat{K}+1}\le1/\rho\). Therefore the supremum can be
taken over the compact set \([0,1]^{\hat K}\times[0,1/\rho]\). The feasible
set is closed because the constraint functions are continuous in \(\bs\) on
compact \(\mS\), and the dual objective is continuous. Hence the supremum is a
maximum.

We now prove the feasibility equivalence. Suppose first that \(\bx\) satisfies
the robust constraint in~\eqref{prob: DR-CCP-GMM}. Then $\inf_{\pi\in\mD}
\sum_{k=1}^{\hat K}
\int_{\mS}G(\bs,\bx)\,d\pi(\hat{\bs}_k,\bs)
\ge\theta$.
By the objective-value equality above and the attained maximum on the dual
side, there exist \(\bbeta\in\mathbb R^{\hat K}\) and \(\beta_{\hat{K}+1}\ge0\) such
that $\sum_{k=1}^{\hat K}\hat w_k\beta_k-\rho\beta_{\hat{K}+1}\ge\theta$ 
and $\beta_k-\beta_{\hat{K}+1}
d_{\mathrm{BW}}^2
\left(
\mN(\hat{\bs}_k),\mN(\bs)
\right)
\le
G(\bs,\bx),
\,\,
\forall \bs\in\mS,\ \forall k\in[\hat K]$. 
Hence we obtain
\(\bbeta\in\mathbb R^{\hat K}\times\mathbb R_+\) satisfying the constraints of
\eqref{prob: semif-infinite}.

Conversely, suppose there exists
\(\bbeta\in\mathbb R^{\hat K}\times\mathbb R_+\) satisfying the constraints of
\eqref{prob: semif-infinite}. Then
\((\beta_1,\ldots,\beta_{\hat K},\beta_{\hat{K}+1})\) is feasible for the dual problem
above, and its objective value is at least \(\theta\). Therefore, the dual
supremum is at least \(\theta\). By the objective-value equality,
\(
\inf_{\pi\in\mD}
\sum_{k=1}^{\hat K}
\int_{\mS}G(\bs,\bx)\,d\pi(\hat{\bs}_k,\bs)
\ge\theta\).
This is exactly the robust constraint in~\eqref{prob: DR-CCP-GMM}. Hence, for
every \(\bx\in\mX\), \(\bx\) is feasible for~\eqref{prob: DR-CCP-GMM} if and
only if there exists \(\bbeta\in\mathbb R^{\hat K}\times\mathbb R_+\) such
that \((\bx,\bbeta)\) is feasible for~\eqref{prob: semif-infinite}.

Since both formulations minimize the same objective \(\bc^\top\bx\), the
feasibility equivalence gives equality of the optimal objective values. It also
gives the stated optimal-solution correspondence. If \(\bx^\star\) is optimal
for~\eqref{prob: DR-CCP-GMM}, then the feasibility equivalence gives some
\(\bbeta^\star\) such that \((\bx^\star,\bbeta^\star)\) is feasible for
\eqref{prob: semif-infinite}. If this pair were not optimal for
\eqref{prob: semif-infinite}, then some feasible
\((\bar{\bx},\bar{\bbeta})\) of~\eqref{prob: semif-infinite} would satisfy
\(\bc^\top\bar{\bx}<\bc^\top\bx^\star\). The feasibility equivalence would make
\(\bar{\bx}\) feasible for~\eqref{prob: DR-CCP-GMM}, contradicting optimality
of \(\bx^\star\). Thus \((\bx^\star,\bbeta^\star)\) is optimal for
\eqref{prob: semif-infinite}. Conversely, if
\((\bx^\star,\bbeta^\star)\) is optimal for~\eqref{prob: semif-infinite}, then
\(\bx^\star\) is feasible for~\eqref{prob: DR-CCP-GMM}. If \(\bx^\star\) were
not optimal for~\eqref{prob: DR-CCP-GMM}, there would exist a feasible
\(\bar{\bx}\) of~\eqref{prob: DR-CCP-GMM} with
\(\bc^\top\bar{\bx}<\bc^\top\bx^\star\). The feasibility equivalence would give
some \(\bar{\bbeta}\) such that \((\bar{\bx},\bar{\bbeta})\) is feasible for
\eqref{prob: semif-infinite}, contradicting optimality of
\((\bx^\star,\bbeta^\star)\). Therefore the optimal \(\bx\)-solutions coincide
whenever an optimal solution exists. \hfill \qedsymbol

\begin{remark}
    The assumption \(\rho>0\) is used only to guarantee attainment of the dual
supremum in the semi-infinite reformulation. It is not needed for finiteness of
the primal inner value, since \(0\le G(\bs,\bx)\le1\) and \(\mD\neq\emptyset\).
When \(\rho>0\), the term \(-\rho\beta_{\hat K+1}\) bounds
\(\beta_{\hat K+1}\) along any maximizing sequence, which yields compactness of
the relevant dual level set. If \(\rho=0\), this coercive term disappears; the
dual supremum may remain finite but may fail to be attained because
\(\beta_{\hat K+1}\) can diverge while relaxing all positive-distance
constraints. Thus, without an additional dual-attainment argument, the value
equality still supports an optimal-value statement, but the existential
semi-infinite reformulation with a finite \(\bbeta\) requires \(\rho>0\).
\end{remark}
\subsection{Reformulated DR-Chance Constrained Program}\label{appndx:cut-alg-proof}

\subsubsection{Proof of Proposition~\ref{prop: infeasibility}. }
    By definition~\eqref{def: G}, $G(\bs, \bx) = 0$. As a result, since $\theta \in (0, 1]$, constraint~\eqref{prob: robust-CC} always fails to satisfy, i.e.,  
    $\min_{\pi \in \mP} \;\sum_{k=1}^{\hat{K}} \int_{\bs \in (\Theta \times \Xi)} 0 \; d\pi \bigl(\hat{\bs}_k, \bs\bigr) = 0 \ngeq \theta$.
    
    Now, for $\bx = \bzero$ with $b \geq 0$, $G(\bs, \bx) = 1$, which translates the worst-case distribution finding problem~\eqref{prob: robust-CC} to a feasibility problem. In particular, any value of $\pi \in \Pi(\hat{\bw}, \bw)$ with some $\bw \in \mathcal P_2(\mS)$  yields $\min_{\pi \in \mP} \;\sum_{k=1}^{\hat{K}} \int_{\bs \in (\Theta \times \Xi)} \; d\pi \bigl(\hat{\bs}_k, \bs\bigr) = 1 > \theta$. Therefore, since $\bw = \hat{\bw} \in \mathcal P_2(\mS)$, $\pi \in \Pi(\hat{\bw}, \hat{\bw})$ is also a solution to the problem~\eqref{prob: robust-CC} yielding same objective values. \hfill \qed

\subsubsection{Proof of Proposition~\ref{prop: infeasibility-1}. }
    If $\|\bbeta\| = 0$, i.e., $\beta^k = 0 \,\, \forall k \in [\hat{K}]$, $\beta_{K+1} \geq 0$ implies that $\sum_{k=1}^{\hat{K}} \hat{w}_k \beta_k - \beta_{\hat{K}+1} \; \rho \leq 0 \ngeq \theta$; a contradiction since $\theta \in (0, 1]$. Hence, $\|\bbeta\| > 0$. 
    
    We also prove \emph{part-(ii)} by contradiction. When $b <0$, for $\bx = 0$, $G(\bs; \bx) = \mathbbm{1}_{\geq 0}(b) = 0$ by definition~\eqref{def: G}. Then for $\hat{\bbeta}\in\mbR^{\hat{K}}\times\mathbb R_+$ be feasible it must satisfy all the semi-infinite constraints:
\[
\beta_k-\beta_{\hat{K}+1}d^2_{\mathrm{BW}}\bigl(\mN(\hbs_k),\mN(\bs)\bigr) - G(s;\bx) \leq 0,
\qquad \forall s\in\Theta\times\Xi,\ \forall k\in[\hat{K}].
\]
Hence, nominal distribution parameter also belongs to the set $\Theta \times \Xi$, i.e., $\bs_k=\hat{\bs}_k \in \Theta \times \Xi$ which leads $d^2_{\mathrm{BW}}\big( \mN(\bs_k), \mN(\hat{\bs}_k)) \big) = 0, \; \forall k \in [\hat{K}]$. To satisfy such a constraint, there must be a solution $\hat{\bbeta}$ such that  $\hat{\beta}_k \leq 0, \; \forall k \in [\hat{K}]$. However, since $\beta_{\hat{K}+1} \in \mbR_+$, we obtain $\sum_{k=1}^{\hat{K}} \hat{w}_k \hat{\beta}_k - \hat{\beta}_{\hat{K}+1} \; \rho \leq 0 \ngeq \theta \in (0, 1]$, which contradicts the feasibility of $\hat{\bbeta}$. 

For \emph{part-iii}, since $ b \geq 0$, given $\hat{\bbeta} \in \mbR^{\hat{K}} \times \mbR_+$ and $\bx = \bzero$, $G(\bs; \bx) = G(\bs; \bzero) = 1$. Hence, any element $\bs \in \Theta \times \Xi$ can be a solution to $\max_{\bs} \hat{\beta}_k - \hat{\beta}_{\hat{K}+1} d^2_{\mathrm{BW}} \big(\mN(\hat{\bs}_k), \mN(\bs) \big) -  1 $, so does the solution $\bs_k = \hat{\bs}_k $ for any $ k \in [\hat{K}]$. \hfill \qed

\subsubsection{Proof of Proposition~\ref{prop: master-prob-reform}. }
First, consider the forward direction of the claim starting from $\bar{\Phi} \left(z ; \bzz\right)  \geq \zeta$. Let us take some $z\geq 0, \; \zeta$ such that $\bar{\Phi}(z; \bzz) \geq \zeta$. Then, $z = y_3 \geq 0$ with $t_3 = 1$ implies that $t_1, t_2 = 0, \, g_i z + g_i^0 \geq \zeta, \; i \in [R]_0$ and $1 \geq \zeta$. Thus, for a given solution $z \geq 0$ satisfying $\bar{\Phi}(z; \bzz) \geq \zeta$, one can construct a feasible solution to the set $\mH^O(\bzz)$ using $t_1, t_2 = 0$, $t_3 = 1$, $\balpha = \bzero$, $y_1, y_2 = 0, z = y_3 \geq 0$. 

Next we consider some $z \in [\check{z}_{-L},  0)$. Then, $z$ will be in $[\check{z}_{-i}, \check{z}_{-i+1}]$ for some $i \in [L]$ which indicates $t_2 = 1, \, t_1 = t_3 = 0$, and therefore, $y_1 = y_3 = 0$. Then there exists some weights $\hat{\alpha}_i, \hat{\alpha}_{i-1} \geq 0$ such that $\hat{\alpha}_i + \hat{\alpha}_{i-1} = 1 = t_2$. Furthermore, from constraint $y_2 = \sum_{i=0}^L \alpha_i \zz_{-i}$ we obtain $y_2 = \hat{\alpha}_i \check{z}_{-i} +  \hat{\alpha}_{i-1} \check{z}_{-i+1}$. Hence, $ \alpha_{i} \Phi(\zz_{-i}) + \alpha_{i-1} \Phi(\zz_{-i+1}) \geq \zeta$. Thus for any given $z \in [\check{z}_{-L},  0)$, one can find an index $i \in [L]$ such that $ \alpha_i = \hat{\alpha}_i, \, \alpha_{i-1} = \hat{\alpha}_{i-1}$ with $y_2 = \hat{\alpha}_i \check{z}_{-i} +  \hat{\alpha}_{i-1} \check{z}_{-i+1}, \, y_1 = y_3 = 0$, and  $\alpha_j = 0 \;\forall j \in [L] \backslash \{i-1, i\}$, which satisfy all the constraints of $\mH^O(\bzz)$. Similarly, if we take some $z < \bzz_{-L}$, then $t_1 = 1,$ and $y_1 = z$. This implies $t_2, t_3, y_2, y_3 = 0, \Phi(\check{z}_{-L}) \geq  \zeta$. Therefore, $\mH^O(\bzz) \neq \emptyset$.

\vspace{1em}
Now we prove that any point in the set $\mH^O(\bzz)$ also satisfies $\bar{\Phi} \left(z ; \bzz\right)  \geq \zeta$ based on the definition of  $\bar{\Phi} \left(z ; \bzz\right)$ in~\eqref{def: phi_pwl_outer}. It is formally shown below using three separate cases:

\begin{enumerate}
    \item[i)] Let us choose $ y_3 = z \geq 0$ with $t_3 = 1$ and $ \, t_1, t_2, y_1, y_2 = 0$ such that $ g_i y_3 + g_i^0 \geq \zeta$ for all $i \in [R]_0$. Then  $\bar{\Phi} \left(z ; \, \bzz\right) = \min_{i \in [R]_0} \{ g_i y_3 + g_i^0\} \geq \zeta$.
    \item[ii)] Let $\zz_{-L} \leq y_2 = z < 0$ with $ t_2=1$ and  $t_1, t_3, y_1, y_3 = 0$ such that $y_2 = \alpha_{\check{i}} \zz_{-\check{i}} + \alpha_{\check{i} - 1} \zz_{-\check{i} + 1}, \quad$  $\alpha_{\check{i}} + \alpha_{\check{i} - 1} = 1 = t_2, $ and $ \alpha_{\check{i}} \Phi(\zz_{-\check{i}}) + \alpha_{\check{i} - 1}  \Phi(\zz_{-\check{i} + 1}) \geq \zeta $ for some $\check{i} \in [L]$ and $\alpha_{j} = 0$ for all $j \in [L] \backslash \{\check{i} - 1, \check{i} \}$. Thus there exists an index $\check{i} \in [L]$ that satisfies $\bar{\Phi} \left(z ; \bzz\right) =  h_{\check{i}} y_2 + h_{\check{i}}^0  \geq \zeta$. Hence, $\bar{\Phi} \left(z ; \bzz\right) =  \max_{i \in [L]}h_{i} y_2 + h_{i}^0  \geq \zeta$ is satisfied.
    \item[iii)] Finally, $ y_1 = z < \check{z}_{-L}$ with $ t_1 = 1, \, t_2, t_3, y_2, y_3 = 0$ indicates that $\Phi(\zz_{-L}) = \bar{\Phi} \left(z ; \bzz\right) \geq \zeta$. 
\end{enumerate}
Additionally, since all of these cases allow $\zeta$ to be at most 1, using the definition of $\bar{\Phi}(z; \bzz)$ in~\eqref{def: phi_pwl_outer}, $\bar{\Phi}(z; \bzz)  \geq \zeta$ is satisfied. \hfill\qedsymbol

\subsubsection{Proof of Proposition~\ref{prop: sub-prob-reform}.}\, First, consider the forward direction of the claim starting from $-\underline{\Phi} \left(z ; \bzz\right)  \geq \zeta'$. Let us take some $z\leq 0, \; \zeta'$ such that $-\underline{\Phi}(z; \bzz) \geq \zeta'$. Then, $z = y_1 \leq 0$ with $t_1 = 1$ implies that $t_2, t_3 = 0, \, -h_i z - h_i^0 \geq \zeta', \; i \in [L]_0$ and $\zeta' \leq 0$. Thus, for a given solution $z \leq 0$ satisfying $-\underline{\Phi}(z; \bzz) \geq \zeta'$, one can construct a feasible solution to the set $\mH^I(\bzz)$ using $t_3, t_2 = 0$, $t_1 = 1$, $\balpha = \bzero$, $y_2, y_3 = 0, z = y_1 \leq 0$. 

Next we consider some $z \in (0, \check{z}_{R}]$. Then, $z$ will be in $[\check{z}_{i-1}, \check{z}_{i}]$ for some $i \in [R]$ which indicates $t_2 = 1, \, t_1 = t_3 = 0$, and therefore, $y_1 = y_3 = 0$. Then there exists some weights $\hat{\alpha}_{i-1}, \hat{\alpha}_i \geq 0$ such that $\hat{\alpha}_{i-1} + \hat{\alpha}_i = 1 = t_2$. Furthermore, from constraint $y_2 = \sum_{i=0}^R \alpha_i \zz_{i}$ we obtain $y_2 = \hat{\alpha}_{i-1} \check{z}_{i-1} + \hat{\alpha}_i \check{z}_{i}$. 
Hence, $  -\alpha_{i-1} \Phi(\zz_{i-1}) - \alpha_{i} \Phi(\zz_{i}) \geq \zeta'$. Thus for any given $z \in (0, \check{z}_{R}]$, one can find an index $i \in [R]$ such that $\alpha_{i-1} = \hat{\alpha}_{i-1}, \;\; \alpha_i = \hat{\alpha}_i$ with $y_2 =  \hat{\alpha}_{i-1} \check{z}_{i-1} + \hat{\alpha}_i \check{z}_{i}, \, y_1 = y_3 = 0$, and  $\alpha_j = 0 \;\forall j \in [R_0] \backslash \{i-1, i\}$, which satisfy all the constraints of $\mH^I(\bzz)$. Similarly, if we take some $z > \bzz_{R}$, then $t_3 = 1$, and $y_3 = z$. This implies $t_1, t_2, y_1, y_2 = 0, -\Phi(\check{z}_{R}) \geq  \zeta'$. Therefore, $\mH^I(\bzz) \neq \emptyset$.

\vspace{1em}
Now we prove that any point in the set $\mH^I(\bzz)$ also satisfies $-\underline{\Phi} \left(z ; \bzz\right)  \geq \zeta'$ based on the definition of  $\underline{\Phi} \left(z ; \bzz\right)$ in~\eqref{def: phi_pwl_inner}. It is formally shown below using three separate cases:

\begin{itemize}
    \item[i)] Let us choose $ y_1 = z \leq 0$ with $t_1 = 1$ and $ \, t_2, t_3, y_2, y_3 = 0$ such that $ -h_i y_1 - h_i^0 \geq \zeta'$ for all $i \in [L]_0$. Then  $-\underline{\Phi} \left(z ; \, \bzz\right) = \min_{i \in [L]_0} \{ -h_i y_1 - h_i^0\} \geq \zeta'$.
    \item[ii)] Let $0 < y_2 = z \leq \zz_{R}$ with $ t_2=1$ and  $t_1, t_3, y_1, y_3 = 0$ such that $y_2 = \alpha_{\check{i}-1} \zz_{\check{i}-1} + \alpha_{\check{i}} \zz_{\check{i}}, \quad$  $\alpha_{\check{i} - 1} + \alpha_{\check{i}} = 1 = t_2, $ and $-\alpha_{\check{i} - 1}  \Phi(\zz_{\check{i} - 1}) - \alpha_{\check{i}} \Phi(\zz_{\check{i}})  \geq \zeta' $ for some $\check{i} \in [R]$ and $\alpha_{j} = 0$ for all $j \in [R_0] \backslash \{\check{i} - 1, \check{i} \}$. Thus there exists an index $\check{i} \in [R]$ that satisfies $-\underline{\Phi} \left(z ; \bzz\right) =  -g_{\check{i}} y_2 - g_{\check{i}}^0  \geq \zeta'$. Hence, $-\underline{\Phi} \left(z ; \bzz\right) =  \max_{i \in [R]} - g_{i} y_2 - g_{i}^0  \geq \zeta'$ is satisfied.
    \item[iii)] Finally, $ y_3 = z > \check{z}_{R}$ with $ t_3 = 1, \, t_1, t_2, y_1, y_2 = 0$ indicates that $-\Phi(\zz_{R}) = -\underline{\Phi} \left(z ; \bzz\right) \geq \zeta'$. 
\end{itemize}
Additionally, since all of these cases allow $\zeta'$ to be at most 0, using the definition of $\underline{\Phi}(z; \bzz)$ in~\eqref{def: phi_pwl_inner}, $-\underline{\Phi}(z; \bzz)  \geq \zeta'$ is satisfied. \hfill\qedsymbol

\subsubsection{Proof of Lemma~\ref{lem:beta-compactification}. }\,

Fix any feasible solution $(\bx,\bbeta)$ of the finite master problem at
iteration $\ell$. Since $\{\hat{\bs}_1,\ldots,\hat{\bs}_{\hat K}\}\subseteq \mS_\ell,$
the master problem contains, for each $k\in[\hat K]$, the constraint generated
by $\bs_l=\hat{\bs}_k$. Hence, $\beta_k
-
\beta_{\hat K+1}
d^2_{\mathrm{BW}}
\bigl(
\mN(\hat{\bs}_k),\mN(\hat{\bs}_k)
\bigr)
\le
G(\hat{\bs}_k;\bx).$
Because $d^2_{\mathrm{BW}}
\bigl(
\mN(\hat{\bs}_k),\mN(\hat{\bs}_k)
\bigr)=0,$ we get $\beta_k\le G(\hat{\bs}_k;\bx).$
Additionally, since $G(\bs;\bx)$ corresponds to a cumulative probability term, $0\le G(\bs;\bx)\le 1$. Together they imply  $\beta_k\le 1,
\,\,\,
\forall k\in[\hat K].$

Next, because $\bbeta\in\mathfrak B$, $\sum_{k=1}^{\hat K}\hat w_k\beta_k-\rho\beta_{\hat K+1}\ge \theta.$ Using $\beta_k\le 1$ for all $k$ and
$\sum_{k=1}^{\hat K}\hat w_k=1$, we obtain $\theta
\le
\sum_{k=1}^{\hat K}\hat w_k\beta_k-\rho\beta_{\hat K+1}
\le
1-\rho\beta_{\hat K+1}.$
Hence, $0\le \beta_{\hat K+1}\le \frac{1-\theta}{\rho},$ 
where the lower bound follows from
$\bbeta\in\mathbb R^{\hat K}\times\mathbb R_+$.

It remains to lower-bound each $\beta_k$. Since $\beta_{\hat K+1}\ge0$ and $\sum_{k=1}^{\hat K}\hat w_k\beta_k-\rho\beta_{\hat K+1}\ge \theta$, we have $-\rho\beta_{\hat K+1}\le0$. Fix any $i \in[\hat K]$. Also, for every $k\neq i$, $\beta_k\le1$. Therefore,
\[
\theta
\le
\hat w_i\beta_i
+
\sum_{k\neq i}\hat w_k\beta_k
-
\rho\beta_{\hat K+1}
\le
\hat w_i\beta_i
+
\sum_{k\neq i}\hat w_k.
\]
Because $\sum_{k\neq i}\hat w_k=1-\hat w_i,$ 
we get $\theta
\le
\hat w_i\beta_i+1-\hat w_i.$
Since $\hat w_i>0$, $\beta_i
\ge
\frac{\theta-(1-\hat w_i)}{\hat w_i}.$ As $i\in[\hat K]$ was arbitrary, the bound holds for all
$k\in[\hat K]$.

Combining these inequalities gives $\frac{\theta-(1-\hat w_k)}{\hat w_k}
\le \beta_k \le 1, \,\, k\in[\hat K],
\,\, \text{and} \,\, 0\le \beta_{\hat K+1} \le \frac{1-\theta}{\rho}.$ Thus, $\bbeta\in\overline{\mathfrak B}$. Since $\overline{\mathfrak B}$ is a finite product of closed and bounded intervals, it is compact. \hfill\qedsymbol

\subsubsection{Proof of Lemma~\ref{lem: optimality-master-problem}. }

     To prove the claim, following ~\cite[Definition~1.3]{gauvin1977differential} we first show that $\mF(\mathtt{0}) := \{ (\bx, \bbeta) \; | \bx \in \mX, \bbeta \in \overline{\mfB}, J_k(\bs_l)\le 0,\;\bs_l \in \mS_\ell, k \in [\hat{K}]\}$ is uniformly compact near $\mathtt{0}$ if the following condition holds: there is a neighborhood $[- \varepsilon, \varepsilon]$ of $0$ for some $\varepsilon > 0$ such that the closure of the set $\bigcup_{\tau' \in [- \varepsilon, \varepsilon]} \mF(\tau')$ is compact, i.e., closed and bounded. Note that by Assumption~\ref{assmp: non-empty-interior}, $\bigcup_{\tau' \in [-\varepsilon, \varepsilon]} \mF(\tau')$ is not empty. Additionally, by Lemma~\ref{lem: G-diff-continuity-x}, $J_k(\bs_l)$ is continuous everywhere in $(\bx, \bbeta) \in \mX \times \overline{\mfB}$ for any $\bs_l \in \mS_\ell, k \in [\hat{K}]$. Then,
given any $\tau' \in [-\varepsilon, \varepsilon]$, each convergent sequence  $\{(\bx_t, \bbeta_t)\}$ with $(\bx_t, \bbeta_t) \in \mF(\tau') \; \forall t$ implies that $(\bar{\bx}, \bar{\bbeta}) := \lim_{t \to \infty} (\bx_t, \bbeta_t) \in \mF(\tau')$.  
Thus, $\mF(\tau')$ also includes all its limit points $(\bar{\bx}, \bar{\bbeta}) $. Therefore, $\bigcup_{\tau' \in [-\varepsilon, \varepsilon]} \mF(\tau')$ is closed. Additionally, since $\mX$ and $\overline{\mfB}$ are bounded (since compact) set, $\mX \cap \{ (\bx, \bbeta) | \, \bbeta \in \overline{\mfB}, J_k(\bs_l) \le \tau',\;\bs_l \in \mS_\ell, k \in [\hat{K}]\} \subseteq \mX \times \overline{\mfB}$ is also bounded and this is true for every $\tau' \in [-\varepsilon, \varepsilon]$. Therefore,  $\bigcup_{\tau' \in [-\varepsilon, \varepsilon]} \mF(\tau')$ is compact.

Next, we show that the MFCQ condition holds at the optimal solution \((\bx^{(\ell)}_*, \bbeta^{(\ell)}_*)\) and thereby continuity of the optimal value function $V^\star_\ell(\mathtt{0})$ at $\mathtt{0}$. Let us define the indicator 
$\delta_{\mX}(\bx)$ and $\delta_{\overline{\mfB}}(\bbeta)$ respectively for the set \(\mX\) and $\overline{\mfB}$ 
and the Lagrangian for some $\bq \in \mbR^{\hat{K} \times |\mS_\ell|}$:
\[
\mL(\bx, \bbeta, u, \bq)=
\bc^\top \bx +\delta_{\mX}(\bx) + \delta_{\overline{\mfB}}(\bbeta) + \bq^\top \, \bH(\bx, \bbeta).
\]
Since \((\bx^{(\ell)}_*, \bbeta^{(\ell)}_*)\) is a local minimizer, continuous differentiability of $G(\bx; \cdot )$ implies continuous differentiability of $J\left((\bx, \bbeta); \bs\right)$ for all $(\bx, \bbeta) \in \mX \times \overline{\mfB}$ by Lemma~\ref{lem: G-diff-continuity-x},
$\exists$ \(\check{\bq} \geq \bzero\), \emph{normal-cone multipliers} \(\check{\bv}\in \partial_{\bx} \delta_{\mX}(\bx^{(\ell)}_*) = \mN_{\mX}(\bx^{(\ell)}_*)\) and \(\check{\bu}\in \partial_{\bbeta} \delta_{\overline{\mfB}}(\bbeta^{(\ell)}_*) = \mN_{\mS}(\bbeta^{(\ell)}_*)\) 
such that primal feasibility, dual feasibility, and complementary conditions hold with $ \bzero=(\bc, \bzero) + (\check{\bv}, \check{\bu})
+\nabla_{\bx, \bbeta}\bH^\top(\bx^{(\ell)}_*, \bbeta^{(\ell)}_*) \, \check{\bq}.$ Additionally, by Assumption~\ref{assmp: LICQ}, there are \emph{no} non-zero \((\bq, \bv)\) with  
\(\bv\in\mN_{\mX}(\check{\bx})\), \(\bq \geq \bzero\) satisfying $ \bzero\;=\;(\check{\bv}, \check{\bu}) 
+\nabla_{\bx, \bbeta}\bH^\top(\bx^{(\ell)}_*, \bbeta^{(\ell)}_*) \, \check{\bq}.$ Then Gordan’s alternative for cones (Thm.\;6.4 in \cite{rockafellar2009variational};  
 Thm.\;2.4.9 in \cite{bazaraa2006nonlinear}) gives
\[
\;
\exists\; \; \mathfrak{z} \in\mT(\bx^{(\ell)}_*, \bbeta^{(\ell)}_*)\;\text{such that}\;
\nabla_{\bx, \bbeta}\bH(\bx^{(\ell)}_*, \bbeta^{(\ell)}_*)^{\!\top}\mathfrak{z}<0,
\]
where \(\mT(\bx^{(\ell)}_*, \bbeta^{(\ell)}_*):=\operatorname{cl}\bigl\{\alpha(\bt-(\bx^{(\ell)}_*, \bbeta^{(\ell)}_*))\mid
\alpha>0,\;\bt\in (\mX \times \bbeta)\) is the Bouligand tangent cone of \(\mX \times \overline{\mfB} \) at \((\bx^{(\ell)}_*, \bbeta^{(\ell)}_*)\).
Thus, the Mangasarian–Fromovitz constraint qualification (MFCQ) holds at \((\bx^{(\ell)}_*, \bbeta^{(\ell)}_*)\). Then, by Theorem 2.6 of~\cite{gauvin1977differential}, optimal value function $V^\star_\ell(\mathtt{0})$ corresponding to master problem is continuous $\mathtt{0}$. 

Finally, Theorem~\ref{thm: complexity-breakpoint} and Proposition~\ref{prop: feasibility-tau} ensure that using $O(\frac{1} {\sqrt{\epsilon}} \sqrt{\log(\frac{1}{\epsilon})})$ breakpoints, a $\epsilon$-feasible solution $\bigl(\bx^{(\ell)}, \bbeta^{(\ell)}\bigr)$ can be obtained. Furthermore, continuity of $V^\star_\ell(\mathtt{0})$ at $\mathtt{0}$ implies that for any given $\epsilon > 0$, $\exists \, \tau^+(\epsilon)$ such that  $|V^\star_\ell(\mathtt{0}) - V^\star_\ell(\tau)| \leq \epsilon $ can be achieved for some $\tau$ such that $0 < \tau \leq \tau^+(\epsilon)$, i.e., for some $\tau \in (0, \epsilon]$. Since $V^\star_\ell(\epsilon) \leq \bc^\top \bx^{(\ell)}$, $\bc^\top \bx^{(\ell)}_* - \bc^\top \bx^{(\ell)}= V^\star_\ell(\mathtt{0}) - \bc^\top \bx^{(\ell)} \leq V^\star_\ell(\mathtt{0}) - V^\star_\ell(\epsilon) \leq \epsilon$. \hfill \qed
\begin{remark}
    We note that strict interior, hence, MFCQ can also be shown by considering a perturbation to $\check{\bs}$ (see Bonnans and Shapiro \cite{bonnans2013perturbation} for details).
\end{remark}

\section{Piecewise linear approximation of a Gaussian CDF}
Here, we present a breakpoint construction approach in Algorithm~\ref{tab:phi-breakpoints} that we followed in our numerical study. Such a construction scheme guarantees that $O(\frac{1} {\sqrt{\tau}} \sqrt{\log(\frac{1}{\tau})})$ breakpoints suffice to achieve piecewise-linear approximation accuracy \(\tau\); see Theorem~\ref{thm: complexity-breakpoint}. We also illustrate the resulting outer and inner piecewise-linear approximations of the standard normal CDF \(\Phi(z)\) in Figure~\ref{fig:pwl_phi_both}

\begin{table}[ht]
\scriptsize
  \centering
  \captionsetup{type=algorithm} 
  \caption{Breakpoints for approximating $\Phi(z)$}
  \label{tab:phi-breakpoints}

  \begin{subtable}[t]{0.48\linewidth}
  \vspace{-6pt}
    \centering
    \caption{\scriptsize Tangent approximation of $\Phi(z), z \ge 0$}
    \label{alg:breakpoint_find}
    \begin{algorithmic}[1]
      \Require{$\tau$}
      \Ensure{$\mathcal{A}^R$}

      \State \textbf{Part 1:} $0\leq z \leq 1$
      \State $i \gets 1, \check{z}_{i} \gets 1, \mathcal{A}^{R} \gets \mathcal{A}^{R} \cup \{\check{z}_{i}\}$
      \State $ \check{z}_{i+1} =  \check{z}_i - \sqrt{\frac{2\tau}{\frac{e^{-0.5}}{\sqrt{2 \pi}}}}$
      \While{$\check{z}_{i+1} > 0$}
      \vspace{5pt}
      \State $\mathcal{A}^{R} \gets \mathcal{A}^{R} \cup \{\check{z}_{i+1}\}, i \gets i + 1$
      \State $\check{z}_{i+1} = \check{z}_{i} -  \sqrt{\frac{2\tau}{\phi(\check{z}_{i})\check{z}_{i}}}$
      \EndWhile
      \State $\mathcal{A}^R \gets \mathcal{A}^{R} \cup \{0\}, i \gets i+1$
      \State Reverse the order of $\mathcal{A}^{R}$
      \State $\mathcal{A}^{R} \gets \mathcal{A}^{R} \backslash \{1\}$
      \State \textbf{Part 2:} $z \geq 1$
      \State $i \gets i + 1, \check{z}_{i-1} \gets 1, \mathcal{A}^{R} \gets \mathcal{A}^{R} \cup \{\check{z}_{i-1}\}$
      \State $\check{z}_i = \check{z}_{i-1} + \sqrt{\frac{2\tau}{\phi(\check{z}_{i-1})
      \check{z}_{i-1}}}$
      \While{$1 - \Phi(\check{z}_i) < \tau$}
      \vspace{5pt}
      \State $\mathcal{A}^{R} \gets \mathcal{A}^{R} \cup \{\check{z}_i\}, i \gets i + 1$
      \State $\check{z}_i = \check{z}_{i-1} + \sqrt{\frac{2\tau}{\phi(\check{z}_{i-1})
      \check{z}_{i-1}}}$
      \EndWhile
      \State $\zz_R \gets \zz_i$
      \State $\mathcal{A}^{R} \gets \mathcal{A}^{R} \cup \{\zz_{R}\}$
    \end{algorithmic}
  \end{subtable}\hfill
  \begin{subtable}[t]{0.48\linewidth}
    \vspace{-6pt}
    \centering
    \caption{\scriptsize Secant approximation of $\Phi(z)$, $z < 0$}
    \label{alg:breakpoint_find_negative}
    \begin{algorithmic}[1]
      \Require{$\tau$}
      \Ensure{$\mathcal{A}^{L}$}

      \State \textbf{Part 1:} $-1\leq z \leq 0$
      \State $i \gets 1, \check{z}_{-i} \gets -1, \mathcal{A}^{L} \gets \mathcal{A}^{L} \cup \{\check{z}_{-i}\}$
      \State $ \check{z}_{-i-1} =  \check{z}_{-i} - 2\sqrt{\frac{2\tau}{\phi(\check{z}_{-i})
      \check{z}_{-i}}}$
      \While{$\check{z}_{-i-1} < 0$}
      \vspace{5pt}
      \State $\mathcal{A}^{L} \gets \mathcal{A}^{L} \cup \{\check{z}_{-i-1}\}, i \gets -i - 1$
      \State $\check{z}_{-i-1} = \check{z}_{-i} - 2\sqrt{\frac{2\tau}{\phi(\check{z}_{-i})
      \check{z}_{-i}}}$
      \EndWhile
      \State Reverse the order of $\mathcal{A}^{L}$
      \State $\mathcal{A}^{L} \gets \mathcal{A}^{L} \backslash {-1}$
      \State \textbf{Part 2:} $z \leq -1$
      \State $i \gets i + 1, \check{z}_{-i+1} \gets -1$,
      
       $\mathcal{A}^{L} \gets \mathcal{A}^{L} \cup \{\check{z}_{-i+1}\}$
      \State $\check{z}_{-i} = \check{z}_{-i+1} + 2 \sqrt{\frac{2\tau}{\phi(\check{z}_{-i+1})
      \check{z}_{-i+1}}}$
      \While{$\Phi(\check{z}_{-i}) < \tau$}
      \vspace{5pt}
      \State $\mathcal{A}^{L} \gets \mathcal{A}^{L} \cup \{\check{z}_{-i}\}, i \gets i + 1$
      \State $\check{z}_{-i} = \check{z}_{-i+1} + 2 \sqrt{\frac{2\tau}{\phi(\check{z}_{-i+1})
      \check{z}_{-i+1}}}$
      \EndWhile
      \State $\zz_{-L} \gets \zz_{-i}$
      \State $\mathcal{A}^{L} \gets \mathcal{A}^{L} \cup \{\zz_{-L}\}$
    \end{algorithmic}
  \end{subtable}
\hfill
\end{table}
\vspace{3em}
\begin{figure}[ht]
\centering
\begin{subfigure}[t]{0.49\textwidth}
    \centering
    \includegraphics[width=\linewidth]{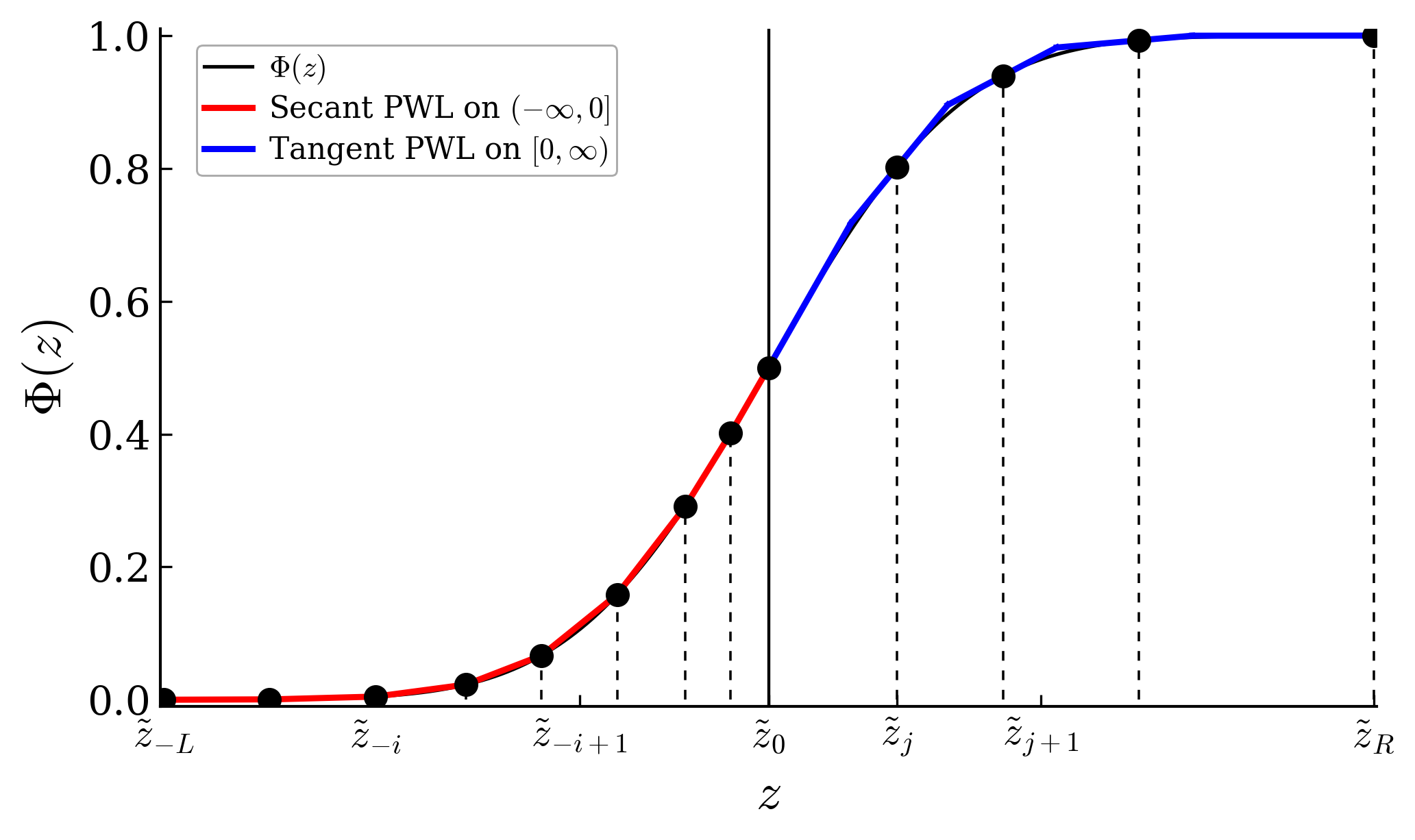}
    \caption{Piecewise linear outer approximation of $\Phi$.}
    \label{fig:pwl_phi_outer}
\end{subfigure}\hfill
\begin{subfigure}[t]{0.49\textwidth}
    \centering
    \includegraphics[width=\linewidth]{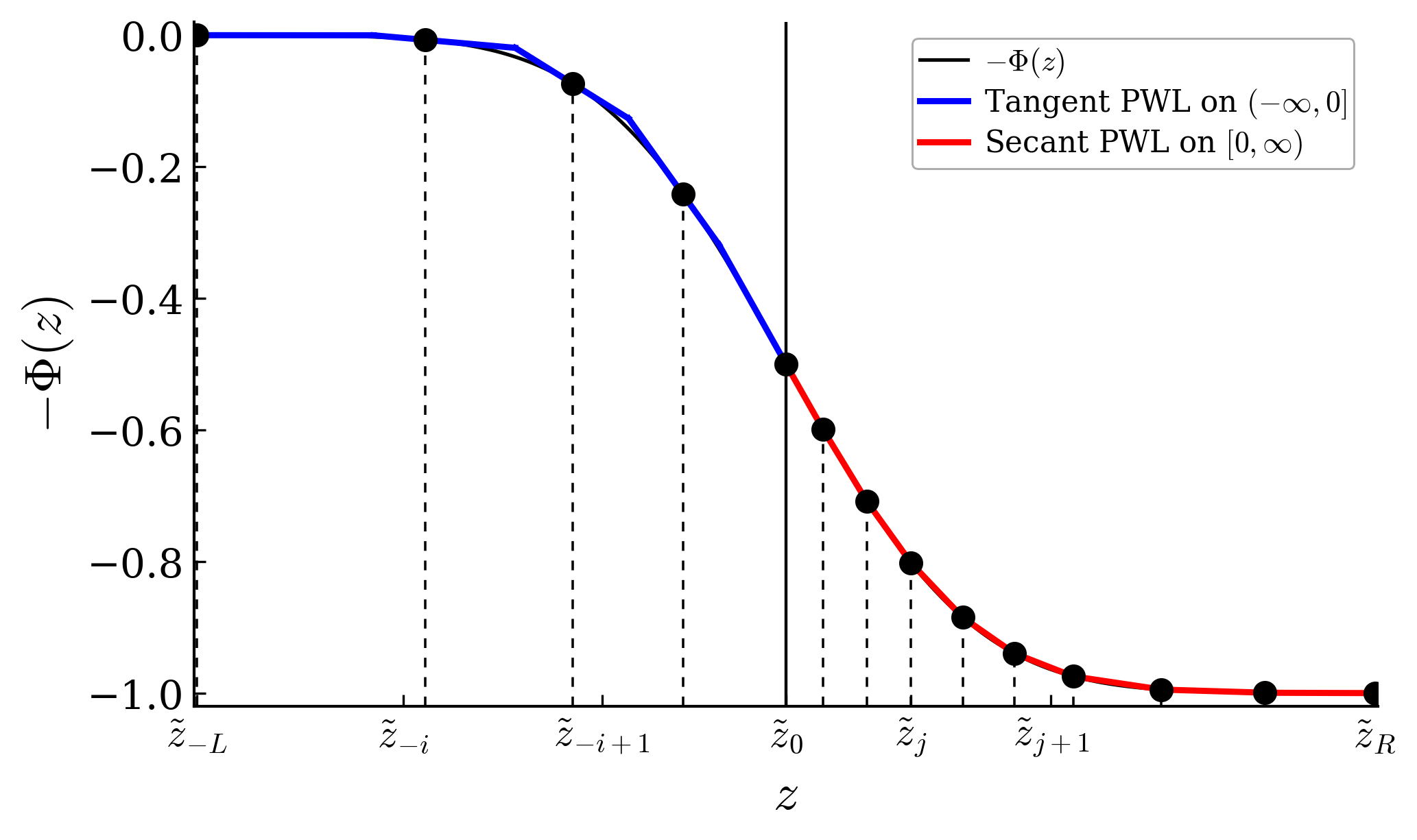}
    \caption{Negation of the piecewise linear inner approximation of $\Phi$.}
    \label{fig:pwl_phi_inner-negative}
\end{subfigure}
\caption{\centering Piecewise linear approximations used in the formulation.}
\label{fig:pwl_phi_both}
\end{figure}

\section{Computational Study Details}\label{appndx: comp-detail}

\subsection{Data Preparation Details}\label{appndx: data-preparation}

For the EV application, the nominal distribution represents daily hourly charging demand and is estimated from session-level charging records. The dataset, sourced from Figshare~\cite{evcharging2025figshare}, contains one record per charging session, including start and end times, delivered energy in kWh, user and charging-post identifiers, location or district name, charging costs in Yuan, termination reason, and weather attributes such as temperature, relative humidity, and precipitation. Its structure is consistent with the previously published high-resolution Jiaxing EV charging transaction dataset~\cite{evcharging2025nature}. Since aggregating the full two-year sample can obscure temporal heterogeneity and collapse the demand distribution toward a dominant mode, we partition the data by quarter and use Q1 observations, i.e., January--March of \(2020\) and \(2021\), for the study.

The session-level records are converted into daily \(24\)-dimensional demand vectors. For a session \(s\) with delivered energy \(E_s\), let \(\Delta_{st}\) denote its overlap duration with hour \(t\in[24]\). The energy allocated to hour \(t\) is \(E_s\Delta_{st}/\sum_{t'\in[24]}\Delta_{st'}\). Summing these allocated contributions over all sessions occurring on the same day yields a per-day hourly-demand vector. The data are processed in Python using \texttt{pandas}: start and end timestamps are parsed as date-time variables, the energy field is converted to numeric values, and records with missing or nonpositive delivered energy are removed. After a fixed random shuffle, \(60\%\) of the daily vectors are used for estimation, while the remaining days are reserved as holdout data for evaluating out-of-sample satisfaction probability.


The objective $\min \sum_{t=1}^{24} \bigl( C^o(x_t) + c_t x_t \bigr)$ having a \emph{linear (fixed) term} and a \emph{piecewise-linear (operational) term} are built as follows. The linear term is $\sum_{t=1}^{24} c_t x_t$, where $c_t$ is the time-of-use cost (Yuan/kWh) in hour $t$; the 24 values $c_t$ are drawn uniformly from $[300,\,350]$. The operational cost in hour $t$ is a convex piecewise-linear function $C^o(x_t)$ of the capacity fraction $x_t$, built from a pivot $(\tau_0, y_0) = (0.8,\,200)$, 16 segments (8 left of the pivot, 8 right), base slope $\eta = 1$, and growth factor $\gamma = 2$: left-segment slopes are $-\eta \gamma^i$ for $i=0,\ldots,7$ (reversed so the steepest negative slope is near $\tau_0$), right-segment slopes are $\eta \gamma^i$ for $i=0,\ldots,7$; knot positions are even-spaced on $[0,\tau_0]$ and $[\tau_0,1]$; intercepts are determined by continuity at the pivot. Thus $C^o(x_t) = \max_{t=1}^{16}(a_t x_t + b_t)$ with $(a_t,b_t)$ from this construction. Model parameters of the considered EV model are further summarized in Table~\ref{tab:ev-master-params} below.

\begin{table}[!hbtp]
\centering
\scriptsize
\caption{Polytopic constraints and coefficients for the EV chance-constrained model.}
\label{tab:ev-master-params}
\renewcommand{\arraystretch}{1.0}
\begin{tabular}{p{2.0cm} p{4.5cm} p{8.5cm}}
\toprule
\textbf{Parameter} & \textbf{Value} & \textbf{Description} \\
\midrule

\multicolumn{3}{l}{\textit{Variable bounds}} \\[2pt]
$x_{\min}$, $x_{\max}$ 
    & $0$, $1$ 
    & Bounds on each $x_t$ (fraction of capacity in hour $t$). \\
\midrule
\multicolumn{3}{l}{\textit{Chance constraint and throughput}} \\[2pt]
$P_{\max}$ 
    & $560.18$
    & Peak capacity (0.98 quantile of daily total demand). \\[6pt]

$E_{Total}$ 
    & $1{,}151{,}359.45$
    & Total energy delivered (0.98 quantile of daily total delivered energy)\\

$E_{min}$
    & $E_{Total}/P_{max}$ 
    &  Minimum daily energy delivery (throughput target); \\

\midrule
\multicolumn{3}{l}{\textit{Objective (two-term structure)}} \\[2pt]
Linear term
    & $\sum_{t=1}^{24} c_t x_t$
    & $c_t\in[300,350]$\,Yuan/kWh (time–of–use); 24 values per instance. \\[6pt]

PWL term
 & \makecell[l]{
 $\displaystyle
C^o(x)=\max_{a\in\mathcal{A}_-\cup\mathcal{A}_+}f(a)$ where\\
 $f(a) = \{\,a\,x + (y_0 - a x_0)\,\}$
 }
    & 
    \makecell[l]{Pivot $(\tau_0,y_0)=(0.8,200)$; $
\,\,
\mathcal{A}_\pm=\{\pm\,2^i : i=0,\dots,7\}.
$ \\
 8 left + 8 right segments;
} \\
\bottomrule
\end{tabular}
\end{table}

\subsection{Master Iteration Runtime Scheduling Strategy}


Because every admitted cut increases the mixed-integer master problem constraint pool, the master solver uses a highly customized stopping rule that balances tractability and solution quality across iterations. Early termination is never allowed before a feasible incumbent has been found. Once feasibility is achieved, the solver may stop only after a preset soft time budget has elapsed and the incumbent has reached the prescribed mixed-integer optimality gap of 0.1\%; otherwise, the solve continues until an iteration-dependent hard cap is reached. This hard cap increases with the iteration count, particularly 2 hours for iterations 0 and 1, then 4, 6, and 10 hours. The attained incumbent gap is recorded at the end of each master solve and carried into the next iteration as a state variable for the stopping logic. This schedule is designed to reflect the systematic growth of the master problem as cuts accumulate: early iterations are kept relatively inexpensive, whereas substantially more time is allowed for the later iterations so that the performance gain attributable to additional cuts can be assessed under a solve budget commensurate with the enlarged constraint set. 
Because each CDR master solve was stopped at its iteration-wise time budget regardless of whether the desired optimality gap of 0.1\% was met, the zeroth CDR solution can provide a solution to the FDR problem. Accordingly, the zeroth iteration problem is also solved separately outside the scheduled CDR iterations. For this standalone solve, we impose the same 0.1\% gap target and a 24-hour time limit, equal to the cumulative budget used by CDR across four iterations, to obtain a fair benchmark for comparison with the CDR objective and OSS probability achieved at its fourth-iteration solution.

\subsection{Tables and Plots from Numerical Findings}\label{appndx: visuals}

\begin{table}[!htbp]
\centering
\caption{\scriptsize Finite-support DR (FDR) solution time, cumulative completion time, and cuts across master iterations of the cutting surface algorithm under continuous-support DR (CDR) for all prescribed service-level targets $\theta \in \{0.95, 0.97, 0.99\}$ and Wasserstein ambiguity radius $\rho \in \{0.001, 0.005, 0.01\}$, shown for mean-support expansion parameter $\varsigma=0.10$.}
\label{tab:cumulative-time-cut-0.1}
\setlength{\tabcolsep}{3pt}
\renewcommand{\arraystretch}{0.5}
\scriptsize
\begin{tabular}{cc|cc|cccccc|cccccc}
\toprule
& & \multicolumn{2}{c|}{FDR} & \multicolumn{12}{c}{CDR} \\
\cmidrule(lr){3-4}\cmidrule(lr){5-16}
& & & & \multicolumn{6}{c|}{Covariance-scaling factor $\underline{\lambda}_f = 0.333$} & \multicolumn{6}{c}{Covariance-scaling factor $\underline{\lambda}_f = 0.5$} \\
\cmidrule(lr){5-10}\cmidrule(lr){11-16}
& &Benchmark & \makecell{From CDR Iter 0} & \multicolumn{2}{c}{Iter 2} & \multicolumn{2}{c}{Iter 3} & \multicolumn{2}{c|}{Iter 4}
& \multicolumn{2}{c}{Iter 2} & \multicolumn{2}{c}{Iter 3} & \multicolumn{2}{c}{Iter 4} \\
\cmidrule(lr){5-6}\cmidrule(lr){7-8}\cmidrule(lr){9-10}
\cmidrule(lr){11-12}\cmidrule(lr){13-14}\cmidrule(lr){15-16}
$\theta$ & $\rho$ & Time (hr) & Time (hr)
& Time (hr) & Cut & Time (hr) & Cut & Time (hr) & Cut
& Time (hr) & Cut & Time (hr) & Cut & Time (hr) & Cut \\
\midrule
     & 0.001 & 0.022 & 0.009 & 0.015 & 15 & 0.700 & 19 & 4.662 & 22 & 0.014 & 15 & 1.077 & 19 & 3.637 & 22 \\
0.95 & 0.005 & 0.026 & 0.006 & 3.968 & 14 & 9.968 & 17 & 19.910 & 21 & 2.770 & 15 & 8.770 & 18 & 18.770 & 22 \\
     & 0.010 & 0.024 & 0.009 & 6.009 & 15 & 12.009 & 18 & 22.009 & 21 & 6.009 & 14 & 12.009 & 18 & 22.009 & 22 \\
\midrule
     & 0.001 & 0.669 & 0.705 & 6.705 & 15 & 12.705 & 19 & 22.705 & 22 & 6.696 & 15 & 12.696 & 19 & 22.696 & 22 \\
0.97 & 0.005 & 0.154 & 0.149 & 6.149 & 15 & 12.149 & 19 & 22.149 & 23 & 6.142 & 15 & 12.142 & 19 & 22.142 & 23 \\
     & 0.010 & 3.300 & 2.000 & 8.000 & 15 & 14.000 & 18 & 24.000 & 22 & 8.000 & 15 & 14.000 & 19 & 24.000 & 23 \\
\midrule
     & 0.001 & 24.000 & 2.000 & 8.000 & 15 & 14.000 & 19 & 24.000 & 22 & 8.000 & 15 & 14.000 & 19 & 24.000 & 22 \\
0.99 & 0.005 & 24.000 & 2.000 & 8.000 & 15 & 14.000 & 19 & 24.000 & 23 & 8.000 & 15 & 14.000 & 19 & 24.000 & 23 \\
     & 0.010 & 24.021 & 2.000 & 8.000 & 15 & 14.000 & 19 & 24.000 & 23 & 8.000 & 15 & 14.000 & 19 & 24.000 & 23 \\
\bottomrule
\end{tabular}
\end{table}
\begin{table}[!htbp]
\centering
\caption{\scriptsize Finite-support DR (FDR) solution time, cumulative completion time, and cuts across master iterations of the cutting surface algorithm under continuous-support DR (CDR) for all prescribed service-level targets $\theta \in \{0.95, 0.97, 0.99\}$ and Wasserstein ambiguity radius $\rho \in \{0.001, 0.005, 0.01\}$, shown for mean-support expansion parameter $\varsigma=0.05$.}
\label{tab:cumulative-time-cut-0.05}
\setlength{\tabcolsep}{3pt}
\renewcommand{\arraystretch}{0.5}
\scriptsize
\begin{tabular}{cc|cc|cccccc|cccccc}
\toprule
& & \multicolumn{2}{c|}{FDR} & \multicolumn{12}{c}{CDR} \\
\cmidrule(lr){3-4}\cmidrule(lr){5-16}
& & & & \multicolumn{6}{c|}{Covariance-scaling factor $\underline{\lambda}_f = 0.333$} & \multicolumn{6}{c}{Covariance-scaling factor $\underline{\lambda}_f = 0.5$} \\
\cmidrule(lr){5-10}\cmidrule(lr){11-16}
& &Benchmark & \makecell{From CDR Iter 0} & \multicolumn{2}{c}{Iter 2} & \multicolumn{2}{c}{Iter 3} & \multicolumn{2}{c|}{Iter 4}
& \multicolumn{2}{c}{Iter 2} & \multicolumn{2}{c}{Iter 3} & \multicolumn{2}{c}{Iter 4} \\
\cmidrule(lr){5-6}\cmidrule(lr){7-8}\cmidrule(lr){9-10}
\cmidrule(lr){11-12}\cmidrule(lr){13-14}\cmidrule(lr){15-16}
$\theta$ & $\rho$ & Time (hr) & Time (hr)
& Time (hr) & Cut & Time (hr) & Cut & Time (hr) & Cut
& Time (hr) & Cut & Time (hr) & Cut & Time (hr) & Cut \\
\midrule
     & 0.001 & 0.022 & 0.008 & 0.849 & 15 & 2.112 & 18 & 9.231 & 21 & 0.558 & 15 & 3.043 & 18 & 5.789 & 22 \\
0.95 & 0.005 & 0.026 & 0.008 & 6.007 & 15 & 12.007 & 19 & 22.007 & 23 & 6.006 & 15 & 12.006 & 19 & 22.006 & 23 \\
     & 0.010 & 0.024 & 0.010 & 6.010 & 14 & 12.010 & 18 & 22.010 & 22 & 6.010 & 14 & 12.010 & 18 & 22.010 & 22 \\
\midrule
     & 0.001 & 0.669 & 0.077 & 5.184 & 15 & 10.418 & 18 & 18.418 & 21 & 5.024 & 15 & 11.024 & 18 & 19.024 & 22 \\
0.97 & 0.005 & 0.154 & 0.139 & 6.139 & 15 & 12.139 & 18 & 22.139 & 22 & 6.138 & 15 & 12.138 & 18 & 22.138 & 22 \\
     & 0.010 & 3.300 & 2.000 & 8.000 & 15 & 14.000 & 19 & 24.000 & 23 & 7.245 & 15 & 13.245 & 19 & 23.245 & 23 \\
\midrule
     & 0.001 & 24.000 & 2.000 & 8.000 & 15 & 14.000 & 18 & 24.000 & 22 & 8.000 & 15 & 14.000 & 18 & 24.000 & 22 \\
0.99 & 0.005 & 24.000 & 2.000 & 8.000 & 14 & 14.000 & 18 & 24.000 & 21 & 8.000 & 14 & 14.000 & 18 & 24.000 & 21 \\
     & 0.010 & 24.021 & 2.000 & 8.000 & 14 & 14.000 & 18 & 24.000 & 22 & 8.000 & 14 & 14.000 & 18 & 24.000 & 22 \\
\bottomrule
\end{tabular}
\end{table}
\vspace{-5em}

\begin{figure}[ht]
\vspace{-10pt}
\centering
\begin{subfigure}[t]{0.30\textwidth}
    \centering
    \includegraphics[width=0.9\linewidth,height=8em]{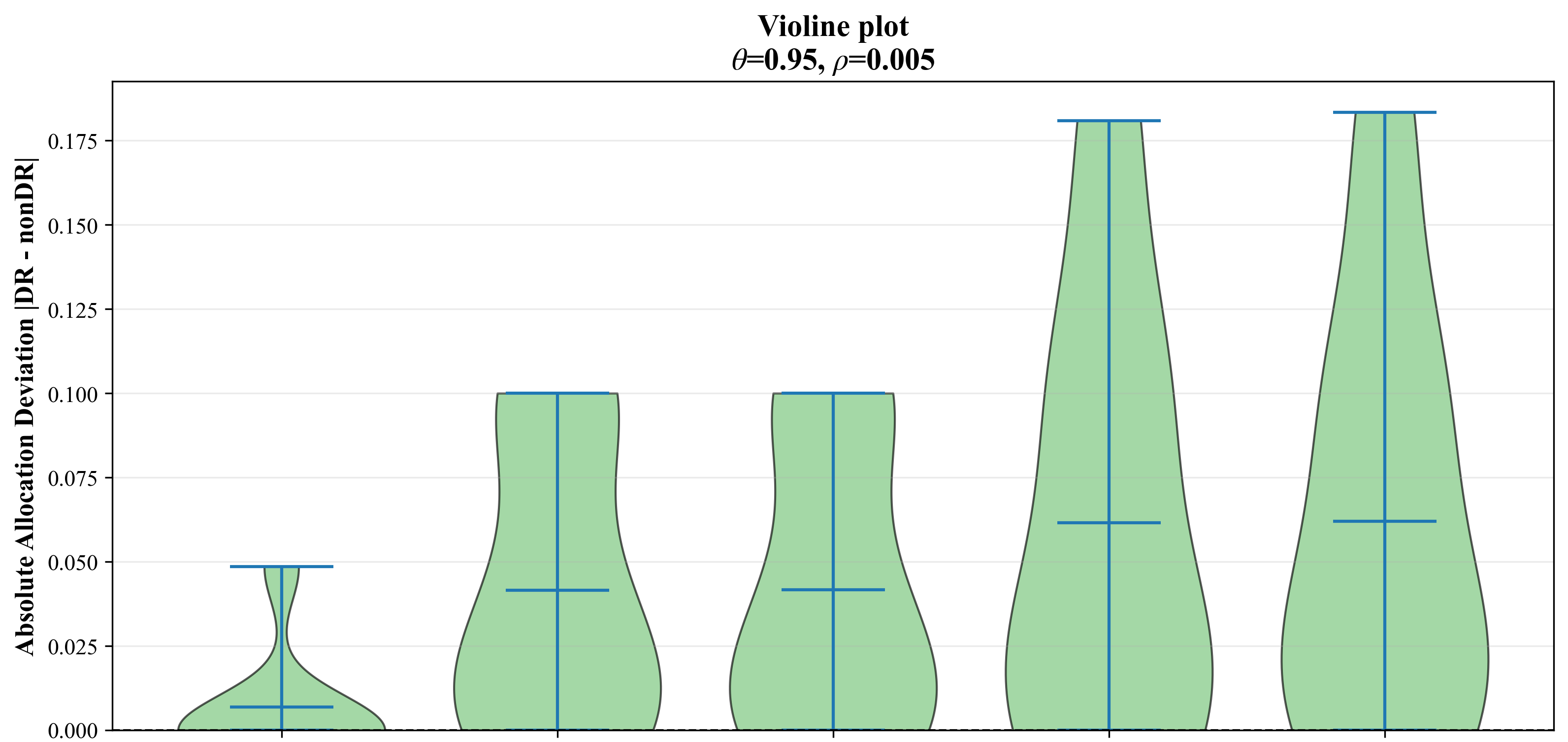}
    \label{fig:violin-0.95-0.005}
\end{subfigure}\hfill
\begin{subfigure}[t]{0.345\textwidth}
    \centering
    \includegraphics[width=\linewidth,height=8em]{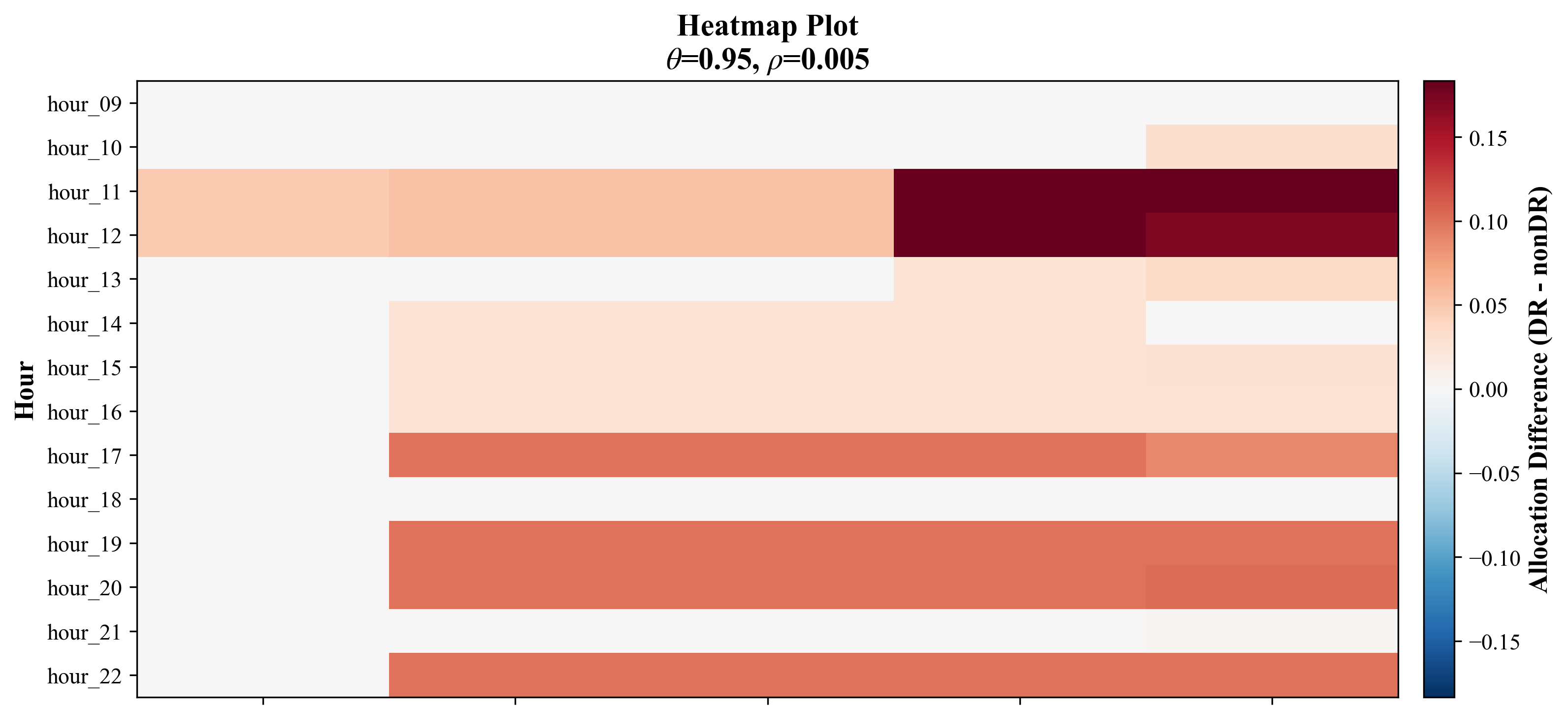}
    \label{fig:heatmap-0.95-0.005}
\end{subfigure}\hfill
\begin{subfigure}[t]{0.345\textwidth}
    \centering
    \includegraphics[width=0.95\linewidth,height=8em]{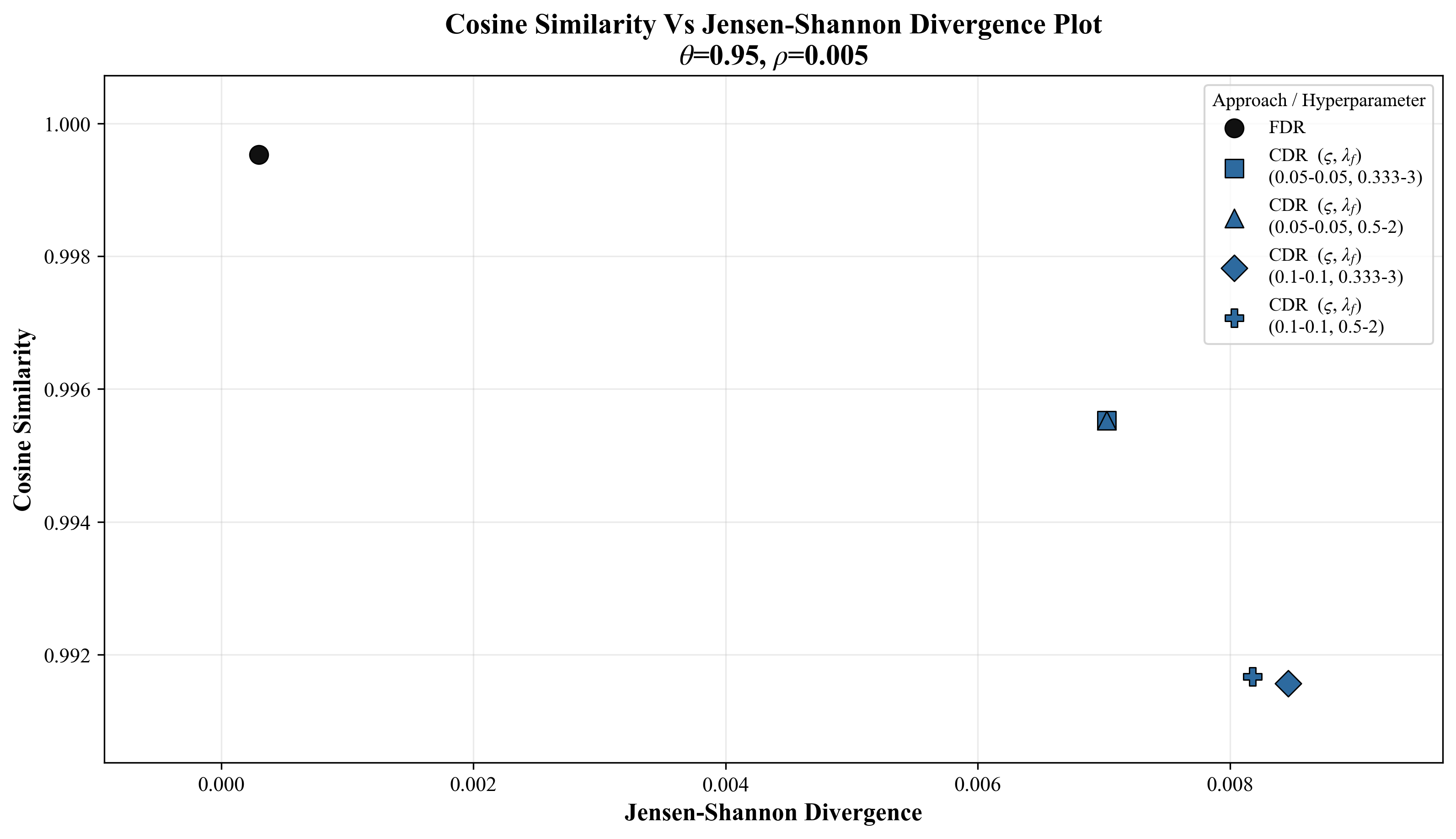}
    \label{fig:scatter-0.95-0.005}
\end{subfigure}

\vspace{-5pt}
\begin{subfigure}[t]{0.30\textwidth}
    \centering
    \includegraphics[width=0.9\linewidth,height=8em]{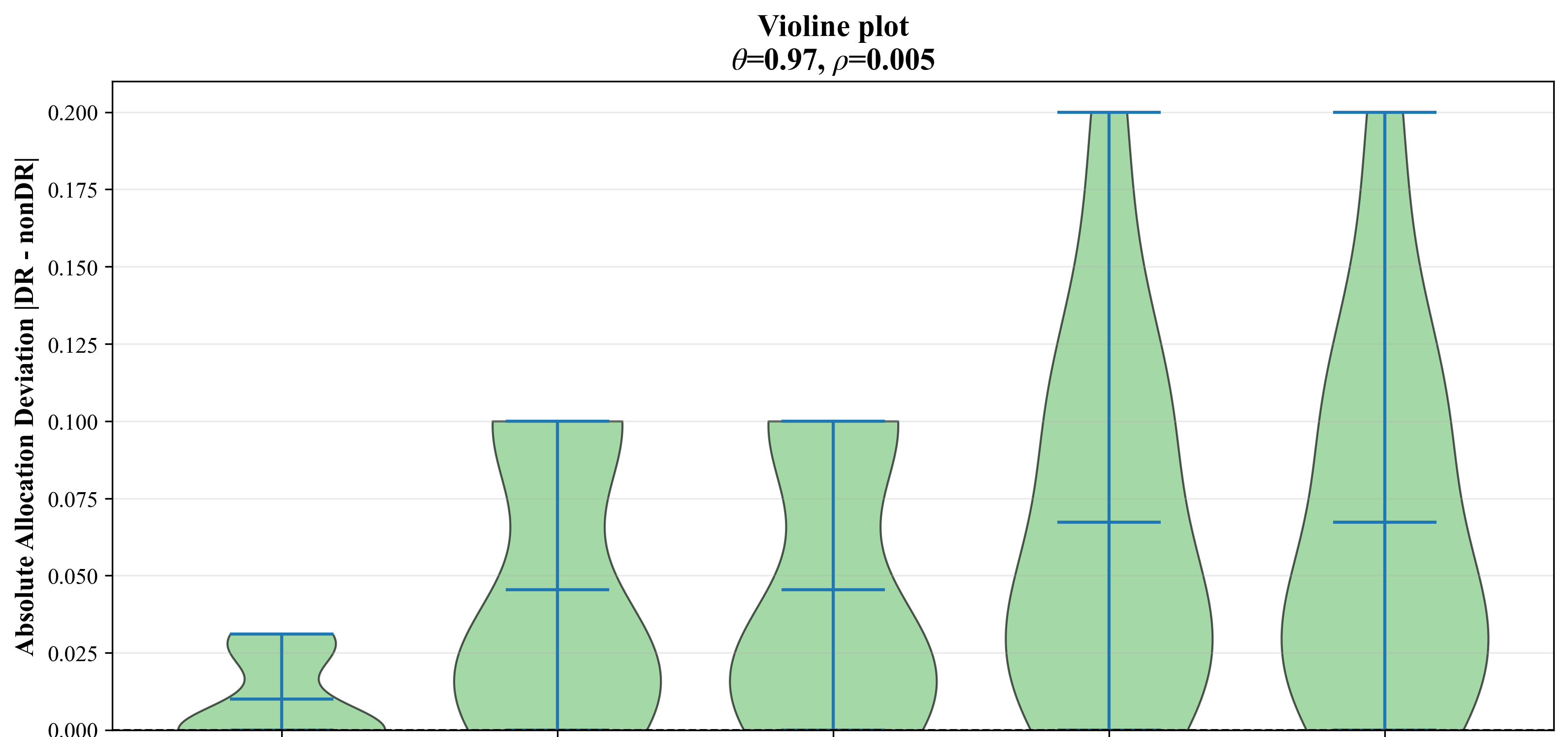}
     \label{fig:violin-0.97-0.005}
\end{subfigure}\hfill
\begin{subfigure}[t]{0.345\textwidth}
    \centering
    \includegraphics[width=\linewidth,height=8em]{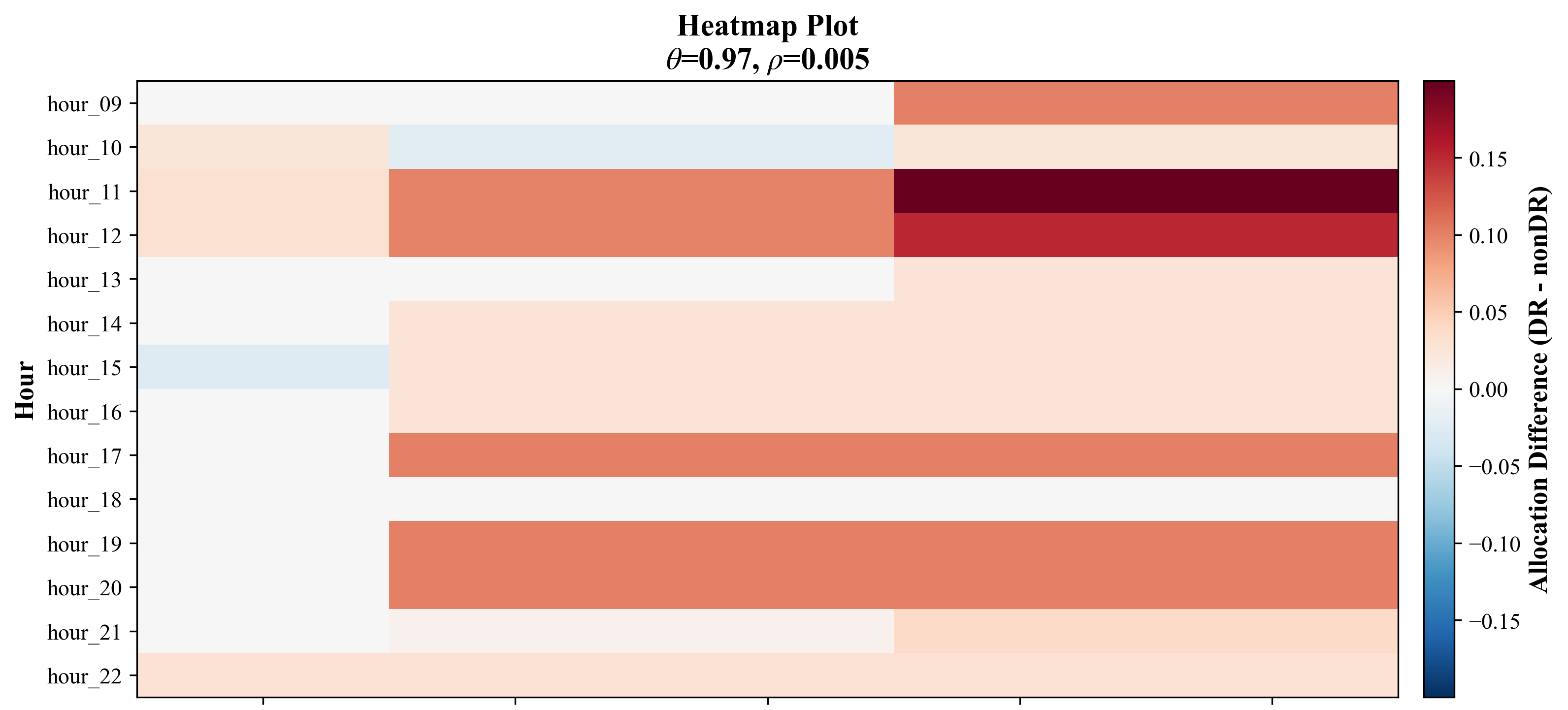}
    \label{fig:heatmap-0.97-0.005}
\end{subfigure}\hfill
\begin{subfigure}[t]{0.345\textwidth}
    \centering
    \includegraphics[width=0.95\linewidth,height=8em]{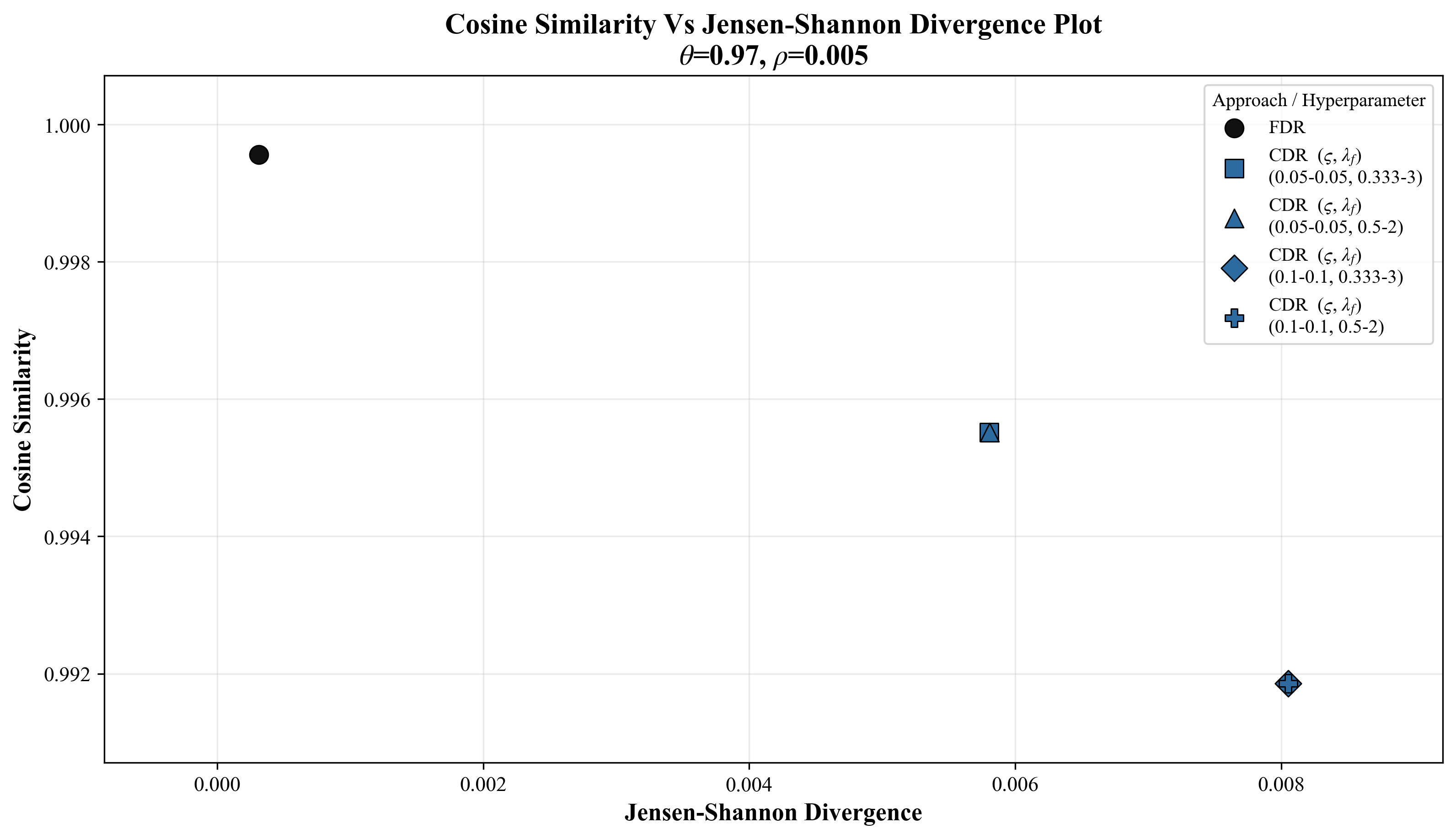}
    \label{fig:scatter-0.97-0.005}
\end{subfigure}

\vspace{-5pt}
\begin{subfigure}[t]{0.30\textwidth}
    \centering
    \includegraphics[width=0.9\linewidth,height=9.5em]{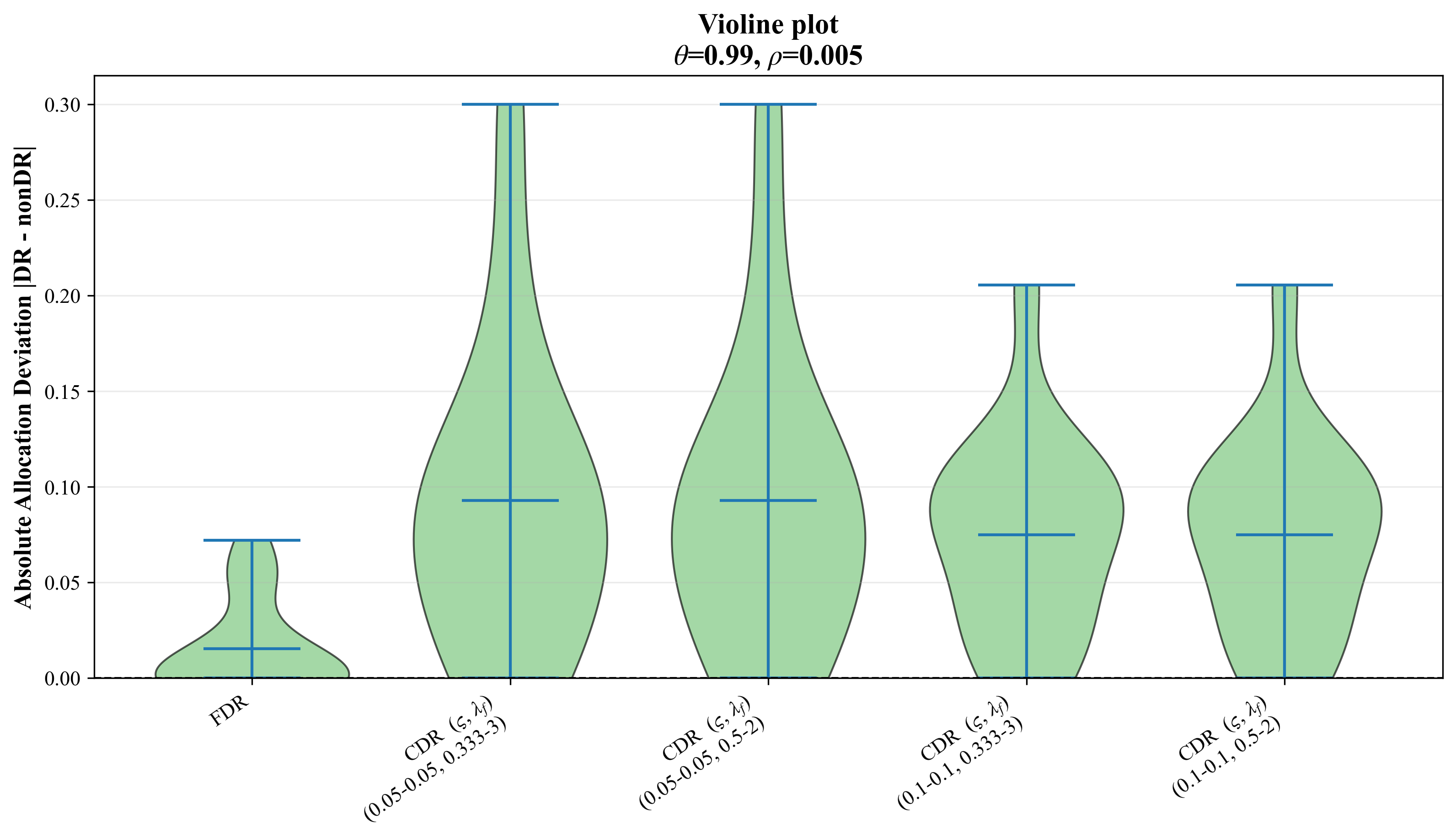}
     \label{fig:violin-0.99-0.005}
\end{subfigure}\hfill
\begin{subfigure}[t]{0.345\textwidth}
    \centering
    \includegraphics[width=\linewidth,height=9.5em]{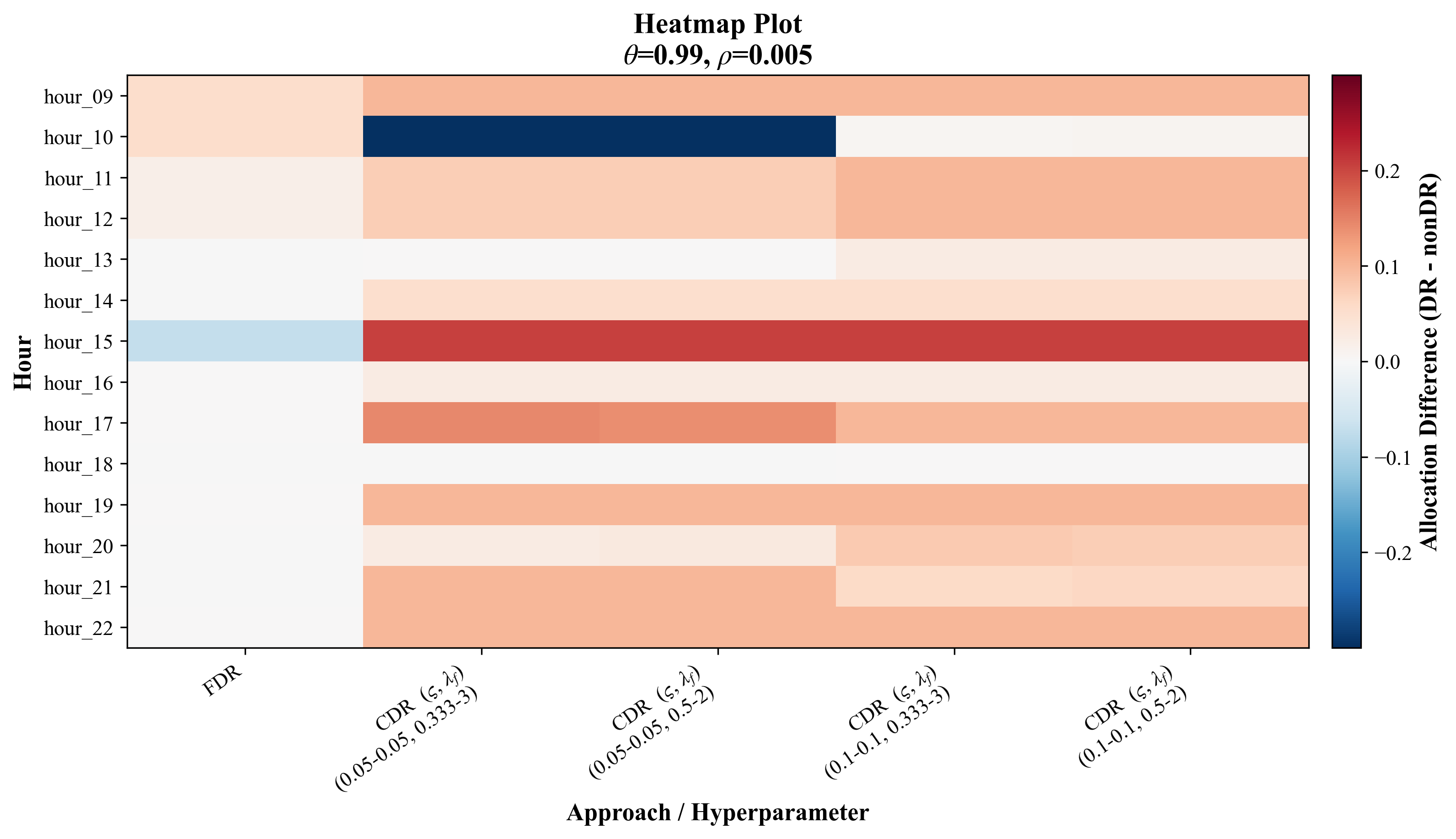}
    \label{fig:heatmap-0.99-0.005}
\end{subfigure}\hfill
\begin{subfigure}[t]{0.345\textwidth}
    \centering
    \includegraphics[width=0.95\linewidth,height=9.5em]{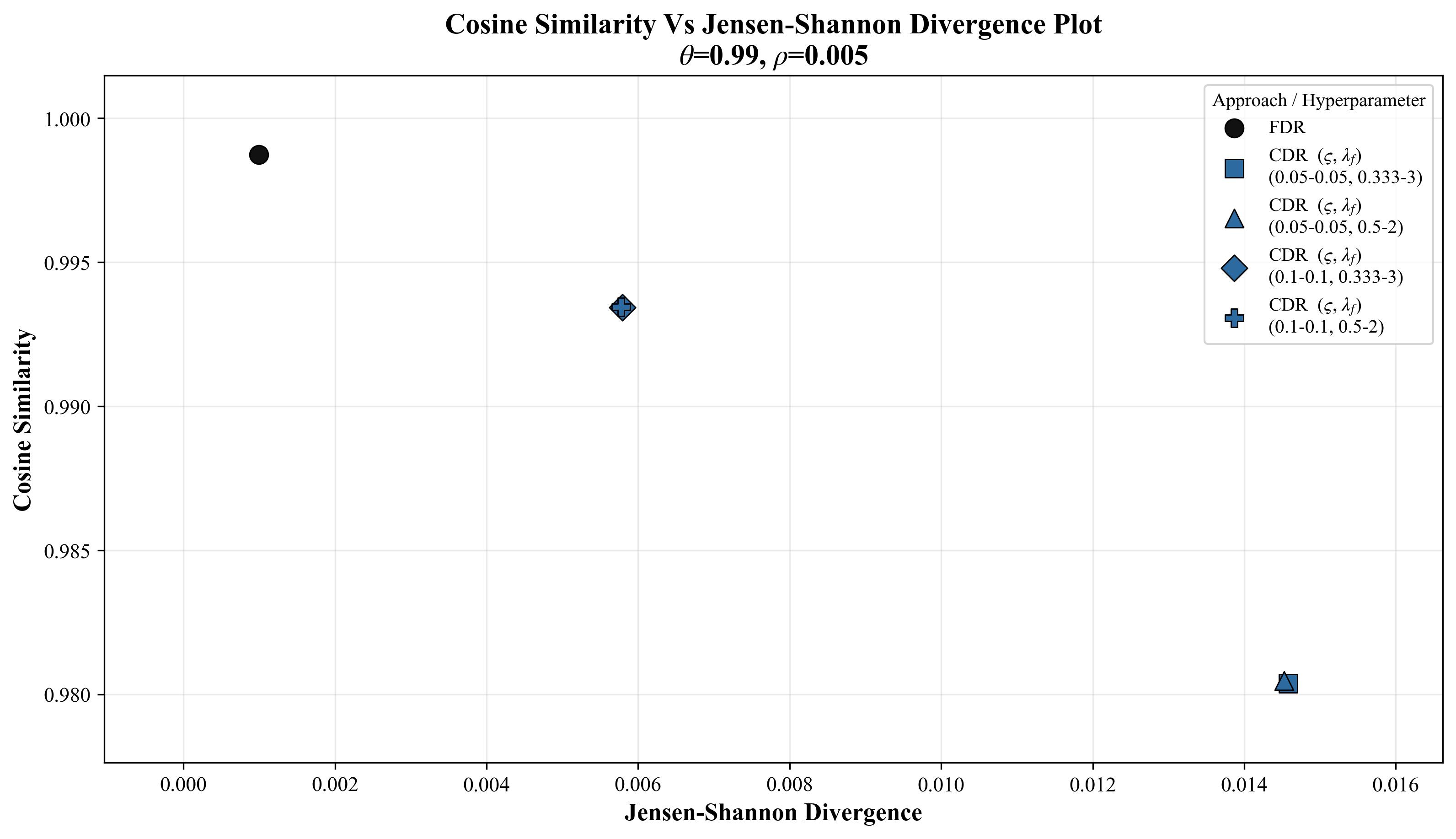}
    \label{fig:scatter-0.99-0.005}
\end{subfigure}
\vspace{-5pt}
\begin{subfigure}[t]{0.30\textwidth}
    \centering
    \includegraphics[width=0.9\linewidth,height=8em]{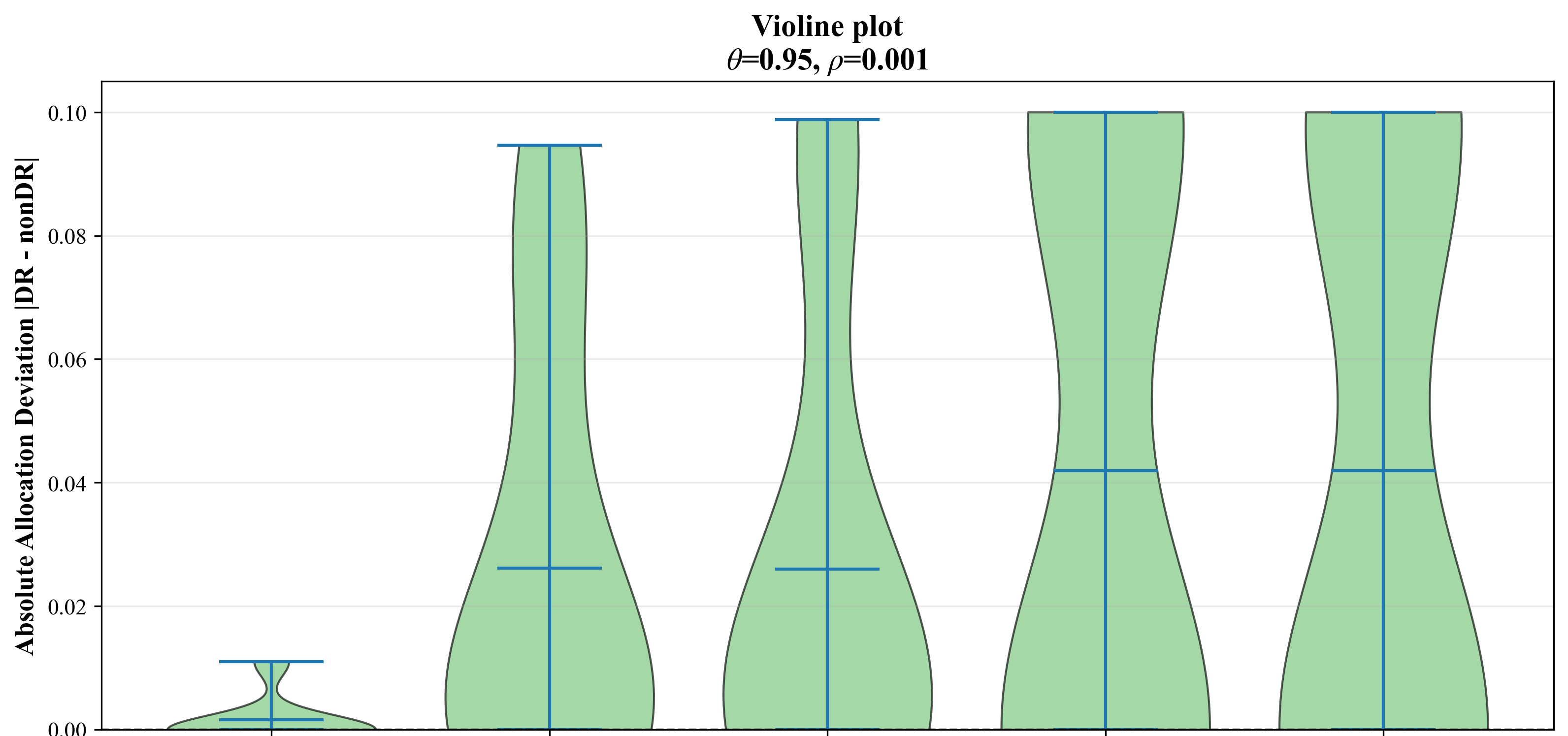}
    \label{fig:violin-0.95-0.001}
\end{subfigure}\hfill
\begin{subfigure}[t]{0.345\textwidth}
    \centering
    \includegraphics[width=\linewidth,height=8em]{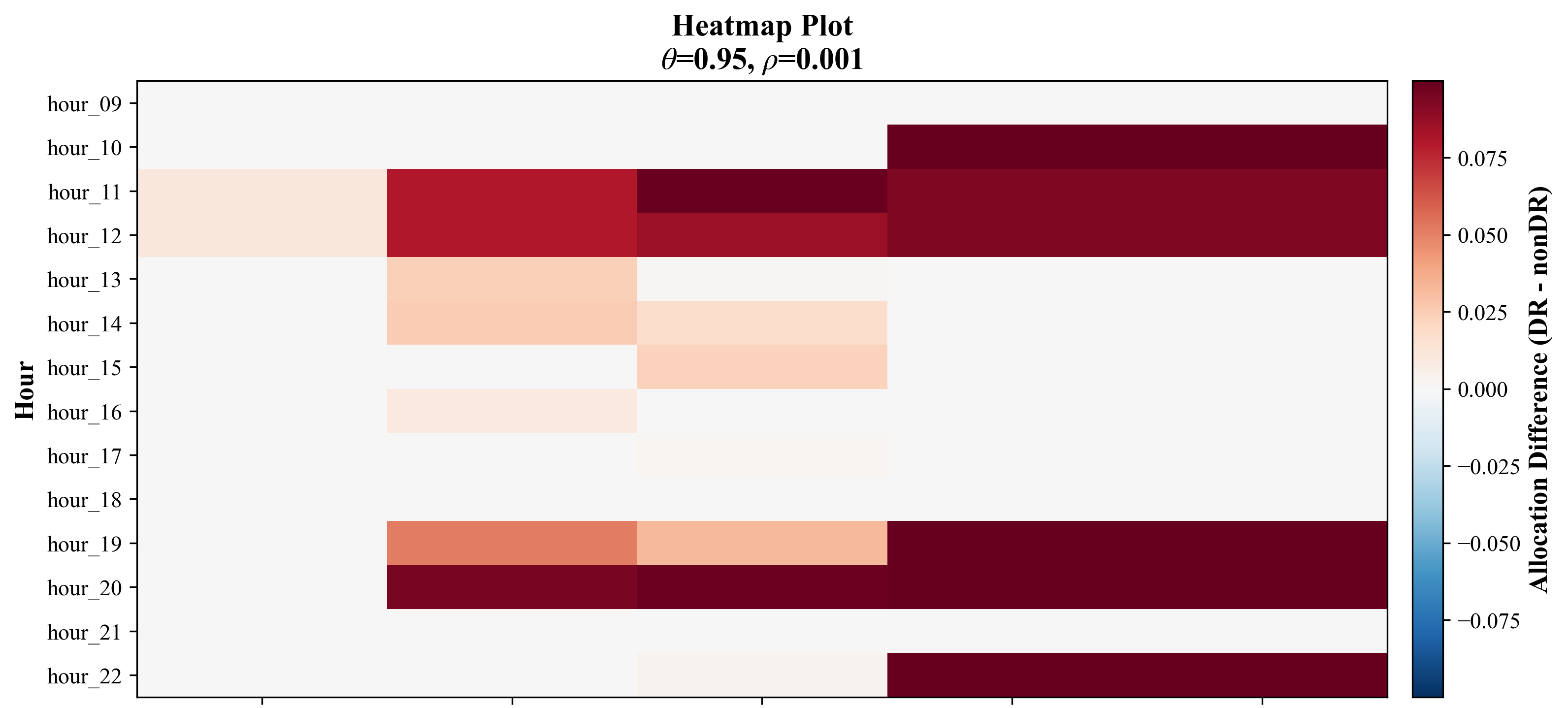}
    \label{fig:heatmap-0.95-0.001}
\end{subfigure}\hfill
\begin{subfigure}[t]{0.345\textwidth}
    \centering
    \includegraphics[width=0.95\linewidth,height=8em]{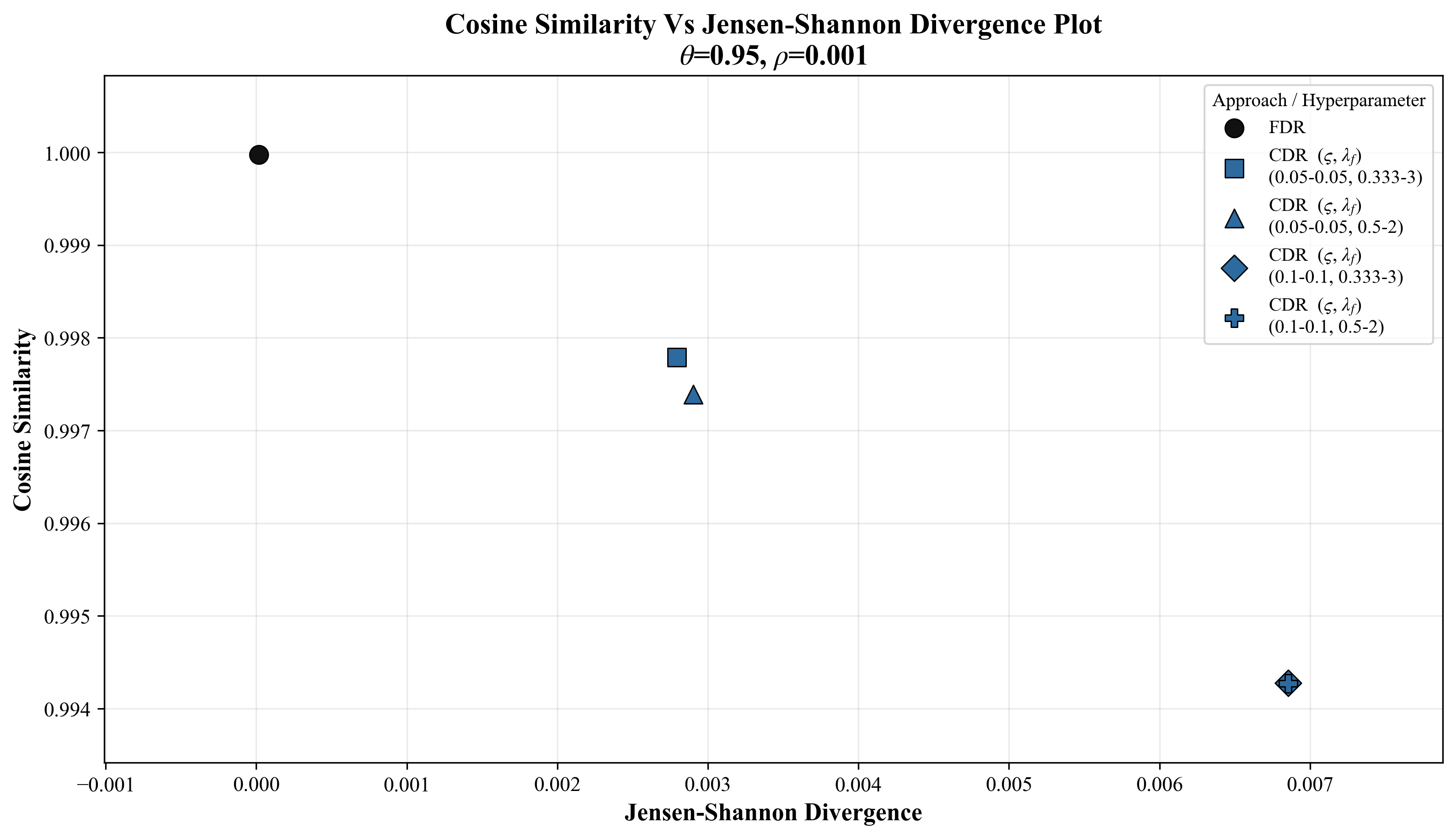}
    \label{fig:scatter-0.95-0.001}
\end{subfigure}

\vspace{-10pt}
\begin{subfigure}[t]{0.30\textwidth}
    \centering
    \includegraphics[width=0.9\linewidth,height=8em]{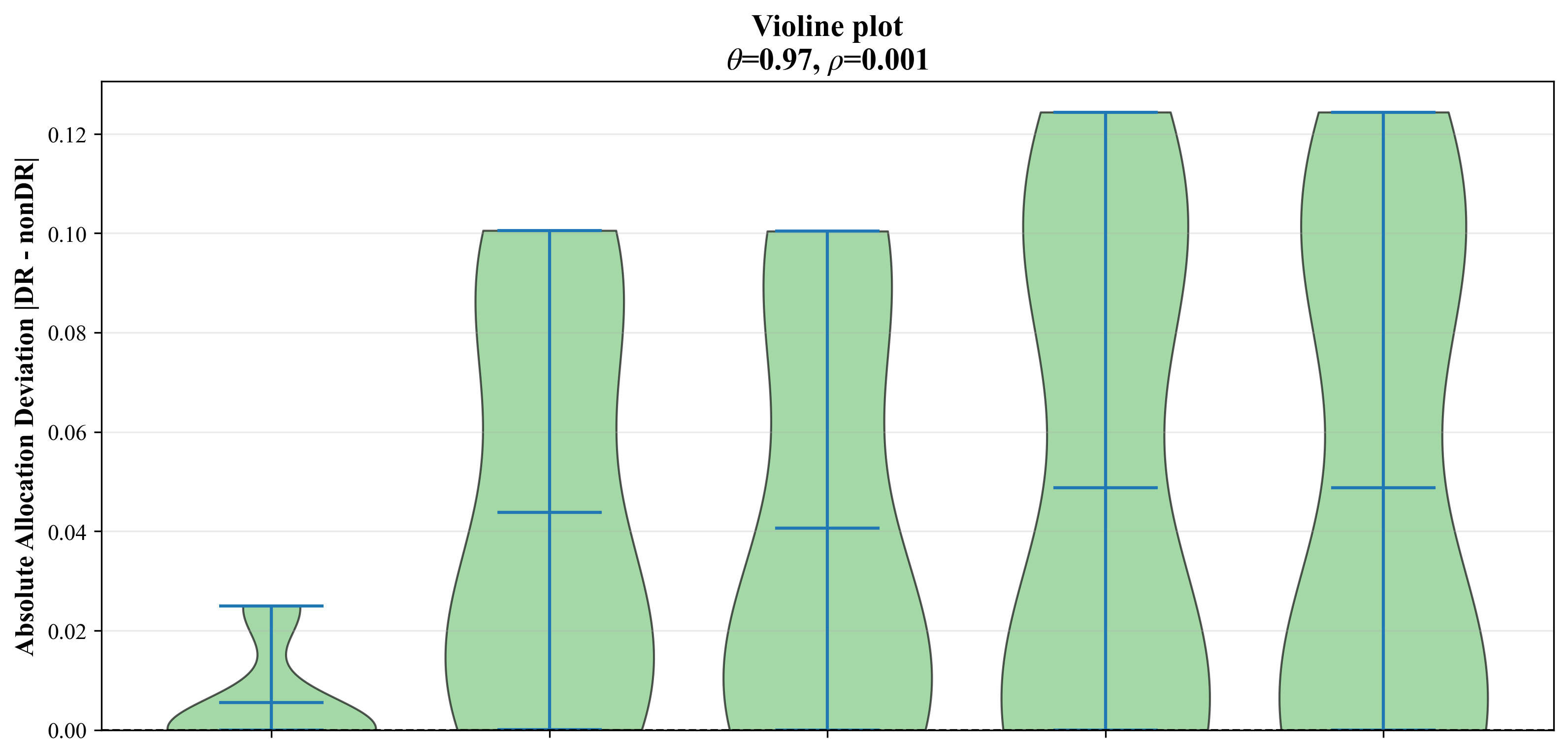}
     \label{fig:violin-0.97-0.001}
\end{subfigure}\hfill
\begin{subfigure}[t]{0.345\textwidth}
    \centering
    \includegraphics[width=\linewidth,height=8em]{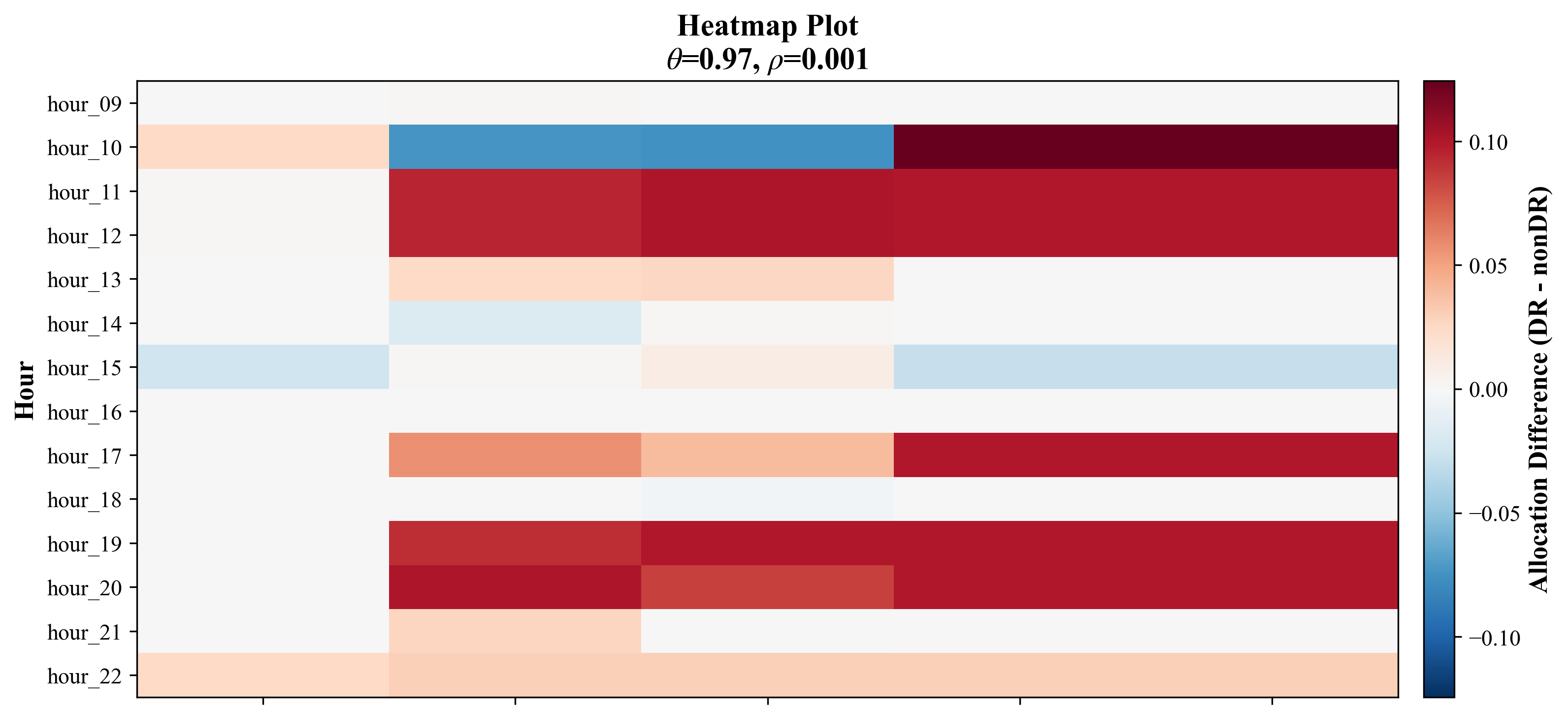}
    \label{fig:heatmap-0.97-0.001}
\end{subfigure}\hfill
\begin{subfigure}[t]{0.345\textwidth}
    \centering
    \includegraphics[width=0.95\linewidth,height=8em]{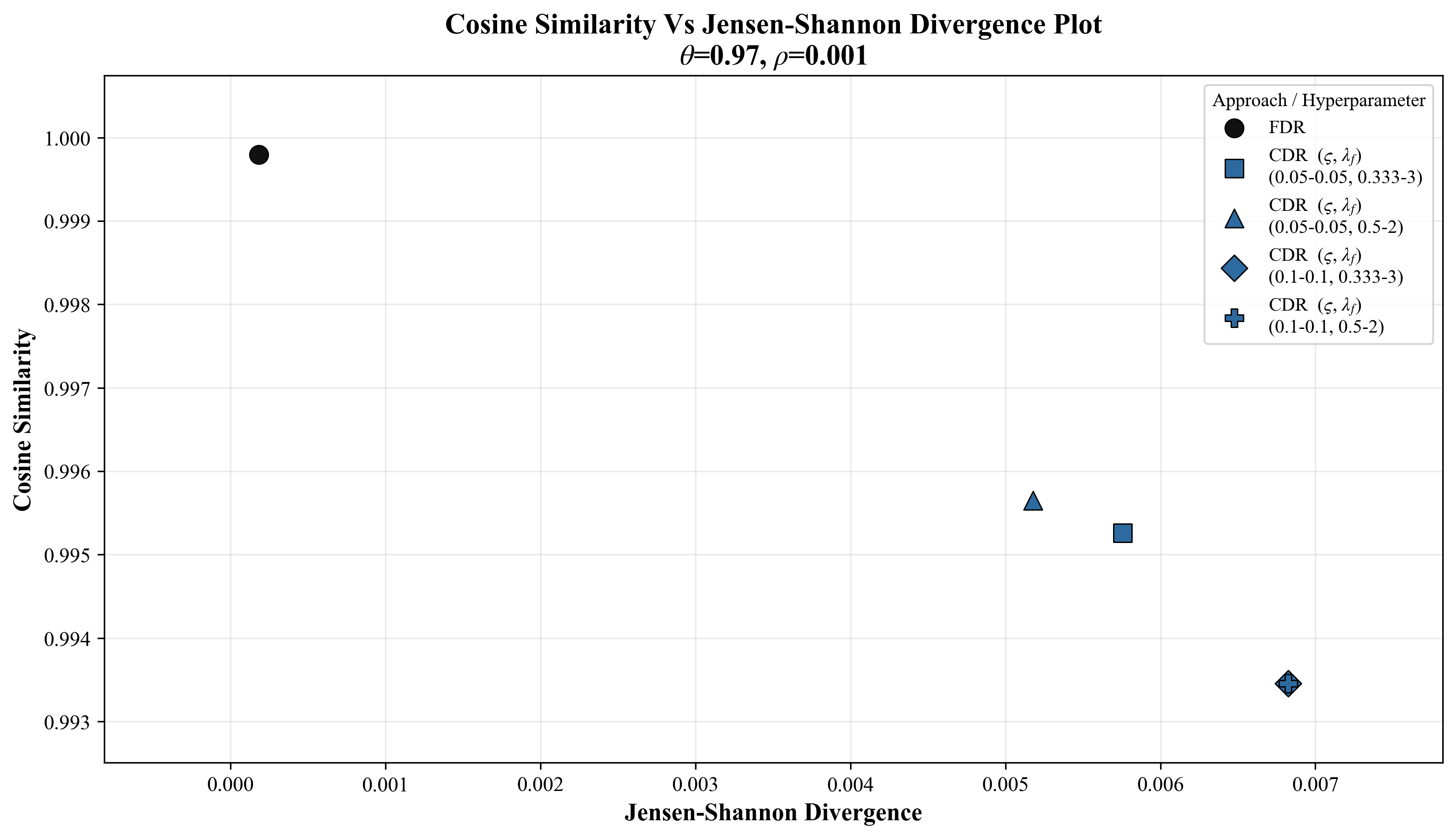}
    \label{fig:scatter-0.97-0.001}
\end{subfigure}

\vspace{-10pt}
\begin{subfigure}[t]{0.30\textwidth}
    \centering
    \includegraphics[width=0.9\linewidth,height=9.5em]{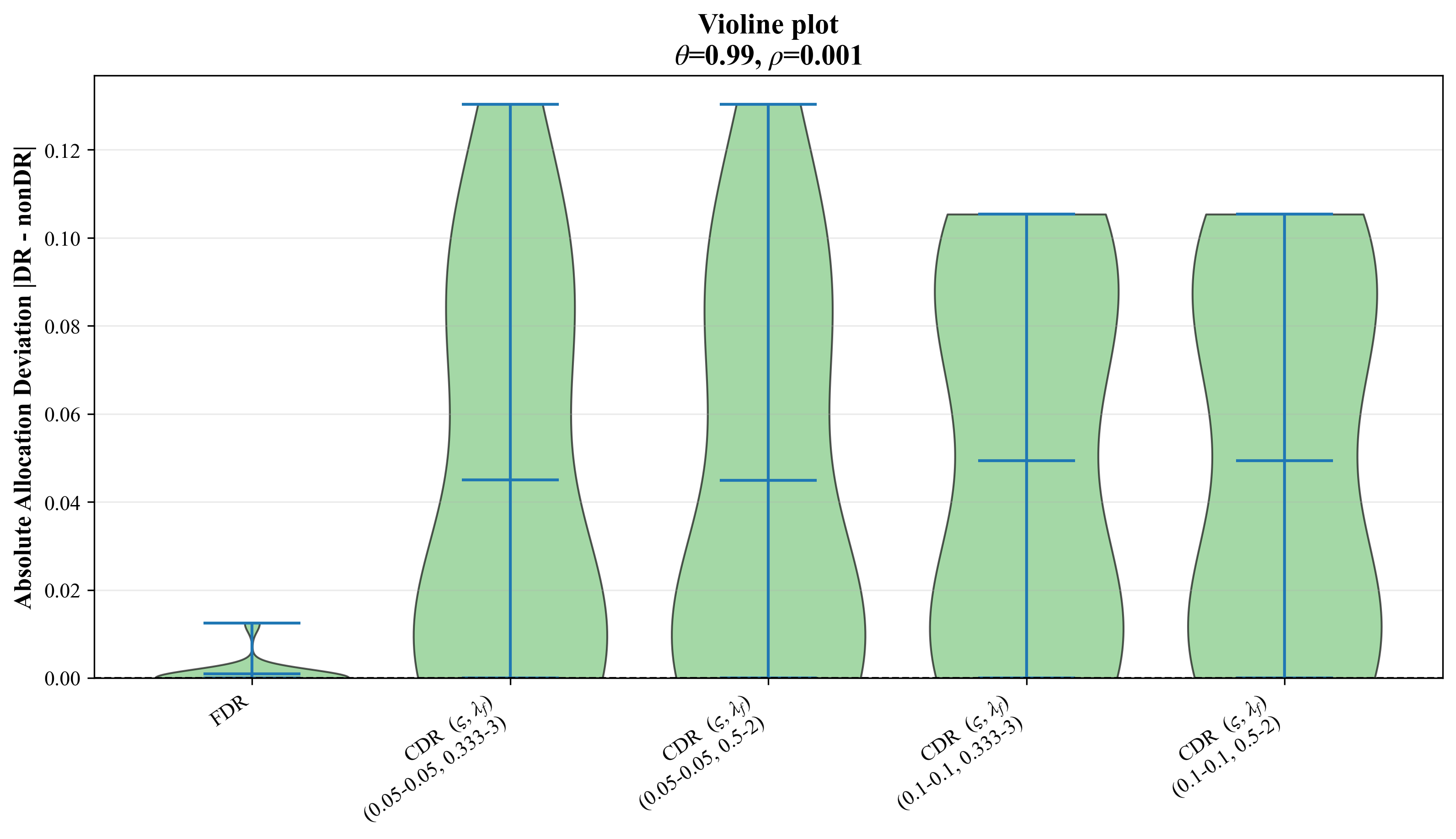}
     \label{fig:violin-0.99-0.001}
\end{subfigure}\hfill
\begin{subfigure}[t]{0.345\textwidth}
    \centering
    \includegraphics[width=\linewidth,height=9.5em]{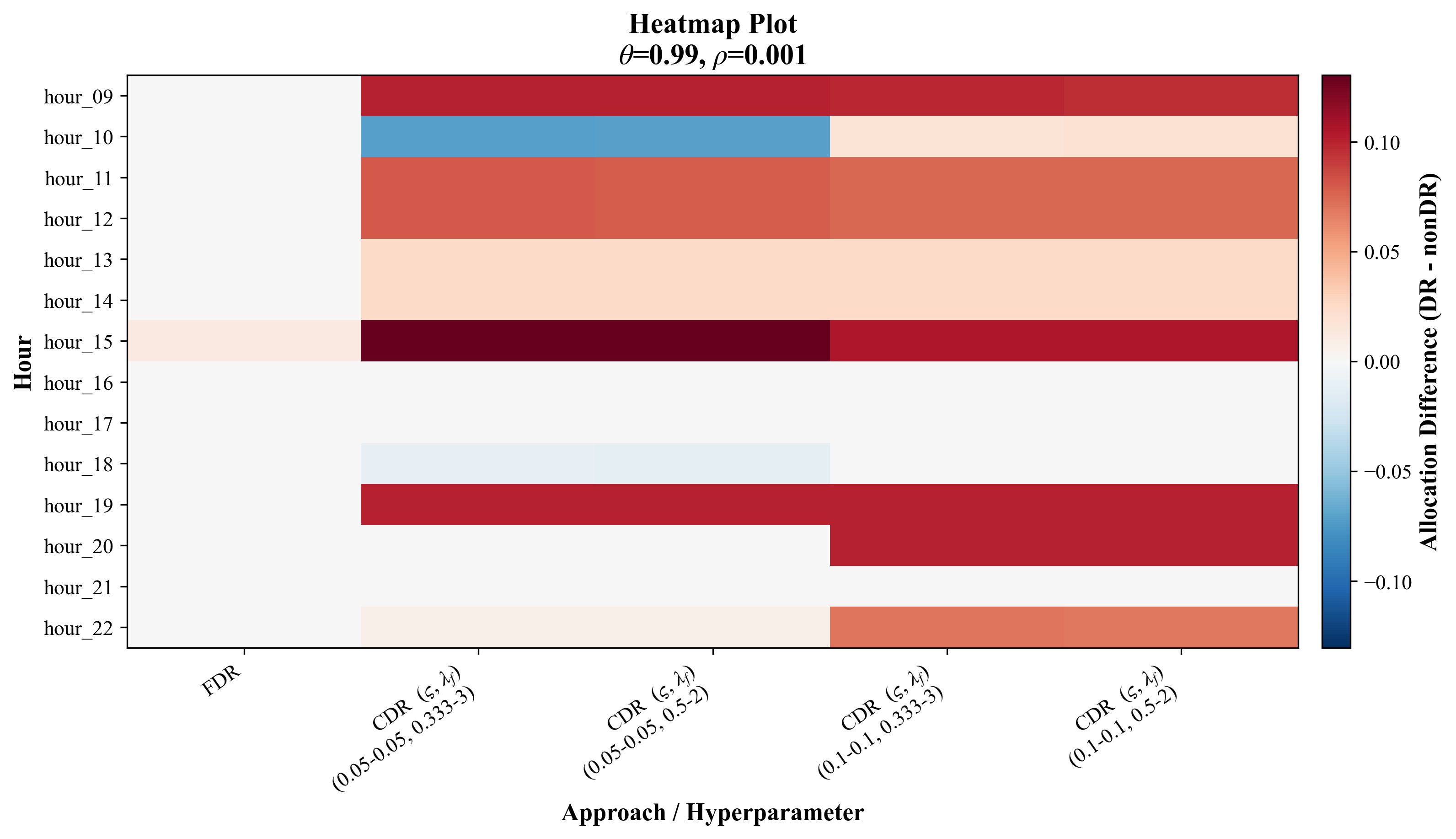}
    \label{fig:heatmap-0.99-0.001}
\end{subfigure}\hfill
\begin{subfigure}[t]{0.345\textwidth}
    \centering
    \includegraphics[width=0.95\linewidth,height=9.5em]{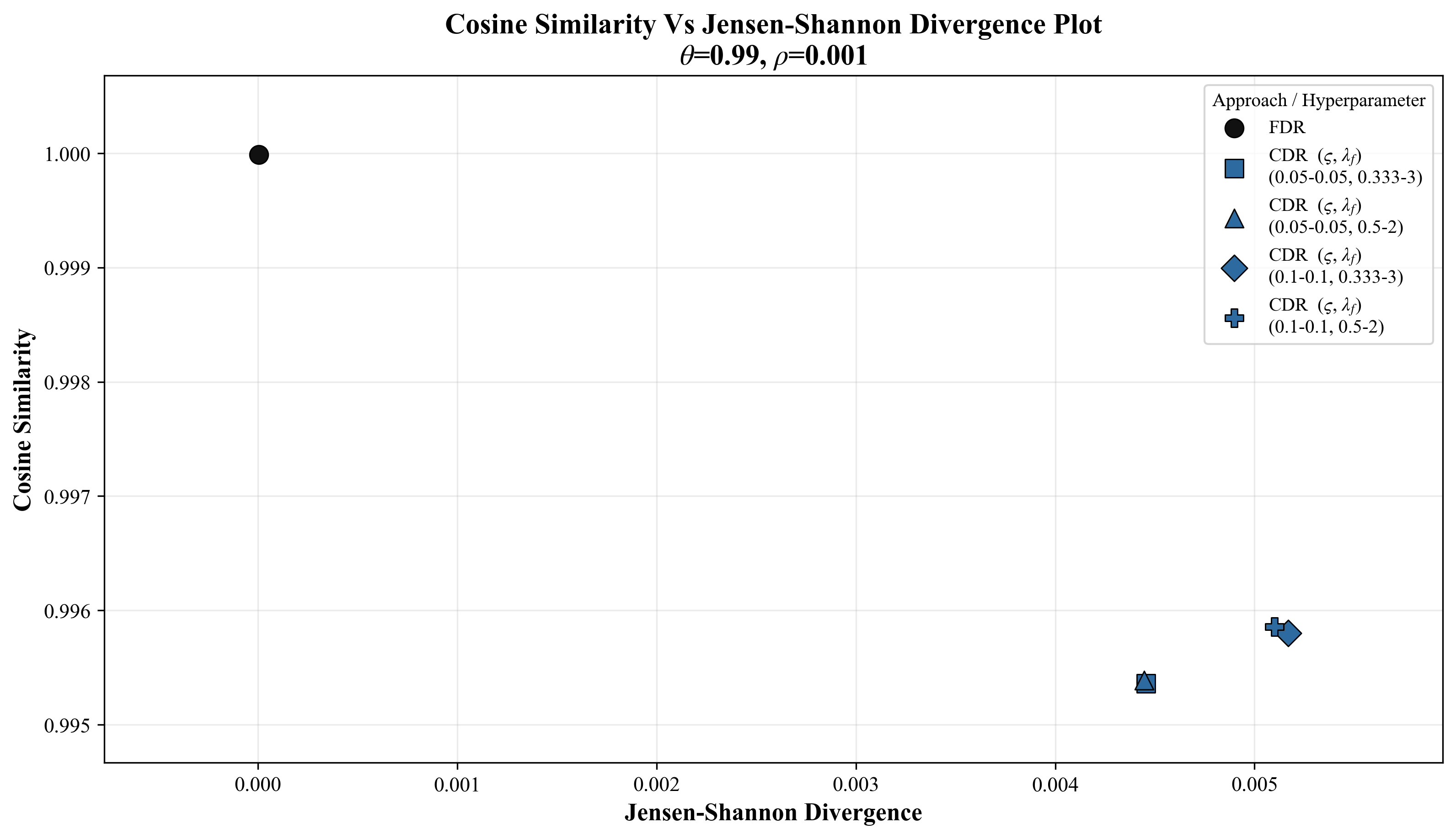}
    \label{fig:scatter-0.99-0.001}
\end{subfigure}
\vspace{-1em}
\caption{\scriptsize 
Comparison of finite- and continuous-support solutions for $(\varsigma,\lambda_{\min})\in\{(0.05,0.333),(0.05,0.5),(0.1,0.333),(0.1,0.5)\}$. The upper and lower row blocks correspond to $\rho=0.005$ and $\rho=0.001$, respectively, with rows in each block ordered from top to bottom by $\theta=0.95,0.97,0.99$. In each row, the left, middle, and right panels show, respectively, the violin plot of absolute deviations from the nominal solution, the heatmap of allocation differences relative to the nominal solution, and the cosine-similarity--Jensen--Shannon-divergence scatter plot.}
\label{fig:combined_three_rows-0.001}
\end{figure}

\end{APPENDICES}

\end{document}